\documentclass[11pt]{article}
\pdfoutput=1

\usepackage[utf8]{inputenc}
\usepackage[T1]{fontenc}

\usepackage[verbose=true,
letterpaper,
textheight=8.4in,
textwidth=6.1in,
margin=1in,
headheight=12pt,
headsep=25pt,
footskip=30pt]{geometry}

\usepackage[toc,page,header]{appendix}
\usepackage{titletoc}

\usepackage{microtype}
\usepackage{mathtools}
\usepackage{booktabs}
\usepackage[round]{natbib}
\usepackage{graphicx}
\usepackage{amsmath,amssymb,amsfonts,amsxtra,bm}
\usepackage{amsthm}
\usepackage{xspace}
\usepackage{paralist}
\usepackage{enumerate}
\usepackage{enumitem}
\usepackage{hyperref}
\usepackage{url}
\usepackage{colortbl}
\usepackage{makecell,multirow}       
\usepackage{nicefrac}       
\usepackage{thmtools}  
\usepackage{xspace}
\usepackage{footnote, tablefootnote}
\usepackage{float}
\floatstyle{plaintop}
\restylefloat{table}
\usepackage{epstopdf}
\usepackage{latexsym}
\usepackage{graphicx}
\usepackage{bbm}

\titlecontents{section}[3em]{\vspace{-1pt}}%
    {\bfseries\contentslabel{2em}}
    {}
    {\titlerule*[0.5pc]{.}\contentspage}%
    
\titlecontents{subsection}[5em]{\vspace{-1pt}}%
    {\contentslabel{2em}}
    {}
    {\titlerule*[0.5pc]{.}\contentspage}%

\titlecontents{subsubsection}[5em]{\vspace{-1pt}}%
    {\footnotesize\contentslabel{3em}}
    {}
    {\titlerule*[0.5pc]{.}\contentspage}%

\numberwithin{equation}{section}
\usepackage{ss}
\usepackage{tikz}

\newtheorem{theorem}{Theorem}
\newtheorem{lemma}{Lemma}

\newtheorem{assumption}{Assumption}
\newtheorem{corollary}[lemma]{Corollary}

\theoremstyle{definition}
\newtheorem{definition}{Definition}
\newtheorem{remark}{Remark}

\newcommand*\at[2]{\left.#1\right|_{#2}}

\newcommand*\del[0]{\partial}

\newcommand*\Z[0]{\mathbb{Z}}

\newcommand*\ddt[0]{\frac{d}{d t}}

\newcommand*\tr[0]{\text{tr}}
\newcommand*\KL[2]{\mathcal{KL}\left(#1\|#2\right)}
\newcommand*\lin[1]{\bm{\left\langle} #1 \bm{\right\rangle}}
\newcommand*\linp[1]{\bm{\langle} #1 \bm{\rangle}}
\newcommand*\linb[1]{\bm{\left\langle} #1 \bm{\right\rangle}_2}

\newcommand*\E[1]{\mathbb{E}\left[{#1}\right]}
\newcommand*\Ep[2]{\mathbb{E}_{{#1}}\left[{#2}\right]}

\renewcommand*\H[0]{{\mathbf{H}}}

\renewcommand*\t[1]{\tilde{#1}}

\newcommand*\h[1]{\hat{#1}}

\newcommand*\N[0]{\mathcal{N}}
\renewcommand*\S[0]{\mathcal{S}}
\newcommand*\Sym[0]{\mathbb{S}}

\newcommand\numberthis{\addtocounter{equation}{1}\tag{\theequation}}
\newcommand*\lrb[1]{{\left[#1\right]}}
\newcommand*\lrbb[1]{\left\{#1\right\}}
\newcommand*\lrp[1]{{(#1)}}
\newcommand*\lrn[1]{{\left\|#1\right\|}}
\newcommand*\lrabs[1]{{\left|#1\right|}}

\newcommand*\cvec[2]{\begin{bmatrix} #1\\#2\end{bmatrix}}

\newcommand*\bmat[1]{\begin{bmatrix} #1 \end{bmatrix}}

\newcommand*\ind[1]{{\mathbbm{1}\lrbb{#1}}}

\renewcommand*{\Re}{\mathbb{R}}

\newcommand*{\cA}{\mathcal{A}}

\newcommand*{\R}{\mathcal{R}}

\newcommand*\circled[1]{\tikz[baseline=(char.base)]{
\node[shape=circle,draw,inner sep=1pt] (char) {#1};}}

\newcommand*{\nablab}{{\bm{\nabla}}}
\newcommand*{\twocase}[4]{\left\{\begin{array}{ll}
        #1 & \text{for } #2\\
        #3 & \text{for } #4
        \end{array}\right.}
        
\newcommand*{\threecase}[6]{\left\{\begin{array}{ll}
        #1, & \text{for } #2\\
        #3, & \text{for } #4\\
        #5, & \text{for } #6
        \end{array}\right.}

\renewcommand*{\Pr}[1]{\mathbb{P}\lrp{{#1}}}

\renewcommand{\L}{\mathcal{L}}

\renewcommand{\epsilon}{\varepsilon}

\newcommand{\JJ}{\mathbf{J}}
\newcommand{\KK}{\mathbf{K}}
\newcommand{\RR}{\mathbf{R}}

\newcommand{\MM}{\mathbf{M}}
\newcommand{\NN}{\mathbf{N}}

\newcommand{\FF}{\mathbf{F}}
\newcommand{\GG}{\mathbf{G}}
\newcommand{\HH}{\mathbf{H}}

\newcommand{\xx}{\mathbf{x}}

\newcommand{\yy}{\mathbf{y}}

\renewcommand{\gg}{\mathbf{g}}

\newcommand{\nnu}{\bm{\nu}}
\newcommand{\hh}{\mathbf{h}}
\renewcommand{\AA}{\mathbf{A}}
\newcommand{\BB}{\mathbf{B}}
\newcommand{\WW}{\mathbf{W}}
\newcommand{\CC}{\mathbf{C}}
\newcommand{\DD}{\mathbf{D}}
\renewcommand{\aa}{\mathbf{a}}

\newcommand{\cc}{\mathbf{c}}
\newcommand{\uu}{\mathbf{u}}
\newcommand{\vv}{\mathbf{v}}

\newcommand{\ww}{\mathbf{w}}
\newcommand{\pp}{\mathbf{p}}

\newcommand{\zz}{\mathbf{z}}

\newcommand{\yt}{{\tilde{y}}}
\newcommand{\zt}{{\tilde{z}}}

\renewcommand{\u}{\mathbf{u}}
\newcommand{\xib}{{\boldsymbol\xi}}

\newcommand{\zb}{{\boldsymbol\zeta}}
\newcommand{\bbeta}{{\boldsymbol\beta}}

\newcommand{\pt}{{\tilde{p}}}
\newcommand{\G}{{\mathcal{G}}}

\newcommand{\F}{{\mathcal{F}}}
\newcommand{\C}{{\mathcal{C}}}
\renewcommand{\a}{{a}}

\newcommand{\elb}[1]{\numberthis \label{#1}}

\renewcommand\r[0]{{r}}
\newcommand\s[0]{{s}}

\DeclareMathOperator{\tc}{{\zeta}}

\newcommand{\dist}{\mathrm{d}}

\DeclareMathOperator{\Exp}{Exp}
\DeclareMathOperator{\Ric}{Ric}

\DeclareMathOperator{\emat}{\mathbf{exp_{mat}}}

\newcommand{\cheng}[1]{\noindent{\textcolor{red}{\{{\textbf{Cheng:}} \em #1\}}}}

\newcommand*{\party}[2]{\Gamma_{#1}^{#2}}

\definecolor{cdarkred}{rgb}{0.55,0.0,0.0}
\definecolor{darkblue}{rgb}{0.0,0.0,0.55}
\definecolor{cgray}{gray}{0.55}
\definecolor{cdarkblue}{RGB}{30,30,200}

\newcommand{\darkred}[1]{{\color{cdarkred} #1}}

\newcommand{\parhead}[1]{{\footnotesize$\color{cgray}{\blacksquare}$}~\darkred{\bfseries #1}}
\newcommand{\sparhead}[1]{{\tiny$\color{cgray}{\blacksquare}$}~\darkred{\bfseries #1}}

\begin{document}

\title{Theory and Algorithms for Diffusion Processes on Riemannian Manifolds}

\author{\name{Xiang Cheng} \email{chengx@mit.edu}\\
   \addr{Massachusetts Institute of Technology}\\[2pt]
   \name{Jingzhao Zhang}\thanks{Part of this work was done when JZ was with MIT.} \email{jzhzhang@mit.edu}\\
   \addr{Tsinghua University, Beijing}\\[2pt]
   \name{Suvrit Sra} \email{suvrit@mit.edu}\\
   \addr{Massachusetts Institute of Technology}
}

\maketitle

\begin{abstract}
  We study geometric stochastic differential equations (SDEs) and their approximations on Riemannian manifolds. In particular, we introduce a simple new construction of geometric SDEs, using which we provide the first (to our knowledge) non-asymptotic bound on the error of the geometric Euler-Murayama discretization. We then bound the distance between the exact SDE and a discrete geometric random walk, where the noise can be non-Gaussian; this analysis is useful for using geometric SDEs to model naturally occurring discrete non-Gaussian stochastic processes. Our results provide convenient new tools for studying MCMC algorithms that adopt  non-standard noise distributions.

\end{abstract}

\vspace*{-6pt}
\section{Introduction}
\lettrine[lines=3]{\color{BrickRed}S}tochastic differential equations (SDEs) offer a powerful formalism for studying diffusion processes, Brownian motion, and algorithms for sampling and optimization. We study the following \emph{geometric stochastic differential equation}:
\begin{equation}
    d x(t) = \beta(x(t)) dt + dB^g_t,
    \label{e:intro_sde}
\end{equation}
that evolves on a Riemannian manifold $(M, g)$. Similar to its Euclidean counterpart, here too $\beta: x \to T_xM$ is a drift ($T_xM$ is the tangent space at $x\in M$); while $dB^g_t$ denotes standard Brownian motion on $M$~\citep[Ch.~3]{hsu2002stochastic}. The notation in~\eqref{e:intro_sde} is a convenient shorthand; for a more precise version see~\citep[Thm.~1.3.6]{hsu2002stochastic}).

Geometric SDEs such as~\eqref{e:intro_sde} play a crucial role in the design and analysis of MCMC algorithms \citep{girolami2011riemann,patterson2013stochastic} that have had much success in solving Bayesian problems on statistical manifolds. These SDEs also directly relate to the tasks of sampling and optimization on manifolds, where often a Lie group structure helps capture symmetries (e.g., the Grassmann manifold, $\text{SO}(n)$, $O(n)$, etc.)~\citep{moitra2020fast,piggott2016geometric}. Moreover, in some settings  MCMC algorithms with respect to carefully chosen metrics can be faster than their Euclidean counterparts \citep{lee2018convergence,chewi2020exponential}.

A basic scheme for simulating the SDE~\eqref{e:intro_sde} is \emph{Geometric Euler-Murayama}~\citep{piggott2016geometric,muniz2021higher}. This method performs the manifold-valued iteration:
\begin{equation}
    x_{k+1} = \Exp_{x_k}\bigl(\delta \beta(x_k) + \sqrt{\delta} \zeta_k\bigr),
    \label{e:intro_euler_murayama}
\end{equation}
where $\Exp_x$ denotes the exponential map for $M$, and $\zeta_k$ is a standard Gaussian with respect to \emph{any} orthonormal basis of $T_{x_k}M$. It has long been known that the discrete-time process~\eqref{e:intro_euler_murayama} converges to the SDE \eqref{e:intro_sde} in the limit~\citep{gangolli1964construction,jorgensen1975central,hsu2002stochastic}.

While several authors~\citep{piggott2016geometric,muniz2021higher,li2020riemannian,moitra2020fast} have analyzed discretization of \eqref{e:intro_sde} in specific settings, much less is known about the \emph{nonasymptotic} error between the SDE~\eqref{e:intro_sde} and its Euler-Murayama discretization on general Riemannian manifolds. Moreover, equally important is a study of other discretizations that approximate~\eqref{e:intro_sde} with non-Gaussian position-dependent noise, as these other discretizations model a richer family of stochastic processes, and may be computationally more advantageous than \eqref{e:intro_euler_murayama}. %

\subsection*{Main Goals and Contributions of This Paper}
\noindent\parhead{Goal I.} Our first goal is to quantify the discretization error of the Geometric Euler-Murayama scheme~\eqref{e:intro_euler_murayama}, in terms of the stepsize $\delta$ and intrinsic quantities such as the curvature of $M$. Unlike Brownian motion on $\Re^d$, it is not always clear how one might sample from the endpoint of a geometric Brownian motion. (On certain special manifolds, exact sampling is possible, e.g. \cite{li2020riemannian}). Our discretization analysis in Lemma \ref{l:discretization-approximation-lipschitz-derivative} guarantees that  \eqref{e:intro_euler_murayama}, which is amenable to efficient computation, is a good approximation to the SDE~\eqref{e:intro_sde}. 

\vskip4pt
\noindent\sparhead{\small Main theoretical contributions associated with Goal I.}
\begin{list}{–}{\leftmargin=2em}
  \vspace{-5pt}
  \setlength{\itemsep}{-1pt}
\item We first construct the exact SDE~\eqref{e:intro_sde} as the limit of a family of increasingly finer Euler-Murayama sequences. These sequences are carefully coupled by our construction which shows that within a short time the trajectory of \eqref{e:intro_euler_murayama} is roughly parallel to \eqref{e:intro_sde}. A significant challenge here lies in analyzing discretizations of geometric Brownian motion and controlling the rotation of orthonormal frames as they get parallel transported along the diffusion path.  
\item Next, using this construction in Lemma \ref{l:discretization-approximation-lipschitz-derivative},  we show that the error between a single $\delta$-step of Euler-Murayama~\eqref{e:intro_euler_murayama} and \eqref{e:intro_sde} evolved over a period of $\delta$ is $O\lrp{\delta^{3/2}}$; this error bound matches the one-step error of Euler-Murayama in $\Re^d$. Importantly, our analysis does not rely on properties of a particular embedding, and hence our error bounds do not depend on extrinsic quantities such as Christoffel symbols. 
\item Finally, in Theorem \ref{t:langevin_mcmc} we bound the $W_1$ distance (with respect to the Riemannian distance on $M$) between \eqref{e:intro_sde} and the Euler-Murayama scheme, for all time $t$. To sample from a distribution $\epsilon$-close to the invariant distribution of~\eqref{e:intro_sde}, Theorem \ref{t:langevin_mcmc} requires $\t{O}(\epsilon^{-2})$ steps of \eqref{e:intro_euler_murayama}, matching the best known iteration complexity of Langevin MCMC in $\Re^d$.
\end{list}


\vskip5pt
\noindent\parhead{Goal II.} Our second goal is to extend the previous error bound to the stochastic process
\begin{equation}
  x_{k+1} = \Exp_{x_k}\bigl(\delta \beta(x_k) + \sqrt{\delta} \xi_k(x_k)\bigr),
  \label{e:intro:discrete_manifold}
\end{equation}
where $\xi_k$ are i.i.d.\ samples of a random vector field $\xi$. In contrast to~\eqref{e:intro_euler_murayama}, process~\eqref{e:intro:discrete_manifold} permits the noise to be non-Gaussian and position dependent. We are interested in cases when $\xi(x)$ is zero-mean with an identity covariance, but its distribution is not Gaussian and might not even be invariant under parallel transport. Our goal is to understand the process~\eqref{e:intro:discrete_manifold}, and how it approximates the SDE~\eqref{e:intro_sde} (see Theorem~\ref{t:main_nongaussian_theorem}). The main motivation behind this goal is that \eqref{e:intro:discrete_manifold} can capture certain naturally occuring stochastic processes; one example is the recent interest in modelling SGD noise as Brownian motion over the manifold embedded in $\Re^d$ whose metric tensor is given by the inverse of the covariance matrix of the stochastic gradient \citep{li2021validity,cheng2020stochastic,pesme2021implicit}. Moreover, \eqref{e:intro:discrete_manifold} can be computationally better than the Geometric Euler-Murayama scheme---e.g., when computing the metric tensor is expensive but generating samples with the right covariance is easy.

\vskip4pt
\noindent\sparhead{\small Main theoretical contributions associated with Goal II.}
\begin{list}{–}{\leftmargin=2em}
  \vspace{-2pt}
  \setlength{\itemsep}{-1pt}
\item We provide a bound on the $W_1$ distance between $K$ steps of \eqref{e:intro:discrete_manifold} and SDE~\eqref{e:intro_sde} taken over a time interval of $K\delta$ (Lemma~\ref{l:theorem_2_corollary}). Specifically, if $x(t)$ is as in~\eqref{e:intro_sde} and $\{x_k\}_k$ is the discrete sequence obtained via~\eqref{e:intro:discrete_manifold}, then Lemma~\ref{l:theorem_2_corollary} shows that for sufficiently large $K$, $\E{\dist(x(K\delta),x_K)} \leq \tilde{O}(\delta^{1/6})$. 
\item Our proof approximates $K$ steps of \eqref{e:intro:discrete_manifold} by a $K$-step-random-walk in $T_{x_0}M$. Theorem \ref{t:main_nongaussian_theorem} bounds the error due to this approximation. The main idea in analyzing this error is to posit a suitable second-order ODE in $T_{x_0} M$ arising from the Jacobi Equation (Lemma~\ref{n:l:w(s)}). This lemma may be useful for analyzing so-called ``trivializations'' of other stochastic algorithms.  

\item Enroute, we also prove an extension of the classic quantitative CLT; this result bounds the distance between $\frac{1}{\sqrt{n}}\sum_{i=1}^n \xi_i$ and $\N(0,I)$, and allows $\xi_i$ to depend on $\xi_0,\ldots,\xi_{i-1}$.
\end{list}

\vspace*{-6pt}
\section{Preliminaries: notation and key assumptions}
We assume some background in Riemannian geometry, and freely use standard notation; we refer the reader to~\citep{jost2008riemannian,lee2006riemannian,petersen2006riemannian} for an in depth treatment. Readers may also find some works on Riemannian optimization useful as additional context: \citep{bacak2014convex,udriste2013convex,absil2009optimization,zhang2016first,boumal2022intro}.

We use capital letters to denote an ordered orthonormal basis at some tangent space, e.g., $F$ is a basis of $T_x M$. We use superscripts to index vectors in the basis, e.g., a basis $F$ is an ordered tuple  $(F^1,\ldots,F^d)$ of vectors in $T_x M$. We use bold lowercase letters $\vv$ to denote Euclidean vectors, and use $\vv \circ F$ as shorthand for $\sum_{i=1}^d \vv_i F^i$. A distribution that we will see frequently in this paper is the one given by $\xib \circ E^x$, where $\xib \sim \N(0,I)$ is a random vector in $\Re^d$, and $x \in M$ and $E^x$ is an orthonormal basis of $T_x M$. We use $\N_x\lrp{0,I}$ to denote the distribution of $\xib \circ E^x$. One can verify that $\N_x\lrp{0,I}$ does not depend on the choice of basis $E^x$. We use $\Pi_a(v)$ to denote projection onto the (Riemannian) $a$-ball, so that $\Pi_a(v) = v$ if $\lrn{v}\leq a$, and $\Pi_a(v) = \frac{av}{\lrn{v}}$, otherwise. We use $\dist(\cdot,\cdot)$ to denote the Riemannian distance on $M$. We use $\nabla$ to denote the Levi Civita connection. Given $x,y\in M$ and $v\in T_x M$, we use $\party{x}{y} v$ to denote parallel transport of $v$ from $x$ to $y$. Given a general curve $\gamma:[0,1] \to M$, we will also use $\party{\gamma(t)}{} v$ to denote the parallel transport of $v$ from $\gamma(0)$ to $\gamma(1)$, along $\gamma$.


Our first assumption is a natural generalization of the dissipativity condition in the Euclidean setting. It assumes that the drift traps the variable within a bounded region.
\begin{assumption}
    \label{ass:distant-dissipativity}
    We call a vector field $\beta$ $(m,q,\R)$-\emph{distant-dissipative} if there exist constants $\R,m,q \geq 0$ such that, for all $x,y$ satisfying $\dist\lrp{x,y} \geq \R$, there exists a minimizing geodesic $\gamma: [0,1] \to M$ with $\gamma(0) = x$ and $\gamma(1) = y$, such that (here $\Gamma_y^x$ denotes parallel transport from $T_yM$ to $T_xM$ along $\gamma$) the inequality
    \begin{alignat*}{1}
      \linp{\party{y}{x}{\beta(y)-\beta(x)}, \gamma'(0)} \leq -m \dist\lrp{x,y}^2,
    \end{alignat*}
    holds, and for all $x,y$ satisfying $\dist(x,y) \leq \R$, there exists a minimizing geodesic $\gamma: [0,1] \to M$ with $\gamma(0) = x$ and $\gamma(1) = y$, such that we have instead the inequality
    \begin{alignat*}{1}
        \linp{\party{y}{x}{\beta(y)-\beta(x)}, \gamma'(0)} \leq q \dist\lrp{x,y}^2.
    \end{alignat*}
\end{assumption}

We then bound the smoothness of the drift and of the Riemannian curvature, before we can carry out the promised nonasymptotic analysis.
\begin{assumption}\label{ass:beta_lipschitz}
  A vector field $\beta$ is $L_{\beta}$-Lipschitz if, for all $x\in M$ and all $v\in T_x M$, $\lrn{\nabla_v \beta(x)} \leq L_\beta\lrn{v}$.
\end{assumption}

\begin{assumption}\label{ass:sectional_curvature_regularity}
  We assume that the manifold $M$ has Riemannian curvature tensor $R$ that satisfies for all $x\in M$, and for all $u,v,w,z\in T_x M$, the bound $\lin{R(u,v)w,z} \leq L_R\lrn{u}\lrn{v}\lrn{w}\lrn{z}$ for some $L_R \ge 0$.
\end{assumption}

\begin{assumption}\label{ass:ricci_curvature_regularity}
    We assume that for all $x\in M$, $u,u\in T_x M$, $\Ric(u,u) \geq - L_{\Ric}$, for some $L_{\Ric} \ge 0$.
\end{assumption}

\noindent Although Assumption \ref{ass:ricci_curvature_regularity} is implied by Assumption \ref{ass:sectional_curvature_regularity} with $L_{\Ric} = d\cdot L_R$, we state them separately as sometimes $L_{\Ric} \ll d\cdot L_R$.


\section{A Construction for SDEs on Manifolds and Langevin MCMC}
\label{Brownian Motion Section}
In this section, we bound the error incurred by the Euler-Murayama discretization. To that end, we  introduce a different construction of geometric SDEs. Specifically, we present a family of processes $\{x^i(t)\}_{i \ge 0}$ (defined in \eqref{d:x^i(t):0} below), where marginally, each $x^i(t)$ corresponds to a Euler-Murayama discretization with stepsize $\delta^i \propto 2^{-i}$. We verify in Lemma \ref{l:x(t)_is_brownian_motion} that $x^i(t)$ has an almost-sure limit $x(t)$, and in Lemma \ref{l:Phi_is_diffusion} that this limit is equivalent to the distribution in \eqref{e:intro_sde} as it has the same generator. 

We choose to use our construction of SDEs instead of the usual approach (e.g., Chapter 2 of \cite{hsu2002stochastic}), as it enables us to easily analyze the discretization error: 
aAdjacent pairs of processes $x^i(t)$ and $x^{i+1}(t)$ are coupled using the manifold analog of synchronous coupling (see Remark \ref{remark:sde_construction}), and the distance between $x^0(t)$ and the limit $x(t)$, can be bound by summing (over $i$) the pairwise distances between $x^i(t)$ and $x^{i+1}(t)$, which are small due to synchronous coupling. We hope that readers unfamiliar with Riemannian geometry, but familiar with stochastic analysis in Euclidean space, will find our construction more accessible than the standard approach, as we only use the basic Jacobi equation and do not invoke non-constructive embeddings of manifolds into $\Re^d$. 

With the above construction, Theorem \ref{t:langevin_mcmc} bounds the difference between a manifold SDE and its Euler-Murayama discretization. We are now ready to present our construction in Section~\ref{ss:discrete_gaussian_walk_construction}. 

\subsection{Construction of Manifold SDE on Dyadic Points}
\label{ss:discrete_gaussian_walk_construction}
To bound the discretization error, we first characterize the manifold SDE~\eqref{e:intro_sde} as the limit of a family of random processes. Let $x_0 \in M$ be an initial point and $E^1,\ldots,E^d$  an orthonormal basis of $T_{x^0}$. Let $\BB(t)$ denote a standard Brownian Motion in $\Re^d$. Let $T\in \Re^+$. Define the initialization
\begin{alignat*}{1}
    & x^0_0 = x_0, \qquad E^{0,j}_0 = E^j,\\
    & x^0_1 = \Exp_{x^0_0}(T \beta(x^0_0) + {\lrp{\BB\lrp{T} - \BB\lrp{0}}} \circ E^{0}_0).
\end{alignat*}
For any $i\in \Z^+$, let $\delta^i := 2^{-i}T$. We define points $x^i_k \in M$ and vectors $E^{i,j}_k$ in suitable tangent spaces for all $i$, all $k\in \{0,\ldots,\nicefrac{T}{\delta^i}\}$, and all $j\in \lrbb{1,\ldots, d}$. Our construction is inductive: Suppose we have already defined $x^i_k$ and $E^{i,j}_k$ for some $i$, for all $k\in \{0,\ldots,\nicefrac{T}{\delta^i}\}$, and for all $j\in\lrbb{1,\ldots,d}$. Then, for $i+1$, we define for all $k=\{0,\ldots,\nicefrac{T}{\delta^i}\}$ the updated values
\begin{alignat*}{1}
    & x^{i+1}_0 := x_0, \qquad E^{i+1,j}_{0} := E^j,\\
    & x^{i+1}_{2k+1} := \Exp_{x^{i+1}_{2k}}\lrp{\delta^{i+1}\beta\lrp{x^{i+1}_{2k}} + {\lrp{\BB\lrp{\lrp{2k+1}\delta^{i+1}} - \BB\lrp{2k\delta^{i+1}}}} \circ E^{i+1}_{2k}},\\
    & E^{i+1}_{2k+1} := \party{x^{i+1}_{2k}}{x^{i+1}_{2k+1}} \lrp{E^{i+1}_{2k}},\\
    & x^{i+1}_{2k+2} := \Exp_{x^{i+1}_{2k+1}}\lrp{\delta^{i+1}\beta\lrp{x^{i+1}_{2k+1}} + {\lrp{\BB\lrp{\lrp{2k+2}\delta^{i+1}} - \BB\lrp{\lrp{2k+1}\delta^{i+1}}}} \circ E^{i+1}_{2k+1}},\\
    & E^{i+1}_{2k+2} := \party{x^{i}_{k+1}}{x^{i+1}_{2k+2}} \lrp{E^{i}_{k+1}}
    \elb{d:x^i_k}.
\end{alignat*}
The above display defines points $x^{i+1}_k$ for all $k = \lrbb{0,\ldots,\nicefrac{T}{\delta^{i+1}}}$. For the parallel transport, if the minimizing geodesic is not unique, any arbitrary choice will do. Finally, we also define for any $i$, and for any $t\in [k\delta^i, (k+1)\delta^i)$, the point $x^i(t)$ to be the ``linear interpolation'' of $x^i_k$ and $x^i_{k+1}$, i.e.,
\begin{alignat*}{1}
  x^{i}(t) := \Exp_{x^{i}_{k}}\bigl((t-k\delta^i)\beta\lrp{x^i_k} + {\tfrac{(t-k\delta^i)}{\delta^i}\lrp{\BB\lrp{t} - \BB\lrp{k\delta^{i}}}} \circ E^{i}_{k}\bigr).
  \elb{d:x^i(t):0}
\end{alignat*}
We verify that $x^i(t)$ defined above coincides with $x^i_k$ defined in \eqref{d:x^i_k} when $t=k\delta^i$ for some $k$. In general, we let 
\begin{alignat*}{1}
  \overline{\Phi}(t; x, E, \beta, \BB, i)
  \elb{d:x^i(t)}
\end{alignat*}
denote the solution to the interpolated process in \eqref{d:x^i(t):0}, initialized at $x_0 = x$. We also let
\begin{alignat*}{1}
  \Phi(t; x, E, \beta, \BB) := \lim_{i \to \infty} \overline{\Phi}(t; x, E, \beta, \BB, i)
  \elb{d:x(t)}
\end{alignat*}
We verify in Lemma \ref{l:x(t)_is_brownian_motion} that $\Phi(t; x, E, \beta, \BB, i)$ exists, and in Lemma \ref{l:Phi_is_diffusion} that $\Phi(t; x, E, \beta, \BB)$ is equal in distribution to \eqref{e:intro_sde}.

\begin{remark}\label{remark:sde_construction}
The choice of basis $E^i_k$ in \eqref{d:x^i_k} can be seen as a combination of ``synchronous coupling'' and ``rolling without slipping'' (see Chapter 2 of \cite{hsu2002stochastic}). In particular, $E^{i+1}_{2k} := \party{x^{i}_{k}}{x^{i+1}_{2k}} \lrp{E^{i}_{k}}$ corresponds to ``synchronous coupling''---the step from $x^{i+1}_{2k}$ to $x^{i+1}_{2k+1}$ is (roughly) parallel to the step from $x^i_{k}$ to $x^i_{k+1}$. On the other hand, $E^{i+1}_{2k+1} := \party{x^{i+1}_{2k}}{x^{i+1}_{2k+1}} \lrp{E^{i+1}_{2k}}$ corresponds to ``rolling without slipping''---the step from $x^{i+1}_{2k+1}$ to $x^{i+1}_{2k+2}$ is with respect to an orthonormal frame that is parallel-transported from $x^{i+1}_{2k}$ to $x^{i+1}_{2k+1}$.
\end{remark}

After constructing $x^i(t)$ via~\eqref{d:x^i(t):0}, the next step is to understand its limit as $i\to \infty$. 
We will soon see that this limit is equivalent to the manifold SDE \eqref{e:intro_sde}. Lemma~\ref{l:x(t)_is_brownian_motion} shows that this limit exists almost surely, uniformly over $[0,T]$, if $\beta$ satisfies \mbox{Assumption~\ref{ass:beta_lipschitz}.}
\begin{lemma}
  \label{l:x(t)_is_brownian_motion}
  Let $x\in M$ be some initial point and $E$ an orthonormal basis of $T_xM$. Let $\BB(t)$ be a Brownian motion in $\Re^d$; and $\beta(x)$ a vector field satisfying Assumption~\ref{ass:beta_lipschitz}. Let $T\in \Re^+$, for $t\in[0,T]$ and let $x^i(t)$ be constructed as per~\eqref{d:x^i(t):0}.
  Then with probability 1, there is a limit $x(t)$ such that for all $\epsilon$, there exists an integer $N$ such that for all $i \geq N$,
  \begin{alignat*}{1}
    \sup_{t\in[0,T]} \dist\lrp{x^i(t), x(t)} \leq \epsilon.
  \end{alignat*}
\end{lemma}
\noindent We defer the proof of Lemma~\ref{l:x(t)_is_brownian_motion} to Section~\ref{appendix:manifold_sde}, and proceed to verify that the limit $x(t)$ indeed
defines an SDE.
\begin{lemma}\label{l:Phi_is_diffusion}
  For any $T$, for $t\in [0,T]$, let $x(t)$ be the limit ensured by Lemma~\ref{l:x(t)_is_brownian_motion}. 
  Then $x(t)$ is diffusion process generated by the operator $L$ whose action on any smooth function $f$ is given by $L f = \lin{\nabla f, \beta} + \frac{1}{2} \Delta(f)$, where $\Delta$ denotes the Laplace Beltrami operator.
\end{lemma}

Lemma~\ref{l:Phi_is_diffusion} follows from \citep{gangolli1964construction,jorgensen1975central}. We provide a proof in Appendix~\ref{ss:phi_is_diffusion} for completeness. By Theorem~1.3.7 of~\citep{hsu2002stochastic}, this implies that $x(t)$ 
exactly corresponds to~\eqref{e:intro_sde}; moreover, its distribution does not depend on the choice of $E$, $\BB$, or time interval $T$.

Next, present a result that quantifies the evolution of distance between two trajectories of \eqref{e:intro_sde}:
\begin{lemma}
  \label{l:g_contraction_without_gradient_lipschitz}
  Assume $\beta$ is $(m,q,\R)$-distant dissipative as per Assumption~\ref{ass:distant-dissipativity}, and that is also satisfies Assumption~\ref{ass:beta_lipschitz}. Further assume that $m \geq 2L_{\Ric}$ and $q + L_{\Ric} \geq 0$. Let $x_0,y_0\in M$, while $E^x$, $E^y$ are some orthonormal bases at $x_0$ and $y_0$, respectively. Let $\BB^x$ and $\BB^y$ be two Brownian motions. Let $T$ be some positive time. Finally, let $x(t)=\Phi(t;x_0,E^x,\beta,\BB^x)$ and $y(t)=\Phi(t;y_0,E^y,\beta,\BB^y)$ be as defined in \eqref{d:x(t)}.
  Then there exists a Lyapunov function $f$, such that $f(r) \geq \frac{1}{2}\exp\lrp{- \lrp{q + L_{\Ric}} \R^2/2} r$ and $\lrabs{f'(r)} \leq 1$, such that 
  \begin{alignat*}{1}
    \inf_{\Omega} \E{f\lrp{\dist\lrp{x(T),y(T)}}}
    \leq& \exp\lrp{-\alpha T}f\lrp{\dist\lrp{x_0,y_0}},
  \end{alignat*}
  where $\Omega$ denotes all couplings between $\BB^x$ and $\BB^y$, and $\alpha := \min\lrbb{\frac{m}{16}, \frac{1}{2 \R^2}} \cdot \exp\lrp{- \frac{1}{2}\lrp{q + L_{\Ric}} \R^2}$.
\end{lemma}
Lemma \ref{l:g_contraction_without_gradient_lipschitz} (with minor variation) was first presented in \cite{eberle2016reflection}. The Lyapunov function $f$ is also taken from \cite{eberle2016reflection} and its form is given in \eqref{d:f}. Since $f(r)\in [\frac{1}{2}e^{- \lrp{q + L_{\Ric}} \R^2/2} r, r]$, Lemma~\ref{l:g_contraction_without_gradient_lipschitz} shows that two trajectores of \eqref{e:intro_sde} (with different initialization) can be coupled in a way such that their distance contracts. This very useful for showing that approximations of the SDE \eqref{e:intro_sde}, such as the geometric Euler-Murayama discretization, also converge in Wasserstein distance to something close to \eqref{e:intro_sde}.

\begin{remark}
\label{remark:kendall-cranston-proof}
The proof of Lemma~\ref{l:g_contraction_without_gradient_lipschitz} is deferred to Appendix~\ref{ss:Evolution_of_Lyapunov_Function_under_Kendall_Cranston_Coupling} 
as it is rather involved. Our proof technique is based on Kendall-Cranston coupling, which we describe in Appendix~\ref{ss:The Kendall Cranston Coupling}. The Lyapunov function $f$ is taken from~\citep{eberle2016reflection}, its exact form corresponds is given in \eqref{d:f} in Appendix~\ref{ss:Lyapunov function and its smooth approximation}. Lemma~\ref{l:g_contraction_without_gradient_lipschitz} is not new. For instance, Theorem 6.6.2 of~\citep{hsu2002stochastic} + It\^{o}'s Lemma, applied to the Lyapunov function $f$ will give the same bound. Nonetheless, we provide a proof of Lemma~\ref{l:g_contraction_without_gradient_lipschitz} for completeness.
\end{remark}

Finally, accoutered with the above construction and the corresponding distance bounds, we are ready to describe our main result for this section.
\subsection{Langevin MCMC on Riemannian Manifolds}
Our first main result is Theorem~\ref{t:langevin_mcmc} on Riemannian Langevin MCMC. It bounds the $W_1$ distance between iterates, and implies that iterates generated by Langevin MCMC converge to the stationary distribution when the discretization step is small enough.
\begin{theorem}[Langevin MCMC on Manifolds under distant dissipativity]
    \label{t:langevin_mcmc}
    Assume the manifold $M$ satisfies Assumptions~\ref{ass:sectional_curvature_regularity} and~\ref{ass:ricci_curvature_regularity}. Additionally, assume that point 1 from Assumption~\ref{ass:higher_curvature_regularity} holds for some $L_R'$ (this last assumption is for analytical convenience; it does not show up in the quantitative bounds).  Let $\beta$ be a vector field satisfying Assumptions \ref{ass:distant-dissipativity} and \ref{ass:beta_lipschitz}; assume in addition that $m\geq 2 L_{\Ric}$ and that $q+ L_{\Ric} \geq 0$. Let $x^*$ be a point with $\beta(x^*) = 0$ and let $x_0$ be a point satisfying $\dist\lrp{x_0,x^*} \leq \R$.  Let $K \in \Z^+$ be some iteration number and $\delta$ be some stepsize. Let $y_k$ denote the exact geometric SDE, defined by $y_{k+1} = \Phi\lrp{\delta,y_0,{E^k},\beta,\BB_k}$, where ${E^k}$ is some orthonormal basis at $y_k$, $\BB_k$ is a $\Re^d$ Brownian Motion, and $\Phi$ is as defined in \eqref{d:x(t)}. Let $x_k$ denote the Euler Murayama discretization of $\Phi$, defined by $x_{k+1} = \Exp_{x_k}\lrp{\delta \beta(x_k) + \sqrt{\delta} \xi_k(x_k)}$, where $\xi_k \sim \N_{x_k}\lrp{0,I}$. 
    
    Then, there exists a constant $\C_0 = poly\lrp{L_\beta', d, L_R, \R, \frac{1}{m}, \log K}$, such that if $\delta \leq \frac{1}{\C_0}$, then there is a coupling between $x_K$ and $y_K$ satisfying the distance bound
    \begin{alignat*}{1}
        \E{\dist\lrp{y_{K},x_{K}}} 
        \leq e^{-\alpha K \delta + \lrp{q + L_{Ric}}\R^2/2}\E{\dist\lrp{y_{0},x_{0}}} + \exp\lrp{\lrp{q + L_{Ric}}\R^2} \cdot \t{O}\lrp{\delta^{1/2}}.
    \end{alignat*}
    $\t{O}$ hides polynomial dependence on $L_\beta', d, L_R, \R, \frac{1}{m}, \log K, \log\frac{1}{\delta}$, and $\alpha := \min\lrbb{\frac{m}{16}, \frac{1}{2 \R^2}} \cdot e^{\lrp{- \frac{1}{2}\lrp{q + L_{\Ric}} \R^2}}$.
\end{theorem}

We defer the proof of Theorem \ref{t:langevin_mcmc} to Section \ref{ss:proof_of_t:langevin_mcmc}, where we state the explicit expressions for $\C_0$.

\subsubsection{Discussion of Theorem \ref{t:langevin_mcmc}}
\label{ss:theorem_1_discussion}
One can view the $x_k$ in Theorem \ref{t:langevin_mcmc} as the manifold version of an Euler-Murayama discretization of the SDE described by $\Phi$. To ensure that the distance between the $x_K$ and $y_K$ is bounded by $\epsilon$, one needs $K = \t{O}\lrp{\frac{1}{\epsilon^2}}$ steps with stepsize $\t{O}\lrp{\epsilon^2}$. This result matches the rate for unadjusted Langevin MCMC on Euclidean space, see e.g.,~\citep{durmus2017nonasymptotic}.

\vskip5pt
\noindent\sparhead{Implementing Langevin MCMC.} A step $x_k$ in Theorem \ref{t:langevin_mcmc} involves computing an exponential map, which may be expensive. But taking a step in a random Gaussian direction is still usually much easier than sampling a Brownian motion on a manifold, even ignoring the effect of $\beta$.

\vskip5pt
\noindent\sparhead{Distant Dissipativity.} This assumption ensures that the manifold SDE (and its approximation) stays within a ball of radius $\R^*$. Let $\beta$ satisfy Assumption~\ref{ass:beta_lipschitz}; let $g: M \to \Re$ have a Hessian that is 1-strongly-convex outside a ball of radius $\R_g$ around $x^*$, then $\beta + 2L_\beta' g$ satisfies Assumption~\ref{ass:distant-dissipativity} with $m=L_\beta', q= L_\beta'$ and $\R = 2 \R_g$. This assumption is commonly used when analyzing nonconvex Langevin MCMC in $\Re^d$ as well~\citep{eberle2016reflection,gorham2019measuring,cheng2020stochastic}.

The requirement $q+L_{\Ric} \geq 0$ is without loss of generality. If $q + L_{\Ric} \leq 0$, we can take $\R=0$ and $m' = \max\lrbb{m, -(q+L_{\Ric})}$, and verify that $\beta$ satisfies Assumption \ref{ass:distant-dissipativity} with $(m',q,0)$. The mixing rate is then $\alpha = \frac{m'}{16}$, which corresponds to the easy setting, e.g. when $\beta$ is the gradient of a function satisfying the $CD(0,\infty)$ condition. The hard case is when $q+ L_{\Ric}$ is positive (e.g., when $\beta$ is the gradient of a concave function, or when the manifold is negatively curved); here the mixing rate $\alpha\propto \exp\lrp{-\lrp{q + L_{\Ric}} \R^2}$ is exponentially small, and this dependence is generally unavoidable.

\vspace*{-5pt}
\subsubsection{Proof Sketch of Theorem \ref{t:langevin_mcmc} and Theoretical Contributions}
There are three key ingredients that make up the proof:
\vskip3pt
\noindent\sparhead{I. Discretization Error Bound.} The first part of our analysis bounds the Wasserstein distance between $x_{k+1} = {\Phi}\lrp{\delta,x_k,{F^k},\beta,\BB_k^x,0}$ and $\bar{x}_{k+1} = \Exp_{x_k}\bigl(\delta \beta(x_k) + \sqrt{\delta} \zeta_k\bigr)$, where $F^k$ is an orthonormal basis of $x_k$ and $\zeta_k \sim \N_{x_k}(0,I)$. We present our bound in Lemma \ref{l:discretization-approximation-lipschitz-derivative} below:
\begin{lemma}
  \label{l:discretization-approximation-lipschitz-derivative}
  Let $\beta(\cdot)$ be a vector field satisfying Assumption \ref{ass:beta_lipschitz}. Consider arbitrary $x_0\in M$ and let $E$ be an orthonormal basis of $T_{x_0} M$. Let $\BB$ be a standard Brownian motion in $\Re^d$. Let $x^0(t)= \overline{\Phi}(t;x,E,\beta,\BB,0)$ and $x(t)= \Phi(t;x,E,\beta,\BB)$ as defined in \eqref{d:x^i(t)} and \eqref{d:x(t)} respectively. (Existence of $x(t)$ follows from Lemma \ref{l:x(t)_is_brownian_motion}). Let $L_1$ be any constant such that $L_1 \geq L_0 := \lrn{\beta(x_0)}$ and assume that $T\leq \min\lrbb{\frac{1}{16{L_\beta'}},  \frac{1}{16L_R d}, \frac{1}{16\sqrt{L_R} L_1}}$. 
Then
  \begin{alignat*}{1}
      \E{\dist\lrp{x^{0}(T),x(T)}} \leq 2^{9}\lrp{T^2 L_1^2+ T^2 {L_\beta'}^2 + T^{3/2} \lrp{d^{3/2}\lrp{L_R + {L_\beta'}^2/{L_1^2}} + L_\beta' \sqrt{d}}}
  \end{alignat*}
\end{lemma}
Taking $T=\delta$, we show that the one-step discretization error is bounded by $\E{\dist\lrp{x_{k+1},\bar{x}_{k+1}}}\lesssim O\lrp{\delta^{3/2}}$. This $\delta^{3/2}$ error scaling is the same as the error of Euler-Murayama in $\Re^d$~\citep{durmus2017nonasymptotic}. We are not aware of existing results bounding the error of Euler-Murayama on general manifolds. We defer the proof to Section \ref{ss:langevin_mcmc_on_manifold}.

\vskip3pt
\noindent\sparhead{II. An SDE Construction.}
The proof of Lemma~\ref{l:discretization-approximation-lipschitz-derivative} works by bounding distance between adjacent trajectories $x^i(t)$ and $x^{i+1}(t)$ defined in \eqref{d:x^i(t):0}. The analysis crucially relies the on the fact that $x^i(t)$ and $x^{i+1}(t)$ are ``almost'' synchronously coupled---see discussion in Section~\ref{ss:discrete_gaussian_walk_construction}.

\vskip3pt
\noindent\sparhead{III. Contraction via Kendall-Cranston (Reflection Coupling).}
The last part involves applying Lemma \ref{l:g_contraction_without_gradient_lipschitz} to show that when $x(t)$ and $y(t)$ evolve under the exact SDE, $f(\dist\lrp{x(t),y(t)})$ contracts with rate $\alpha$. We again highlight the fact that Lemma 3 is an existing result, see Remark~\ref{remark:kendall-cranston-proof}.



\vspace*{-5pt}
\section{Diffusion Approximation of Non-Gaussian walk}
\vspace*{-5pt}
We now study how to approximate the SDE~\eqref{e:intro_sde} by a discrete random walk with non-Gaussian noise. Let $\beta$ be a deterministic and $\xi$ a random vector field $\xi$; given an initial point $y\in M$, we define a one-step walk with drift vector field $\beta$ and noise $\xi$ (in the form of a random vector field) as
    \begin{alignat*}{1}
        \Psi(\delta;y,\beta,\xi) := \Exp_{y}\bigl(\delta \beta(y) + \sqrt{\delta} \xi(y)\bigr).
        \elb{d:y_k}
    \end{alignat*}
Next, we introduce key assumptions on the manifold and the noise. Our first assumption bounds the first and second derivatives of the Riemannian curvature tensor:
\begin{assumption}\label{ass:higher_curvature_regularity}
    We assume that the manifold $M$ has Riemannian curvature tensor $R$ satisfying
    \begin{enumerate}
        \item For all $x\in M$, $u,v,w,z,a\in T_x M$, $\lin{(\nabla_a R)(u,v)w,z} \leq L_R'\lrn{u}\lrn{v}\lrn{w}\lrn{z}\lrn{a}$
        \item For all $x\in M$, $u,v,w,z,a,b\in T_x M$, $\lin{(\nabla_b (\nabla_a R))(u,v)w,z} \leq L_R''\lrn{u}\lrn{v}\lrn{w}\lrn{z}\lrn{a}\lrn{b}$
    \end{enumerate}
\end{assumption}
Our second assumption bounds the magnitude of $\xi$ and its derivatives up to second order:
\begin{assumption}[Noise regularity]\label{ass:regularity_of_xi}
    The random vector field $\xi$ satisfies, with probability 1,
    \begin{inparaenum}[(i)]
        \item $\lrn{\xi(x)} \leq L_{\xi}$;
        \item $\lrn{\lrp{\nabla_{v_1} \xi}(x)} \leq L_{\xi}'$;
        \item $\lrn{\lrp{\nabla_{v_2}\nabla_{v_1} \xi}(x)} \leq L_{\xi}''$.
    \end{inparaenum}
    for all $x$ and for all $v_1, v_2, v_3 \in T_x M$ with $\lrn{v_1} = \lrn{v_2} = \lrn{v_3} = 1$.
\end{assumption}
Lastly, we assume that the noise $\xi$ has $0$ mean and identity covariance:
\begin{assumption}[Moments of $\xi$]\label{ass:moments_of_xi}
    For all $x\in M$, the random vector field $\xi$ in satisfies:
    \begin{align*}
        1.\ \E{\xi(x)}=0 \qquad \qquad 2.\ \mathrm{Cov}(\xi(x)) = I.
    \end{align*}
    More precisely, let $E_1,\ldots,E_d$ be any orthonormal basis of $T_x M$. Let $\xib \in \Re^d$ be the coordinates of $\xi$ wrt $E_1,\ldots,E_d$. Then 2. is saying that $\E{\xib(x) \xib(x)^T} = I$.
\end{assumption}
Given these assumptions, we are ready to state the second main result of this paper:
\begin{theorem}\label{t:main_nongaussian_theorem}
    Assume $M$ satisfies Assumptions \ref{ass:sectional_curvature_regularity}, \ref{ass:ricci_curvature_regularity} and \ref{ass:higher_curvature_regularity}. Let $\beta$ be a vector field satisfying Assumption \ref{ass:beta_lipschitz}. Assume in addition that there exists $L_\beta$ such that for all $x\in M$, $\lrn{\beta(x)}\leq L_\beta$. Let $\xi$ be a vector field that satisfies Assumption \ref{ass:regularity_of_xi} and Assumption \ref{ass:moments_of_xi}. There exist a constant $\C_1$, which depend polynomially on $L_\beta,L_\beta', L_\xi,L_\xi',L_\xi'',L_R,L_R'$, such that for any positive $T \leq \frac{1}{\C_1}$, the following holds:
    
    Let $\delta := {T}^3$ and $K:= T/\delta$. Let $y_0$ be arbitrary, let $x^*$ be any point with $\beta(x^*)=0$. For $k\in \Z^+$, let $y_{k+1} = \Psi(\delta;y_k,\beta,\xi_k)$, where $\Psi$ is as defined \eqref{d:y_k}, and let $\bar{y}_{k} := \overline{\Phi}(k\delta;y_0,E,\beta,\BB,0)$, where $\overline{\Phi}$ is as defined in \eqref{d:x^i(t)}, $E$ is an orthonormal basis at $T_{y_0}M$, $\BB$ is a standard Euclidean Brownian motion., there is a coupling between $y_K$ and $\bar{y}_K$ for which
    \begin{alignat*}{1}
        \E{\dist\lrp{\bar{y}_{K},y_K}} \leq \t{O}\lrp{T^{3/2} \lrp{1 + L_\beta^2}}
    \end{alignat*}
    where $\t{O}$ hides polynomial dependence on $L_R, L_R', L_\xi, L_\xi', L_\xi'', \log\lrp{\frac{1}{T}}$.
\end{theorem}

We defer the proof of Theorem \ref{t:main_nongaussian_theorem} to Appendix \ref{s:proof_of_t:main_nongaussian_theorem}.

\vspace*{-5pt}
\subsection{Discussion of Theorem \ref{t:main_nongaussian_theorem}}

\vskip3pt
\noindent\sparhead{Short-time discretization bound} Theorem \ref{t:main_nongaussian_theorem} bounds the distance between a discrete semimartingale with non-Gaussian increments ($y_k$), and an single geometric Euler Murayama step ($\bar{y}_k$). We assume that $\xi_k$, the $0$-mean noise in $y_k$, has bounded derivatives (Assumption \ref{ass:regularity_of_xi}), and identity covariance (Assumption \ref{ass:moments_of_xi}), but can otherwise be arbitrary. Notably, the distribution of $\xi$, need \emph{not} be rotationally invariant, and hence not invariant under parallel transport.

Theorem \ref{t:main_nongaussian_theorem} bounds the distance between $y_k$ and $\bar{y}_k$ by $O\lrp{T^{3/2}}$, and is thus meaningful only when $T=K\delta$ is small. Thus Theorem \ref{t:main_nongaussian_theorem} is most analogous to Lemma \ref{l:discretization-approximation-lipschitz-derivative}, which shows that the distance between the exact SDE in \eqref{e:intro_sde} and its Euler-Murayama discretization is also bounded by $O(T^{3/2})$.

\vskip3pt
\noindent\sparhead{Long-time behavior and distance to exact SDE.} Similar to how Theorem \ref{t:langevin_mcmc} applies Lemma \ref{l:discretization-approximation-lipschitz-derivative} recursively to show that the Euler-Murayama discretization converges to the exact SDE over a long period of time, we can also apply Theorem \ref{t:main_nongaussian_theorem}+Lemma \ref{l:discretization-approximation-lipschitz-derivative} recursively to show that $y_k$ converges to the same SDE over a long period of time. Our result is summarized in Lemma \ref{l:theorem_2_corollary} below. According to Lemma \ref{l:theorem_2_corollary}, for the non-Gaussian random semimartingale $y_{NK}$ to be $\epsilon$-close to the exact SDE $x_{NK}$ in Wasserstein distance, $\delta$ needs to be order of $\epsilon^6$. The first term requires number of steps $NK = \t{O}\lrp{\log(1/\epsilon)/\delta} = \t{O}\lrp{\epsilon^{-6}}$

\begin{lemma}
    \label{l:theorem_2_corollary}
    Assume $M,\beta,\xi$ satisfy the same assumptions as Theorem \ref{t:main_nongaussian_theorem}. Further let $\beta$ satisfy Assumption \ref{ass:distant-dissipativity} with $m,q,\R$, such that $m \geq 2L_{\Ric}$, and $q + L_{\Ric} \geq 0$. Let $\alpha$ be as in Theorem \ref{t:langevin_mcmc}. 
    
    Let $T\in \Re^+$. Let $\delta := {T}^3$ and $K:= T/\delta$. Let $x^*$ be any point with $\beta(x^*)=0$. For $k\in \Z^+$, let $y_{k+1} = \Psi(\delta;y_k,\beta,\xi_k)$,where $\Psi$ is as defined \eqref{d:y_k}. For $i\in \Z^+$, let $x_{k+1} := \Phi(K\delta;x_{iK},{F^{iK}},\beta,\BB^{iK})$, where $\Psi$ is as defined in  \eqref{d:y_k} and $\Phi$ is defined by~\eqref{d:x(t)}, $F^{iK}$ is an orthonormal basis at $T_{x_{iK}}M$, and $\BB^{iK}$ are independent Brownian motions. Assume that $\dist\lrp{x_0,x^*} \leq 2\R$ and $\dist\lrp{y_0,x^*}\leq 2\R$.
    
    For any $N$, there exist a constant $\C_2$, which depends polynomially on $L_\beta,L_\beta', L_\xi,L_\xi',L_\xi'',L_R,L_R'$,\\
    $d,\R,\log N$, such that if $T \leq \frac{1}{\C_2}$, then there is a coupling between $x_{NK}$ and $y_{NK}$ satisfying
    \begin{alignat*}{1}
        \E{\dist\lrp{x_{NK},{y}_{NK}}}
        \leq \exp\lrp{-\alpha N K\delta + \lrp{q + L_{Ric}}\R^2}\dist\lrp{x_{0},y_{0}} + \exp\lrp{\lrp{q + L_{Ric}}\R^2} \cdot \t{O}\lrp{\delta^{1/6}},
    \end{alignat*}
    where $\t{O}$ hides polynomial dependency on $L_\beta,L_\beta', L_\xi,L_\xi',L_\xi'',L_R,L_R',L_R'', \frac{1}{m},d,\R,\log{N},\log\frac{1}{\delta}$.
\end{lemma}
We defer the proof of Lemma \ref{l:theorem_2_corollary} to Appendix \ref{s:proof_of_t:main_nongaussian_theorem}. 

\subsection{Proof Sketch and Theoretical Contributions}
Below, we sketch the proof of Theorem~\ref{t:main_nongaussian_theorem}. \emph{To simplify presentation, we will assume that $\beta=0$ in this sketch. This still captures the essence of the proof as the main bottleneck in the bound is due to the noise terms $\xi_k$.}

\begin{enumerate}
\setlength{\itemsep}{-1pt}
    \item Define $\bar{y}_{K} := \Exp_{y_0}\lrp{\sqrt{K \delta} \zeta}$ where $\zeta \sim \N_{y_0}\lrp{0,I}$, i.e. $\bar{y}_K$ is a single Euler Murayama step, of stepsize $K\delta$. We show that, conditioned on $\xi_0,\ldots,\xi_{K-1}$, there is some coupling between $\zeta$ and $\xi_{0},\ldots,\xi_{K -1}$, such that
    $\E{\dist\lrp{y_{K},\bar{y}_{K}}} = \t{O}\lrp{\lrp{K\delta}^{3/2}}$.
    \item We now bound the distance between $\bar{y}_{K}$ and the exact SDE $\hat{y}_{K} := \Phi(K\delta;y_{0},{E},\hat{\BB}^i)$. By Lemma \ref{l:discretization-approximation-lipschitz-derivative}, $\E{\dist\lrp{\bar{y}_{K},\hat{y}_{K}}} \leq \t{O}\lrp{\lrp{K\delta}^{3/2}}$.
\end{enumerate}
\vskip3pt
\noindent\sparhead{I. A random walk on $T_{y_0} M$.} 
Our first step is to construct a discrete stochastic process $\t{z}_k$ in $T_{y_0} M$, such that $\t{y}_k:=\Exp_{z_k}$ approximates $y_k$ for $k=0...K$. This corresponds to Step 1 in the proof of Theorem \ref{t:main_nongaussian_theorem} in Section \ref{s:proof_of_t:main_nongaussian_theorem}. Specifically, $\t{z}_k$ and $\t{y}_k$ are defined as 
\begin{alignat*}{1}
    & \t{y}_0 = y_0, \qquad \t{z}_0 = 0,\\
    & \gamma_k(t) = \Exp_{\t{y}_0} \lrp{(1-t) \t{z}_k}, \quad
     \t{z}_{k+1} = \t{z}_k + \sqrt{\delta} \party{\gamma_k(t)}{}\lrp{\xi_k (\t{y}_k)}, 
    \quad \t{y}_{k+1} = \Exp_{y_0} \lrp{\t{z}_{k+1}}.
    \elb{d:ty_k:main_discussion}
\end{alignat*}
Our goal in step I is to show the bound
\begin{alignat*}{1}
    \E{\dist\lrp{\t{y}_{K},y_{K}}} \leq \t{O}\lrp{\lrp{K\delta}^{3/2}} = \t{O}\lrp{\lrp{T}^{3/2}}, 
    \elb{e:t:sdasd:0}
\end{alignat*}
The key to proving the above inequality is Lemma \ref{l:triangle_distortion_G}: given $x\in M$ and $u\in T_x M$, let $x' := \Exp_x(u)$, then for any $v\in T_x M$, there exists a tensor $G$, that depends only on $x$, such that \\
$$\dist\lrp{\Exp_{x}(u+v), \Exp_{x'}(\party{x}{x'}G(v))}\leq O\lrp{\lrn{u}^2\lrn{v}^4}.$$
Furthermore, $\lrn{G - Id} = O\lrp{\lrn{u}^2}$.

Consider a fixed step $k$, and let $x = y_0, x' = \t{y}_k, u = \t{z}_k, v = \party{\t{y}_k}{y_0} \lrp{\sqrt{\delta}\xi_k\lrp{\t{y}_k}}$. Notice that $\t{z}_k$ is the sum of $k$ random variables each of variance scaled by $\delta$. Thus with high probability, $\lrn{u}^2 \leq O\lrp{K\delta}$ and $\lrn{v}^2 \leq O\lrp{\delta}$, so that
$O\lrp{\lrn{u}^2\lrn{v}^4} = O\lrp{K\delta^3}$. Lemma \ref{l:triangle_distortion_G} thus guarantees that
\begin{alignat*}{1}
    \dist\lrp{\t{y}_{k+1}, \Exp_{\t{y}_{k}}\lrp{\sqrt{\delta} G(\xi(\t{y}_k))}}^2\leq O\lrp{K\delta^3}
    \elb{e:t:sdasd:1}
\end{alignat*}
On the other hand, we can use a basic "synchronous-coupling-like" analysis (Lemma \ref{l:discrete-approximate-synchronous-coupling}) to show that
\begin{alignat*}{1}
    & \E{\dist\lrp{y_{k+1}, \Exp_{\t{y}_{k}}\lrp{\sqrt{\delta} G(\xi(\t{y}_k))}}^2} \\
    \lesssim& \lrp{1+\delta}\E{\dist\lrp{y_{k}, \t{y}_k}^2} - 2\sqrt{\delta} \E{\lin{\Exp_{\t{y}_k}^{-1}, G(\xi_k(\t{y}_k)) - \party{y_k}{\t{y}_k} \xi_k(y_k)}}\\
    &\quad + \delta \E{\lrn{G(\xi_k(\t{y}_k)) - \party{y_k}{\t{y}_k} \xi_k(y_k)}^2}\\
    \leq& \lrp{1+\delta}\E{\dist\lrp{y_{k}, \t{y}_k}^2} + K^2 \delta^3
    \elb{e:t:sdasd:2}
\end{alignat*}
Note that the above is simplified and keeps only the $K$ and $\delta$ dependence. The last line uses the facts that $1. \E{\xi_k(x)} = 0$, $2. G(v)$ is linear in $v$, and $3. \lrn{G - Id}^2 \leq O\lrp{\lrn{u}^4} = O\lrp{K^2\delta^2}$. Combining \eqref{e:t:sdasd:1} and \eqref{e:t:sdasd:2}, and recursing over $k= K...0$ gives \eqref{e:t:sdasd:0}. Detailed steps can be found in the proof of Corollary \ref{c:K-step-retraction-bound-yk}.
\begin{remark}
    Lemma \ref{l:triangle_distortion_G} is similar to Lemma 3 of \cite{sun2019escaping}, which we restate as Lemma \ref{l:triangle_distortion} in the appendix. Roughly speaking, this Lemma provides a bound of the form \\
    $\dist\lrp{\Exp_{x}(u+v), \Exp_{x'}(\party{x}{x'}G(v))}\leq O\lrp{\lrn{u}^2\lrn{v}^2 \lrp{\lrn{u}^2 + \lrn{v}^2}}$, which ends up being $O\lrp{K^2 \delta^3}$. Compared with \eqref{e:t:sdasd:1}, we see that there is an additional factor of $K$, which would eventually lead to a bound of $O\lrp{K^2 \delta^{3/2}}$ instead of \eqref{e:t:sdasd:0}. The essential difference is our bound is that we pull out the linear error term $G()$ separately; later on, when bounding \eqref{e:t:sdasd:2}, the cross term involving $G(\xi(\cdot))$ has $0$ expectation due to linearity, so that it does not pose an issue. We believe that Lemma \ref{l:triangle_distortion_G} may be of independent interest when one needs to analyze retractions of semimartingales in other settings.    
\end{remark}

\vskip3pt
\noindent\sparhead{II. A variant of Central Limit Theorem.} 
Our goal in this step is to bound $W_2\lrp{\t{z}_K, \sqrt{K\delta} \zeta}$, where $\zeta \sim \N_{y_0}(0,I)$. Observe that $\t{z}_K$ is the sum of $\sqrt{\delta} \party{\t{y}_k}{y_0}\xi_k(\t{y}_k)$ for $k=0,\ldots,K-1$. We assumed that $\xi_k(x)$ has $0$-mean and identity covariance for all $x$. Lemma \ref{l:identity_covariance_parallel_transport} implies that the parallel transport of $\xi_k$ must also have $0$-mean and identity covariance. If the $\xi_k$'s are independent, then standard quantitative CLT (e.g. \cite{zhai2018high}) immediately imply that $\t{z}_K$ is close to a $\N(0,K\delta I)$. However, each $\xi_k(\t{y}_k)$ depends on $\xi_0(\t{y_0}),\ldots,\xi_{k-1}(\t{y}_{k-1})$, because it depends on $\t{y}_k$. In Lemma \ref{l:clt:main}, we show that the CLT approximation still holds if 1. $\xi_k$'s have the correct covariance, and 2. $\xi_k$'s have bounded derivatives, up to second order. Quantitatively, Lemma \ref{l:clt:main} guarantees that
\begin{alignat*}{1}
    W_2\lrp{\t{z}_K, \sqrt{K\delta} \zeta} \leq \t{O}\lrp{\delta^{1/2}} = \t{O}\lrp{T^{3/2}}
    \elb{e:t:sdasd:3}
\end{alignat*} 
where we use the definition of $\delta = (K\delta)^3 = T^{3}$.

Theorem \ref{t:main_nongaussian_theorem} follows almost immediately from two observations: First, $y_K = \overline{\Phi}(K\delta; y_0, E,\beta, \BB, 0)$ is equal in distribution to $\Exp_{y_0}\lrp{\sqrt{K\delta} \zeta}$. Second, $\dist\lrp{\t{y}_K, \bar{y}_K} = \dist\lrp{\Exp_{y_0}\lrp{\t{z}_K}, \Exp_{y_0}(\sqrt{K\delta} \zeta)} \lesssim \lrn{\t{z}_K, \sqrt{K\delta}\zeta}$, where we use Lemma \ref{l:discrete-approximate-synchronous-coupling} in the last step. Theorem \ref{t:main_nongaussian_theorem} follows immediately from summing \eqref{e:t:sdasd:0} and \eqref{e:t:sdasd:3}.

\vspace*{-5pt}
\section{Stochastic Gradient Descent and Generalization Error}
\label{s:SGD}
\vspace*{-5pt}

In this section, we show how our analysis of the diffusion on Riemannian manifold can lead to a generalization bound for SGD. We begin with setting up the notations and explain why SGD can be viewed as a Riemannian diffusion process.

\subsection{Setup}
To distinguish from $\nabla$ (Levi Civita connection), we use the bold version $\nablab$ to denote the \emph{Euclidean derivative} of a function $h: \Re^d \to \Re$, i.e. $\frac{d}{dt} h(x + t v) = \sum_{i=1}^d \lrp{\nablab h(x)}_i \cdot v_i$.

Let $\ell(x,s): \Re^d \times \S \to \Re$ denote a loss function, with $x$ being model parameters (e.g. neural network weights) and $s$ being a sample (e.g. vector of pixel values). Let $\S_n = \{s_1, s_2, s_{n-1}, s_n\}$ denote a set of $n$ i.i.d samples. Let $\ell_n(x) := \frac{1}{n} \sum_{s\in \S_n} \ell(x,s)$ be the expected gradient wrt uniform distribution over the training set. We are interested in the following sequence:
\begin{alignat*}{1}
    x_{k+1} = x_k - \delta \nablab  \ell_n(x_k) +  \sqrt{\delta}\xi_k(x_k),
    \elb{e:sgd_section:noisy_sgd}
\end{alignat*}
where $\xi_k$ are i.i.d random vector fields satisfying $\E{\xi_k(x)}=0$ and\\
$\E{\xi_k(x) \xi_k(x)^T} = \E{\lrp{\nablab \ell(x,s) - \nablab \ell_n(x)}\lrp{\nablab \ell(x,s) - \nablab \ell_n(x)}^T}$. Let us denote $A_n(x) := \E{\xi_k(x) \xi_k(x)^T}$, where expectation is wrt $s\sim Uniform (\S_n)$. Let us also define the following Euclidean SDE, initialized at $y(0) = x_0$:
\begin{alignat*}{1}
    d y(t) = -\nablab \ell_n(y(t))dt + A_n(y(t))^{1/2} d\WW(t)
    \elb{d:sgd_section:euclidean_sde}
\end{alignat*}
where $\WW(t)$ is a Brownian motion in $\Re^d$. We will see in Section \ref{ss:theorem_4_discussion} that $x_k$ in \eqref{e:sgd_section:noisy_sgd} approximates \eqref{d:sgd_section:euclidean_sde}. Finally, let us define the Riemannian manifold $(\Re^d, g)$, where $g$ is the metric tensor, given by $g(x) = A_n(x)^{-1}$, where $A_n(x)$ is the empirical covariance of gradient at $x$ defined earlier. 

Notice that the noise scaling is $\sqrt{\delta}$, whereas in standard SGD, the $0$-mean noise term has scaling $\delta$. In the study of SGD noise and generalization error, various authors \cite{hoffer2017train,cheng2020stochastic,li2021validity} have proposed the injection of additional noise into SGD iterates, so that the second moment of the noise term is constant even as stepsize goes to $0$. In a related vein, the "linear-scaling-rule", which suggests to multiply stepsize by $c$ whenever batchsize is multiplied by $c$, also works to effectively maintain a constant ratio for $\frac{\lrn{A_n}_2}{\delta}$.  (\cite{krizhevsky2009learning,goyal2017accurate}, see also the discussion in \cite{li2021validity}) Our motivation here is to analyze these "noisy SGD" variants. We do not specify an explicit form for $\xi$, but only impose various regularity assumptions such as \eqref{ass:regularity_of_xi}. This is so our analysis applies for different noise injection schemes, such as adding the difference of indpendently sampled gradients \cite{cheng2020stochastic,li2021validity}, or multiplying the stochastic gradient by a Gaussian noise \cite{hoffer2017train}.

\subsection{The Generalization Gap from Riemannian Diffusion}
We study the generalization performance of SGD via Riemannian diffusion following the stability approach in ~\cite{bousquet2002stability,hardt2016train}. Recall the definitions of $\S_n = \{s_1, s_2, s_{n-1}, s_n\}$, $\ell_n$, $g_n$, $A_n$ from the previous sections. Let $\t{\S}_n := \{s_1', s_2, s_{n-1}, s_n\}$ be an adjacent set of samples, where $s_1'$ is independently sampled, i.e., $\S_n$ and $\t{\S}_n$ differ only in the first sample.
Let $\t{\ell}_n$ denote the expected empirical loss on $\S_n$, and let $\t{g}_n(x)$ denote the empirical covariance of $\nablab \ell(x,\cdot)$. Let $\t{A}_n(x) := \t{g}_n(x)^{-1}$. Let $\t{\xi}_k$ be i.i.d random vector fields with $0$-mean and covariance $\t{A}_n(x)$.
\begin{alignat*}{1}
    \t{x}_{k+1} = \t{x}_k - \delta \nablab  \t{\ell}_n(\t{x}_k) +  \sqrt{\delta}\t{\xi}_k(\t{x}_k),
    \elb{e:sgd_section:noisy_sgd_tilde}
\end{alignat*}

\begin{theorem}[Algorithm stability, informal]\label{t:sgd-informal}
    Denote $x_t, x_t'$ as the iterates generated by equation~\ref{e:sgd_section:noisy_sgd} and \ref{e:sgd_section:noisy_sgd_tilde} with data set $\S_n$ and $\S_n'$ respectively. Then for positive integers $K = \mathcal{O}(n^4), \delta = \mathcal{O}(n^{-6})$, we have that for $t = K, 2K, 3K...$, 
    \begin{align*}
        W_1(x_t, x_t') \le \min\{ \mathcal{O}(t\delta n^{-1}), \mathcal{O}(n^{-1})\},
    \end{align*}
    where the hidden constants depend on constants in assumptions but are independent of $n, t, \delta, K$.
\end{theorem}

Detailed theorem statement, along with the assumptions needed, can be found in Theorem~\ref{t:sgd_stability} in Appendix \ref{ss:main_sgd_result}. We discuss Theorem \ref{t:sgd-informal} in more detail at the end in Section \ref{ss:theorem_4_discussion}. Before that, we state the following generalization guarantee on SGD:

\begin{lemma}[Generalization bound]\label{thm:gen-gap}
Denote $x_t,$ as the iterates generated by equation~\ref{e:sgd_section:noisy_sgd} with data set $\S_n$. Then for positive integers $K = \mathcal{O}(n^4), \delta = \mathcal{O}(n^{-6})$, we have that for $t = K, 2K, 3K...$, the generalization gap satisfies
    \begin{align*}
       \mathbb{E}_{S, s' \cA} \bigl[ f(s', x_t) - \frac{1}{n}\sum_{i=1}^n f(s_i, x_t)\bigr] \le \min\{ \mathcal{O}(t\delta n^{-1}), \mathcal{O}(n^{-1})\},
    \end{align*}
    where $s'$ denotes an independent test sample. The hidden constants depend on constants in assumptions but are independent of $n, t, \delta, K$.
\end{lemma}

\subsection{Discussion and Proof Outline of Theorem \ref{t:sgd_stability}} \label{ss:theorem_4_discussion}
There are numerous assumptions used in Theorem \ref{t:sgd_stability}. They fall into three categories:
\begin{enumerate}[label=\Roman*]
    \item That $\lambda_A I \prec A(x) \prec L_A I$ (Assumption \ref{ass:g_A_regularity}). This essentially ensures that, up to constants, the Riemannian metric is equivalent to the Euclidean metric. This assumption is satisfied if the covariance matrix of $\xi$ is lower and upper bounded.
    \item That the vector field $\nablab \ell_n(x) + \phi(x)$ ($\phi$ being the trace of the Christoffel symbol defined in \eqref{d:phi_for_christoffel_symbol}) 
    satisfies the \emph{distant-dissipativity} assumption (Assumption \ref{ass:distant-dissipativity}). This is required to apply Lemma \ref{l:g_contraction_without_gradient_lipschitz} in order to show that \eqref{d:sgd_section:euclidean_sde} mixes in Riemannian distance.
    \item Various smoothness assumptions, such as $\lrn{\nablab \ell(x,s) - \nablab \ell(y,s)}_2 \leq L_{\ell}'\lrn{x-y}_2$, $\lrn{\xi(x) - \xi(y)}_2 \leq L_{\xi}' \lrn{x-y}_2$ (Assumption \ref{ass:euclidean_xi_assumption}) and $\tr\lrp{A_n(x) - A_n(y)}^2 \leq {L_A'}^2 \lrn{x-y}_2^2$ (Assumption \ref{ass:g_A_Lipschitz}). These assumptions are satisfied by assuming bounded derivatives of sufficient order on $\ell$ and $\xi$.
\end{enumerate}

The proof of Theorem \ref{t:sgd_stability} can be decomposed as follows: First, in Lemma \ref{l:euclidean_walk_to_manifold_walk}, we show that $x_k$ (as defined in \eqref{e:sgd_section:noisy_sgd}) can be approximated by the sequence $z_{k+1} = \Exp_{z_k} \lrp{- \delta \nablab  \ell_n(x_k) +  \sqrt{\delta}\xi_k(x_k) + \frac{\delta}{2} \phi(z_k)}$, with error $\E{\lrn{x_K - z_K}_2^2} = O\lrp{K\delta^2}$. The $\phi$ is a second-order correction term arising due to interaction between $\sqrt{\delta} \xi$ and the Riemannian curvature. Together with Theorem \ref{t:main_nongaussian_theorem} and Lemma \ref{l:manifold_sde_to_euclidean_sde}, we show in Lemma \ref{l:euclidean_clt} that $\E{\lrn{x_K - y(K\delta)}_2} \leq K\delta \cdot \delta^{1/6}$, where $y(t)$ is as defined in \eqref{d:sgd_section:euclidean_sde}.

    We highlight that Lemma \ref{l:euclidean_walk_to_manifold_walk} may be of independent interest. For example, the discrete walk in Theorem \ref{t:langevin_mcmc} may not be easily computable on a general manifold. Often, the algorithm has to work in some Euclidean coordinate system, and Lemma \ref{l:euclidean_walk_to_manifold_walk} allows us to approximate one step of \eqref{e:intro:discrete_manifold} by a step in Euclidean coordinates. 

    We repeat Step 1 to show that $\t{x}_{k+1}$ in \eqref{e:sgd_section:noisy_sgd_tilde} is similarly approximated by $d \t{y}(t) = -\nablab \t{\ell}_n(\t{y}(t))dt + \t{A}_n(\t{y}(t))^{1/2} d\WW(t)$. We then define $d \bar{y}(t) = -\nablab \ell_n(\bar{y}(t))dt + {A}_n(\t{y}(t))^{1/2} d\WW(t)$. We show in Lemma \ref{l:euclidean_norm_divergence_under_constant_perturbation} that $\E{\lrn{\t{y}(t) - \bar{y}(t)}_2^2} = O\lrp{t/(n^2)}$. We use Lemma \ref{l:manifold_sde_to_euclidean_sde} to verify that $y$ and $\bar{y}$ are equivalent to SDEs on the manifold $(\Re^d, g)$. Lemma \ref{l:g_contraction_without_gradient_lipschitz} the implies that $\dist\lrp{y(K\delta),\bar{y}(K\delta)}$ contracts in terms of a Lyapunov function (this step crucially relies on II). 
    
    Finally, we remark that as an alternative, we can prove the contraction of $y(K\delta)-\bar{y}(K\delta)$, under assumptions of Euclidean distant-dissipativity instead (see, e.g. Assumption 3.2 of \cite{gorham2019measuring}). Euclidean distant dissipativity is easier to verify, since it is guaranteed by adding a sufficiently strongly convex regularizer outside a ball, but may lead to worse mixing rate. We remark that our stability bound in Theorem \ref{t:main_nongaussian_theorem} is unchanged in the big-O sense if we use this approach.

\bibliographystyle{plainnat}
\setlength{\bibsep}{3pt}
\bibliography{references} 

\newpage
\appendix
\begin{center}
  \Large\bfseries Contents of the Appendices
\end{center}
  
\startcontents[sections]
\printcontents[sections]{l}{1}{\setcounter{tocdepth}{2}}

\section{Manifold SDE}
\label{appendix:manifold_sde}
In this section, we state and prove key Lemmas related to our construction in Section \ref{ss:discrete_gaussian_walk_construction}. 

We prove the existence of limit in Lemma \ref{l:x(t)_is_brownian_motion}, the equivalence to manifold SDE in Lemma \ref{l:Phi_is_diffusion}. We also state our key discretization error bound for Euler Murayama discretization in Lemma \ref{l:discretization-approximation-lipschitz-derivative}. Finally, we provide the proof of Theorem \ref{t:langevin_mcmc}.

\subsection{Existence}
\begin{lemma}\label{l:key_brownian_limit_lemma}
  Let $T$ be any positive constant. Let $x^i(t)$ be as defined in \eqref{d:x^i(t)}. 
  
  Assume that there is are constants $L_\beta, L_\beta'$ such that for all $x,y\in M$, $\lrn{\beta(x)} \leq L_\beta$ and \\
  $\lrn{\beta(x) - \party{y}{x} \beta(y)} \leq L_\beta' \lrn{x-y}$.
  
  Then there exists a constant $\C$, which depends on $L_R, T, d$, such that for all $i \geq\C$,
  \begin{alignat*}{1}
      & \Pr{\sup_{t \in [0,T]} \dist\lrp{x^i(t), x^{i+1}(t)} \geq 2^{- \frac{i}{4} - 2}} \\
      & \leq e^{\lrp{2^{6-i}T {L_R} L_\beta^2 + 2 L_R d + L_\beta'}T}  T^2 \cdot \lrp{{\delta^i}^3 L_R^2 L_\beta^6 + L_R^2 d^3 + {\delta^i} {L_\beta'}^2 L_\beta^2 +{L_\beta'}^2d} \cdot 2^{-i/2 + 14}
  \end{alignat*}
\end{lemma}
\begin{proof}
  Let us define
  \begin{alignat*}{1}
      & a_k := \delta^{i+1} \beta\lrp{x^{i+1}_{2k}} + {\lrp{\BB\lrp{\lrp{2k+1}\delta^{i+1}} - \BB\lrp{2k\delta^{i+1}}}} \circ E^{i+1}_{2k}\\
      & b_k := \delta^{i+1} \beta\lrp{x^{i+1}_{2k}} + {\lrp{\BB\lrp{\lrp{2k+2}\delta^{i+1}} - \BB\lrp{\lrp{2k+1}\delta^{i+1}}}} \circ E^{i+1}_{2k}
      \elb{e:t:rjlqkwn:0}
  \end{alignat*}

  Our proof breaks down the bound of $\dist\lrp{x^i_k,x^{i+1}_{2k+2}}$ into two parts: by Young's inequality,
  \begin{alignat*}{1}
      \dist\lrp{x^i_{k+1}, x^{i+1}_{2k+2}}^2
      \leq& \lrp{\dist\lrp{x^i_{k+1}, \Exp_{x^{i+1}_{2k}}\lrp{a_k + b_k}} + \dist\lrp{\Exp_{x^{i+1}_{2k}}\lrp{a_k + b_k},x^{i+1}_{2k+2}}}^2\\
      \leq& \lrp{1 + \frac{1}{2K}}\dist\lrp{x^i_{k+1}, \Exp_{x^{i+1}_{2k}}\lrp{a_k + b_k}}^2 + K\dist\lrp{\Exp_{x^{i+1}_{2k}}\lrp{a_k + b_k},x^{i+1}_{2k+2}}^2
      \elb{e:t:rjlqkwn:1}
  \end{alignat*}
  We now bound the first term of \eqref{e:t:rjlqkwn:1}. From definition in \eqref{d:x^i_k} and \eqref{e:t:rjlqkwn:0},
  \begin{alignat*}{1}
      x^i_{k+1} =& \Exp_{x^i_k}\lrp{\delta^i \beta\lrp{x^i_k} + {\lrp{\BB\lrp{\lrp{k+1}\delta^{i}} - \BB\lrp{k\delta^{i}}}} \circ E^{i}_{k}}\\
      \Exp_{x^{i+1}_{2k}}\lrp{a_k + b_k} 
      =& \Exp_{x^i_k}\lrp{\delta^i \beta\lrp{x^{i+1}_{2k}} + {\lrp{\BB\lrp{\lrp{2k+2}\delta^{i+1}} - \BB\lrp{2k\delta^{i+1}}}} \circ E^{i+1}_{2k}}
  \end{alignat*}

  We thus apply Lemma \ref{l:discrete-approximate-synchronous-coupling}, with $x := x^i_k$, $y := x^{i+1}_{2k}$, $u := \delta^{i} \beta\lrp{x^i_k} + {\lrp{\BB\lrp{\lrp{k+1}\delta^{i}} - \BB\lrp{k\delta^{i}}}} \circ E^{i}_{k}$, $v := \delta^{i} \beta\lrp{x^{i+1}_{2k}} + {\lrp{\BB\lrp{\lrp{2k+2}\delta^{i+1}} - \BB\lrp{2k\delta^{i+1}}}} \circ E^{i+1}_{k}$. Let $\gamma(t), u(t), v(t)$ be as defined in Lemma \ref{l:discrete-approximate-synchronous-coupling}. Then Lemma \ref{l:discrete-approximate-synchronous-coupling} bounds
  \begin{alignat*}{1}
      \dist\lrp{\Exp_{x}(u), \Exp_y(v)}^2 \leq& \lrp{1+ 4 \C_k^2 e^{4\C_k}} \dist\lrp{x,y}^2 + 32 e^{\C_k} \lrn{v(0) - u(0)}^2 + 2\lin{\gamma'(0), v(0) - u(0)} 
      \elb{e:t:pmcqwd:1}
  \end{alignat*}
  where $\C_k := \sqrt{L_R} \lrp{\lrn{u} + \lrn{v}} \leq 2\sqrt{L_R}\lrp{\delta^i L_\beta + \lrn{\BB((k+1)\delta^i) - \BB(k\delta^i)}_2}$. 
  
  Some of the terms above can be simplified. We begin by bounding the $\lrn{u(0) - v(0)}$ term. By assumption that $\beta$ is Lipschitz, $\lrn{\delta^i \beta(x^i_k) - \party{x^{i+1}_{2k}}{x^i_k}\delta^i \beta(x^{i+1}_{2k})} \leq \delta^i L_\beta' \dist\lrp{x^i_k, x^{i+1}_{2k}}$. By definition of $E^{i+1}_{2k}$ from \eqref{d:x^i_k}, 
  \begin{alignat*}{1}
      & \party{x^{i+1}_{2k}}{x^i_k} \lrp{{\lrp{\BB\lrp{\lrp{2k+2}\delta^{i+1}} - \BB\lrp{2k\delta^{i+1}}}} \circ E^{i+1}_{2k}}\\
      =& {\lrp{\BB\lrp{\lrp{2k+2}\delta^{i+1}} - \BB\lrp{2k\delta^{i+1}}}} \circ \lrp{\party{x^{i+1}_{2k}}{x^i_k}E^{i+1}_{2k}}\\
      =& {\lrp{\BB\lrp{\lrp{k+1}\delta^{i}} - \BB\lrp{k\delta^{i}}}} \circ E^{i}_{k}
  \end{alignat*}
  where the last line is because $\delta^i = 2 \delta^{i+1}$ and because $E^{i+1}_{2k} := \party{x^{i}_{k}}{x^{i+1}_{2k}} \lrp{E^{i}_{k}}$ from \eqref{d:x^i_k}.

  Thus
  \begin{alignat*}{1}
      & {\lrp{\BB\lrp{\lrp{k+1}\delta^{i}} - \BB\lrp{k\delta^{i}}}} \circ E^{i}_{k} - \party{x^{i+1}_{2k}}{x^i_k} \lrp{{\lrp{\BB\lrp{\lrp{2k+2}\delta^{i+1}} - \BB\lrp{2k\delta^{i+1}}}} \circ E^{i+1}_{2k}}\\
      =& {\lrp{\BB\lrp{\lrp{k+1}\delta^{i}} - \BB\lrp{k\delta^{i}}}} \circ E^{i}_{k} -{\lrp{\BB\lrp{\lrp{k+1}\delta^{i}} - \BB\lrp{k\delta^{i}}}} \circ \party{x^{i+1}_{2k}}{x^i_k}E^{i}_{k}\\
      =& 0
  \end{alignat*}

  We can thus bound via Young's Inequality:
  \begin{alignat*}{1}
      \lrn{u(0) - v(0)}^2 \leq 2 {\delta^i}^2 {L_\beta'}^2 \dist\lrp{x^i_k,x^{i+1}_{2k}}^2
  \end{alignat*}
  Finally,
  \begin{alignat*}{1}
      & 2\lin{\gamma'(0), v(0) - u(0)} 
      = 2\lin{\Exp_{x^i_k}^{-1} \lrp{x^{i+1}_{2k}}, \party{x^{i+1}_{2k}}{x^i_k}\delta^i \beta(x^{i+1}_{2k}) - \delta^i \beta(x^i_k) }
      \leq 2\delta^i L_\beta' \dist\lrp{x^i_k, x^{i+1}_{2k}}^2 
  \end{alignat*}
  Plugging into \ref{e:t:pmcqwd:1}
  \begin{alignat*}{1}
      & \dist\lrp{x^i_{k+1}, \Exp_{x^{i+1}_{2k}}\lrp{a_k + b_k}}^2 
      \leq \lrp{1+ 4 \C_k^2 e^{4\C_k} + 64e^{\C_k}\delta^iL_\beta'}\dist\lrp{x^i_k,x^{i+1}_{2k}}^2
  \end{alignat*}

  We now bound the second term of \eqref{e:t:rjlqkwn:1}. Let us introduce two more convenient definitions:
  \begin{alignat*}{1}
      & b'_k := \delta^i \beta\lrp{x^{i+1}_{2k+1}} + {\lrp{\BB\lrp{\lrp{2k+2}\delta^{i+1}} - \BB\lrp{\lrp{2k+1}\delta^{i+1}}}} \circ E^{i+1}_{2k+1}\\
      & z := \Exp_{x^{i+1}_{2k}}\lrp{a_k}
  \end{alignat*}
  It follows from definition that 
  \begin{alignat*}{1}
      x^{i+1}_{2k+2} = \Exp_{z}\lrp{b'_k}
  \end{alignat*}

  We break the bound on $\dist\lrp{\Exp_{x^{i+1}_{2k}}\lrp{a_k + b_k},x^{i+1}_{2k+2}}$ into two terms:
  \begin{alignat*}{1}
      \dist\lrp{\Exp_{x^{i+1}_{2k}}\lrp{a_k + b_k},x^{i+1}_{2k+2}}
      \leq& \dist\lrp{\Exp_{x^{i+1}_{2k}}\lrp{a_k + b_k},\Exp_{z}\lrp{ \party{x^{i+1}_{2k}}{x^{i+1}_{2k+1}}b_k}} + \dist\lrp{\Exp_{z}\lrp{ \party{x^{i+1}_{2k}}{x^{i+1}_{2k+1}}b_k},x^{i+1}_{2k+2}}\\
      =& \dist\lrp{\Exp_{x^{i+1}_{2k}}\lrp{a_k + b_k},\Exp_{z}\lrp{ \party{x^{i+1}_{2k}}{x^{i+1}_{2k+1}}b_k}} + \dist\lrp{\Exp_z\lrp{ \party{x^{i+1}_{2k}}{x^{i+1}_{2k+1}}b_k}, \Exp_z\lrp{b'_k}}
  \end{alignat*}
  To bound the first term, we apply Lemma \ref{l:triangle_distortion} with $x = x^{i}_{2k}$, $a = a_k$, $y = b_k$, to get
  \begin{alignat*}{1}
      & \dist\lrp{\Exp_{x^{i+1}_{2k}}\lrp{a_k + b_k},\Exp_{z}\lrp{ \party{x^{i+1}_{2k}}{x^{i+1}_{2k+1}}b_k}}\\
      \leq& L_R \lrn{a_k}\lrn{b_k}\lrp{\lrn{a_k} + \lrn{b_k}} e^{\sqrt{L_R} \lrp{\lrn{a_k}+\lrn{b_k}}}
  \end{alignat*}

  To bound the second term, we apply Lemma \ref{l:discrete-approximate-synchronous-coupling} (with $x=y=z$), so that
  \begin{alignat*}{1}
      & \dist\lrp{\Exp_z\lrp{ \party{x^{i+1}_{2k}}{x^{i+1}_{2k+1}}b_k}, \Exp_z\lrp{b'_k}}^2 \\
      \leq& 32 e^{\C_k'} \lrn{\party{x^{i+1}_{2k}}{x^{i+1}_{2k+1}}b_k - b'_k}^2\\
      \leq& 64 e^{\C_k'} {\delta^{i+1}}^2 {L_\beta'}^2 \dist\lrp{x^{i+1}_{2k}, x^{i+1}_{2k+1}}^2 \\
      \leq& 128 e^{\C_k'} {\delta^{i+1}}^2 {L_\beta'}^2 \lrn{a_k}_2^2
  \end{alignat*}
  where we define $\C_k':= \sqrt{L_R} \lrp{\lrn{b_k} + \lrn{b_k'}}$.

  Plugging everything into \eqref{e:t:rjlqkwn:1},
  \begin{alignat*}{1}
      \dist\lrp{x^i_{k+1}, x^{i+1}_{2k+2}}^2
      \leq& \lrp{1 + \frac{1}{2K}}\lrp{1+ 4 \C_k^2 e^{4\C_k} + 64e^{\C_k}\delta^iL_\beta'}\dist\lrp{x^i_k,x^{i+1}_{2k}}^2\\
      &\qquad + 32K L_R^2 \lrp{\lrn{a_k}^6 + \lrn{b_k}^6} e^{2\sqrt{L_R} \lrp{\lrn{a_k}+\lrn{b_k}}} + 256K e^{\C_k'} {\delta^{i+1}}^2 {L_\beta'}^2 \lrn{a_k}^2
  \end{alignat*}

  In fact, if we consider any $t\in [k\delta^i, (k+1)\delta^i)$, and using the definition of $x^i(t)$ from \eqref{d:x^i(t)} as the linear interpolation between $x^i_k$ and $x^i_{k+1}$, we can extend the bound to
  \begin{alignat*}{1}
      & \sup_{t\in [k\delta^i, (k+1)\delta^i)} \dist\lrp{x^i(t), x^{i+1}(t)}^2 \\\leq& \lrp{1 + \frac{1}{2K}}\lrp{1+ 4 \C_k^2 e^{4\C_k} + 64e^{\C_k}\delta^iL_\beta'}\dist\lrp{x^i_k,x^{i+1}_{2k}}^2\\
      &\qquad + 32K L_R^2 \lrp{\lrn{a_k}^6 + \lrn{b_k}^6} e^{2\sqrt{L_R} \lrp{\lrn{a_k}+\lrn{b_k}}} + 256K e^{\C_k'} {\delta^{i+1}}^2 {L_\beta'}^2 \lrn{a_k}^2
      \elb{e:t:qomlkqm}
  \end{alignat*}

  Let us define
  \begin{alignat*}{1}
      r_0 =& 0\\
      r_{k+1}^2 
      :=& \lrp{1 + \frac{1}{2K}}\lrp{1+ 4 \C_k^2 e^{4\C_k} + 64e^{\C_k}\delta^iL_\beta'}r_k^2\\
      &\qquad + 32K L_R^2 \lrp{\lrn{a_k}^6 + \lrn{b_k}^6} e^{2\sqrt{L_R} \lrp{\lrn{a_k}+\lrn{b_k}}} + 256K e^{\C_k'} {\delta^{i+1}}^2 {L_\beta'}^2 \lrn{a_k}^2
  \end{alignat*}
  It follows from \eqref{e:t:qomlkqm} that $r_k \geq \sup_{t\in [(k-1)\delta^i, k\delta^i)} \dist\lrp{x^i(t), x^{i+1}(t)}$ and that $r_{k+1} \geq r_k$ with probability 1, for all $k$, so that $\sup_{t\leq T} \dist\lrp{x^i(t), x^{i+1}(t)} \leq r_K$. We will now bound $\E{r_K^2}$, and then apply Markov's Inequality. Let us define $\F_k$ to be the $\sigma$-field generated by $\BB(t)$ for $t\in [0, k\delta^i)$. Then
  \begin{alignat*}{1}
      \Ep{\F_k}{r_{k+1}^2} 
      \leq& \Ep{\F_k}{\lrp{1 + \frac{1}{2K}}\lrp{1+ 4 \C_k^2 e^{4\C_k} + 64e^{\C_k}\delta^iL_\beta'}} r_k^2 \\
      &\qquad + \Ep{\F_k}{32K L_R^2 \lrp{\lrn{a_k}^6 + \lrn{b_k}^6} e^{2\sqrt{L_R} \lrp{\lrn{a_k}+\lrn{b_k}}} + 256K e^{\C_k'} {\delta^{i+1}}^2 {L_\beta'}^2 \lrn{a_k}^2}
  \end{alignat*}
  We will bound the terms above one by one. First, note from definition that \\
  $\C_k \leq \sqrt{L_R}\lrp{2 \delta^i L_\beta + 2\lrn{\BB((k+1)\delta^i)-\BB(k\delta^i)}_2}$. Let $\eta^i_k:=\BB((k+1)\delta^i)-\BB(k\delta^i)$.
  
  For sufficiently large $i$, $\delta^i \leq {\sqrt{L_R} L_\beta/8}$. Simplifying,
  \begin{alignat*}{1}
      & \Ep{\F_k}{\lrp{1 + \frac{1}{2K}}\lrp{1+ 4 \C_k^2 e^{4\C_k} + 64e^{\C_k}\delta^iL_\beta'}}\\
      \leq& 1 + \frac{1}{2K} + 16{\delta^i}^2 L_R L_\beta^2 + 16 L_R \E{\lrn{\eta^i_k}^2} + 16{\delta^i}^2 L_R L_\beta^2 \E{e^{2\sqrt{L_R} \lrn{\eta^i_k}}}\\
      &\quad + \frac{8L_R}{\delta^i d} \E{\lrn{\eta^i_k}^4} + 8L_R \delta^i d \E{e^{4\sqrt{L_R} \lrn{\eta^i_k}}} + 128\delta^i L_\beta' \E{e^{2\sqrt{L_R} \lrn{\eta^i_k}}}\\
      \leq& 1 + \frac{1}{2K} + 16 L_R \lrp{{\delta^i}^2 L_\beta^2 + \E{\lrn{\eta^i_k}^2} + \frac{1}{\delta^i d}\E{\lrn{\eta^i_k}^4}} + 8\delta^i \lrp{\delta^i L_R L_\beta^2 + L_R d + 16 L_\beta'}\E{e^{4\sqrt{L_R} \lrn{\eta^i_k}}}\\
      \leq& 1 + \frac{1}{2K} + 40 \lrp{{\delta^{i}}^2 L_R L_\beta^2 + 2\delta^i L_R d + \delta^i L_\beta'}
  \end{alignat*}
  where we use
  \begin{alignat*}{1}
      & \E{\lrn{\eta^i_k}^2} = \delta^i d\\
      & \E{\lrn{\eta^i_k}^2} \leq 2 {\delta^i}^2 d^2\\
      & \E{e^{4\sqrt{L_R} \lrn{\eta^i_k}}} \leq 2\E{e^{8L_R \lrn{\eta^i_k}^2}} \leq 2e^{16L_R \delta^i d} \leq 4 
  \end{alignat*}
  where we use Lemma \ref{l:subexponential-chi-square}, and the fact that $\delta^i \leq \frac{1}{32L_R d}$ for sufficiently large $i$. 

  Next, we bound $\Ep{\F_k}{32K L_R^2 \lrp{\lrn{a_k}^6 + \lrn{b_k}^6} e^{2\sqrt{L_R} \lrp{\lrn{a_k}+\lrn{b_k}}}}$. Note that $\lrn{a_k}\leq \frac{\delta^i}{2} L_\beta + \lrn{\eta^{i+1}_{2k}}$ and $\lrn{b_k}\leq \frac{\delta^i}{2} L_\beta + \lrn{\eta^{i+1}_{2k+1}}$. By similar argument as above, 
  \begin{alignat*}{1}
      & \Ep{\F_k}{32K L_R^2 \lrp{\lrn{a_k}^6 + \lrn{b_k}^6} e^{2\sqrt{L_R} \lrp{\lrn{a_k}+\lrn{b_k}}}}\\
      \leq& 2048KL_R^2 e^{2\sqrt{L_R}\delta^i L_\beta} \Ep{\F_k}{\lrp{2^{-5} {\delta^i}^6 L_\beta^6 + \lrn{\eta^{i+1}_{2k}}^6 + \lrn{\eta^{i+1}_{2k}}^6} \cdot e^{2\sqrt{L_R} \lrp{\lrn{\eta^{i+1}_{2k}} + \lrn{\eta^{i+1}_{2k+1}}}}}\\
      \leq& K L_R^2 \lrp{512 {\delta^i}^6 L_\beta^6 + 2048 {\delta^i}^3 d^3 }
  \end{alignat*}

  Finally, note that $\C_k' \leq 2\sqrt{L_R} \lrp{L_\beta + \lrn{\eta^{i+1}_{2k+1}}}$, so that
  \begin{alignat*}{1}
      \Ep{\F_k}{256K e^{\C_k'} {\delta^{i+1}}^2 {L_\beta'}^2 \lrn{a_k}^2}
      \leq& 256K{\delta^i}^2 {L_\beta'}^2 \lrp{{\delta^i}^2 L_\beta^2 + \delta^i d}
  \end{alignat*}

  Put together,
  \begin{alignat*}{1}
      &\Ep{\F_k}{r_{k+1}^2} \\
      \leq& \lrp{1 + \frac{1}{2K} + 40 \lrp{{\delta^{i}}^2 L_R L_\beta^2 + 2\delta^i L_R d + \delta^i L_\beta'}} \lrp{K L_R^2 \lrp{512 {\delta^i}^6 L_\beta^6 + 2048 {\delta^i}^3 d^3 } + 256K{\delta^i}^2 {L_\beta'}^2 \lrp{{\delta^i}^2 L_\beta^2 + \delta^i d}}
      \elb{e:t:lqmf:0}
  \end{alignat*}
  Applying the above recursively and simplifying,
  \begin{alignat*}{1}
      \E{r_{K}^2} 
      \leq e^{40K{\delta^i}^2 L_R L_\beta^2 + 2K\delta^i L_R d + K\delta^i L_\beta'} \lrp{K {\delta^i}}^2 \lrp{{\delta^i}^4 L_R^2 L_\beta^6 + {\delta^i} L_R^2 d^3 + {\delta^i}^2 {L_\beta'}^2 L_\beta^2 + {\delta^i} {L_\beta'}^2d} \cdot 2^{10}
      \elb{e:t:lqmf}
  \end{alignat*}
  By Markov's Inequality, and recalling that $r_k$ is w.p. 1 non-decreasing and \\ 
  $\sup_{t\leq T} \dist\lrp{x^i(t), x^{i+1}(t)} \leq R_K$,
  \begin{alignat*}{1}
      & \Pr{\sup_{t\in[0,T]} \dist\lrp{x^i(t), x^{i+1}(t)}^2 \geq 2^{- i/2 -4}}\\
      \leq& \E{r_{K}^2} \cdot 2^{i/2 + 4}\\
      \leq& e^{\lrp{40\delta^i {L_R} L_\beta^2 + 2 L_R d + L_\beta'}T} T^2 \cdot \lrp{{\delta^i}^3 L_R^2 L_\beta^6 + L_R^2 d^3 + {\delta^i} {L_\beta'}^2 L_\beta^2 +{L_\beta'}^2d} \cdot 2^{-i/2 + 14}
      \elb{e:t:rjlqkwn:3}
  \end{alignat*}
\end{proof}

\begin{proof}[Proof of Lemma \ref{l:x(t)_is_brownian_motion}]
  Let us define $L_0 := \lrn{\beta(x(0))}$.

  \textbf{Step 1: Bounding the probability of deviation between $x^i$ and $x^{i+1}$}\\

  We would like to apply Lemma \ref{l:key_brownian_limit_lemma}. However, note that Lemma \ref{l:key_brownian_limit_lemma} assumes that $\lrn{\beta(x)} \leq L_\beta$ globally, which we do not assume here. We must therefore approximate $\beta$ by a sequence of Lipschitz vector fields. 

  Let us define 
  \begin{alignat*}{1}
      \beta^j(x):= \twocase{\beta(x)}{\lrn{\beta(x)} \leq 2^{j/2}}{\beta(x)\cdot\frac{2^{j/2}}{\lrn{\beta(x)}}}{\lrn{\beta(x)} > 2^{j/2}}
  \end{alignat*}

  Let us denote by $L_{\beta^j} := 2^{j/2}$. We verify that for all $x,y\in M$, $\lrn{\beta^j(x)} \leq L_{\beta^j}$ and \\
  $\lrn{\beta^j(x) - \party{y}{x} \beta^j(y)} \leq L_\beta' \lrn{x-y}$.

  Finally, for any let $\t{x}^{i,j}(t)$ be as defined in \eqref{d:x^i(t)}, with $\beta$ replaced by $\beta^j$. Lemma \ref{l:key_brownian_limit_lemma} immediately implies that, for all $i \geq \C$ (where $\C$ is some constant depending on $L_R, T, d$),
  \begin{alignat*}{1}
      & \Pr{\sup_{t \in [0,T]} \dist\lrp{\t{x}^{i,i}(t), \t{x}^{i+1,i}(t)} \geq 2^{- \frac{i}{4} - 2}} \\
      &\quad \leq e^{\lrp{40T {L_R} + 2 L_R d + L_\beta'}T}  T^2 \cdot \lrp{T^3L_R^2 + L_R^2 d^3 + T {L_\beta'}^2  +{L_\beta'}^2d} \cdot 2^{-i/2 + 14}
  \end{alignat*}
  where we use the fact that $\delta^i {L_\beta^i}^2 = T$ by definition.

  Recalling that $\beta^j(x) = \beta(x)$ unless $\lrn{\beta(x)} \geq 2^{j/2}$,
  \begin{alignat*}{1}
      \Pr{\exists_{t\in[0,T]} x^i(t) \neq \t{x}^{i,i}(t)} = \Pr{\exists_{k\in \lrbb{0...2^i}} x^i_k \neq  \t{x}^{i,i}_k} \leq \Pr{\sup_{k\in \lrbb{0...2^i}} \lrn{\beta(x^i_k)} \geq 2^{i/2}}
  \end{alignat*}

  We can bound $\lrn{\beta(x)} \leq L_0 + L_\beta' \dist\lrp{x,x_0}$, so that
  \begin{alignat*}{1}
      & \Pr{\sup_{k\in \lrbb{0...2^i}} \lrn{\beta(x^i_k)} \geq 2^{i/2}}\\
      \leq& \Pr{\sup_{k\in \lrbb{0...2^i}} \dist\lrp{x^i_k} \geq \frac{2^{i/2} - L_0}{L_\beta'}}\\
      \leq& \exp\lrp{2 + 8T L_\beta' +T L_R d + TL_R L_0^2} \cdot \lrp{2T d + 4T^2 L_0^2} \cdot {L_\beta'}^2 \cdot 2^{-i+2}
  \end{alignat*}
  where we use Lemma \ref{l:near_tail_bound_L2}, with $K=2^i$, and assume that $i$ satisfies $2^{i/2} \geq L_0$ and $2^i \geq T$.

  Using identical steps, we can also bound
  \begin{alignat*}{1}
      \Pr{\exists_{t\in[0,T]} x^{i+1}(t) \neq \t{x}^{i+1,i}(t)}
      \leq \exp\lrp{2 + 8T L_\beta' +T L_R d + T L_R L_0^2} \cdot \lrp{2T d + 4T^2 L_0^2} \cdot {L_\beta'}^2 \cdot 2^{-i+2}
  \end{alignat*}

  Put together,
  \begin{alignat*}{1}
      & \Pr{\sup_{t \in [0,T]} \dist\lrp{{x}^{i}(t), {x}^{i+1}(t)} \geq 2^{- \frac{i}{4} - 2}} \\
      \leq& \Pr{\sup_{t \in [0,T]} \dist\lrp{\t{x}^{i,i}(t), \t{x}^{i+1,i}(t)} \geq 2^{- \frac{i}{4} - 2}}  + \Pr{\exists_{t\in[0,T]} x^{i}(t) \neq \t{x}^{i,i}(t)} + \Pr{\exists_{t\in[0,T]} x^{i+1}(t) \neq \t{x}^{i+1,i}(t)}\\
      \leq& \C_2 \cdot 2^{-i/2}
  \end{alignat*}
  where $\C_2$ is a constant that depends on $T, L_R, L_\beta', L_0, d$, but does not depend on $i$.

  \textbf{Step 2: Apply Borel-Cantelli to show uniformly-Cauchy sequence with probability 1}\\
  Thus
  \begin{alignat*}{1}
      \sum_{i=\C_1}^{\infty} \Pr{\sup_{t}\dist\lrp{x^i(t), x^{i+1}(t)} \geq 2^{- \frac{i}{4}}} < \infty
  \end{alignat*}
  By the Borel-Cantelli Lemma,
  \begin{alignat*}{1}
      \Pr{\sup_{t}\dist\lrp{x^i(t), x^{i+1}(t)} \geq 2^{- \frac{i}{4}} \text{ for infinitely many $i$}} = 0
  \end{alignat*}
  Equivalently, with probability $1$, for all $\epsilon$, there exists a $N$ such that for all $i\geq N$,\\
  $\sup_{t}\dist\lrp{x^i(t), x^{i+1}(t)} \leq 2^{- \frac{i}{4}}$. For any $j \geq i \geq N$, it then follows that
  \begin{alignat*}{1}
      \sup_t \dist\lrp{x^i(t), x^j(t)}
      \leq& \sum_{\ell=i}^j \dist\lrp{x^\ell(t), x^{\ell+1}(t)}\\
      \leq& \sum_{\ell=i}^j 2^{- \frac{\ell}{4}}\\
      \leq& 6 \cdot 2^{- i/4}
  \end{alignat*}

  \textbf{Step 3: Uniform-Cauchy sequence implies uniform convergence to limit using standard arguments}
  Therefore, with probability 1, $x^i(t)$ is a uniformly Cauchy sequence. Let $x(t)$ be the point-wise limit of $x^i(t)$, as $i\to \infty$. It follows \footnote{A nice clean proof can be seen at \url{https://math.stackexchange.com/questions/1287669/uniformly-cauchy-sequences}} that with probability 1, for any $\epsilon$, there exists a $N$ such that for all $i \geq N$, 
  \begin{alignat*}{1}
      \sup_{t\in[0,T]} \dist\lrp{x^i(t), x(t)} \leq \epsilon
  \end{alignat*}
\end{proof}
  \subsection{$\Phi$ is an SDE}\label{ss:phi_is_diffusion}
  \begin{proof}[Proof of Lemma \ref{l:Phi_is_diffusion}]
  Let $\F_t$ denote the sigma field generated by $\BB(s) : s\in[0,t]$.

  Consider any $f: M \to \Re$ with $\lrn{f'}\leq C$, $\lrn{f''}\leq C$, $\lrn{f'''}\leq C$ globally. Let $x^i(t)=\overline{\Phi}(t;x,E,\beta,\BB,i)$ be as defined in \eqref{d:x^i(t)}. We will verify that $f(x(t)) - f(x(0)) - \int_0^t L f(x(t)) dt$ is a martingale.

  To begin, let $s,t\in [0,T]$ be such that $s = j \delta^a$ and $t = j' \delta^a$ for some positive integers $j\leq j', a$, where $\delta^i = T/2^i$. We will show that conditioned on $x(s)$, $f(x(t)) - f(x(s)) - \int_s^t L f(x(t)) dt$ is a martingale. Let us define
  \begin{alignat*}{1}
    u^i_k = \delta^i \beta(x^i_k) + \lrp{\BB((k+1)\delta^i) -\BB(k\delta^i)} \circ E^i_k
  \end{alignat*}
  so that $x^i(t) = \Exp_{x^i_k}\lrp{\frac{t-k\delta^i}{\delta^i}u^i_k}$, where $t\in[k\delta^i,(k+1)\delta^i]$.
  
  Consider an arbitrary $\ell \geq a$. Consider the sum
  \begin{alignat*}{1}
    \sum_{k=s/\delta^i}^{t/\delta^i-1} f(x^{\ell}((k+1)\delta^\ell))
    - f(x^{\ell}(k\delta^\ell)) - \lin{u^\ell_k, \nabla f(x^{\ell}(k\delta^\ell))} - \nabla^2 f(x^\ell_k)[u^\ell_k,u^\ell_k]
    \elb{e:t:alkdalksm:1}
  \end{alignat*}
  By Taylor's theorem, 
  \begin{alignat*}{1}
    & \lrabs{f(x^{\ell}((k+1)\delta^\ell))
    - f(x^{\ell}(k\delta^\ell)) - \lin{u^\ell_k, \nabla f(x^{\ell}(k\delta^\ell))} - \nabla^2 f(x^\ell_k)[u^\ell_k,u^\ell_k]}\\
    \leq& C\lrn{u^\ell_k}^3 \\
    \leq& 8 C^3{\delta^\ell}^3 \lrn{\beta(x^\ell_k)}^3 + 8C^3 \lrn{\BB((k+1)\delta^\ell)-\BB(k\delta^\ell)}_2^3
  \end{alignat*}

  Notice also that
  \begin{alignat*}{1}
    \lin{u^\ell_k, \nabla f(x^{\ell}(k\delta^\ell))}
    =& \delta^{\ell} \underbrace{\lin{\beta(x^\ell), \nabla f(x^{\ell}(k\delta^\ell))}}_{\lrn{\cdot}^2 \leq {\delta^{\ell}}^2 C^2 \lrn{\beta(x^\ell(k\delta^{\ell}))}} + \underbrace{\lin{\nabla f(x^{\ell}(k\delta^\ell)), \lrp{\BB((k+1)\delta^\ell)-\BB(k\delta^\ell)}\circ E^\ell_k}}_{\Ep{\F_{k\delta^\ell}}{\cdot}=0}
  \end{alignat*}
  and that $\Ep{\F_{k\delta^\ell}}{\nabla^2 f(x^{\ell}(k\delta^\ell)) [u^\ell_k,u^\ell_k] - \delta^\ell \Delta f(x^{\ell}(k\delta^\ell))} = 0$, and that
  \begin{alignat*}{1}
    \lrabs{\nabla^2 f(x^{\ell}(k\delta^\ell)) [u^\ell_k,u^\ell_k] - \delta^\ell \Delta f(x^{\ell}(k\delta^\ell))}^2 
    \leq& {\delta^\ell}^4 C^2 \lrn{\beta(x^\ell(k\delta^\ell))}^4 \\
    &\qquad + {\delta^\ell}^2 \lrn{\BB((k+1)\delta^\ell)-\BB(k\delta^\ell)}_2^4 + {\delta^{\ell}}^2 C^2 d
  \end{alignat*}

  Finally, note that there exists a constant $C'$, which depends on $T,d,L_\beta'\lrn{\beta(x_0)}$, such that for all $\ell$, for all $t\in [0,T]$, $\E{\lrn{\beta(x^\ell(t)}^6} \leq C'$. The proof is similar to Lemma \ref{l:near_tail_bound_L4} and we omit it here.

  Plugging into \eqref{e:t:alkdalksm:1} and taking expectation conditioned on $\BB(t) : t\in[0,s]$, we get that
  \begin{alignat*}{1}
    &\Ep{\F_s}{\lrabs{f(x^\ell(t)) - f(x^\ell(s)) + \sum_{k=s/\delta^\ell}^{t/\delta^\ell-1} - \delta^{\ell} \lin{\beta(x^{\ell}(k\delta^\ell)), \nabla f(x^{\ell}(k\delta^\ell))} - \frac{\delta^{\ell}}{2} \Delta f(x^{\ell}(k\delta^\ell))}^2} \\
    &\Ep{\F_s}{\lrabs{\sum_{k=s/\delta^\ell}^{t/\delta^\ell-1} f(x^{\ell}((k+1)\delta^\ell)) - f(x^{\ell}(k\delta^\ell)) - \delta^{\ell} \lin{\beta(x^{\ell}(k\delta^\ell)), \nabla f(x^{\ell}(k\delta^\ell))} - \frac{\delta^{\ell}}{2} \Delta f(x^{\ell}(k\delta^\ell))}^2} \\
    \leq& poly(C,C',d) \sum_{k=s/\delta^\ell}^{t/\delta^\ell-1} {\delta^{\ell}}^{2}
    \leq poly(C,C',d,T) {\delta^{\ell}}
  \end{alignat*}
  Next, define $g^\ell_k := \delta^{\ell} \lin{\beta(x^{\ell}(k\delta^\ell)), \nabla f(x^{\ell}(k\delta^\ell))} + \frac{\delta^{\ell}}{2} \Delta f(x^{\ell}(k\delta^\ell))$. Using the smoothness of $\beta$ and $f$, and the fact that $\dist\lrp{x^\ell(t),x(k\delta^\ell)} \leq \delta^{\ell}\lrn{\beta(x(k\delta^\ell))} + \lrn{\BB((k+1)\delta^\ell)-\BB(k\delta^\ell)}_2$, we verify that for any $k$,
  \begin{alignat*}{1}
    \Ep{\F_s}{\lrabs{g^\ell_k - \int_{k\delta^\ell}^{(k+1)\delta^\ell} \lin{\beta(x(r)), \nabla f(x(r))} + \frac{1}{2} \Delta f(x(r)) dr}} \leq poly \lrp{C,C',d} {\delta^{\ell}}^{3/2}
  \end{alignat*}
  Putting everything together, we get
  \begin{alignat*}{1}
    \Ep{\F_s}{\lrabs{f(x^\ell(t)) - f(x^\ell(s)) + \int_s^t - \lin{\beta(x^{\ell}(r)), \nabla f(x^{\ell}(r))} - \frac{\delta^{\ell}}{2} \Delta f(x^{\ell}(r)) dr}} \leq  poly(C,C',d,T) {\delta^{\ell}}^{1/2}
  \end{alignat*}
  By Lemma \ref{l:x(t)_is_brownian_motion}, $\sup_{t\in[0,T]} \dist\lrp{x^\ell(t), x(t)}$ converges to $0$ almost surely as $\ell \to \infty$. By Dominated Convergence Theorem, and by smoothness of $f$ and $\beta$, 
  
  \begin{alignat*}{1}
    & \lim_{\ell \to \infty} \Ep{\F_s}{\lrabs{f(x^\ell(t)) - f(x^\ell(s)) + \int_s^t - \lin{\beta(x^{\ell}(r)), \nabla f(x^{\ell}(r))} - \frac{\delta^{\ell}}{2} \Delta f(x^{\ell}(r)) dr}}\\
    = &\Ep{\F_s}{\lrabs{f(x(t)) - f(x(s)) + \int_s^t - \lin{\beta(x(r)), \nabla f(x(r))} - \frac{\delta}{2} \Delta f(x(r)) dr}} = 0
  \end{alignat*}
  The $=0$ is because $\lim_{\ell\to 0} poly(C,C',d,T) {\delta^{\ell}}^{1/2} = 0$.

  Recall that we assumed that $s$ and $t$ are integral multiples of $T/2^a$ for some positive integer $a$. To extend to the case of general $s,t$, we can define $s^\ell:= \lfloor\frac{s}{\delta^\ell}\rfloor$ and $\lfloor\frac{t}{\delta^\ell}\rfloor$. We know that $\Ep{\F_{s^\ell}}{\lrabs{f(x(t)) - f(x(s)) + \int_{s^\ell}^{t^\ell} - \lin{\beta(x(r)), \nabla f(x(r))} - \frac{\delta}{2} \Delta f(x(r)) dr}} = 0$ for all $\ell$, then take the limit of $\ell to \infty$.
  
\end{proof}

\subsection{Langevin MCMC on Manifold}
\label{ss:langevin_mcmc_on_manifold}
\begin{lemma}
    \label{l:discretization-approximation-lipschitz}
    Let $\beta(\cdot)$ be a vector field satisfying Assumption \ref{ass:beta_lipschitz}. Assume also that there exists $L_\beta$ such that $\lrn{\beta(x)}\leq L_\beta$ for all $x$. Consider arbitrary $x_0\in M$ and let $E$ be an orthonormal basis of $T_{x_0} M$. Let $\BB$ be a standard Brownian motion in $\Re^d$. Let $x^i(t)= \overline{\Phi}(t;x,E,\beta,\BB,i)$ and $x(t)= \Phi(t;x,E,\beta,\BB)$ as defined in \eqref{d:x^i(t)} and \eqref{d:x(t)} respectively. (Existence of $x(t)$ follows from Lemma \ref{l:x(t)_is_brownian_motion}).
    
    Then for any non-negative integer $\ell$,
    \begin{alignat*}{1}
        \E{\sup_{t\in[0,T]}\dist\lrp{x^\ell(t),x(t)}^2} 
        \leq& 2^{14}e^{40T\delta^\ell L_R L_\beta^2 + 2T L_R d + T L_\beta'} T^3\lrp{{\delta^\ell}^3 L_R^2 L_\beta^6 + L_R^2 d^3 + {\delta^\ell} {L_\beta'}^2 L_\beta^2 + {L_\beta'}^2d} \cdot 2^{-\ell}
    \end{alignat*}
    where $\delta^i := 2^{-i} T$
\end{lemma}

\begin{proof}
    Consider any fixed $i$, let $\delta^i := T/2^i$ and let $K:= T/\delta^i = 2^i$ as in \eqref{d:x^i_k}.

    Following the same steps leading up to \eqref{e:t:lqmf}, we can bound
    \begin{alignat*}{1}
        &\E{\sup_{t\in[0,T]}\dist\lrp{x^i(t),x^{i+1}(t)}^2} \\
        \leq& e^{40K{\delta^i}^2 L_R L_\beta^2 + 2K\delta^i L_R d + K\delta^i L_\beta'} \lrp{K {\delta^i}}^2 \lrp{{\delta^i}^4 L_R^2 L_\beta^6 + {\delta^i} L_R^2 d^3 + {\delta^i}^2 {L_\beta'}^2 L_\beta^2 + {\delta^i} {L_\beta'}^2d} \cdot 2^{10}\\
        =& \underbrace{2^{10}e^{40T\delta^i L_R L_\beta^2 + 2T L_R d + T L_\beta'} T^3\lrp{{\delta^i}^3 L_R^2 L_\beta^6 + L_R^2 d^3 + {\delta^i} {L_\beta'}^2 L_\beta^2 + {L_\beta'}^2d}}_{:=s_i} \cdot 2^{-i}
    \end{alignat*}
    where we use the fact that $K \delta^i =T$ by definition.

    By repeated application of Young's Inequality, we can bound, for any $\ell$ and any $j\geq \ell$,
    \begin{alignat*}{1}
        \E{\sup_{t\in[0,T]}\dist\lrp{x^{\ell}(t),x^{j}(t)}^2} 
        \leq& \sum_{i=\ell}^{j-1} 3 \lrp{\frac{3}{2}}^{i-\ell}\E{\sup_{t\in[0,T]}\dist\lrp{x^i(t), x^{i+1}(t)}^2}\\
        \leq& \sum_{i=\ell}^{j-1} 3 \lrp{\frac{3}{2}}^{i-\ell}\cdot 2^{-i} s_i \\
        \leq& 12 \cdot 2^{-\ell} \cdot s_\ell
    \end{alignat*}

    Since the above holds for any $j$, we can take the limit of $j\to\infty$ and 
    \begin{alignat*}{1}
        \E{\sup_{t\in[0,T]}\dist\lrp{x^\ell(t),x(t)}^2} 
        \leq& 12 \cdot 2^{-\ell} \cdot s_\ell
    \end{alignat*}
    where we use the fact that $\dist\lrp{x^j_{2^j},x(T)}$ converges almost surely to $0$, from Lemma \ref{l:x(t)_is_brownian_motion}.
\end{proof}

\begin{proof}[Proof of Lemma \ref{l:discretization-approximation-lipschitz-derivative}]
    Let us define 
    \begin{alignat*}{1}
        \beta^j(x):= \twocase{\beta(x)}{\lrn{\beta(x)} \leq L_1 2^{j/2+1}}{\beta(x)\cdot\frac{L_1 2^{j/2+1}}{\lrn{\beta(x)}}}{\lrn{\beta(x)} > L_1 2^{j/2+1}}
    \end{alignat*}
    Let us also define $L_{\beta^j} := L_1 2^{j/2}$. We verify that $\lrn{\beta^j(x)} \leq L_{\beta^j}$ for all $x$.

    Let $\t{x}^{i,j}(t):= \overline{\Phi}(t;x,E,\beta^j,\BB,i)$ and $\t{x}^{\cdot,j}(t):= \overline{\Phi}(t;x,E,\beta^i,\BB,i)$. By Young's Inequality,
    \begin{alignat*}{1}
        \dist\lrp{x^i(t),x^{i+1}(t)}
        \leq& \dist\lrp{x^i(t),\t{x}^{i,i}(t)} + \dist\lrp{x^{i+1}(t),\t{x}^{i+1,i}(t)} + \dist\lrp{\t{x}^{i,i}(t),\t{x}^{i+1,i}(t)}
        \elb{e:t:oqmflwkf:1}
    \end{alignat*}
    We first bound the last term. Let $\delta^i := T/2^i$. Using identical argument as Lemma \ref{l:discretization-approximation-lipschitz},
    \begin{alignat*}{1}
        &\E{\dist\lrp{\t{x}^{i,i}(T),\t{x}^{i+1,i}(T)}^2}\\
        \leq& 2^{10}e^{40T\delta^i L_R L_{\beta^i}^2 + 2T L_R d + T L_\beta'} T^3\lrp{{\delta^i}^3 L_R^2 L_{\beta^i}^6 + L_R^2 d^3 + {\delta^i} {L_\beta'}^2 L_{\beta^i}^2 + {L_\beta'}^2d} \cdot 2^{-i}\\
        =& 2^{10}e^{40T^2 L_R L_1^2 + 2T L_R d + T L_\beta'} T^3\lrp{T^3 L_R^2 L_1^6 + L_R^2 d^3 + T {L_\beta'}^2 L_1^2 + {L_\beta'}^2d} \cdot 2^{-i}\\
        \leq& 2^{14-i} T^3\lrp{T^3 L_R^2 L_1^6 + L_R^2 d^3 + T {L_\beta'}^2 L_1^2 + {L_\beta'}^2d} 
    \end{alignat*}

    On the other hand, by triangle inequality and by Assumption \ref{ass:beta_lipschitz},
    \begin{alignat*}{1}
        \dist\lrp{x^i(T),\t{x}^{i,i}(T)}
        \leq& \ind{\sup_{s\in[0,T]} \dist\lrp{x^i(s),x_0}> \frac{2^{i/2+1}L_1-L_0}{L_\beta'}}\lrp{\dist\lrp{x^i(T),x_0} + \dist\lrp{\t{x}^{i,i}(T),x_0}}\\
        \leq& \ind{\sup_{s\in[0,T]} \dist\lrp{x^i(s),x_0}> \frac{2^{i/2}L_1}{L_\beta'}}\lrp{\dist\lrp{x^i(T),x_0} + \dist\lrp{\t{x}^{i,i}(T),x_0}}
    \end{alignat*}
    Define $x^i_k := x^i(k\delta^i)$. Using the fact that by definition of \eqref{d:x^i(t)}, $x^i(t)$ are linear interpolations of $x^i_k$ for $k=0...2^i$,
    \begin{alignat*}{1}
        &\E{\dist\lrp{x^i(T),\t{x}^{i,i}(T)}} \\
        \leq& 2\sqrt{\Pr{\sup_{k\in\lrbb{0..2^i}} \dist\lrp{x^i_k,x_0}> \frac{2^{i/2}L_1}{L_\beta'}}} \cdot \lrp{\sqrt{\E{\dist\lrp{x^i(T),x_0}^2}} + \sqrt{\E{\dist\lrp{\t{x}^{i,i}(T),x_0}^2}}}
        \elb{e:t:oqmflwkf:2}
    \end{alignat*}

    From Lemma \ref{l:near_tail_bound_L4}, and our assumed bound on $T$,
    \begin{alignat*}{1}
        \sqrt{\Pr{\sup_{k\in\lrbb{0...2^i}} \dist\lrp{x^i_k,x_0} \geq \frac{2^{i/2}L_1}{L_\beta'}}}
        \leq& \frac{{L_\beta'}^2}{L_1^2 2^i}\exp\lrp{1 + 8 T L_\beta' + 2T L_R d + 2T\delta^i L_R L_0^2}\lrp{3T d + 8T^2 L_0^2}\\
        \leq& \frac{{L_\beta'}^2\lrp{Td + T^2 L_0^2}}{L_1^2} \cdot{2^{4-i}}
    \end{alignat*}

    From Lemma \ref{l:near_tail_bound_L2},
    \begin{alignat*}{1}
        \E{\dist\lrp{x^i(T),x_0}^2} 
        \leq& \exp\lrp{1 + 8T L_\beta' + T L_R d + T^2 L_R L_0^2} \cdot \lrp{T d + 8T^2 L_0^2}\\
        \leq& 4 \lrp{T d + 8T^2 L_0^2}
    \end{alignat*}
    The same upper bound also applies to $\E{\dist\lrp{\t{x}^{i,i}(T),x_0}^2}$. Plugging into \eqref{e:t:oqmflwkf:2}, 
    \begin{alignat*}{1}
        \E{\dist\lrp{x^i(T),\t{x}^{i,i}(T)}} \leq 2^{6-i} \frac{{L_\beta'}^2\lrp{Td + T^2 L_0^2}^{3/2}}{L_1^2} 
    \end{alignat*}

    By identical steps, we can also upper bound
    \begin{alignat*}{1}
        \E{\dist\lrp{x^{i+1}(T),\t{x}^{i+1,i}(T)}} \leq 2^{6-i} \frac{{L_\beta'}^2\lrp{Td + T^2 L_0^2}^{3/2}}{L_1^2}
    \end{alignat*}

    Plugging into \eqref{e:t:oqmflwkf:1}, 
    \begin{alignat*}{1}
        \E{\dist\lrp{x^i(T),x^{i+1}(T)}} 
        \leq& 2^{7-i/2} T^{3/2}\lrp{T^3 L_R^2 L_1^6 + L_R^2 d^3 + T {L_\beta'}^2 L_1^2 + {L_\beta'}^2d}^{1/2} + 2^{7-i} \frac{{L_\beta'}^2\lrp{Td + T^2 L_0^2}^{3/2}}{L_1^2}\\
        \leq& 2^{8-i/2}\lrp{T^2 L_1^2+ T^2 {L_\beta'}^2 + T^{3/2} \lrp{d^{3/2}\lrp{L_R + {L_\beta'}^2/{L_1^2}} + L_\beta' \sqrt{d}}}
    \end{alignat*}

    Summing over $i$, we can bound, for any integers $0\leq \ell$,
    \begin{alignat*}{1}
        \E{\dist\lrp{x^{0}(T),x^{\ell'}(T)}} 
        \leq& \sum_{i=0}^{\ell}2^{8-i/2}\lrp{T^2 L_1^2+ T^2 {L_\beta'}^2 + T^{3/2} \lrp{d^{3/2}\lrp{L_R + {L_\beta'}^2/{L_1^2}} + L_\beta' \sqrt{d}}}\\
        \leq& 2^{9}\lrp{T^2 L_1^2+ T^2 {L_\beta'}^2 + T^{3/2} \lrp{d^{3/2}\lrp{L_R + {L_\beta'}^2/{L_1^2}} + L_\beta' \sqrt{d}}}
    \end{alignat*}
    Since the above holds for all $\ell$, by dominated convergence together with Lemma \ref{l:x(t)_is_brownian_motion}, 
    \begin{alignat*}{1}
        \E{\dist\lrp{x^{0}(T),x(T)}} \leq 2^{9}\lrp{T^2 L_1^2+ T^2 {L_\beta'}^2 + T^{3/2} \lrp{d^{3/2}\lrp{L_R + {L_\beta'}^2/{L_1^2}} + L_\beta' \sqrt{d}}}
    \end{alignat*}
\end{proof}

\subsection{Proof of Theorem \ref{t:langevin_mcmc}}\label{ss:proof_of_t:langevin_mcmc}

\begin{proof}[Proof of Theorem \ref{t:langevin_mcmc}]
    Let us define
    \begin{alignat*}{1}
        & r_0 := 64 \lrp{\R + \frac{\sqrt{L_R} d}{m} + \sqrt{\frac{d}{m}}\lrp{\log\lrp{{\frac{L_R {L_\beta'}^4d^2}{m^5} + L_\beta' \R^2}}}}
    \end{alignat*}
    Let $\delta \in \Re^+$ be some fixed stepsize satisfying
    \begin{alignat*}{1}
        \delta \leq& \min\Bigg\{\frac{m}{32 {L_\beta'}^2}, \frac{1}{16L_R d}, \frac{d}{4m}\\
        & \frac{m}{32\sqrt{L_R} {L_\beta'}^2 r_0}, \sqrt{\frac{m^3}{2^{22} d L_R {L_\beta'}^4\log^2 K} }, \sqrt{\frac{m^3}{2^{22} d L_R {L_\beta'}^4 \log^2 \lrp{d L_R {L_\beta'}^4/m^3}}}\\
        & \frac{d}{16m\sqrt{L_R} {L_\beta'}^2 r_0}, \sqrt{\frac{d}{2^{22} m L_R \log^2 K} }, \sqrt{\frac{d}{2^{22} m L_R \log^2 \lrp{mL_R/d}}}\Bigg\}
        \elb{e:t:qodwnsak:0}
    \end{alignat*}
    and let $r\in \Re^+$ be a radius given by
    \begin{alignat*}{1}
        & r := r_0 + 64 \sqrt{\frac{d}{m}} \log\lrp{\frac{K}{\delta}}
    \end{alignat*}
    Let $A_k$ denote the event $\lrbb{\max_{i\leq k} \dist\lrp{x_i,y_i}} \leq r$.

    \textbf{Step 1: Tail Bound:}\\
    By Lemma \ref{l:far-tail-bound-truncated} and Lemma \ref{l:far-tail-bound-l2-brownian} and Lemma \ref{l:far-tail-bound-l2},
    \begin{alignat*}{1}
        & \E{\ind{A_k^c}} 
        = \Pr{A_k^c} 
        \leq 32k\delta m \exp\lrp{\frac{m\R^2}{d} + \frac{2 L_R d}{m} - \frac{m r^2}{32 d}}\\
        & \E{\dist\lrp{x_k,x^*}^2} \leq \frac{2^{12} L_R {L_\beta'}^4 d^2}{m^{6}} + \frac{9 {L_\beta'} \R^2}{m} \\
        & \E{\dist\lrp{y_k,x^*}^2} \leq \frac{2^{12} L_R {L_\beta'}^4 d^2}{m^{6}} + \frac{9 {L_\beta'} \R^2}{m} 
    \end{alignat*}
    By our definition of $r$ and applying Cauchy Schwarz, we verify that
    \begin{alignat*}{1}
        \E{\ind{A_K^c} \dist\lrp{y_K,x_K}} \leq \sqrt{\delta}
        \elb{e:t:qodwnsak:3}
    \end{alignat*}
    Note that Lemma \ref{l:far-tail-bound-truncated} requires $\delta \leq \min\lrbb{\frac{m}{16 {L_\beta'}^2 {\sqrt{L_R} r}}, \frac{d}{4m {\sqrt{L_R} r}}}$. This bound follows from \eqref{e:t:qodwnsak:0}, using Lemma \ref{l:useful_xlogx}, together with some algebra. 
    
    \textbf{Step 1: One-step analysis:}\\
    We will bound $\E{\ind{A_{k+1}} \dist\lrp{y_{k+1},{x}_{k+1}}}$ in terms of $\E{\ind{A_k} \dist\lrp{y_{k},{x}_{k}}}$. Consider a fixed but arbitrary $k$. We define a useful intermediate variable
    \begin{alignat*}{1}
        \bar{x}_{k+1} := \Phi\lrp{\delta,x_k,{E^k}^x,\beta,\BB_k^x}
    \end{alignat*}
    where ${E^k}^x$ is an orthonormal basis at $x_k$. In words, $\bar{x}_{k+1}$ is the result of a $\delta$-time exact diffusion, starting at $x_k$. From Lemma \ref{l:g_contraction_without_gradient_lipschitz}, there exists a coupling such that
    \begin{alignat*}{1}
        \E{\ind{A_{k}} f\lrp{\dist\lrp{y_{k+1},\bar{x}_{k+1}}}} \leq e^{-\alpha \delta} f\lrp{\dist\lrp{y_{k},x_{k}}}
        \elb{e:t:qodwnsak:1}
    \end{alignat*}
    where $f$ is a Lyapunov function satisfying $f(r) \geq \frac{1}{2}\exp\lrp{- \lrp{q + L_{\Ric}} \R^2/2} r$ and $\lrabs{f'(r)} \leq 1$ and $\alpha := \min\lrbb{\frac{m}{16}, \frac{1}{2 \R^2}} \cdot \exp\lrp{- \frac{1}{2}\lrp{q + L_{\Ric}} \R^2}$ are as defined in Lemma \ref{l:g_contraction_without_gradient_lipschitz}.

    Next, we will bound $\E{\ind{A_{k+1}} \dist\lrp{x_{k+1},\bar{x}_{k+1}}}$. We would like to apply Lemma \ref{l:discretization-approximation-lipschitz-derivative} with $L_1 = L_\beta' r$. We verify that under the event $A_k$, $\beta(x_k)$ is indeed bounded by $L_1$. Lemma \ref{l:discretization-approximation-lipschitz-derivative} requires that $\delta$ be upper bounded by a few quantities, most of these are immediately satisfied by our definition of $\delta$ in \eqref{e:t:qodwnsak:0}, but the assumption that $\delta \leq  \frac{1}{16\sqrt{L_R} L_1} =  \frac{1}{16\sqrt{L_R} L_\beta' r}$ is slightly tricky, because our definition of $r$ itself depends on $\delta$; we verify that this bound indeed holds by using Lemma \ref{l:useful_xlogx}.

    Lemma \ref{l:discretization-approximation-lipschitz-derivative} thus guarantees that
    \begin{alignat*}{1}
        \E{\ind{A_k}\dist\lrp{x_{k+1},\bar{x}_{k+1}}} 
        \leq& 2^{9}\lrp{\delta^2 L_1^2+ \delta^2 {L_\beta'}^2 + \delta^{3/2} \lrp{d^{3/2}\lrp{L_R + {L_\beta'}^2/{L_1^2}} + L_\beta' \sqrt{d}}}\\
        \leq& 2^{10}\lrp{\delta^2 {L_\beta'}^2\lrp{1 + r^2} + \delta^{3/2} \lrp{d^{3/2}\lrp{L_R + 1/{r^2}} + L_\beta' \sqrt{d}}}\\
        \leq& \t{O}\lrp{\delta^{3/2}}
        \elb{e:t:qodwnsak:2}
    \end{alignat*}
    where $\t{O}$ hides polynomial dependency on $L_\beta', d, L_R, \R, \frac{1}{m}, \log K$.. Combining, \eqref{e:t:qodwnsak:1} and \eqref{e:t:qodwnsak:2}, and using triangle inequality, together with the fact that $A_{k+1} \subset A_k$, and the fact that $\lrabs{f'} \leq 1$,
    \begin{alignat*}{1}
        \E{\ind{A_{k+1}}f\lrp{\dist\lrp{y_{k+1},x_{k+1}}}} \leq e^{-\alpha \delta} \E{\ind{A_k}f\lrp{\dist\lrp{y_{k},x_{k}}}} + \t{O}\lrp{\delta^{3/2}}
    \end{alignat*}
    Applying the above recursively, we have that
    \begin{alignat*}{1}
        \E{\ind{A_{K}}f\lrp{\dist\lrp{y_{K},x_{K}}}} \leq e^{-\alpha K \delta}\E{f\lrp{\dist\lrp{y_{0},x_{0}}}} + \frac{1}{\alpha} \cdot \t{O}\lrp{\delta^{1/2}}
        \elb{e:t:qodwnsak:4}
    \end{alignat*}
    Combining \eqref{e:t:qodwnsak:3} and \eqref{e:t:qodwnsak:4}, we get
    \begin{alignat*}{1}
        \E{f\lrp{\dist\lrp{y_{K},x_{K}}}} \leq e^{-\alpha K \delta}\E{f\lrp{\dist\lrp{y_{0},x_{0}}}} + \frac{1}{\alpha} \cdot \t{O}\lrp{\delta^{1/2}}
    \end{alignat*}
    Using the fact that $\frac{1}{2}\exp\lrp{- \lrp{q + L_{\Ric}} \R^2/2} r \leq f(r) \leq r$, we have
    \begin{alignat*}{1}
        \E{\dist\lrp{y_{K},x_{K}}} 
        \leq& e^{-\alpha K \delta + \lrp{q + L_{Ric}}\R^2/2}\E{\dist\lrp{y_{0},x_{0}}} + \frac{\exp\lrp{\lrp{q + L_{Ric}}\R^2/2}}{\alpha} \cdot \t{O}\lrp{\delta^{1/2}}\\
        =& e^{-\alpha K \delta + \lrp{q + L_{Ric}}\R^2/2}\E{\dist\lrp{y_{0},x_{0}}} + \exp\lrp{\lrp{q + L_{Ric}}\R^2} \cdot \t{O}\lrp{\delta^{1/2}}
    \end{alignat*}
    where $\t{O}$ hides polynomial dependency on $L_\beta', d, L_R, \R, \frac{1}{m}, \log K$.

    Finally, the constant $\C_0$ in the theorem statement is simply 1 over the right side of \eqref{e:t:qodwnsak:0}.
\end{proof}

\section{Distance Contraction under Kendall Cranston Coupling}
\label{sec:distance-contraction}
In this section, we prove Lemma~\ref{l:g_contraction_without_gradient_lipschitz}, which is the main tool for proving mixing of manifold diffusion processes under the distant dissipativity assumption. We note again that the proof is entirely based on existing results, and is only included for completeness.

\subsection{The Kendall Cranston Coupling}
\label{ss:The Kendall Cranston Coupling}

\begin{lemma}\label{l:kendall-cranston}
    Let $T\in \Re^+$ be some fixed time. Assume that there is are constants $L_\beta, L_\beta'$ such that for all $x,y\in M$, $\lrn{\beta(x)} \leq L_\beta$ and $\lrn{\beta(x) - \party{y}{x} \beta(y)} \leq L_\beta' \lrn{x-y}$. Let $i$ be some integer satisfying 
    $i \geq \max\lrbb{\log_2 \lrp{32T \sqrt{L_R} L_\beta}, \log_2 \lrp{32Td}, \log_2\lrp{32L_\beta T}}$.

    Let $\kappa(r):= \frac{1}{r^2} \sup_{\dist\lrp{x,y} =  r} \lin{\party{y}{x} \beta(y) - \beta(x), \Exp_{x}^{-1}(y)}$. 
    
    Let $x,y\in M$ and let $E^x$ be an arbitrary orthonormal basis of $T_x M$ and let $E^y$ be an arbitrary orthonormal basis of $T_y$.  let $x^i(t):= \overline{\Phi}(t;x,E^x,\beta,\BB^x,i)$ and $y^i(t) := \overline{\Phi}(t;y,E^y,\beta,\BB^y,i)$ where $\BB^x$ and $\BB^y$ are standard Brownian motion in $\Re^d$, and where $\overline{\Phi}$ is as defined in \eqref{d:x^i(t)}. 
    
    For any $\epsilon$, there exists a coupling between $\BB^x$ and $\BB^y$, and Brownian motion $\WW^i$ over $\Re$, such that for all $k\in \lrbb{0...2^i}$,
    \begin{alignat*}{1}
        \dist\lrp{x^i_{k+1}, y^i_{k+1}}^2 
        \leq& \lrp{1+\delta^i \lrp{2\kappa\lrp{\dist\lrp{x^i_{k}, y^i_{k}}} + L_{Ric}}} \dist\lrp{x^i_{k}, y^i_{k}}^2 \\
        &\quad + \ind{\dist\lrp{x^i_{k}, y^i_{k}} > \epsilon}\lrp{4 \delta^i- 4 \dist\lrp{x^i_{k}, y^i_{k}}  \lrp{\WW^i\lrp{(k+1)\delta^i} - \WW^i\lrp{k\delta^{i}} }}\\
        &\quad + \tau^i_k
    \end{alignat*}
    where $\tau^i_k$ satisfies
    \begin{alignat*}{1}
        & \Ep{\F_k}{\lrabs{\tau^i_k}} \leq \C_1 {\delta^i}^{3/2} \lrp{1 + L_\beta^4}\lrp{1+ \dist\lrp{x^i_k,y^i_k}^2}\\
        & \Ep{\F_k}{{{\tau}^i_k}^2 }\leq \C_1 {\lrp{1+\dist\lrp{x^i_k,y^i_k}^4} {\delta^i}^{2}}
    \end{alignat*}
    where $\C_1$ is a constant depending on $L_R, L_R', d, T$.
\end{lemma}

\begin{proof}
    We set up some notation: throughout this proof, consider a fixed $i$. Recall that $\delta^i:= T/2^i$, and assume $i$ is large enough such that $\delta^i \leq \frac{1}{32\sqrt{L_R}L_\beta}$. Let $x^i_k$ be as defined in \eqref{d:x^i_k} so that $x^i_k = x^i(k\delta^i)$. Let us also define $K := 2^i$, so that $T = K\delta^i$.

    \textbf{Step 1: defining the coupling}
    By definition, for any $k\in \lrbb{0...K}$,
    \begin{alignat*}{1}
        & x^{i}_{k+1} := \Exp_{x^{i}_{k}}\lrp{\delta^i \beta\lrp{x^i_k} + {\lrp{\BB^x\lrp{(k+1)\delta^i} - \BB^x\lrp{k\delta^{i}}}} \circ E^{i}_{k}}\\
        & y^{i}_{k+1} := \Exp_{y^{i}_{k}}\lrp{\delta^i \beta\lrp{y^i_k} + {\lrp{\BB^y\lrp{(k+1)\delta^i} - \BB^y\lrp{k\delta^{i}}}} \circ \t{E}^{i}_{k}}
    \end{alignat*}

    Let $\gamma^i_k:[0,1] \to M$ denote a minimizing geodesic from $x^i_k$ to $y^i_k$.

    Let $F^i_k$ be an orthonormal basis at $T_{y^i_k} M$, obtained from the parallel transport of $E^i_k$ along $\gamma^i_k$, i.e. for all $j=1...d$,
    \begin{alignat*}{1}
        F^{i,j}_k = \party{\gamma^i_k}{} E^{i,j}_k
    \end{alignat*}
    Let us define $\MM^i_k \in \Re^{d\times d}$ as matrix whose $a,b$ entry is 
    \begin{alignat*}{1}
        \lrb{\MM^i_k}_{a,b} = \lin{F^{i,a}_k, \t{E}^{i,b}_k}
    \end{alignat*}
    one can verify that $\MM^i_k$ is an orthogonal matrix, and that for all $\vv\in \Re^d$, $\vv \circ F^i_k = \MM \vv \circ \t{E}^i_k$. 

    Let us define $\bar{\nnu}^i_k$ denote the unique coordinates of $\frac{{\gamma^i_k}'(1)}{\lrn{{\gamma^i_k}'(1)}}$ wrt $F^i_k$ (equivalently the coordinates of $\frac{{\gamma^i_k}'(0)}{\lrn{{\gamma^i_k}'(0)}}$ wrt $E^i_k$). We define $\nnu^i_k := \ind{\dist\lrp{x^i_{k}, y^i_{k}} > \epsilon} \bar{\nnu}^i_k$.
    
    We now define a coupling between $\BB^x(t)$ and $\BB^y(t)$ as follows:
    \begin{alignat*}{1}
        \BB^y(t) := \int_{0}^{T} \ind{t\in[k\delta^i,(k+1)\delta^i)}\MM^i_k  \lrp{I - 2 \nnu^i_k {\nnu^i_k}^T}  d \BB^x(t)
    \end{alignat*}
    For this to be a valid coupling, it suffices to verify that \\
    $\int_{0}^{T} \ind{t\in[k\delta^i,(k+1)\delta^i)} \MM^i_k \lrp{I - 2 \nnu^i_k {\nnu^i_k}^T} d \BB^x(t)$ is indeed a standard Brownian motion. This can be done by verifying that the definition satisfies Levy's characterization of Brownian motion. We omit the proof, but highlight two important facts: 1. $\int_{0}^{T} \ind{t\in[k\delta^i,(k+1)\delta^i)} \MM^i_k\lrp{I - 2 \nnu^i_k {\nnu^i_k}^T} $ is adapted to the natural filtration of $\BB^x(t)$, and 2. $\MM^i_k\lrp{I - 2 \nnu^i_k {\nnu^i_k}^T} $ is an orthogonal matrix. We have thus defined a coupling between $\BB^x$ and $\BB^y$, and consequently, a coupling between $x^i(t)$ and $y^i(t)$ for all $t$.

    \textbf{Step 2: Applying Lemma \ref{l:discrete-approximate-synchronous-coupling-ricci}}

    Having defined a coupling between $x^i_k$ and $y^i_k$, we bound $\E{\dist\lrp{x^i_K, y^i_K}^2}$ for $K := T/\delta^i = 2^i$ by applying Lemma \ref{l:discrete-approximate-synchronous-coupling-ricci} , with $x=x^i_k$, $y=y^i_k$, $u= \delta^i \beta\lrp{x^i_k} + {\lrp{\BB\lrp{(k+1)\delta^i} - \BB\lrp{k\delta^{i}}}} \circ E^{i}_{k}$, $v = \delta^i \beta\lrp{y^i_k} + {\lrp{\t{\BB}\lrp{(k+1)\delta^i} - \t{\BB}\lrp{k\delta^{i}}}} \circ \t{E}^{i}_{k}$ and $\gamma := \gamma^i_k$.

    Following the notation in Lemma \ref{l:discrete-approximate-synchronous-coupling-ricci}, let $u(t)$ and $v(t)$ be the parallel transport of $u$ and $v$ along $\gamma(t)$. We verify that $u(s) = \delta^i \party{x^i_k}{\gamma^i_k(s)} \beta(x^i_k) + \lrp{\BB\lrp{(k+1)\delta^i} - \BB\lrp{k\delta^{i}}} \circ \party{x^i_k}{\gamma^i_k(s)} E^i_k$ and that
    \begin{alignat*}{1}
        v(s) 
        =& \delta^i \party{y^i_k}{\gamma^i_k(s)} \beta(y^i_k) +  \MM^i_k\lrp{I - 2 \nnu^i_k {\nnu^i_k}^T}\lrp{\BB\lrp{(k+1)\delta^i} - \BB\lrp{k\delta^{i}}} \circ \party{x^i_k}{\gamma^i_k(s)} \t{E}^i_k\\
        =& \delta^i \party{y^i_k}{\gamma^i_k(s)} \beta(y^i_k) + \lrp{I - 2 \nnu^i_k {\nnu^i_k}^T } \lrp{\BB\lrp{(k+1)\delta^i} - \BB\lrp{k\delta^{i}}} \circ \party{y^i_k}{\gamma^i_k(s)} F^i_k\\
        =& \delta^i \party{y^i_k}{\gamma^i_k(s)} \beta(y^i_k) + \lrp{I - 2 \nnu^i_k {\nnu^i_k}^T }\lrp{\BB\lrp{(k+1)\delta^i} - \BB\lrp{k\delta^{i}}} \circ \party{x^i_k}{\gamma^i_k(s)} E^i_k\\
        =& \delta^i \party{y^i_k}{\gamma^i_k(s)} \beta(y^i_k) + \lrp{\BB\lrp{(k+1)\delta^i} - \BB\lrp{k\delta^{i}}} \circ \party{x^i_k}{\gamma^i_k(s)} E^i_k\\
        &\quad  - 2 \lin{\nnu^i_k, \BB\lrp{(k+1)\delta^i} - \BB\lrp{k\delta^{i}}} \frac{{\gamma^i_k}'(s)}{\lrn{{\gamma^i_k}'(s)}}
    \end{alignat*}
    where the second equality is by definition of $\MM^i_k$, the third equality is by definition of $F^i_k$, the fourth equality is by definition of $\nnu^i_k$ and the fact that $\gamma^i_k$ is a geodesic. It is convenient subsequently to note the following:
    \begin{alignat*}{1}
        & v(s) - u(s)\\
        =& \delta^i \lrp{\party{y^i_k}{\gamma^i_k(s)} \beta(y^i_k) - \party{x^i_k}{\gamma^i_k(s)} \beta(x^i_k)} - 2 \lin{\nnu^i_k, \BB\lrp{(k+1)\delta^i} - \BB\lrp{k\delta^{i}}} \frac{{\gamma^i_k}'(s)}{\lrn{{\gamma^i_k}'(s)}}
    \end{alignat*}
    and
    \begin{alignat*}{1}
        & (1-s) u(s) + s v(s) \\
        =& (1-s) \delta^i \party{x^i_k}{\gamma^i_k(s)} \beta(x^i_k) + s \delta^i \party{y^i_k}{\gamma^i_k(s)} \beta(y^i_k)\\
        &\quad + \lrp{\BB\lrp{(k+1)\delta^i} - \BB\lrp{k\delta^{i}}} \circ \party{x^i_k}{\gamma^i_k(s)} E^i_k \\
        &\quad - 2 s \lin{\nnu^i_k, \BB\lrp{(k+1)\delta^i} - \BB\lrp{k\delta^{i}}} \frac{{\gamma^i_k}'(s)}{\lrn{{\gamma^i_k}'(s)}}
    \end{alignat*}
    
    \textbf{Step 3: Reorganizing Lemma \ref{l:discrete-approximate-synchronous-coupling-ricci}}\\
    With $u,v$ as defined above, Lemma \ref{l:discrete-approximate-synchronous-coupling-ricci} implies that
    \begin{alignat*}{1}
        &\dist\lrp{x^i_{k+1}, y^i_{k+1}}^2 - \dist\lrp{x^i_{k}, y^i_{k}}^2\\
        \leq& 2\lin{{\gamma^i_k}'(0), v(0) - u(0)} + \lrn{v(0) - u(0)}^2 \\
        &\quad -2\int_0^1 \lin{R\lrp{{\gamma^i_k}'(s),(1-s) u(s) + s v(s)}(1-s) u(s) + s v(s),{\gamma^i_k}'(s)} ds \\
        &\quad + \lrp{2\C^2 e^{\C} + 18\C^4 e^{2\C}} \lrn{v(0) - u(0)}^2 + \lrp{18\C^4 e^{2\C} + 4\C'} \dist\lrp{x^i_{k}, y^i_{k}}^2\\
        &\quad + 4 \C^2 e^{2\C} \dist\lrp{x^i_{k}, y^i_{k}} \lrn{v(0) - u(0)}
        \elb{e:l:kendall-cranston:step1.1}
    \end{alignat*}
    where $\C := \sqrt{L_R} \lrp{\lrn{u} + \lrn{v}}$ and $\C' := L_R' \lrp{\lrn{u} + \lrn{v}}^3$.

    Below, we bound each of the terms above
    \begin{alignat*}{2}
        & 2\lin{{\gamma^i_k}'(0), v(0) - u(0)} 
        &&= 2 \delta^i \lin{{\gamma^i_k}'(0), \party{y^i_k}{x^i_k} \beta(y^i_k) - \beta(x^i_k)} - 4 \lrn{{\gamma^i_k}'(0)}\lin{\nnu^i_k, \BB\lrp{(k+1)\delta^i} - \BB\lrp{k\delta^{i}}}\\
        & \lrn{v(0) - u(0)}^2 &&\leq 4 \lin{\nnu^i_k, \BB\lrp{(k+1)\delta^i} - \BB\lrp{k\delta^{i}}}^2\\
        & &&\quad + \underbrace{{\delta^i}^2 L_\beta^2 + 4 \delta^i L_\beta \lin{\nnu^i_k, \BB\lrp{(k+1)\delta^i} - \BB\lrp{k\delta^{i}}}}_{\tau^i_{k,1}}
    \end{alignat*}
    \begin{alignat*}{1}
        &\int_0^1 \lin{R\lrp{{\gamma^i_k}'(s),(1-s) u(s) + s v(s)}(1-s) u(s) + s v(s),{\gamma^i_k}'(s)} ds \\
        \leq& \int_0^1 \lin{R\lrp{{\gamma^i_k}'(s),\lrp{\BB\lrp{(k+1)\delta^i} - \BB\lrp{k\delta^{i}}}\circ \party{x^i_k}{\gamma^i_k(s)} E^i_k }\lrp{\BB\lrp{(k+1)\delta^i} - \BB\lrp{k\delta^{i}}}\circ \party{x^i_k}{\gamma^i_k(s)} E^i_k  , {\gamma^i_k}'(s)} ds\\
        &\quad + \frac{4\lin{\nnu^i_k, \BB\lrp{(k+1)\delta^i} - \BB\lrp{k\delta^{i}}}^2}{\lrn{{\gamma^i_k}'(0)}^2} \int_0^1 s^2 \underbrace{\lin{R\lrp{{\gamma^i_k}'(s), {\gamma^i_k}'(s)} {\gamma^i_k}'(s), {\gamma^i_k}'(s)}}_{=0 \text{ as $R(u,u) = 0$ for all $u$}} ds\\
        &\quad + \underbrace{{\delta^i}^2 L_R \dist\lrp{x^i_k,y^i_k}^2 L_\beta^2 + 4 \delta^i L_R \dist\lrp{x^i_k,y^i_k}^2 L_\beta \lrn{\BB\lrp{(k+1)\delta^i} - \BB\lrp{k\delta^{i}}}_2}_{\tau^i_{k,2}}
    \end{alignat*}

    Finally, we will take the remaining terms, and denote them by
    \begin{alignat*}{1}
        \tau^i_{k,3} :=& \lrp{2\C^2 e^{\C} + 18\C^4 e^{2\C}} \lrn{v(0) - u(0)}^2 \\
        &\quad + \lrp{18\C^4 e^{2\C} + 4\C'} \dist\lrp{x^i_{k}, y^i_{k}}^2 + 4 \C^2 e^{2\C} \dist\lrp{x^i_{k}, y^i_{k}} \lrn{v(0) - u(0)}
    \end{alignat*}

    We claim that under our assumption on $i$,
    \begin{alignat*}{1}
        \Ep{\F_k}{\lrabs{\tau^i_{k,1} + \tau^i_{k,2} + \tau^i_{k,3}}} = O\lrp{{\delta^i}^{3/2} \lrp{1 + L_\beta^4}\lrp{1+ \dist\lrp{x^i_k,y^i_k}^2}}
    \end{alignat*}
    where $O()$ hides dependencies on $L_R, L_R', d, T$.

    We omit the proof for the above claim, which involves some tedious but straightforward algebra, but we note that the proof uses $\E{\lrn{\BB\lrp{(k+1)\delta^i} - \BB\lrp{k\delta^{i}}}_2^j} = O\lrp{{\delta^i}^{j/2}}$ (for all integer $j$) and that $\E{\exp\lrp{a\lrn{\BB\lrp{(k+1)\delta^i} - \BB\lrp{k\delta^{i}}}_2}} \leq 4\exp\lrp{2a^2 \delta^i d} \leq 8$ (for all positive $a$). It is also important to use our assumed upper bound on $\delta^i$ in the Lemma statement.


    We simplify \eqref{e:l:kendall-cranston:step1.1} to
    \begin{alignat*}{1}
        & \dist\lrp{x^i_{k+1}, y^i_{k+1}}^2 - \dist\lrp{x^i_{k}, y^i_{k}}^2\\
        \leq& 2 \delta^i \kappa\lrp{\dist\lrp{x^i_{k}, y^i_{k}}}\dist\lrp{x^i_{k}, y^i_{k}}^2 - 4 \dist\lrp{x^i_{k}, y^i_{k}}\lin{\nnu^i_k, \BB\lrp{(k+1)\delta^i} - \BB\lrp{k\delta^{i}}}\\
        &\quad + 4 \lin{\nnu^i_k, \BB\lrp{(k+1)\delta^i} - \BB\lrp{k\delta^{i}}}^2\\
        &\quad + 2\int_0^1 \lin{R\lrp{{\gamma^i_k}'(s),\lrp{\BB\lrp{(k+1)\delta^i} - \BB\lrp{k\delta^{i}}}\circ \party{x^i_k}{\gamma^i_k(s)} E^i_k }\lrp{\BB\lrp{(k+1)\delta^i} - \BB\lrp{k\delta^{i}}}\circ \party{x^i_k}{\gamma^i_k(s)} E^i_k  , {\gamma^i_k}'(s)} ds\\
        &\quad + \tau^i_{k,1} + \tau^i_{k,2} + \tau^i_{k,3}
        \elb{e:l:kendall-cranston:step1.2}
    \end{alignat*}

    \textbf{Step 4: Pulling out the expectation}\\
    We will further simplify \eqref{e:l:kendall-cranston:step1.2} by replacing a few terms by their expectations. Define
    \begin{alignat*}{1}
        & {\tau}^i_{k,4}:= \int_0^1 \lin{R\lrp{{\gamma^i_k}'(s),\lrp{\BB\lrp{(k+1)\delta^i} - \BB\lrp{k\delta^{i}}}\circ \party{x^i_k}{\gamma^i_k(s)} E^i_k }, \lrp{\BB\lrp{(k+1)\delta^i} - \BB\lrp{k\delta^{i}}}\circ \party{x^i_k}{\gamma^i_k(s)} E^i_k , {\gamma^i_k}'(s)}\\
        &\qquad - \delta^iRic\lrp{{\gamma^i_k}'(s)} ds \\
        & {\tau}^i_{k,5}:= \delta^i - \lin{\nnu^i_k, \BB\lrp{(k+1)\delta^i} - \BB\lrp{k\delta^{i}}}^2\\
    \end{alignat*}

    By definition of Ricci Curvature,  
    \begin{alignat*}{1}
        & \Ep{\F_k}{\int_0^1 \lin{R\lrp{{\gamma^i_k}'(s),\lrp{\BB\lrp{(k+1)\delta^i} - \BB\lrp{k\delta^{i}}}\circ \party{x^i_k}{\gamma^i_k(s)} E^i_k }, \lrp{\BB\lrp{(k+1)\delta^i} - \BB\lrp{k\delta^{i}}}\circ \party{x^i_k}{\gamma^i_k(s)} E^i_k , {\gamma^i_k}'(s)} ds}\\
        =& \int \delta^iRic\lrp{{\gamma^i_k}'(s)} ds
    \end{alignat*}
    By definition of $\nnu^i_k$,  $\E{\lin{\nnu^i_k, \BB\lrp{(k+1)\delta^i} - \BB\lrp{k\delta^{i}}}^2} = \delta^i \ind{\dist\lrp{x^i_k,y^i_k}>\epsilon}$. 
    
    Let $\tau^i_k:= \tau^i_{k,1}+\tau^i_{k,2}+\tau^i_{k,3}$.
    We can thus further simplify \eqref{e:l:kendall-cranston:step1.2} to
    \begin{alignat*}{1}
        & \dist\lrp{x^i_{k+1}, y^i_{k+1}}^2 - \dist\lrp{x^i_{k}, y^i_{k}}^2\\
        \leq& 2 \delta^i \kappa\lrp{\dist\lrp{x^i_{k}, y^i_{k}}}\dist\lrp{x^i_{k}, y^i_{k}}^2 - 4 \dist\lrp{x^i_{k}, y^i_{k}}\lin{\nnu^i_k, \BB\lrp{(k+1)\delta^i} - \BB\lrp{k\delta^{i}}}\\
        &\quad + 4 \delta^i \ind{\dist\lrp{x^i_k,y^i_k}>\epsilon} +  2\delta^i \int_0^1Ric\lrp{{\gamma^i_k}'(s)} ds\\
        &\quad + \tau^i_k\\
        \leq& \delta^i \lrp{2\kappa\lrp{\dist\lrp{x^i_{k}, y^i_{k}}} + 2L_{Ric}} \dist\lrp{x^i_{k}, y^i_{k}}^2\\
        &\quad - 4 \dist\lrp{x^i_{k}, y^i_{k}}\lin{\nnu^i_k, \BB\lrp{(k+1)\delta^i} - \BB\lrp{k\delta^{i}}} + 4 \delta^i \ind{\dist\lrp{x^i_k,y^i_k}>\epsilon}\\
        &\quad + \tau^i_k
        \elb{e:l:kendall-cranston:step2}
    \end{alignat*}  
    the conclusion follows by defining $\WW^i(t):= \int_0^t \ind{t\in[k\delta^i,(k+1)\delta^i]}\lin{\bar{\nnu}^i_k, \BB\lrp{(k+1)\delta^i} - \BB\lrp{k\delta^{i}}}$ and verifying that it is a Brownian motion. (Recall our definition that $\nnu^i_k := \ind{\dist\lrp{x^i_{k}, y^i_{k}} > \epsilon} \bar{\nnu}^i_k$)
\end{proof}

\subsection{Lyapunov function and its smooth approximation}
\label{ss:Lyapunov function and its smooth approximation}
In this section, we consider a Lyapunov function $f$ taken from \cite{eberle2016reflection}. By analyzing how $f(\dist\lrp{x^i_k,y^i_k})$ evolves under the dynamic in Lemma \ref{l:kendall-cranston}, one can demonstrate that the distance function contracts.

Let $\L,\R\in \Re^+$. We will see later that $\L$ and $\R$ will correspond to distant-dissipativity parameters in \eqref{ass:distant-dissipativity}.

Let $\epsilon\in[0,\infty)$. One should think of $\epsilon$ as being arbitrarily small, as eventually we are only interested in the limit as $\epsilon \to 0$.
    
Define functions $\psi_\epsilon(r)$, $\Psi_\epsilon(r)$ and $\nu(r)$, all from $ \Re^+$ to $\Re$:
\begin{align*}
&\mu_\epsilon(r) = \threecase{1}{r\leq \R}{1 - (r-\R)/\lrp{\epsilon}}{r\in{\R, \R + \epsilon}}{0}{r \geq \R + \epsilon}\\
&\nu_\epsilon(r) := 1- \frac{1}{2}
\frac{\int_0^{r}\frac{ \mu_\epsilon(s) \Psi_\epsilon(s)}{\psi_\epsilon(s)} ds}{\int_0^{\infty}\frac{ \mu_\epsilon(s) \Psi_\epsilon(s)}{\psi_\epsilon(s)}ds}
&\psi_\epsilon(r) := e^{- \frac{\L \int_0^r r \mu_\epsilon(r) dr}{2}}\\
&\Psi_\epsilon(r) := \int_0^r \psi_\epsilon(s) ds, \\
\end{align*}

We defined an $\epsilon$-smoothed Lyapunov function as
\begin{definition}
    \label{d:f_epsilon}
    \begin{alignat*}{1}
        & f_\epsilon(r):= \int_0^r \psi_\epsilon(s) \nu_\epsilon(s) ds\\
        & g_\epsilon(s) = f_\epsilon\lrp{\sqrt{s + \epsilon}}
    \end{alignat*}
\end{definition}
The case when $\epsilon=0$ (when there is no smoothing) will be of particular interest to us:
\begin{definition}
    \label{d:f}
    \begin{alignat*}{1}
        & f(r):= f_0(r) = g_0(r)
    \end{alignat*}
\end{definition}

\begin{remark}
    The Lyapunov function from \cite{eberle2016reflection} is more general, but for the specific case of $\L , \R$ distant dissipative functions, it is equal to $f$ as defined in \eqref{d:f}.
\end{remark}
\begin{lemma}
    \label{l:fproperties}
    Assume $\epsilon \in (0, \min\lrbb{1/4,1/(4\sqrt{\L}),1/\lrp{4\L \R}}]$
    \begin{alignat*}{2}
        &1.\ f_\epsilon(r) \in [\frac{1}{2}\exp\lrp{-  (1+\epsilon) \L \R^2/2} r, r] \qquad && \text{for all $r$} \\
        &2.\ f_\epsilon'(r) \in [\frac{1}{2}\exp\lrp{-  (1+\epsilon) \L \R^2/2}, 1] \qquad && \text{for all $r$} \\
        &3.\ f_\epsilon''(r) \in [-4{\L}^{3/2}, 0] \qquad && \text{for all $r$} \\
        &4.\ f_\epsilon''(r) + \L r f_\epsilon'(r)  \leq -\frac{\exp\lrp{- (1+\epsilon)\L \R^2/2}}{\lrp{1+\epsilon}^2 \R^2} f_\epsilon(r) \qquad && \text{for $r\in [0,\R]$}\\
        &5.\ \lrabs{f_\epsilon'''(r)} \leq \frac{256 \sqrt{\L}}{\epsilon} \qquad && \text{for all $r$}
    \end{alignat*}
\end{lemma}
\begin{proof}
    We can verify that
    \begin{alignat*}{2}
        & f_\epsilon'(r) &&= \psi_\epsilon(r) \nu_\epsilon(r)\\
        & f_\epsilon''(r) &&= \psi_\epsilon'(r) \nu_\epsilon(r) + \psi_\epsilon(r) \nu_\epsilon'(r)\\
        & &&= - \L \mu_\epsilon(r) r\psi_\epsilon(r) \nu_\epsilon(r) + \psi_\epsilon(r) \nu_\epsilon'(r)\\
        & f_\epsilon'''(r) &&= -\L \psi_\epsilon(r) \nu_\epsilon(r) + \L r \psi_\epsilon(r) \mu_\epsilon'(r) + \L^2 r^2 \psi_\epsilon(r) \nu_\epsilon(r) - 2\L r \psi_\epsilon(r) \nu_\epsilon'(r) + \psi_\epsilon(r) \nu_\epsilon''(r)
    \end{alignat*}
    1. follows from integrating 2.
    
    2. follows from $\nu_\epsilon(r) \in [1/2,1]$ and $\psi_\epsilon \in [\exp\lrp{-  (1+\epsilon) \L \R^2/2}, 1]$ and the expression for $f_\epsilon'(r)$ above.

    3. follows from $\mu_\epsilon, \psi_\epsilon, \nu_\epsilon \geq 0$ and $\nu_\epsilon' \leq 0$, and the fact that $r \psi_\epsilon(r) \leq 2 \sqrt{\L}$ and \eqref{e:t:qoidmas}.

    4. is a little more involved. First note that over $r\in [0,\R]$, $\mu_\epsilon(r) = 1$. This will simplify some calculations. From the expression for $f_\epsilon''$ above, we verify
    \begin{alignat*}{1}
        f_\epsilon''(r) + \L r f_\epsilon'(r) 
        = \psi_\epsilon(r) \nu_\epsilon'(r)
        = - \frac{\Psi_\epsilon(r)}{2\int_0^{\infty}\frac{ \mu_\epsilon(s) \Psi_\epsilon(s)}{\psi_\epsilon(s)}ds}
    \end{alignat*}
    We can bound the denominator as
    \begin{alignat*}{1}
        & \int_0^{\infty}\frac{ \mu_\epsilon(s) \Psi_\epsilon(s)}{\psi_\epsilon(s)}ds
        \leq \int_0^{\R + \epsilon}\frac{\Psi_\epsilon(s)}{\psi_\epsilon(s)}ds
        \leq \frac{\int_0^{\R + \epsilon} \Psi_\epsilon(s) ds}{\psi\lrp{\R + \epsilon}} 
        \leq \frac{(1+\epsilon)^2\R^2}{2\exp\lrp{-\L (1+\epsilon) \R^2/2}}\\
    \end{alignat*}
    where the first inequality is by $\mu_\epsilon(s) \leq 1$, and $\mu_\epsilon(r) = 0$ for $r\geq \R + \epsilon$ the second inequality is by $\psi_\epsilon(r)$ being monotonically decreasing, and the third inequality is by $\Psi_\epsilon(r) \leq r$.Finally, note that $\Psi_\epsilon(r) \geq f_\epsilon(r)$. Put together,
    \begin{alignat*}{1}
        f_\epsilon''(r) + \L r f_\epsilon'(r)  \leq -\frac{\exp\lrp{- (1+\epsilon)\L \R^2/2}}{\lrp{1+\epsilon}^2 \R^2} f_\epsilon(r)
    \end{alignat*}

    We now prove the bound for 5. It is useful to recall that $\psi_\epsilon(r) \leq 1$ and $\nu_\epsilon(r) \leq 1$. 
    \begin{alignat*}{1}
        \Psi_\epsilon(r) = \int_0^r \exp\lrp{-\L s^2} ds \leq \frac{4}{\sqrt{\L}}
    \end{alignat*}
    \begin{alignat*}{1}
        \int_0^{\infty}\frac{ \mu_\epsilon(s) \Psi_\epsilon(s)}{\psi_\epsilon(s)}ds
        \geq& \int_0^{\R}\frac{\Psi_\epsilon(s)}{\psi_\epsilon(s)}ds
        \geq \frac{1}{2}\int_0^{1/\sqrt{2\L}} \Psi_\epsilon(s) ds \geq \frac{1}{16\L}
        \elb{e:t:qoidmas}
    \end{alignat*}
    \begin{alignat*}{1}
        \lrabs{\psi_\epsilon(r) \nu'_\epsilon(r)} \leq \frac{\Psi\lrp{\R + \epsilon}}{ 2\int_0^{\infty}\frac{ \mu_\epsilon(s) \Psi_\epsilon(s)}{\psi_\epsilon(s)}ds}
        \leq 8 \sqrt{\L}
    \end{alignat*}

    For $r\in [0, \R + \epsilon]$ ($\nu_\epsilon'' = 0$ outside this range),
    \begin{alignat*}{1}
        \lrabs{\psi_\epsilon(r) \nu_\epsilon''(r)} \leq \frac{\frac{1}{\epsilon}\Psi\lrp{r} + r \psi_\epsilon(r)/\epsilon + \psi_\epsilon(r) + 2r\Psi_\epsilon(r)/\psi_\epsilon(r)}{ 2\int_0^{\infty}\frac{ \mu_\epsilon(s) \Psi_\epsilon(s)}{\psi_\epsilon(s)}ds} \leq 32\L \cdot \Psi_\epsilon(r) \cdot \lrp{\frac{2}{\epsilon} + 2\L \R} \leq \frac{128 \sqrt{\L}}{\epsilon}
    \end{alignat*}

    We can thus bound $\lrabs{f'''(r)}$ as
    \begin{alignat*}{1}
        \lrabs{f_\epsilon'''(r)} \leq 2\L + 16 \L^{3/2} \R + \frac{128 \sqrt{\L}}{\epsilon} \leq \frac{256 \sqrt{\L}}{\epsilon}
    \end{alignat*}
\end{proof}

\begin{lemma}
    \label{l:gproperties}
    Assume $\epsilon \in (0, \min\lrbb{1/4,1/(4\sqrt{\L}),1/\lrp{4\L \R}}]$
    \begin{alignat*}{2}
        &1.\ {g_\epsilon'(s)} = \frac{1}{2\sqrt{s+\epsilon}} f_\epsilon'(\sqrt{s+\epsilon})\\
        &2.\ {g_\epsilon''(s)} = \frac{1}{4\lrp{s+\epsilon}} f_\epsilon''(\sqrt{s+\epsilon}) - \frac{1}{4\lrp{s+\epsilon}^{3/2}} f_\epsilon'(\sqrt{s+\epsilon})\\
        &3.\ {g_\epsilon'''(s)} = \frac{1}{8\lrp{s+\epsilon}^{3/2}} f_\epsilon'''(\sqrt{s+\epsilon}) - \frac{1}{8\lrp{s+\epsilon}^{2}} f_\epsilon''(\sqrt{s+\epsilon}) + \frac{1}{6\lrp{s+\epsilon}^{5/2}} f_\epsilon'(\sqrt{s+\epsilon})\\
        &4.\ \lrabs{g_\epsilon'''(s)} \leq O\lrp{\epsilon^{-5/2}} \qquad \text{for all $s$}
    \end{alignat*}
    where $O()$ notation hides dependency on $\L$ and $\R$.
\end{lemma}
\begin{proof}
    The first 3 points follow from chain rule.

    The last point follows from point 5 from Lemma \ref{l:fproperties}. 
    \begin{alignat*}{1}
        \lrabs{g_\epsilon'''(s)} \leq \frac{64\sqrt{\L}}{\epsilon^{5/2} \R} + \frac{\sqrt{\L}}{\epsilon^2} + \frac{1}{\epsilon^{5/2}}
    \end{alignat*}
\end{proof}

\subsection{Contraction of Lyapunov Function under Kendall Cranston Coupling}
\label{ss:Evolution_of_Lyapunov_Function_under_Kendall_Cranston_Coupling}

\begin{lemma}
    \label{l:g_epsiilon_evolution_beta_lipschitz}
    Consider the same setup as Lemma \ref{l:kendall-cranston}. Let $\kappa(r):= \frac{1}{r^2} \sup_{\dist\lrp{x,y} =  r} \lin{\party{y}{x} \beta(y) - \beta(x), \Exp_{x}^{-1}(y)}$. Assume there exists $\R\geq 0, q\leq 0$ such that $\kappa(r) \leq q$ for all $r\leq \R$. Let $\L = q + L_{Ric}$. Let  $\epsilon \in (0, \min\lrbb{1/4,1/(4\sqrt{\L}),1/\lrp{4\L \R}}]$. Let $g_\epsilon$ be as defined in \ref{d:f_epsilon} with parameters $\L$ and $\R$. Let $\F_k$ denote the natural filtration generated by $x^i_k$ and $y^i_k$.
    
    There exists a constant $c_1$, depending on $L_\beta, L_\beta', L_R, T, d$, and some constant $c_2$, depending on $L_\beta', L_{Ric}, \R$ such that for any $i>c_1$ and $\epsilon>c_2$, there exists a coupling between $x^i_k$ and $y^i_k$ such that
    \begin{alignat*}{1}
        &\E{g_\epsilon\lrp{\dist\lrp{x^i_{k+1}, y^i_{k+1}}^2}} \\
        \leq& \E{\ind{r > \R}\delta^i \lrp{\lrp{\kappa(r_k) + L_{Ric}}\exp\lrp{-  (1+\epsilon) \L \R^2/2}/8} g_\epsilon\lrp{\dist\lrp{x^i_{k+1}, y^i_{k+1}}^2}}\\
        &\quad - \frac{\exp\lrp{- (1+\epsilon)\L \R^2/2}}{2\lrp{1+\epsilon}^2 \R^2} \delta^i \E{\ind{r \leq \R} g_\epsilon\lrp{\dist\lrp{x^i_{k+1}, y^i_{k+1}}^2}} + O\lrp{\delta^i{\epsilon^{1/2}} + \epsilon^{-5/2} {\delta^i}^{3/2}}
    \end{alignat*}
    where $O\lrp{}$ hides dependency on $L_R, L_\beta', T, d$.
\end{lemma}
\begin{proof}
    Let us define, for convenience, $r_k := \dist\lrp{x^i_{k}, y^i_{k}}$. By Lemma \ref{l:kendall-cranston}, for any $i$ and any $\epsilon$, there exists a coupling satisfying
    \begin{alignat*}{1}
        r_{k+1}^2 
        \leq& \lrp{1+\delta^i \lrp{2\kappa\lrp{r_k} + 2L_{Ric}}} r_k^2 \\
        &\quad +\ind{r_k > \epsilon^{1/3}} \lrp{4 \delta^i - 4 r_k \WW^i\lrp{(k+1)\delta^i} - \WW^i\lrp{k\delta^{i}}} + \tau^i_k
    \end{alignat*}
    where $\tau^i_k$ satisfies
    \begin{alignat*}{1}
        & \E{\lrabs{\tau^i_k}} \leq O \lrp{{\delta^i}^{3/2} \lrp{1 + L_\beta^4}\lrp{1+ r_k^2}} \qquad \E{{{\tau}^i_k}^2 }\leq O\lrp{{\lrp{1+r_k^4} {\delta^i}^{2}}}
    \end{alignat*}
    where $O\lrp{}$ hides dependencies on $L_R, L_R', d, T$.

    By third order Taylor expansion,
    \begin{alignat*}{1}
        & \E{g_{\epsilon}\lrp{r_{k+1}^2}}\\
        =& \E{g_{\epsilon}\lrp{r_k^2}}\\
        &\quad + \E{g_\epsilon'\lrp{r_k^2} \cdot  \lrp{\delta^i \lrp{2\kappa\lrp{r_k} + 2L_{Ric}}} r_k^2}\\
        &\quad + \E{g_\epsilon'\lrp{r_k^2} \cdot 4\delta^i }\\
        &\quad + \E{\frac{1}{2} g_\epsilon''\lrp{r_k^2} \cdot \lrp{ 4 r_k \ind{r_k > \epsilon^{1/3}} \lrp{\WW^i\lrp{(k+1)\delta^i} - \WW^i\lrp{k\delta^{i}} }}^2}\\
        &\quad + O\lrp{\epsilon^{-5/2} {\delta^i}^{3/2}}
        \elb{e:t:qknqoiwdn:0}
    \end{alignat*}
    Note that the last line uses two facts:
    \begin{enumerate}
        \item From Lemma \ref{l:near_tail_bound_one_step}, for any $j$, there exists a constant $\C$, depending on $T,d, L_R, L_\beta', L_\xi$, but independent of $L_\beta$, such that for all $i,k$, $\E{\dist\lrp{x^i_k,x_0}^{2j}} < \C$ and $\E{\dist\lrp{y^i_k,y_0}^{2j}} < \C$.
        \item Roughly speaking, $\E{\dist\lrp{x^i_{k+1},y^i_{k+1}}^2 - \dist\lrp{x^i_{k},y^i_{k}}^2} = O\lrp{{\delta^i}^{3/2}}$. More specifically:
        \begin{alignat*}{1}
            & \lrabs{\dist\lrp{x^i_{k+1},y^i_{k+1}} - \dist\lrp{x^i_{k},y^i_{k}}}\\
            & \leq 2\dist\lrp{x^i_k x^i_{k+1}} + 2\dist\lrp{y^i_k y^i_{k+1}} \\
            & \leq 2\delta^i\lrp{\lrn{\beta(x_0)} + \lrn{\beta(y_0)} + L_\beta' \dist\lrp{\dist\lrp{x^i_k,x_0}} + L_\beta' \dist\lrp{\dist\lrp{x^i_k,x_0}}}\\
            &\quad + 4 \lrn{\BB((k+1)\delta^i) - \BB(k\delta^i)}_2
        \end{alignat*}
    \end{enumerate}

    Plugging in the definition of $g_\epsilon'$ and $g_\epsilon''$,
    \begin{alignat*}{1}
        & g_\epsilon'\lrp{r_k^2} \cdot  \lrp{\delta^i \lrp{\lrp{2\kappa\lrp{r_k} + 2L_{Ric}}} r_k^2+ 4\ind{r_k > \epsilon^{1/3}}\delta^i}\\
        =& \frac{\delta^i}{2\sqrt{r_k^2 + \epsilon}} f_\epsilon'\lrp{\sqrt{r_k^2+\epsilon}} \lrp{\lrp{2\kappa\lrp{r_k} + 2L_{Ric}} r_k^2 + 4\ind{r_k > \epsilon^{1/3}}}\\
        \leq&  \frac{\delta^i}{2\sqrt{r_k^2 + \epsilon}} f_\epsilon'\lrp{\sqrt{r_k^2+\epsilon}} \lrp{\lrp{2\kappa\lrp{r_k} + 2L_{Ric}} r_k^2} + 2\ind{r_k> \epsilon^{1/3}} \frac{\delta^i f_\epsilon'\lrp{\sqrt{r_k^2+\epsilon}}}{\sqrt{r_k^2 + \epsilon}}
    \end{alignat*}
    where we use the assumption that $\epsilon \leq \frac{1}{4\R^2}$ and $m \geq 2L_{Ric}$ and $\epsilon < 1$.

    On the other hand, 
    \begin{alignat*}{1}
        & \Ep{\F_k}{\frac{1}{2} g_\epsilon''\lrp{r_k^2} \cdot \lrp{ 4 r_k \ind{r_k > \epsilon^{1/3}} \lrp{\WW^i\lrp{(k+1)\delta^i} - \WW^i\lrp{k\delta^{i}} }}^2}\\
        =& 8 \delta^i r_k^2 g_\epsilon''\lrp{r_k^2} \cdot \ind{r_k>\epsilon^{1/3}}\\
        =& \ind{r_k>\epsilon^{1/3}}\frac{2 \delta^i r_k^2}{r_k^2 + \epsilon} f_\epsilon''\lrp{\sqrt{r_k^2+\epsilon}} - \ind{r_k>\epsilon^{1/3}} \frac{2\delta^i r_k^2f_\epsilon'\lrp{\sqrt{r_k^2+\epsilon}}}{\lrp{r_k^2 + \epsilon}^{3/2}}
    \end{alignat*}

    Note that $r_k > \epsilon^{1/3}$ implies that $\frac{r_k^2}{\lrp{r_k^2 + \epsilon}^{3/2}} \geq \frac{1}{1+\epsilon^{1/3}}$. Thus
    \begin{alignat*}{1}
        2\ind{r_k> \epsilon^{1/3}} \frac{\delta^i f_\epsilon'\lrp{\sqrt{r_k^2+\epsilon}}}{\sqrt{r_k^2 + \epsilon}}- \ind{r_k>\epsilon^{1/3}} \frac{2\delta^i r_k^2f_\epsilon'\lrp{\sqrt{r_k^2+\epsilon}}}{\lrp{r_k^2 + \epsilon}^{3/2}}\leq 4\delta^i \epsilon^{1/3}
        \elb{e:t:qknqoiwdn:1}
    \end{alignat*}
    where we use the fact that $\lrabs{f_\epsilon'}\leq 1$.

    We now bound $\frac{\delta^i}{2\sqrt{r_k^2 + \epsilon}} f_\epsilon'\lrp{\sqrt{r_k^2+\epsilon}} \lrp{\lrp{2\kappa\lrp{r_k} + 2L_{Ric}} r_k^2} + \ind{r_k>\epsilon^{1/3}}\frac{2 \delta^i r_k^2}{r_k^2 + \epsilon} f_\epsilon''\lrp{\sqrt{r_k^2+\epsilon}}$. Consider three cases:
    \begin{enumerate}
        \item $r_k \leq \epsilon^{1/3}$: 
        \begin{alignat*}{1}
            & \frac{\delta^i}{2\sqrt{r_k^2 + \epsilon}} f_\epsilon'\lrp{\sqrt{r_k^2+\epsilon}} \lrp{\lrp{2\kappa\lrp{r_k} + 2L_{Ric}} r_k^2}
            \leq {\delta^i \lrp{q + L_{Ric}}}{\epsilon^{1/2}} 
        \end{alignat*}
        \item $r_k \in (\epsilon^{1/3},\R]$: 
        \begin{alignat*}{1}
            & \frac{\delta^i}{2\sqrt{r_k^2 + \epsilon}} f_\epsilon'\lrp{\sqrt{r_k^2+\epsilon}} \lrp{\lrp{2\kappa\lrp{r_k} + 2L_{Ric}} r_k^2} + \frac{2 \delta^i r_k^2}{r_k^2 + \epsilon} f_\epsilon''\lrp{\sqrt{r_k^2+\epsilon}}\\
            \leq& \frac{\delta^i r_k^2}{{r_k^2 + \epsilon}}\lrp{\L f_\epsilon'\lrp{\sqrt{r_k^2 + \epsilon}}\sqrt{r_k^2 + \epsilon} + 2f_\epsilon''\lrp{\sqrt{r_k^2+\epsilon}}}\\
            \leq& - \frac{\exp\lrp{- (1+\epsilon)\L \R^2/2}}{2\lrp{1+\epsilon}^2 \R^2}  \delta^i f_\epsilon\lrp{\sqrt{r_k^2 + \epsilon}}
        \end{alignat*}
        where we use Lemma \ref{l:fproperties} and the definition of $\L$.
        \item $r_k > \R$: We use the fact that $f_\epsilon''(r) \leq 0$ for all $r \geq \R \geq \epsilon$. Thus
        \begin{alignat*}{1}
            & \frac{\delta^i}{2\sqrt{r_k^2 + \epsilon}} f_\epsilon'\lrp{\sqrt{r_k^2+\epsilon}} \lrp{\lrp{2\kappa\lrp{r_k} + 2L_{Ric}} r_k^2} + \frac{2 \delta^i r_k^2}{r_k^2 + \epsilon} f_\epsilon''\lrp{\sqrt{r_k^2+\epsilon}}\\
            \leq& \frac{\delta^i}{2\sqrt{r_k^2 + \epsilon}} f_\epsilon'\lrp{\sqrt{r_k^2+\epsilon}} \lrp{\lrp{2\kappa\lrp{r_k} + 2L_{Ric}} r_k^2}\\
            \leq& - \frac{\delta^i \lrp{\lrp{\kappa(r_k) + L_{Ric}}} r_k^2 f_\epsilon'\lrp{\sqrt{r_k^2 + \epsilon}}}{8\sqrt{r_k^2 + \epsilon}}\\
            \leq& - \frac{1}{8}\delta^i \lrp{\lrp{\kappa(r_k) + L_{Ric}}\exp\lrp{-  (1+\epsilon) \L \R^2/2}} f_\epsilon\lrp{\sqrt{r_k^2 + \epsilon}}
        \end{alignat*}
    \end{enumerate}
\end{proof}

\begin{proof}[Proof of Lemma \ref{l:g_contraction_without_gradient_lipschitz}]
    Lemma \ref{l:g_epsiilon_evolution_beta_lipschitz} almost gives us what we need. However, because we assumed that $\beta$ satisfies Assumption \ref{ass:distant-dissipativity}, the assumption that $\lrn{\beta(x)} \leq L_\beta$ cannot possibly hold. We thus need to approximate $\beta$ by a sequence of increasingly non-Lipschitz functions.

    Consider a fixed $i$.

    Let $s^j$ be a sequence of increasing radius, such that $s^j \to \infty$ as $j\to \infty$. Let $\beta^j$ denote the truncation of $\beta$ to norm $s^j$, i.e.
    \begin{alignat*}{1}
        \beta^j(x):= \twocase{\beta(x)}{\lrn{\beta(x)} \leq s^j}{\beta(x)\cdot\frac{s^j}{\lrn{\beta(x)}}}{\lrn{\beta(x)} > s^j}
    \end{alignat*}
    We verify that $\beta^j$ also satisfies Assumption \ref{ass:beta_lipschitz} with the same $L_\beta'$ as $\beta$.

    Recall that we defined
    \begin{alignat*}{1}
        & x^i(t) := \overline{\Phi}(t;x_0,E^x,\beta,\BB^x,i) \\
        & y^i(t) := \overline{\Phi}(t;y_0,E^y,\beta,\BB^y,i)\\
        & x(t) := {\Phi}(t;x_0,E^x,\beta,\BB^x) \\
        & y(t) := {\Phi}(t;y_0,E^y,\beta,\BB^y)
    \end{alignat*}
    where $\Phi$ and $\overline{\Phi}$ are as defined in \eqref{d:x^i(t)} and \eqref{d:x(t)}, and where $\BB^x$ and $\BB^y$ are two Brownian motions which may be coupled in a non-trivial way. Furthermore, define, for all $i$,
    \begin{alignat*}{1}
        & \t{x}^{i,j}(t) := \overline{\Phi}(t;x_0,E^x,\beta^j,\BB^x,i) \\
        & \t{y}^{i,j}(t) := \overline{\Phi}(t;y_0,E^y,\beta^j,\BB^y,i) \\
        & \t{x}^{\cdot,j}(t) := {\Phi}(t;x_0,E^x,\beta^j,\BB^x) \\
        & \t{y}^{\cdot,j}(t) := {\Phi}(t;y_0,E^y,\beta^j,\BB^y) 
    \end{alignat*}
    note that the above definition implies a non-trivial coupling between $\t{x}^{i,j}(t)$ and $x^{i}(t)$ for all $j$, via the shared Brownian motion $\BB^x$.

    Let $L_0 := \max\lrbb{\lrn{\beta(x_0)},\lrn{\beta(y_0)}}$. Let us define ${\t{r}^{i,j}_k} := \dist\lrp{\t{x}^{i,j}_k,\t{y}^{i,j}_k}$. Using Assumption \ref{ass:distant-dissipativity} and the assumption that $m \geq 2L_{Ric}$, we verify that
    \begin{alignat*}{1}
        & \ind{{\t{r}^{i,j}_k} > \R}\lrp{\kappa({\t{r}^{i,j}_k}) + L_{Ric}} \\
        <& \ind{\R < {\t{r}^{i,j}_k} \leq \frac{s^j - L_0}{L_\beta'}} \lrp{- m/2} + \ind{\R < {\t{r}^{i,j}_k}, \frac{s^j - L_0}{L_\beta'}\leq {\t{r}^{i,j}_k}} \lrp{\frac{s^j}{{\t{r}^{i,j}_k}} + L_{Ric}}
    \end{alignat*}

    
    Let $q,\R$ be the parameters in Assumption \ref{ass:distant-dissipativity}. This implies that $\kappa(r) \leq q$ for all $r\leq \R$. Let $\L := q + L_{Ric}$ and $\epsilon$ be as defined in Lemma \ref{l:g_epsiilon_evolution_beta_lipschitz}. Then by Lemma \ref{l:g_epsiilon_evolution_beta_lipschitz}:
    \begin{alignat*}{1}
        &\E{g_\epsilon\lrp{r_{k+1}^2}} \\
        \leq& \E{\ind{r > \R}\delta^i \lrp{\lrp{\kappa({\t{r}^{i,j}_k}) + L_{Ric}}\exp\lrp{-  (1+\epsilon) \L \R^2/2}/8} g_\epsilon\lrp{r_{k+1}^2}} \\
        &\quad - \frac{\exp\lrp{- (1+\epsilon)\L \R^2/2}}{2\lrp{1+\epsilon}^2 \R^2} \delta^i \E{\ind{r \leq \R} g_\epsilon\lrp{r_{k+1}^2}} + O\lrp{\delta^i{\epsilon^{1/2}} + \epsilon^{-5/2} {\delta^i}^{3/2}}\\
        \leq& -\frac{\delta^i m\exp\lrp{-  (1+\epsilon) \L \R^2/2}}{16} \E{\ind{\R < {\t{r}^{i,j}_k} \leq \frac{s^j - L_0}{L_\beta'}} g_\epsilon\lrp{r_{k+1}^2}} \\
        &\quad - \frac{\delta^i \exp\lrp{- (1+\epsilon)\L \R^2/2}}{2\lrp{1+\epsilon}^2 \R^2} \E{\ind{r \leq \R} g_\epsilon\lrp{r_{k+1}^2}}\\
        &\quad + \delta^i \E{\ind{\frac{s^j - L_0}{L_\beta'}\leq {\t{r}^{i,j}_k}} \lrp{{s^j}+ L_{Ric}{\t{r}^{i,j}_k}}} + O\lrp{\delta^i{\epsilon^{1/2}} + \epsilon^{-5/2} {\delta^i}^{3/2}}\\
        \leq& - \alpha_\epsilon \delta^i \E{g_\epsilon\lrp{r_{k+1}^2}} + \delta^i \E{\ind{\frac{s^j - L_0}{L_\beta'}\leq {\t{r}^{i,j}_k}} \lrp{{s^j}+ \lrp{m + L_{Ric}}{\t{r}^{i,j}_k}}} + O\lrp{\delta^i{\epsilon^{1/2}} + \epsilon^{-5/2} {\delta^i}^{3/2}}\\
        \leq& - \alpha_\epsilon \delta^i \E{g_\epsilon\lrp{r_{k+1}^2}} + \delta^i \E{\ind{\frac{s^j - L_0}{L_\beta'}\leq {\t{r}^{i,j}_k}} \lrp{{s^j}+ \lrp{m + L_{Ric}}{\t{r}^{i,j}_k}}} + O\lrp{\delta^i{\epsilon^{1/2}} + \epsilon^{-5/2} {\delta^i}^{3/2}}
    \end{alignat*}
    where we define $\alpha_\epsilon := \min\lrbb{\frac{m}{16}, \frac{1}{2\lrp{1+\epsilon}^2 \R^2}}\exp\lrp{-  \frac{1}{2}(1+\epsilon) \L \R^2}$.

    Applying the above recursively for $k=0...K$, where $K = T/2^i$, we get that
    \begin{alignat*}{1}
        &\E{g_{\epsilon}\lrp{\lrp{\t{r}^{i,j}_K}^2}} \\
        \leq& \exp\lrp{-\alpha_\epsilon K\delta^i}\E{g_{\epsilon}\lrp{r_{0}^2}} + O\lrp{T \epsilon^{1/2} + T\epsilon^{-5/2} {\delta^i}^{1/2}}\\
        &\quad + \E{\ind{\frac{s^j - L_0}{L_\beta'}\leq \sup_{k\leq K} {\t{r}^{i,j}_k}} \sum_{k=0}^{K} \delta^i {\t{r}^{i,j}_k}}\\
        \leq& \exp\lrp{-\alpha_\epsilon K\delta^i}\E{g_{\epsilon}\lrp{r_{0}^2}} + O\lrp{T \epsilon^{1/2} + T\epsilon^{-5/2} {\delta^i}^{1/2}}\\
        &\quad + \Pr{\frac{s^j - L_0}{L_\beta'}\leq \sup_{k\leq K} {\t{r}^{i,j}_k}}^{1/2} \sqrt{\E{\lrp{\sum_{k=0}^{K} \delta^i {\t{r}^{i,j}_k}}^2}}
    \end{alignat*}

    From Lemma \ref{l:near_tail_bound_L2}, $\Pr{\frac{s^j - L_0}{L_\beta'}\leq \sup_{k\leq K} {\t{r}^{i,j}_k}}^{1/2} = O\lrp{\frac{1}{s^j}}$ and $\sqrt{\E{\lrp{\sum_{k=0}^{K} \delta^i {\t{r}^{i,j}_k}}^2}} = O\lrp{T}$. Recalling the definition of $\t{r}$, and the fact that $\dist\lrp{\t{x}^i_K,\t{y}^i_K}:= \dist\lrp{\t{x}^i(T),\t{y}^i(T)}$,
    \begin{alignat*}{1}
        \E{g_{\epsilon}\lrp{\dist\lrp{\t{x}^{i,j}(T),\t{y}^{i,j}(T}^2}} 
        \leq& \exp\lrp{-\alpha_\epsilon K\delta^i}g_{\epsilon}\lrp{\dist\lrp{x_0,y_0}^2} + O\lrp{\epsilon^{1/2} + \epsilon^{-5/2} {\delta^i}^{1/2} + \frac{1}{s^j}}
    \end{alignat*}

    First, by taking the limit of $i$ to infinity (e.g. for each $i$, we see that for any $j$ and any $\epsilon$,
    \begin{alignat*}{1}
        \lim_{i \to \infty} \E{g_{\epsilon}\lrp{\dist\lrp{\t{x}^{i,j}(T),\t{y}^{i,j}(T)}^2}}  \leq& \exp\lrp{-\alpha_\epsilon K\delta^i}g_{\epsilon}\lrp{\dist\lrp{x_0,y_0}^2}  + O\lrp{\epsilon^{1/2} + \frac{1}{s^j}}
    \end{alignat*}
    Let us define $\t{x}^{\cdot,j}(t)$ as the almost sure limit of $\t{x}^{i,j}(t)$, as $i\to \infty$, whose existence is shown in Lemma \ref{l:x(t)_is_brownian_motion} (similarly for $\t{y}^{\cdot,j}(t)$). It follows that $g_{\epsilon}\lrp{\dist\lrp{\t{x}^{i,j}(T),\t{y}^{i,j}(T}^2}$ converges almost surely to $g_{\epsilon}\lrp{\dist\lrp{\t{x}^{\cdot,j}(T),\t{y}^{\cdot,j}(T)}^2}$ as $i\to \infty$. By dominated convergence (Lemma \ref{l:near_tail_bound_L2} implies a single constant upper bounds $\E{\dist\lrp{\t{x}^{i,j}(T),\t{y}^{i,j}(T}^2}$ for all $i$), $\E{g_{\epsilon}\lrp{\dist\lrp{\t{x}^{i,j}(T),\t{y}^{i,j}(T)}^2}}$ converges to $\E{g_{\epsilon}\lrp{\dist\lrp{\t{x}^{\cdot,j}(T),\t{y}^{\cdot,j}(T)}^2}}$ as $i\to \infty$. Let $\Omega$ denote the set of all couplings between the $\BB(t)$ and $\t{\BB}(t)$, which induces a coupling between $\t{x}^{\cdot,j}(t)$ and $\t{y}^{\cdot,j}(t)$. Then 
    \begin{alignat*}{1}
        & \inf_{\Omega} \E{g_{\epsilon}\lrp{\dist\lrp{\t{x}^{\cdot,j}(T),\t{y}^{\cdot,j}(T)}^2}}\\
        \leq & \lim_{i \to \infty} \E{g_{\epsilon}\lrp{\dist\lrp{\t{x}^{i,j}(T),\t{y}^{i,j}(T)}^2}} \\
        \leq& \exp\lrp{-\alpha_\epsilon K\delta^i}g_{\epsilon}\lrp{\dist\lrp{x_0,y_0}^2}  + O\lrp{\epsilon^{1/2} + \frac{1}{s^j}}
        \elb{e:t:pmfalkfma}
    \end{alignat*}




    From Lemma \ref{l:near_tail_bound_L2}, we know there exists a constant $\C$ (depending on $T, L_R, L_\beta', \lrn{\beta(x_0)}$), such that for all $i$, 
    \begin{alignat*}{1}
        \Pr{\sup_{t} \dist\lrp{x^i(t),x_0}^2 \geq s} \leq \frac{\C}{s}
    \end{alignat*}
    (by definition in \eqref{d:x^i(t)}, $x^i(t)$ are linear interpolations of $x^i(k)$). Thus $\Pr{\sup_{t} \dist\lrp{x^i(t),x_0} = \infty} = 0$. Next, notice that when $\sup_{t\in[0,T]} \dist\lrp{x(t), x_0} \leq \frac{s^j - L_0}{L_\beta'}$, $x(t) = \t{x}^{\cdot,j}(t)$ for all $t\in [0,T]$. It thus follows that as $s^j \to \infty$, $\dist\lrp{x(t), \t{x}^{\cdot,j}(t)}$ converges to $0$ almost surely (similarly for $y$). Thus $g_\epsilon\lrp{\dist\lrp{\t{x}^{\cdot,j}(T), \t{y}^{\cdot,j}(T)}}$ converges to $g_\epsilon\lrp{\dist\lrp{x(T),y(T)}}$ almost surely, and by dominated convergence, in $L_1$ as well. Thus taking limit of \eqref{e:t:pmfalkfma} as $j\to \infty$, aka $s^j \to \infty$,
    \begin{alignat*}{1}
        & \inf_{\Omega} \E{g_{\epsilon}\lrp{\dist\lrp{x(T),y(T)}^2}}\\
        \leq& \exp\lrp{-\alpha_\epsilon K\delta^i}g_{\epsilon}\lrp{\dist\lrp{x_0,y_0}^2}  + O\lrp{\epsilon^{1/2}}
    \end{alignat*}
    Finally, take the limit of $\epsilon \to 0$. Note that $g_\epsilon(r^2) \to g_0(r^2) = f(r)$. Note also that $\alpha_\epsilon \to \alpha$ as defined in the lemma statement. Finally, the properties of $f$ follows from Lemma \ref{l:fproperties}.

\end{proof}

\subsection{Useful Miscellaneous Results}
    The following Lemma is taken from \cite{sun2019escaping}:
    \begin{lemma}\label{l:triangle_distortion}
        For any $x\in M$, $a,y\in T_x M$
        \begin{alignat*}{1}
            \dist\lrp{\Exp_x(y+a), \Exp_{\Exp_x(a)} \lrp{\party{x}{\Exp_x(a)}y}}
            \leq& L_R \lrn{a}\lrn{y}\lrp{\lrn{a} + \lrn{y}} e^{\sqrt{L_R} \lrp{\lrn{a}+\lrn{y}}} 
        \end{alignat*}
    \end{lemma}
    \begin{proof}
        From the proof of Lemma 3 from \cite{sun2019escaping} (which is in turn a refinement of the proof from \cite{karcher1977riemannian})
        \begin{alignat*}{1}
            &\dist\lrp{\Exp_x(y+a), \Exp_{\Exp_x(a)} \lrp{\party{x}{\Exp_x(a)}y}} \\
            \leq& \int_0^1 \frac{\cosh(\sqrt{L_R}\lrn{y + (1-t) a})-\frac{\sinh(\sqrt{L_R}\lrn{y + (1-t) a})}{\sqrt{L_R}\lrn{y + (1-t) a}}}{\lrn{y + (1-t) a}} dt \cdot \lrn{a}\lrn{y}\\
            \leq& \sqrt{L_R} \int_0^1 \sqrt{L_R} \lrn{y + (1-t) a} e^{\sqrt{L_R} \lrn{y + (1-t) a}} dt \cdot \lrn{a}\lrn{y}\\
            \leq& L_R \lrn{a}\lrn{y}\lrp{\lrn{a} + \lrn{y}} e^{\sqrt{L_R} \lrp{\lrn{a}+\lrn{y}}} 
        \end{alignat*}
        where we use the fact from Lemma \ref{l:sinh_bounds} that for all $r \geq 0$,
        \begin{alignat*}{1}
            \frac{\cosh(r)}{r} - \frac{\sinh(r)}{r^2} \leq r e^r
        \end{alignat*}
    \end{proof}

    \begin{lemma}\label{l:sinh_bounds}
        For all $r \geq 0$,
        \begin{alignat*}{1}
            & \sinh(r) \leq r e^r\\
            & \cosh(r) - 1 \leq \frac{r^2}{2} e^r\\
            & \frac{\cosh(r)}{r} - \frac{\sinh(r)}{r^2} \leq r e^r\\
            & \cosh(r) \leq e^r\\
            & \frac{\sinh(r)}{r} - 1 \leq r^2 e^3
        \end{alignat*}
    
    \end{lemma}
    \begin{proof}
        Wolfram alpha says that $\sinh(r) - r e^r$ has global maximum at $r=0$ with value $\sinh(0) = 0$.
    
        Wolfram alpha also says that $\cosh(r) - 1 = 2 \sinh^2(r/2) \leq r^2 e^{r}/2$ by the first bound.
    
        The last two inequalities are all from Wolfram alpha.

    \end{proof}

\section{Non-Gaussian Approximation}
In this section, we prove Theorem \ref{t:main_nongaussian_theorem}, which bounds the distance between a manifold diffusion, and a random walk on the manifold with non-Gaussian noise term. 

\subsection{Proof of Theorem \ref{t:main_nongaussian_theorem} and Lemma \ref{l:theorem_2_corollary}}
\label{s:proof_of_t:main_nongaussian_theorem}
\begin{proof}[Proof of Theorem \ref{t:main_nongaussian_theorem}]
    \textbf{Step 1: Tangent Space Non-Gaussian Walk}\\
    For $k\leq K-1$, define
    \begin{alignat*}{1}
        & y_{k+1} = \Psi(\delta;y_{0},\beta,\xi_k)\\
        & \t{z}_{k+1} = \t{\Psi}\lrp{\delta;y_{0},\t{z}_{k},\beta,\xi_k}\\
        & \t{y}_{k+1} = \Exp_{y_{0}} \lrp{\t{z}_{k +1}}
    \end{alignat*}
    where $\Psi$ and $\t{\Psi}$ are as definde in \eqref{d:y_k} and \eqref{d:ty_k} respectively. Assume $T$ satisfies
    \begin{alignat*}{1}
        & T \leq \min\lrbb{\frac{1}{64L_\beta'}, \frac{1}{64 {L_\xi'}^2}, \frac{1}{16 \sqrt{L_R} L_\beta}, \frac{1}{256 L_R L_\xi^2}, \frac{\C_r}{16 L_\beta},\frac{\C_r^2}{1024 L_R L_\xi^2 \log\lrp{L_R / \C_r^2}}, \frac{\C_r^{2/3}}{2 L_\xi^{2/3}}}
        \elb{e:t:aasdakdsjk:1}
    \end{alignat*}
    where $\C_r$ is defined in \eqref{d:c_r}. We can apply Corollary \ref{c:K-step-retraction-bound-yk} to bound
    \begin{alignat*}{1}
        \E{\dist\lrp{\t{y}_{K},y_{K}}} \leq \t{O}\lrp{\lrp{K\delta}^{3/2} \lrp{1 + L_\beta^2}}
    \end{alignat*}
    where $\t{O}$ hides polynomial dependency on $L_R, L_R', L_\beta', L_\xi, \log\frac{1}{K\delta}$.

    \textbf{Step 2: Tangent Space Central Limit Theorem}\\    
    In the previous step, we have shown that, $y_{K}$ can be well approximated by $\t{y}_{K} = \Exp_{y_{0}}\lrp{\t{z}_{K}}$. In this step, we show that $\t{z}_K$ is close to a Gaussian distribution. Recall that $E$ is an orthonormal basis of $y_0$ (the exact choice does not matter). For any $v\in T_{y_{0}} M$, we let $\vv$ denote its coordinates wrt $E$. Let us define the random function $H : T_{y_{0}} M \to T_{y_{0}} M$ as
    \begin{alignat*}{1}
        H_k(z) = \party{\Exp_{y_{0}}(z)}{y_{0}} \lrp{\xi_k\lrp{\Exp_{y_{0}}(z)}}
    \end{alignat*}
    where the parallel transport is along the geodesic $\gamma(t) = \Exp_{y_{0}}(t z)$
    We can verify, from definition of \eqref{d:ty_k}, that the following are equivalent:
    \begin{alignat*}{1}
        & \t{z}_{k+1} = \t{\Psi}\lrp{\delta;y_{0},\t{z}_{k},\beta,\xi_k}\\
        & \t{z}_{k+1} = \t{z}_{k} + \delta \beta(y_{0}) + \sqrt{\delta} H_k(\t{z}_{k}) 
        \elb{e:t:qoiaspoqkdm:1}
    \end{alignat*}

    Let us define $\H_k : \Re^d \to \Re^d$ as the vector-valued function, whose $j^{th}$ coordinate is given by
    \begin{alignat*}{1}
        \H_{k,j}(\zz) = \lin{H_k\lrp{\zz \circ E}, E_j}
    \end{alignat*}
    (recall that $\zz \circ E = \sum_{j=1}^d \zz_j E_j$).

    Finally, let us define $\bbeta$ denote the coordinates of $\beta(y_{0})$ wrt $E$.

    Let us define, for $k\leq K-1$,
    \begin{alignat*}{1}
        \t{\zz}_{k + 1} = \t{\zz}_k + \delta \bbeta + \sqrt{\delta} \H_k(\t{\zz}_{k})
        \elb{e:t:qoiaspoqkdm:2}
    \end{alignat*}
    We can verify the processes in \eqref{e:t:qoiaspoqkdm:1} and \eqref{e:t:qoiaspoqkdm:2} are equivalent in the following sense: $\t{\zz}_{k} \circ E = \t{z}_{k}$.

    The main goal of this section is to bound the Euclidean distance between $\t{\zz}_k$ in \eqref{e:t:qoiaspoqkdm:2} and $\zz_k$ defined below:
    \begin{alignat*}{1}
        {\zz}_{k + 1} = {\zz}_{k}  + \delta \bbeta + \sqrt{\delta} \zb_{k + 1}
        \elb{e:t:qoiaspoqkdm:3}
    \end{alignat*}
    where $\zb_k$ are iid samples from $\N(0,I)$.
    
    Before going into the main proof, we state two important assumptions, and their consequences.

    \textbf{Step 2.1: Regularity of $\H$}
    Under Assumption \ref{ass:regularity_of_xi}, we can use Lemma \ref{l:derivative_of_pullback_H} to bound the derivatives of $\H_k$, for all $\lrn{\zz} \leq \frac{1}{\sqrt{L_R}}$. Let $L_{\H} := L_\xi$. We claim that there exists constants $L_\H,L_\H',L_\H'' = poly\lrp{L_R,L_R',L_R'',L_\xi,L_\xi',L_\xi''}$ such that
    \begin{alignat*}{2}
        & \text{For all $\lrn{\zz}_2$:}&&\qquad \lrn{\H(\zz)}_2 \leq L_\H \\
        & \text{For all $\lrn{\zz}_2 \leq \frac{1}{\sqrt{L_R}}$:}&& \qquad  \lrn{\nabla \H(\zz)}_2 \leq L_\H'\qquad  \lrn{\nabla^2 \H(\zz)}_2 \leq L_\H''
        \numberthis \label{ass:clt_H_lipschitz}
    \end{alignat*}
    The first line of \eqref{ass:clt_H_lipschitz} is immediate, since $H_k$ is simply the parallel transport of $\xi_k$. The second line of \eqref{ass:clt_H_lipschitz} follows from Lemma \ref{l:derivative_of_pullback_H}. The exact expression of $L_\H,L_\H',L_\H''$ can be found in the statement of Lemma \ref{l:derivative_of_pullback_H}.

    \textbf{Step 2.2: Moments of $\H$}
    
    Again using the fact that $H_k$ is the parallel transport of $\xi_k$, and using Assumption \ref{ass:moments_of_xi} and Lemma \ref{l:identity_covariance_parallel_transport}, we verify that for all $\uu$, $\E{\H(\uu)} = 0$ and $\E{\H(\uu)\H(\uu)^T} = I$.

    \textbf{Step 2.3: Assumptions on $\delta$ and $K\delta$}
    Let us assume that
    \begin{alignat*}{1}
        &T \leq \min\lrbb{\frac{1}{16 L_\H^2}, \frac{1}{16 {L_\H'}^2}, \frac{1}{32 \sqrt{L_R} L_\beta}, \frac{1}{128 L_R L_\H^2},\frac{1}{256 L_RL_\H^2 \log \lrp{2^{10} L_R}}}\\
        \elb{ass:clt_constants_0}
    \end{alignat*}
    We will apply Lemma \ref{l:clt:main}, which shows that there exists a coupling between $\zz_k$ and $\t{\zz}_k$ such that
    \begin{alignat*}{1}
        \sqrt{\E{\lrn{\t{\zz}_{K}- \zz_K}_2^2}}
            \leq& \lambda_6 \cdot \log\lrp{\frac{1}{T}}^4 \cdot T^{3/2}
            \elb{e:t:main_delta_restriction}
    \end{alignat*}
    where $\lambda_6 = poly\lrp{L_\H, L_\H', L_\H'', L_R}=poly\lrp{L_R,L_R',L_R'',L_\xi,L_\xi',L_\xi''}$ is some positive constant. We will apply Lemma \ref{l:discrete-approximate-synchronous-coupling} to convert this bound on $\lrn{\zz_K - \t{\zz}_K}$ to a bound on $\dist\lrp{\Exp_{y_0}(z_K),\Exp_{y_0}(\t{z}_K)}$: let $\C := \sqrt{L_R} \lrp{\lrn{\zz_K}_2 + \lrn{\t{\zz}_K}_2}$, then
    \begin{alignat*}{1}
        \dist\lrp{\Exp_{y_0}(z_K), \Exp_{y_0}(\t{z}_K)}^2 \leq& 32 e^{\C} \lrn{\zz_k - \t{\zz}_k}_2^2
    \end{alignat*}
    Taking square-root, then taking expectation and applying Cauchy Schwarz,
    \begin{alignat*}{1}
        \E{\dist\lrp{\Exp_{y_0}(z_K), \Exp_{y_0}(\t{z}_K)}} \leq 8\sqrt{\E{e^{\sqrt{L_R} \lrp{\lrn{\zz_K}_2 + \lrn{\t{\zz}_K}_2}}}} \cdot \sqrt{\E{\lrn{\zz_k - \t{\zz}_k}_2^2}}
    \end{alignat*}
    The $\sqrt{\E{\lrn{\zz_k - \t{\zz}_k}_2^2}}$ has been bound in \eqref{e:t:main_delta_restriction}. The $\E{e^{\sqrt{L_R} \lrp{\lrn{\zz_K}_2 + \lrn{\t{\zz}_K}_2}}}$ can be bound using Lemma \ref{l:clt_gaussian_tail}: Using our assumption that $K\delta \leq \frac{1}{L_R L_\H^2}$,
    \begin{alignat*}{1}
        \E{e^{\sqrt{L_R} \lrn{\t{\zz}_K}_2}}
        \leq& 4 \E{e^{L_R \lrn{\t{\zz}_K}_2^2}}
        \leq 4 \E{e^{\frac{1}{K\delta L_\H^2} \lrn{\zz_K}_2^2}} \leq 8
    \end{alignat*}
    Similarly, recalling that $L_\H^2 \geq d$ since $Cov(\H) = I$, 
    \begin{alignat*}{1}
        \E{e^{\sqrt{L_R} \lrn{\zz_K}_2}}
        \leq& 4 \E{e^{\frac{1}{K\delta d} \lrn{\zz_K}_2^2}}
        \leq 4 \lrp{\E{e^{\frac{1}{K\delta} \lrn{\zz_K}_2^2}}}^{1/d}
        \leq 8
    \end{alignat*}
    Putting everything together,
    \begin{alignat*}{1}
        \E{\dist\lrp{\Exp_{y_0}(z_K), \Exp_{y_0}(\t{z}_K)}} \leq 512 \lambda_6 \cdot \log\lrp{\frac{1}{T}}^4 \cdot T^{3/2}
    \end{alignat*}
    Recall the definition of $\bar{y}_k = \overline{\Phi}(k\delta;y_0,E,\beta,\BB,0)$, and notice from the definition of $\overline{\Phi}$ in \eqref{d:x^i(t)} that $\bar{y}_K$ has the same distribution as $\Exp_{y_{0}}\lrp{{\zz}_{K}}$. This implies
    \begin{alignat*}{1}
        \E{\dist\lrp{\bar{y}_{K}, \tilde{y}_{K}}}\leq \t{O}\lrp{\lrp{K\delta}^{3/2}}
    \end{alignat*}

    By triangle inequality, we can verify that $\E{\dist\lrp{\bar{y}_{K}, y_{K}}}\leq \t{O}\lrp{\lrp{K\delta}^{3/2}\lrp{1+L_\beta^2}}$. 
\end{proof}

\begin{proof}[Proof of Lemma \ref{l:theorem_2_corollary}]
    Let us define a radius $r$ and the event $A_k$ (for $k\in \Z^+$) as
    \begin{alignat*}{1}
        & r := 64 \lrp{\R + \frac{\sqrt{L_R} L_\xi^2}{m}} + 64 \sqrt{\frac{L_\xi^2}{m}} \log\lrp{\frac{NK}{\delta}}\\
        & A_k := \max_{j\leq k} \dist\lrp{{y}_j, x^*} \leq r
        \elb{n:d:A_k:main}
    \end{alignat*}
    Recall that $K = \delta^{-2}$ so that $r$ depends polynomially on $N$ and $\log(1/K\delta)$.
    \textbf{Step 1: $K$-step Approximation to Euler Murayama Step}\\
    Let us define $L_\beta := L_\beta' r$, and $\beta^r(x) := \Pi_{L_\beta} \lrp{\beta(x)}$. Note that $\beta^r$ also satisfies Assumption \ref{ass:beta_lipschitz} with parameter $L_\beta'$. 
    
    Consider a fixed $i$. Let $\t{y}_{iK} := y_{iK}$, and for $k=0...K-1$, 
    \begin{alignat*}{1}
        \t{y}_{iK + k + 1} = \Psi(\delta; \t{y}_{iK + k}, \beta^r, \xi_k)
    \end{alignat*}
    where $\Psi$ is as defined in \eqref{d:y_k}. Let us also define
    \begin{alignat*}{1}
        \bar{y}_{(i+1)K} := \overline{\Phi}(K\delta;y_{iK},E^{iK},\beta,\bar{\BB}^{iK})
    \end{alignat*}
    where $\Phi$ is as defined in \eqref{d:x(t)}, $E^{iK}$ is an orthonormal basis at $T_{y_{iK}}M$, $\bar{\BB}^{iK}$ is a standard Euclidean Brownian motion. From the definition of $\beta^r$, $\overline{\Phi}(K\delta;y_{iK},E^{iK},\beta,\bar{\BB}^{iK}) = \overline{\Phi}(K\delta;y_{iK},E^{iK},\beta^r,\bar{\BB}^{iK})$ under the event $A_{iK}$, because $\dist\lrp{y_{iK},x^*} \leq r \Rightarrow \lrn{\beta(y_{iK})} \leq L_\beta \Rightarrow \beta(y_{iK}) = \beta^r(y_{iK})$. By Theorem \ref{t:main_nongaussian_theorem} (with $\beta^r$ as the drift vector field), 
    \begin{alignat*}{1}
        \E{\ind{A_{iK}}\lrp{\dist\lrp{\bar{y}_{(i+1)K},\t{y}_{(i+1)K}}}} \leq {O}\lrp{\lrp{K\delta}^{3/2}\lrp{1+\s^3}} = \t{O}\lrp{\lrp{K\delta}^{3/2}}
    \end{alignat*}
    where $\t{O}$ hides polynomial dependence on $L_R, L_R', L_\beta', L_\xi, L_\xi', L_\xi'', \log(N), \log\lrp{\frac{1}{K\delta}}$. Note that Theorem \ref{t:main_nongaussian_theorem} requires that $K\delta \leq 1/\C_1$ for some $\C_1 = poly(L_\beta,L_\beta', L_\xi,L_\xi',L_\xi'',L_R,L_R')$\\
    $ = poly(r,L_\beta', L_\xi,L_\xi',L_\xi'',L_R,L_R')$. By Lemma \ref{l:useful_xlogx} and some algebra, we verify that there exists $\C_2$ depending polynomially on $\log N$ (and polynomially on the various Lipschitz parameters) such that $T = K\delta \leq \C_2$ satisfies the requirement for Theorem \ref{t:main_nongaussian_theorem}. Finally, let 
    \begin{alignat*}{1}
        \hat{y}_{(i+1)K} := {\Phi}(K\delta;y_{iK},E^{iK},\beta,\bar{\BB}^{iK})
    \end{alignat*}

    By Lemma \ref{l:discretization-approximation-lipschitz-derivative}
    \begin{alignat*}{1}
        \E{\ind{A_{iK}} \dist\lrp{\hat{y}_{(i+1)K},\bar{y}_{(i+1)K}}} 
        \leq& 2^{9}\lrp{T^2 (L_\beta' r)^2+ T^2 {L_\beta'}^2 + T^{3/2} \lrp{d^{3/2}\lrp{L_R + {L_\beta'}^2/{L_\beta^2}} + L_\beta' \sqrt{d}}}\\
        =& O\lrp{\lrp{K\delta}^{3/2}}
    \end{alignat*}. Applying triangle inequality,
    \begin{alignat*}{1}
        \E{\ind{A_{iK}} \dist\lrp{y_{(i+1)K}, \hat{y}_{(i+1)K}}} = \t{O}\lrp{\lrp{K\delta}^{3/2}}
    \end{alignat*}

    \textbf{Step 2: Contraction under Exact Diffusion}
    Let $f$ be as defined in \eqref{d:f}. By Lemma \ref{l:g_contraction_without_gradient_lipschitz}, 
    \begin{alignat*}{1}
        \E{\ind{A_{iK}}f\lrp{\dist\lrp{x_{(i+1)K},\dist\lrp{\hat{y}_{(i+1)K}}}}}
        \leq& \E{\ind{A_{iK}}\exp\lrp{-\alpha K\delta}f\lrp{\dist\lrp{x_{iK},y_{iK}}}}
    \end{alignat*}
    Combined with our previous bound on $\E{\ind{A_{iK} \dist\lrp{y_{(i+1)K}, \hat{y}_{(i+1)K}}}}$, as well as the fact that $\lrabs{f'} \leq 1$ from Lemma \ref{l:fproperties},
    \begin{alignat*}{1}
        \E{\ind{A_{iK}}f\lrp{\dist\lrp{x_{(i+1)K},{y}_{(i+1)K}}}} 
        \leq& \E{\ind{A_{iK}}\exp\lrp{-\alpha K\delta}f\lrp{\dist\lrp{x_{iK},y_{iK}}}} + \t{O}\lrp{\lrp{K\delta}^{3/2}}
    \end{alignat*}

    Applying the above recursively, 
    \begin{alignat*}{1}
        \E{\ind{A_{NK}}f\lrp{\dist\lrp{x_{NK},{y}_{NK}}}} 
        \leq& \exp\lrp{-\alpha N K\delta}f\lrp{\dist\lrp{x_{0},y_{0}}} + \frac{1}{\alpha}\t{O}\lrp{\lrp{K\delta}^{1/2}}
        \elb{e:t:oqiwdasn:1}
    \end{alignat*}

    \textbf{Step 3: Tail Bound}\\
    We now consider the low probability event that $A_{NK}$ fails to hold (i.e. the iterate $y_k$ was at some point further than distance $r$ from $x^*$). Specifically, our goal is to verify the following bound:
    \begin{alignat*}{1}
        \E{\ind{A^c_{NK}}\dist\lrp{{y}_{NK}, x_{NK}}} \leq O\lrp{(K\delta)^2}
        \elb{e:t:qimdwqdm:1}
    \end{alignat*}

    By Lemma \ref{l:far-tail-bound-nongaussian-fixed} and Lemma \ref{l:far-tail-bound-l2-brownian} and Lemma \ref{l:far-tail-bound-l2-nongaussian},
    \begin{alignat*}{1}
        & \E{\ind{A_k^c}} 
        = \Pr{A_k^c} 
        \leq 32k\delta m \exp\lrp{\frac{m\R^2}{L_\xi^2} + \frac{2 L_R L_\xi^2}{m} - \frac{m t^2}{32 L_\xi^2}}\\
        & \E{\dist\lrp{x_k,x^*}^2} \leq \frac{2^{12} L_R {L_\beta'}^4 d^2}{m^{6}} + \frac{9 {L_\beta'} \R^2}{m} \\
        & \E{\dist\lrp{y_k,x^*}^2} \leq \frac{2^{12} L_R {L_\beta'}^4 L_\xi^2}{m^{6}} + \frac{9 {L_\beta'} \R^2}{m} 
    \end{alignat*}
    By our definition of $r$ and applying Cauchy Schwarz, we verify that
    \begin{alignat*}{1}
        \E{\ind{A_K^c} \dist\lrp{y_K,x_K}} \leq O\lrp{\delta} \leq O(K\delta)^2)
        \elb{e:t:oqiwdasn:2}
    \end{alignat*}
    Note that Lemma \ref{l:far-tail-bound-nongaussian-fixed} requires $\delta \leq \min\lrbb{\frac{m}{16 {L_\beta'}^2 {\sqrt{L_R} r}}}$. Using Lemma \ref{l:useful_xlogx}, there exists $\C' = poly\lrp{\R, L_R, L_\xi, d, \frac{1}{m}, L_\beta', \log N}$ such that $\delta \leq \C'$ implies the desired bound. 

    Combining \eqref{e:t:oqiwdasn:1} and \eqref{e:t:oqiwdasn:2}, and using the fact that $\frac{1}{2}\exp\lrp{- \lrp{q + L_{\Ric}} \R^2/2} r \leq f(r) \leq r$ and the fact that $\delta = (K\delta)^3$,
    \begin{alignat*}{1}
        \E{\dist\lrp{x_{NK},{y}_{NK}}}
        \leq& \exp\lrp{-\alpha N K\delta + \lrp{q + L_{Ric}}\R^2}\dist\lrp{x_{0},y_{0}} + \exp\lrp{\lrp{q + L_{Ric}}\R^2} \cdot \t{O}\lrp{\delta^{1/6}}
    \end{alignat*}
    where $\t{O}$ hides polynomial dependency on $L_R, L_R', L_\beta', L_\xi, L_\xi', L_\xi'', \log N, \log\lrp{\frac{1}{\delta}}$.
\end{proof}

\subsection{Tangent Space Approximation}
We define an approximation to \eqref{d:y_k}, where each step is a "straight line" in the tangent space $T_{y_0} M$. We will define a constant, which will be useful throughout the proof.
\begin{alignat*}{1}
    & \C_r := \frac{1}{16} \min\lrbb{{L_R'}^{-1/3}, \frac{1}{8 \sqrt{L_R}}}
    \elb{d:c_r}
\end{alignat*}
\begin{definition}
    Given a deterministic vector field $\beta$, a random vector field $\xi$, an initial point $y\in M$, we define
    \begin{alignat*}{1}
        \t{\Psi}(\delta;y, z,\beta,\xi) := z + \delta \beta(y) + \party{\Exp_{y}(z)}{y} \lrp{\xi\lrp{\Exp_{y}(z)}}
        \elb{d:ty_k}
    \end{alignat*}
    where the parallel transport is along the geodesic $\gamma(t) = \Exp_{y}\lrp{t z}$.
\end{definition}

We will be interested in the sequence
\begin{alignat*}{1}
    & \t{z}_0 = 0\\
    & \t{z}_{k+1} = \t{\Psi}\lrp{\delta;y_0,\t{z}_k,\beta,\xi_k}
    \elb{d:ty_k:-1}
\end{alignat*}
Define $\t{y}_k := \Exp_{y_0}\lrp{z_k}$. We will show that $\t{y}_k$ is a good approximation to
\begin{alignat*}{1}
    y_{k+1} = \Psi(\delta;y_k,\beta,\xi_k)
    \elb{d:y_k:0}
\end{alignat*}
for $\Psi$ as defined in \eqref{d:y_k}. It will help to write \eqref{d:ty_k} in the following equivalent form:
\begin{alignat*}{1}
    & \t{y}_0 = y_0 \qquad \t{z}_0 = 0\\
    & \gamma_k(t) = \Exp_{\t{y}_0} \lrp{(1-t) \t{z}_k} \\
    & \t{z}_{k+1} = \t{z}_k + \delta \beta(y_0) + \sqrt{\delta} \party{\gamma_k(t)}{}\lrp{\xi_k (\t{y}_k)} \\
    & \t{y}_{k+1} = \Exp_{y_0} \lrp{\t{z}_{k+1}}
    \elb{d:ty_k:0}
\end{alignat*}
Note that by definition, for all $k$, $\gamma_k(t):[0,1]\to M$ is a geodesic, with $\gamma_k(0)= \t{y}_k$ and $\gamma_k(1) = \t{y}_0$.

The main goal of this subsection is to prove the following approximation bound between \eqref{d:y_k:0} and \eqref{d:ty_k:0}.

\begin{corollary}\label{c:K-step-retraction-bound-yk}
    Let $y_0$ be some initial point. For $i\in\Z^+$, let $\xi_i$ be independent samples of some random function $\xi$. Let $\beta$ be a vector field satisfying Assumption \ref{ass:beta_lipschitz}, in addition, assume that there exists a constant $L_\beta$, such that for all $x\in M$, $\lrn{\beta(x)}\leq L_\beta$. Assume $\xi$ satisfies \eqref{ass:moments_of_xi} and assume that there exists $L_\xi$ such that $\lrn{\xi(x)} \leq L_\xi$ for all $x$. Let $\t{z}_0 := 0$. Define
    \begin{alignat*}{1}
        & y_{k+1} = \Psi(\delta;y_k,\beta,\xi_k)\\
        & \t{z}_{k+1} = \t{\Psi}\lrp{\delta;y_0,\t{z}_k,\beta,\xi_k}\\
        & \t{y}_{k+1} = \Exp_{y_0} \lrp{\t{z}_{k+1}}
    \end{alignat*}
    with $\t{y}_0 = y_0$, where $\Psi$ and $\t{\Psi}$ are as defined in \eqref{d:y_k} and \eqref{d:ty_k} respectively. Let $K,\delta$ satisfy
    \begin{alignat*}{1}
        & \delta\leq \min\lrbb{\frac{\C_r^2}{4 L_\xi^2}}\\
        & K\delta \leq \min\lrbb{\frac{1}{64L_\beta'}, \frac{1}{64 {L_\xi'}^2}, \frac{1}{64 L_\xi^2}, \frac{1}{16 \sqrt{L_R} L_\beta}, \frac{1}{256 L_R L_\xi^2}, \frac{\C_r}{16 L_\beta},\frac{\C_r^2}{1024 L_R L_\xi^2 \log\lrp{L_R / \C_r^2}}}
        \elb{e:t:aasdakdsjk:1}
    \end{alignat*}
    where $\C_r$ is as defined in \eqref{d:c_r}.
    Then
    \begin{alignat*}{1}
        \E{\dist\lrp{\t{y}_K,y_K}} \leq \t{O}\lrp{\lrp{K\delta}^{3/2} \lrp{1 + L_\beta^2}}
    \end{alignat*}
    where $\t{O}$ hides polynomial dependency on $L_R, L_R', L_\beta', L_\xi, \log\frac{1}{K\delta}$.
\end{corollary}
\begin{proof}[Proof of Corollary \ref{c:K-step-retraction-bound-yk}]
    Let us first define a radius 
    \begin{alignat*}{1}
        r_1 := \sqrt{64  K^2\delta^2 L_\beta^2 + 256  K\delta L_\xi^2 \log\lrp{\frac{1}{32K\delta L_R L_\xi^2}}}
        \elb{d:r_1}
    \end{alignat*}
    We verify that assumptions on $K,\delta$ and the definition of $r_1$ satisfy the conditions for Lemma \ref{l:one-step-retraction-bound-yk}, which we need shortly. In particular, we wish to verify that $r_1 \leq \C_r$. By \eqref{e:t:aasdakdsjk:1}, $64 K^2 \delta^2 L_\beta^2 \leq \C_r^2/2$, so that it suffices to verify $256  K\delta L_\xi^2 \log\lrp{\frac{1}{32K\delta L_R L_\xi^2}} \leq \frac{\C_r^2 }{2}$, moving terms around, this is equivalent to $\frac{1}{32 K\delta L_R L_\xi^2 \log \lrp{\frac{1}{32K\delta L_R L_\xi^2}}} \geq \frac{4 L_R}{\C_r^2}$. Notice from definition that $L_R / \C_r^2 \leq 1$, so that we can apply Lemma \ref{l:useful_xlogx}, which guarantees that this holds as along as $\frac{1}{32 K\delta L_R L_\xi^2} \geq \frac{12L_R}{\C_r^2} \log\lrp{\frac{4L_R}{\C_r^2}}$, this is guaranteed by our assumption that $K\delta \leq \frac{\C_r^2}{1024 L_R L_\xi^2 \log\lrp{L_R / \C_r^2}}$

    We can decompose $\E{\dist\lrp{\t{y}_{K}, y_{K}}^2}$ as
    \begin{alignat*}{1}
        \E{\dist\lrp{\t{y}_{K}, y_{K}}^2}
        \leq& \E{\ind{A_K} \dist\lrp{\t{y}_{K}, y_{K}}^2} + \E{\ind{A_K^c} \dist\lrp{\yt_K, y_k}^2}
    \end{alignat*}

    To bound the first term, we simply apply the result of Lemma \ref{l:one-step-retraction-bound-yk} recursively over $k=0...K-1$.
    \begin{alignat*}{1}
        &\E{\ind{A_K}\dist\lrp{\t{y}_{K}, y_{K}}^2}\\
        \leq& \E{\ind{A_{K-1}}\dist\lrp{\t{y}_{K}, y_{K}}^2}\\
        \leq& \lrp{1 + \frac{1}{K}}\E{\ind{A_{K-2}} \dist\lrp{y_{K-1}, \t{y}_{K-1}}^2} \\
        &\quad + \lrp{ 9K\delta^2 L_R^2 r_1^4 L_\beta^2 + 
        9K\delta^2 {L_\beta'}^2 r_1^2 + 8 \delta L_R^2 r_1^4 L_\xi^2 + 2^{27} K \lrp{{L_R'}^2 r_1^4 + L_R^2 r_1^2} \lrp{\delta^4 L_\beta^4 + \delta^2 L_\xi^4}}\\
        \leq& \lrp{1 + \frac{1}{K}}\E{\ind{A_{K-2}} \dist\lrp{y_{K-1}, \t{y}_{K-1}}^2} + \t{O}\lrp{K^2 \delta^3 \lrp{1+ L_\beta^2}}\\
        \leq& ...\\
        \leq& \t{O}\lrp{\lrp{K\delta}^3 \lrp{1+ L_\beta^2}}
    \end{alignat*}
    the first inequality is because $\ind{A_K} \leq \ind{A_{k-1}}$ for all $k$, the third inequality is because $r_1 = \t{O}\lrp{K\delta}$. The $\t{O}()$ hides polynomial dependency on $L_R, L_R', L_\beta', L_\xi, \log\frac{1}{K\delta}$.

    To bound the second term, we first bound the probability of $A_K^c$. We can apply Lemma \ref{l:near_tail_bound_tangent_space_no_stopping}: let $E$ denote any orthonormal basis of $T_{y_0} M$, let $\vv$ denote the coordinates of $\beta(y_0)$ wrt $E$, and $\H_k$ denote the coordinates of $\party{\t{y}_k}{y_0} \xi_k(\t{y}_k)$. We verify that $\E{\H_k} = 0$, $\E{\H_k \H_k^T} = I$ (using Assumption \ref{ass:moments_of_xi} together with Lemma \ref{l:identity_covariance_parallel_transport}), and that $\lrn{\H_k(x)} \leq L_{\H} := L_\xi$ for all $x$. Let $\zz_k$ denote the coordinates of $\t{z}_k$ wrt $E$. Finally, we verify that $\lrn{\vv}_2 \leq \delta {L_\beta'} \dist\lrp{y_0,x^*}$. Thus by Lemma \ref{l:near_tail_bound_tangent_space_no_stopping},
    \begin{alignat*}{1}
        \Pr{\max_{k\leq K} \lrn{\zt_k} \geq t} = \Pr{\max_{k\leq K} \lrn{\zz_k}_2 \geq t} \leq \exp\lrp{ \frac{{64 K^2\delta^2 L_\beta^2 + 8K\delta L_\xi^2} - t^2}{128 K\delta L_\xi^2}}
    \end{alignat*}
    Note that Lemma \ref{l:near_tail_bound_tangent_space_no_stopping} explicitly deals with $\Re^d$, but it also works for $\t{z}$ which lies in $T_{y_0} M$ which is isometric to $\Re^d$. Similarly, by Lemma \ref{l:near_tail_bound_no_stopping},
    \begin{alignat*}{1}
        \Pr{\max_{k\leq K} \dist\lrp{y_k,y_0} \geq t} \leq \exp\lrp{\frac{32K^2\delta^2 {L_\beta}^2 + 8K\delta L_\xi^2 - t^2}{128 K\delta L_\xi^2}}
    \end{alignat*}
    
    By union bound,
    \begin{alignat*}{1}
        \Pr{A_k} 
        \leq& \Pr{\max_{k\leq K} \lrn{\zt_k} \geq \r_1} + \Pr{\max_{k\leq K} \dist\lrp{y_k,y_0} \geq \r_1} \\
        \leq& 2\exp\lrp{\frac{64  K^2\delta^2 L_\beta^2 + 8 K\delta L_\xi^2 - \r_1^2}{128  K\delta L_\xi^2}}
    \end{alignat*}
    
    By Lemma \ref{l:near_tail_bound_L2},
    \begin{alignat*}{1}
        \E{\dist\lrp{y_k,y_0}^2} 
        \leq& \exp\lrp{1 + 8K\delta L_\beta' + K\delta L_R L_\xi^2 + K\delta^2 L_R L_\beta^2} \cdot \lrp{2K\delta L_\xi^2 + 8K^2\delta^2 L_\beta^2}\\
        \leq& 16 \lrp{K\delta L_\xi^2 + K^2\delta^2 L_\beta^2}
    \end{alignat*}

    From the definition of $\lrn{\t{z}_k}$, we can also bound
    \begin{alignat*}{1}
        \E{\lrn{\t{z}_k}^2} 
        \leq& 2\lrp{K^2\delta^2 L_\beta^2 + K\delta L_\xi^2}
    \end{alignat*}
    We can therefore bound using Cauchy Schwarz and triangle inequality 
    \begin{alignat*}{1}
        & \E{\ind{A_K^c} \dist\lrp{\t{y}_K,y_K}}\\
        \leq& 32 \exp\lrp{\frac{64  K^2\delta^2 L_\beta^2 + 8 K\delta L_\xi^2 - \r_1^2}{256  K\delta L_\xi^2}} \sqrt{K^2\delta^2 L_\beta^2 + K\delta L_\xi^2}
    \end{alignat*}
    From our choice of $r_1$, we verify that
    \begin{alignat*}{1}
        r_1 \geq \sqrt{64  K^2\delta^2 L_\beta^2 + 8 K\delta L_\xi^2 + 128 K\delta L_\xi^2 \log\lrp{\frac{1}{32K\delta L_R L_\xi^2}}}
    \end{alignat*}
    Note that we use the fact that $\log\lrp{\frac{1}{32K\delta L_R L_\xi^2}} \geq \log\lrp{8} \geq 1$, using the upper bound on $K\delta$ in the assumption in \eqref{e:t:aasdakdsjk:1}. It follows that
    \begin{alignat*}{1}
        \E{\ind{A_K^c} \dist\lrp{\t{y}_K,y_K}} \leq \lrp{K\delta}^{3/2} \lrp{L_\beta \sqrt{L_R} L_\xi + L_R L_\xi^3} = {O}\lrp{\lrp{K\delta}^3 \lrp{1+ L_\beta^2}}
    \end{alignat*}
    Combining the above with our earlier bound that $\E{\ind{A_K}\dist\lrp{\t{y}_{K}, y_{K}}^2} \leq \t{O}\lrp{\lrp{K\delta}^3 \lrp{1+ L_\beta^2}}$ gives the desired bound in the Lemma.

\end{proof}
\begin{lemma}\label{l:one-step-retraction-bound-yk}
    Let $y_0$ be some initial point. For $i\in\Z^+$, let $\xi_i$ be independent samples of some random function $\xi$ satisfying $\lrn{\xi(x)} \leq L_\xi$ for all $x$. Let $\beta$ be a vector field satisfying Assumption \ref{ass:beta_lipschitz}, in addition, assume that there exists a constant $L_\beta$, such that for all $x\in M$, $\lrn{\beta(x)}\leq L_\beta$. Let $\t{z}_0 := 0$. Define the two stochastic processes
    \begin{alignat*}{1}
        & y_{k+1} = \Psi\lrp{\delta, y_k, \beta, \xi_k}\\
        & \t{z}_{k+1} = \t{\Psi}\lrp{\delta;y_0,\t{z}_k,\beta,\xi_k}\\
        & \t{y}_{k+1} = \Exp_{y_0}\lrp{\t{z}_{k+1}}
    \end{alignat*}
    where $\Psi$ and $\t{y}_k = \t{\Psi}$ are as defined in \eqref{d:y_k} and \eqref{d:ty_k} respectively. Let $\r_1, K,\delta$ satisfy
    \begin{alignat*}{1}
        & \delta\leq \min\lrbb{\frac{\C_r}{2 L_\beta}, \frac{\C_r^2}{4 L_\xi^2}}\\
        & K\delta \leq \min\lrbb{\frac{1}{64L_\beta'}, \frac{1}{64 {L_\xi'}^2}, \frac{1}{16 \sqrt{L_R} L_\beta}, \frac{1}{256 L_R L_\xi^2}}\\
        & \r_1 \leq \C_r
        \elb{ass:retraction_constants}
    \end{alignat*}
    where $\C_r$ is defined in \eqref{d:c_r}.
    
    Let ${A_k}$ be the event defined as $A_k := \lrbb{\sup_{i\leq k} \dist\lrp{y_i, y_0} \leq \r_1 } \cap \lrbb{\sup_{i\leq k} \lrn{\zt_i} \leq \r_1}$. Then
    \begin{alignat*}{1}
        &\ind{A_k}\Ep{k}{\dist\lrp{\t{y}_{k+1}, y_{k+1}}^2}\\
        \leq& \ind{A_{k-1}}\left(\lrp{1 + \frac{1}{K}} \dist\lrp{y_k, \t{y}_k}^2 +  9K\delta^2 L_R^2 r_1^4 L_\beta^2 + 
        9K\delta^2 {L_\beta'}^2 r_1^2 + 8 \delta L_R^2 r_1^4 L_\xi^2 \right.\\
        &\quad \left. + 2^{27} K \lrp{{L_R'}^2 r_1^4 + L_R^2 r_1^2} \lrp{\delta^4 L_\beta^4 + \delta^2 L_\xi^4}\right) 
    \end{alignat*}
    where $\Ep{k}{\cdot}$ denotes expectation conditioned on $\xi_0...\xi_{k-1}$.
\end{lemma}
\begin{proof}[Proof of Lemma \ref{l:one-step-retraction-bound-yk}]
    For the rest of this proof, consider a fixed $k$. We will always condition on the event $A_k$, but to reduce notation we will not state this explicitly throughout the proof. Whenever we use $\party{y_0}{\t{y}_k}$, we refer to parallel transport along the geodesic $\Exp_{y_0}\lrp{t \t{z}_k}$.

    By our assumptions in \eqref{ass:retraction_constants}, and by definition of the event $A_k$, we guarantee that
        \begin{alignat*}{1}
            &1.\ \dist\lrp{y_k, y_0} \leq \r_1 \leq \C_r\\
            &2.\ \lrn{\zt_k} \leq \r_1 \leq \C_r\\
            &3.\ \lrn{\delta \beta (y_k) + \sqrt{\delta} \xi_k(y_k)} \leq \delta L_\beta + \sqrt{\delta}L_\xi \leq \C_r\\
            &4.\ \lrn{\delta \beta(y_0) + \sqrt{\delta} \party{y_k}{y_0}\lrp{\xi_k\lrp{\t{y}_k}}}  \leq \delta L_\beta + \sqrt{\delta}L_\xi \leq \C_r
            \numberthis \label{e:t:regularity_bounds}
        \end{alignat*}
        where the last two lines use \eqref{ass:retraction_constants}.  
        
        Now consider a fixed $k$. Let us apply Lemma \ref{l:triangle_distortion_G} with $x:=y_0$, $u:=\t{z}_k$, $v:=\delta \beta(y_0) + \sqrt{\delta} \party{\t{y}_k}{y_0}\lrp{\xi_k (\t{y}_k)}$, and $x':= \Exp_x(u) = \t{y}_k$ (note that the conditions on $\lrn{u}\leq \C_r$ and $\lrn{v}\leq \C_r$ are verified in \eqref{e:t:regularity_bounds}), which guarantees that there exists a tensor $G$ such that 
        \begin{alignat*}{1}
            \dist\lrp{\Exp_{x}\lrp{u+v}, \Exp_{x'}\lrp{\party{x}{x'} G(v)} } 
            \leq& 2^{12}\lrp{L_R' \lrp{\lrn{u} + \lrn{v}}^2 + L_R \lrp{\lrn{u} + \lrn{v}}}\lrn{v}^2\\
            \leq& 2^{13} \lrp{L_R' r_1^2 + L_R r_1} \lrp{\delta^2 L_\beta^2 + \delta L_\xi^2}
        \end{alignat*}
        Notice from our definition of $\t{y}_k$ that the $\t{y}_{k+1} = \Exp_{x}(u,v)$, so that
        \begin{alignat*}{1}
            \dist\lrp{\t{y}_{k+1}, \Exp_{x'}\lrp{\party{x}{x'} G(v)} } \leq 2^{13} \lrp{L_R' r_1^2 + L_R r_1} \lrp{\delta^2 L_\beta^2 + \delta L_\xi^2}
            \elb{e:t:alkdsmalds:1}
        \end{alignat*}
        
        We will establish shortly that 
        \begin{alignat*}{1}
            & \Ep{k}{\dist\lrp{y_{k+1}, \Exp_{x'}\lrp{\party{x}{x'} G(v)} }^2 }\\
            \leq& \lrp{1 + \frac{1}{K}} \dist\lrp{y_k, \t{y}_k}^2 +  9K\delta^2 L_R^2 r_1^4 L_\beta^2 + 
            9K\delta^2 {L_\beta'}^2 r_1^2 + 8 \delta L_R^2 r_1^4 L_\xi^2
            \elb{e:t:alkdsmalds:2}
        \end{alignat*}
        where $\Ep{k}{\cdot}$ denotes expectation with respect to the randomness in $\xi_k$, conditioned on $\xi_0...\xi_{k-1}$. Combining \eqref{e:t:alkdsmalds:1} and \eqref{e:t:alkdsmalds:2}, and explicitly writing the $\ind{A_k}$ term,
        \begin{alignat*}{1}
            & \ind{A_k}\Ep{k}{\dist\lrp{\t{y}_{k+1}, y_{k+1}}^2}\\
            \leq& \ind{A_k}\lrp{\lrp{1 + \frac{1}{2K}}\Ep{k}{\dist\lrp{\Exp_{x'}\lrp{\party{x}{x'} G(v)}, {y}_{k+1}}^2} + K\Ep{k}{\dist\lrp{\Exp_{x'}\lrp{\party{x}{x'} G(v)}, \t{y}_{k+1}}^2 }}\\
            \leq& \ind{A_k}\left(\lrp{1 + \frac{1}{K}} \dist\lrp{y_k, \t{y}_k}^2 +  9K\delta^2 L_R^2 r_1^4 L_\beta^2 + 
            9K\delta^2 {L_\beta'}^2 r_1^2 + 8 \delta L_R^2 r_1^4 L_\xi^2 \right.\\
            &\quad \left. + 2^{27} K \lrp{{L_R'}^2 r_1^4 + L_R^2 r_1^2} \lrp{\delta^4 L_\beta^4 + \delta^2 L_\xi^4}\right)\\
            \leq& \ind{A_{k-1}}\left(\lrp{1 + \frac{1}{K}} \dist\lrp{y_k, \t{y}_k}^2 +  9K\delta^2 L_R^2 r_1^4 L_\beta^2 + 
            9K\delta^2 {L_\beta'}^2 r_1^2 + 8 \delta L_R^2 r_1^4 L_\xi^2 \right.\\
            &\quad \left. + 2^{27} K \lrp{{L_R'}^2 r_1^4 + L_R^2 r_1^2} \lrp{\delta^4 L_\beta^4 + \delta^2 L_\xi^4}\right)
        \end{alignat*}
        where the last inequallity is because $\ind{A_k} \leq \ind{A_{k-1}}$ for all $k$, by definition of $A_k$. This proves the lemma.

        For the rest of this proof, we will prove the bound in \eqref{e:t:alkdsmalds:2}. Recall that we earlier defined $v := \delta \beta(y_0) + \sqrt{\delta} \party{\t{y}_k}{y_0}  \xi_k(\t{y}_k)$. From the definition of $G(v)$ (see Lemma \ref{l:triangle_distortion_G}), we verify that $G(v)$ is linear in $v$. In particular,
        \begin{alignat*}{1}
            G(v) = \delta G\lrp{\beta(y_0)} + \sqrt{\delta} \party{\t{y}_k}{y_0} G\lrp{\xi_k(\t{y}_k)}
        \end{alignat*}
        Let us also define 
        \begin{alignat*}{1}
            \t{v} := \party{y_0}{\t{y}_k}G(v) \qquad \qquad \hat{v} := \delta \beta(y_k) + \sqrt{\delta} \xi_k(y_k)
        \end{alignat*}
        so that $y_{k+1} = \Exp_{y_k}(\hat{v})$ and $\Exp_{x'}\lrp{\party{x}{x'} G(v)} = \Exp_{\t{y}_k}(\t{v})$. Bounding $\Ep{k}{\dist\lrp{y_{k+1}, \Exp_{x'}\lrp{\party{x}{x'} G(v)}}^2}$ from \eqref{e:t:alkdsmalds:2} is thus equivalent to bounding $\Ep{k}{\dist\lrp{\Exp_{y_k}(\hat{v}),\Exp_{\t{y}_k}(\t{v})}^2}$. 
        
        Our main tool is Lemma \ref{l:discrete-approximate-synchronous-coupling}. To this end, we will first verify that $L_R \lrp{\lrn{\t{v}}^2 + \lrn{\hat{v}}^2} \leq 1/4$: From Lemma \ref{l:triangle_distortion_G}, we can bound $\lrn{G - Id} \leq \frac{1}{3} L_R \lrn{u}^2 \leq \frac{1}{3} L_R r_1^2 \leq 1$, where the last step uses \eqref{e:t:regularity_bounds} and the definition of $\C_r$ in \eqref{d:c_r}. Therefore, $L_R \lrn{\t{v}} = L_R \lrn{G(v)}^2 \leq 4 L_R \lrn{v}^2 \leq \frac{1}{8}$. Similarly, we can bound $L_R \lrn{v}^2 \leq \frac{1}{8}$. 

        By Lemma \ref{l:discrete-approximate-synchronous-coupling}, 
        \begin{alignat*}{1}
            & \dist\lrp{\Exp_{y_k}(\hat{v}),\Exp_{\t{y}_k}(\t{v})}^2\\
            \leq& \lrp{1 + 8 L_R \lrp{\lrn{G(v)} + \lrn{\hat{v}}}^2} \dist\lrp{y_k, \t{y}_k}^2 \\
            &\quad + 32 \lrn{\t{v} - \party{y_k}{\t{y}_k}\hat{v}}^2 + 2 \lin{\gamma'(0), \party{y_k}{\t{y}_k}\hat{v} - \t{v}}
            \elb{e:t:qondnqowid:1}
        \end{alignat*}
        where $\gamma(t):[0,1] \to M$ is any minimizing geodesic from $\t{y}_k$ to $y_k$, and $\party{y_k}{\t{y}_k}$ denotes parallel transport along $\gamma$. 

        Using our bound on $\lrn{G} \leq 1 + \lrn{G - Id} \leq 2$ above, the first term of \eqref{e:t:qondnqowid:1} can be bounded as
        \begin{alignat*}{1}
            & \lrp{1 + 8 L_R \lrp{\lrn{G(v)} + \lrn{\hat{v}}}^2} \dist\lrp{y_k, \t{y}_k}^2 \\
            \leq& \lrp{1 + 16 L_R \lrp{\lrn{v} + \lrn{\hat{v}}}^2} \dist\lrp{y_k, \t{y}_k}^2\\
            \leq& \lrp{1 + 32 L_R \lrp{\delta^2 L_\beta^2 + \delta L_\xi^2}} \dist\lrp{y_k, \t{y}_k}^2\\
            \leq& \lrp{1 + \frac{1}{8K}} \dist\lrp{y_k, \t{y}_k}^2
        \end{alignat*}
        where the last inequality uses the bound on $K\delta$ from \eqref{ass:retraction_constants}.

        We can bound the second term of \eqref{e:t:qondnqowid:1} as
        \begin{alignat*}{1}
            & \lrn{\t{v} - \party{y_k}{\t{y}_k}\hat{v}}^2\\
            \leq&  \lrn{\delta \party{y_0}{\t{y}_k}G\lrp{\beta(y_0)} + \sqrt{\delta} G\lrp{\xi_k(\t{y}_k)} -  \delta \party{y_k}{\t{y}_k}\beta(y_k) - \sqrt{\delta} \party{y_k}{\t{y}_k}\xi_k(y_k)}^2\\
            \leq& 8 \delta^2 \lrn{G\lrp{\beta(y_0)} - \beta(y_0)}^2 + 8 \delta^2 \lrn{\party{y_0}{\t{y}_k}\beta(y_0) - \beta(\t{y}_k)}^2 + 8 \delta^2 \lrn{\beta(\t{y}_k) - \party{y_k}{\t{y}_k}\beta(y_k)}^2 \\
            &\quad + 8 \delta \lrn{G(\xi_k(\t{y}_k) - \xi_k(\t{y}_k)}^2 + 8 \lrn{\xi_k(\t{y}_k) - \xi_k(y_k)}^2
        \end{alignat*}
        By \eqref{e:t:regularity_bounds}, Assumption \ref{ass:beta_lipschitz}, and our bound on $\lrn{G-I}$, $\lrn{G\lrp{\beta(y_0)} - \beta(y_0)} \leq L_R r_1^2 \cdot L_\beta$. Similarly, $ \lrn{G(\xi_k(\t{y}_k) - \xi_k(\t{y}_k)} \leq L_R r_1^2 L_\xi$. By Assumption \ref{ass:regularity_of_xi}, $\lrn{\xi_k(\t{y}_k) - \xi_k(y_k)} \leq {L_\xi'} \dist\lrp{y_k,\t{y}_k}$. By Assumption \ref{ass:beta_lipschitz}, $\lrn{\party{y_0}{\t{y}_k}\beta(y_0) - \beta(\t{y}_k)} \leq {L_\beta'} r_1$ and $\lrn{\beta(\t{y}_k) - \party{y_k}{\t{y}_k}\beta(y_k)} \leq {L_\beta'} \dist\lrp{y_k,\t{y}_k}$. Put together, we can bound
        \begin{alignat*}{1}
            \lrn{\t{v} - \party{y_k}{\t{y}_k}\hat{v}}^2 
            \leq& 8\lrp{\delta^2 {L_\beta'}^2 + \delta {L_\xi'}^2} \dist\lrp{y_k,\t{y}_k}^2 +  8 L_R^2 r_1^4 \lrp{\delta^2 L_\beta^2 + \delta L_\xi^2} + 
            8\delta^2 {L_\beta'}^2 r_1^2\\
            \leq& \frac{1}{8K}\dist\lrp{y_k,\t{y}_k}^2 + 8 L_R^2 r_1^4 \lrp{\delta^2 L_\beta^2 + \delta L_\xi^2} + 
            8\delta^2 {L_\beta'}^2 r_1^2
            \elb{e:t:pwojmfkds}
        \end{alignat*}
        where the last inequality uses the bound on $K\delta$ from \eqref{ass:retraction_constants}.
        
        Finally, we bound the third term of \eqref{e:t:qondnqowid:1}:
        \begin{alignat*}{1}
            & \lin{\gamma'(0), \party{y_k}{\t{y}_k}\hat{v} - \t{v} + \lrp{\sqrt{\delta} \party{y_k}{\t{y}_k}\xi_k(y_k) - \sqrt{\delta} G\lrp{\xi_k(\t{y}_k)}}} + \lin{\gamma'(0), \sqrt{\delta} \party{y_k}{\t{y}_k}\xi_k(y_k) - \sqrt{\delta} G\lrp{\xi_k(\t{y}_k)}}\\
            \leq& \dist\lrp{t_k,\t{y}_k} \lrn{\party{y_k}{\t{y}_k}\hat{v} - \t{v} + \lrp{\sqrt{\delta} \party{y_k}{\t{y}_k}\xi_k(y_k) - \sqrt{\delta} G\lrp{\xi_k(\t{y}_k)}}} + \lin{\gamma'(0), \sqrt{\delta} \party{y_k}{\t{y}_k}\xi_k(y_k) - \sqrt{\delta} G\lrp{\xi_k(\t{y}_k)}}\\
            \leq& \frac{1}{2K}\dist\lrp{t_k,\t{y}_k}^2 + K\lrn{\party{y_k}{\t{y}_k}\hat{v} - \t{v} + \lrp{\sqrt{\delta} \party{y_k}{\t{y}_k}\xi_k(y_k) - \sqrt{\delta} G\lrp{\xi_k(\t{y}_k)}}}^2 + \lin{\gamma'(0), \sqrt{\delta} \party{y_k}{\t{y}_k}\xi_k(y_k) - \sqrt{\delta} G\lrp{\xi_k(\t{y}_k)}}\\
            \leq& \frac{1}{2K}\dist\lrp{t_k,\t{y}_k}^2 + 8K\delta^2 {L_\beta'}^2 \dist\lrp{y_k,\t{y}_k}^2 +  8K\delta^2 L_R^2 r_1^4 L_\beta^2 + 
            8K\delta^2 {L_\beta'}^2 r_1^2 + \lin{\gamma'(0), \sqrt{\delta} \party{y_k}{\t{y}_k}\xi_k(y_k) - \sqrt{\delta} G\lrp{\xi_k(\t{y}_k)}}\\
            \leq& \frac{3}{4K}\dist\lrp{t_k,\t{y}_k}^2 +  8K\delta^2 L_R^2 r_1^4 L_\beta^2 + 
            8K\delta^2 {L_\beta'}^2 r_1^2 + \lin{\gamma'(0), \sqrt{\delta} \party{y_k}{\t{y}_k}\xi_k(y_k) - \sqrt{\delta} G\lrp{\xi_k(\t{y}_k)}}
        \end{alignat*}
        where the second inequality uses very similar steps as \eqref{e:t:pwojmfkds}, and the last inequality uses the bound on $K\delta$ from \eqref{ass:retraction_constants}. The reasoning for our decomposition above is as follows: let $\Ep{k}{}$ denote expectation wrt randomness in $\xi_k$ conditioned on $\xi_0...\xi_{k-1}$. We verify that $\Ep{k}{\lin{\gamma'(0), \sqrt{\delta} G\lrp{\xi_k(\t{y}_k)} -  \sqrt{\delta} \party{y_k}{\t{y}_k}\xi_k(y_k)}} = 0$. In particular, this is because $\xi_k$ and $G(\xi_k)$ have $0$ mean, because $G(v)$ is linear in $v$. Plugging everything back into \eqref{e:t:qondnqowid:1} and taking expectation wrt $\Ep{k}{\cdot}$,
        \begin{alignat*}{1}
            \Ep{k}{\dist\lrp{\Exp_{y_k}(\hat{v}),\Exp_{\t{y}_k}(\t{v})}^2}
            \leq& \lrp{1 + \frac{1}{K}} \dist\lrp{y_k, \t{y}_k}^2 +  9K\delta^2 L_R^2 r_1^4 L_\beta^2 + 
            9K\delta^2 {L_\beta'}^2 r_1^2 + 8 \delta L_R^2 r_1^4 L_\xi^2
        \end{alignat*}
        where the second inequality is by bounds on $K\delta$ from \eqref{ass:retraction_constants}. This concludes the proof of \eqref{e:t:alkdsmalds:2}.
\end{proof}

\begin{lemma}
    \label{l:useful_xlogx}
    Let $c\in \Re^+$ be such that $c \geq 3$. For any $x$ satisfying $x\geq 3 c \log c$, we have that 
    \begin{alignat*}{1}
        \frac{x} {\log x} \geq c
    \end{alignat*}
\end{lemma}
\begin{proof}
    Let $y := 3 c \log c$. We will first verify that $\frac{y}{\log(y)} \geq c \log c$:
    \begin{alignat*}{1}
        c \log (y)
        = c \lrp{\log c + \log \log c + \log 3}
        \leq 3 c \log c = y
    \end{alignat*}
    where we use the fact that $\log \log c \leq \log c$ for all $c$, and $\log 3 \leq \log c$ as $3\leq c$ by assumption. Dividing both sides by $\log (y)$ gives the desired conclusion.

    Next, we note that $\frac{x}{\log(x)}$ is monotonically increasing for $x\geq e$. Again by assumption that $c\geq 3$, we verify that $y = 3 c \log c \geq 3 c \geq e$. Therefore, for any $x \geq y$, $\frac{x}{\log x} \geq \frac{y}{\log y} \geq c$. This concludes our proof
\end{proof}

\subsection{Orthonormal Frame Computations} \label{ss:orthonormal_frame_computations}

    \begin{lemma}\label{l:derivative_of_pullback_H}
        Given $x\in M$, $u,v\in T_x M$, vector field $\xi: M \to T M$. Let $E=\lrbb{E_1...E_d}$ be an arbitrary orthonormal basis of $T_x M$. Let $\uu,\vv \in \Re^d$ be vectors of indices of $u$ and $v$ wrt $E_i$. 
        
        Let $\xi(x)$ be a vector field that satisfies \ref{ass:regularity_of_xi}. Let $H(z) : T_x M \to T_x M$ be defined as
        \begin{alignat*}{1}
            H(z) = \party{\Exp_{x}(z)}{x} \lrp{\xi\lrp{\Exp_{x}(z)}}
        \end{alignat*}
        where the parallel transport is along the geodesic $\gamma(t) = \Exp_{x}(t z)$. Let us define $\H_k : \Re^d \to \Re^d$ as the vector-valued function, whose $j^{th}$ coordinate is given by
        \begin{alignat*}{1}
            \H_{j}(\zz) = \lin{H\lrp{\zz \circ E}, E_j}
        \end{alignat*}
        (recall that $\zz \circ E = \sum_{j=1}^d \zz_j E_j$).

        Assume that $L_R \lrn{\uu}_2^2 \leq 1$. Then for all $\vv,\ww,\zz$, 
        \begin{alignat*}{2}
            &\lrabs{\lin{\nabla \H_j(\uu), \vv}} 
            &&\leq 4 \lrp{L_\xi' + L_\xi L_R} \lrn{\vv}\\
            &\lrabs{\lin{\nabla^2 \H_j(\uu), \vv \ww^T }_2} 
            &&\leq 208 \lrp{L_\xi'' + L_\xi' L_R\lrn{\uu} + L_R' L_\xi\lrn{\uu} + L_R L_\xi} \lrn{\vv}\lrn{\ww}\\
            &\frac{\lrabs{\lin{\nabla^3 \H_j(\uu), \vv \otimes \ww \otimes \zz}_3}}{\lrn{\vv}\lrn{\ww}\lrn{\zz}}
            &&\leq 8 \lrp{L_\xi''' + L_\xi'' L_R \lrn{\uu}}\\
            & &&\quad + 832\lrp{L_\xi''L_R \lrn{\uu} + L_\xi' \lrp{L_R'\lrn{\uu} + L_R}}\\
            & &&\quad + 416 L_\xi' \lrp{L_R'\lrn{\uu} + L_R}\\
            & &&\quad + 2^{16} L_\xi \lrp{\lrp{L_R'\lrn{\uu} + L_R}\lrp{L_R'\lrn{\uu}^2 + L_R \lrn{\uu}} + L_R' + L_R L_R'' \lrn{\uu}^3} 
        \end{alignat*}
    \end{lemma}

    \begin{proof}[Proof of Lemma \ref{l:derivative_of_pullback_H}]
        Let $\Lambda(s,t) = \exp_x\lrp{t(u+sv)}$.

        Let us define a parallel orthonormal frame $E(s,t) = \lrbb{E_1(s,t)...E_(s,t)}$ as the parallel transport of $E$ along the geodesic $\gamma_s(t):=\Lambda(s,t)$.

        For any $i$, tt follows from our definition that
        \begin{alignat*}{1}
            \H_j(\uu+s\vv) 
            :=& \lin{\party{\Exp_x(u+sv)}{x}\lrp{\xi\lrp{\exp_x\lrp{u+sv}}}, E_i(s,0)}\\
            =& \lin{\xi\lrp{\exp_x\lrp{u+sv}}, E_i(s,1)}
        \end{alignat*}
        where $u = \uu \circ E$ and $v = \vv \circ E$.
        
        Then 
        \begin{alignat*}{1}
            \at{\frac{d}{ds} \H_j\lrp{\uu + s \vv}}{s=0}
            =& \at{\frac{d}{ds} \lin{\xi\lrp{\exp_x\lrp{u+sv}}, E_i(s,1)}}{s=0}\\
            =& \at{D_s \lin{\xi\lrp{\exp_x\lrp{u+sv}}, E_i(s,t)}}{s=0,t=1}
        \end{alignat*}
        
        We now bound $D_s \lin{\xi\lrp{\Lambda(s,t)}, E_i(s,t)}$.
        \begin{alignat*}{1}
            D_s \lin{\xi\lrp{\Lambda(s,t)}, E_i(s,t)}
            =& \lin{D_s \xi\lrp{\Lambda(s,t)}, E_i(s,t)} + \lin{\xi\lrp{\Lambda(s,t)}, D_s E_i(s,t)}
        \end{alignat*}

        By definition, 
        \begin{alignat*}{1}
            \lrn{D_s \xi\lrp{\Lambda(s,t)}} 
            =& \lrn{\lin{\nabla \xi\lrp{\Lambda(s,t)}, J(s,t)}} \\
            \leq& \lrn{\nabla \xi}\lrn{J} \\
            \leq& L_\xi' \lrn{v} \lrp{1+L_R \lrn{u}^2 \lrp{L_R \lrn{u}^2 e^{L_R \lrn{u}^2} + 1}}\\
            \leq& 4 L_\xi' \lrn{v}
        \end{alignat*}
        (using assumption that $\lrn{u}^2 L_R \leq 1$)

        The second term can be simplified using $D_t E_i(s,t) = 0$, so
        \begin{alignat*}{1}
            D_t D_s E_i(s,t)
            = - R\lrp{J, \gamma'} E_i(s,t)
        \end{alignat*}
        Integrating the above then gives
        \begin{alignat*}{1}
            \lrn{D_s E_i(s,t)} 
            \leq& \int_0^t \lrn{R\lrp{J(s,r), E_i(s,r)}} dr\\
            =& L_R \lrn{J} \\
            \leq& 4L_R \lrn{v}
        \end{alignat*}
        Thus
        \begin{alignat*}{1}
            \lrn{\lin{\xi\lrp{\exp_x\lrp{u+sv}}, D_s E_i(s,t)}} \leq 4L_\xi L_R \lrn{v}
        \end{alignat*}
        Summing the bounds on $D_s \lin{\xi\lrp{\Lambda(s,t)}, E_i(s,t)}$ and $\lrn{D_s E_i(s,t)}$ gives
        \begin{alignat*}{1}
            \lrn{D_s \lin{\xi\lrp{\exp_x\lrp{u+sv}}, E_i(s,t)}} \leq 4\lrp{L_R + L_\xi}\lrn{v}
        \end{alignat*}

        Let us now define a field $\Lambda(r,s,t) = \exp_x \lrp{t\lrp{u + rw + sv}}$. We bound $\lrn{\frac{\del}{\del r} \frac{\del}{\del s} \H\lrp{\uu + s\vv + r\ww}}$. We can verify that this corresponds to
        \begin{alignat*}{1}
            D_r D_s \lin{\xi\lrp{r,s,t}, E_i(r,s,t)}
            =& \lin{D_r D_s \xi\lrp{r,s,t}, E_i(r,s,t)}\\
            &\quad + \lin{D_s \xi\lrp{r,s,t}, D_r E_i(r,s,t)}\\
            &\quad + \lin{D_r \xi\lrp{r,s,t}, D_s E_i(r,s,t)}\\
            &\quad + \lin{\xi\lrp{r,s,t}, D_r D_s E_i(r,s,t)}\\
            \lrn{D_r D_s \lin{\xi\lrp{r,s,t}, E_i(r,s,t)}}
            \leq& L_\xi'' \lrn{J^s}^2\\
            &\quad + L_\xi' \lrn{J^s}\lrn{D_r E_i(r,s,t)}\\
            &\quad + L_\xi' \lrn{J^r}\lrn{D_s E_i(r,s,t)}\\
            &\quad + L_\xi \lrn{D_s D_s E_i(s,t)}\\
            \leq& 2L_\xi''\lrn{v}\lrn{w}\\
            &\quad + 4L_\xi' L_R \lrn{u}\lrn{v}\lrn{w}\\
            &\quad + 4L_\xi' L_R \lrn{u}\lrn{v}\lrn{w}\\
            &\quad + 208 L_\xi \lrp{L_R'\lrn{u} + L_R}\lrn{v}\lrn{w}\\
            \leq& 208 \lrp{L_\xi'' + L_\xi' L_R\lrn{u} + L_R' L_\xi\lrn{u} + L_R L_\xi} \lrn{v}\lrn{w}
        \end{alignat*}
        Where we apply Lemma \ref{l:approx_Ds_Ei} and Lemma \ref{l:approx_Dr_Ds_Ei}.

        Finally, define a field $\Lambda(r,s,t) = \exp_x \lrp{t\lrp{u + rw + sv}}$. We bound $\lrn{\frac{\del}{\del q}\frac{\del}{\del r} \frac{\del}{\del s} \H\lrp{\uu + s\vv + r\ww + q \zz}}$. We can verify that this corresponds to
        \begin{alignat*}{1}
            D_q D_r D_s \lin{\xi\lrp{q,r,s,t}, E_i(q,r,s,t)}
            =& D_q \lrp{\lin{D_r D_s \xi\lrp{q,r,s,t}, E_i(q,r,s,t)}}\\
            &\quad + D_q \lrp{\lin{D_s \xi\lrp{q,r,s,t}, D_r E_i(q,r,s,t)}}\\
            &\quad + D_q \lrp{\lin{D_r \xi\lrp{q,r,s,t}, D_s E_i(q,r,s,t)}}\\
            &\quad + D_q \lrp{\lin{\xi\lrp{q,r,s,t}, D_r D_s E_i(q,r,s,t)}}
        \end{alignat*}

        We will bound the norms of each of the terms above separately
        \begin{alignat*}{1}
            D_q \lrp{\lin{D_r D_s \xi\lrp{q,r,s,t}, E_i(q,r,s,t)}}
            =& \lin{D_q D_r D_s \xi\lrp{q,r,s,t}, E_i(q,r,s,t)}\\
            &\quad + \lin{D_r D_s \xi\lrp{q,r,s,t}, D_q E_i(q,r,s,t)}\\
            \lrn{D_q \lrp{\lin{D_r D_s \xi\lrp{q,r,s,t}, E_i(q,r,s,t)}}}
            \leq& \lrn{D_q D_r D_s \xi\lrp{q,r,s,t}}\\
            &\quad + \lrn{D_r D_s \xi\lrp{q,r,s,t}} \lrn{D_q E_i(q,r,s,t)}\\
            \leq& L_\xi'''\lrn{J^q}\lrn{J^r}\lrn{J^s}\\
            &\quad + L_\xi''\lrn{J^r}\lrn{J^s} \lrn{D_q E_i(q,r,s,t)}\\
            \leq& 8 L_\xi''' \lrn{v}\lrn{w}\lrn{z}\\
            &\quad + 8 L_\xi''L_R \lrn{u}\lrn{v}\lrn{w}\lrn{z}\\
            \leq& 8 \lrp{L_\xi''' + L_\xi'' L_R \lrn{u}} \lrn{v}\lrn{w}\lrn{z}
        \end{alignat*}
        
        \begin{alignat*}{1}
            D_q \lrp{\lin{D_s \xi\lrp{q,r,s,t}, D_r E_i(q,r,s,t)}}
            =& \lin{D_q D_s \xi\lrp{q,r,s,t}, D_r E_i(q,r,s,t)}\\
            &\quad + \lin{D_s \xi\lrp{q,r,s,t}, D_q D_r E_i(q,r,s,t)}\\
            \lrn{D_q \lrp{\lin{D_s \xi\lrp{q,r,s,t}, D_r E_i(q,r,s,t)}}}
            \leq& L_\xi'' \lrn{J^q} \lrn{J^s} \lrn{D_r E_i(q,r,s,t)}\\
            &\quad + L_\xi' \lrn{J^s} \lrn{D_q D_r E_i(q,r,s,t)}\\
            \leq& 8 L_\xi''L_R \lrn{u}\lrn{v}\lrn{w}\lrn{z}\\
            &\quad + 416 L_\xi' \lrp{L_R'\lrn{u} + L_R}\lrn{v}\lrn{w}\lrn{z}\\
            \leq& 416\lrp{L_\xi''L_R \lrn{u} + L_\xi' \lrp{L_R'\lrn{u} + L_R}}\lrn{v}\lrn{w}\lrn{z}
        \end{alignat*}

        By symmetry,
        \begin{alignat*}{1}
            \lrn{D_q \lrp{\lin{D_r \xi\lrp{q,r,s,t}, D_s E_i(q,r,s,t)}}}
            \leq& 416\lrp{L_\xi''L_R \lrn{u} + L_\xi' \lrp{L_R'\lrn{u} + L_R}}\lrn{v}\lrn{w}\lrn{z}
        \end{alignat*}

        Finally,
        \begin{alignat*}{1}
            &D_q \lin{\xi\lrp{r,s,t}, D_r D_s E_i(r,s,t)}\\
            =& \lin{D_q \xi\lrp{r,s,t}, D_r D_s E_i(r,s,t)}\\
            &\quad + \lin{\xi\lrp{r,s,t}, D_q D_r D_s E_i(r,s,t)}\\
            &\lrn{D_q \lin{\xi\lrp{r,s,t}, D_r D_s E_i(r,s,t)}}\\
            \leq& L_\xi' \lrn{J^q} \lrn{D_r D_s E_i(r,s,t)}\\
            &\quad + L_\xi \lrn{D_q D_r D_s E_i(r,s,t)}\\
            \leq& 416 L_\xi' \lrp{L_R'\lrn{u} + L_R}\lrn{v}\lrn{w}\lrn{z}\\
            &\quad + 2^{16} L_\xi \lrp{\lrp{L_R'\lrn{u} + L_R}\lrp{L_R'\lrn{u}^2 + L_R \lrn{u}} + L_R' + L_R L_R'' \lrn{u}^3} \lrn{v}\lrn{w}\lrn{z}
        \end{alignat*}
        Where we use Lemma \ref{l:approx_Dq_Dr_Ds_Ei}. Thus
        \begin{alignat*}{1}
            &\lrn{D_q D_r D_s \lin{\xi\lrp{q,r,s,t}, E_i(q,r,s,t)}}\\
            \leq& 8 \lrp{L_\xi''' + L_\xi'' L_R \lrn{u}} \lrn{v}\lrn{w}\lrn{z}\\
            &\quad + 832\lrp{L_\xi''L_R \lrn{u} + L_\xi' \lrp{L_R'\lrn{u} + L_R}}\lrn{v}\lrn{w}\lrn{z}\\
            &\quad + 416 L_\xi' \lrp{L_R'\lrn{u} + L_R}\lrn{v}\lrn{w}\lrn{z} \\
            &\quad + 2^{16} L_\xi \lrp{\lrp{L_R'\lrn{u} + L_R}\lrp{L_R'\lrn{u}^2 + L_R \lrn{u}} + L_R' + L_R L_R'' \lrn{u}^3} \lrn{v}\lrn{w}\lrn{z}
        \end{alignat*}

\end{proof}

\begin{lemma}\label{l:approx_Dr_Js}
    Let $\Lambda(r,s,t) := \exp_x \lrp{t\lrp{u + rw + sv}}$. Assume that $L_R \lrn{u}^2 \leq \frac{1}{4}$. Then for all $t\leq 1$, and for $s=r=0$,
    \begin{alignat*}{1}
        \lrn{D_r J^s} \leq& 48 \lrp{ L_R' \lrn{u}^2  + L_R \lrn{u}}\lrn{v}\lrn{w}\\
        \lrn{D_t D_r J^s} 
        \leq& 96  \lrp{ L_R' \lrn{u}^2  + L_R \lrn{u}}\lrn{v}\lrn{w}
    \end{alignat*}
\end{lemma}
\begin{proof}[Proof of Lemma \ref{l:approx_Dr_Js}]
        From the Jacobi Field formula,
        \begin{alignat*}{1}
            D_r D_t D_t J^s
            =& - D_r \lrp{R\lrp{J^s, \gamma'} \gamma'}
        \end{alignat*}
        Exchanging $D_r$ and $D_t$,
        \begin{alignat*}{1}
            D_t D_r D_t J^s
            =& - D_r \lrp{R\lrp{J^s, \gamma'} \gamma'} - R\lrp{J^r, \gamma'} D_t J^s\\
            D_t D_t D_r J^s
            =& - D_r \lrp{R\lrp{J^s, \gamma'} \gamma'} - R\lrp{J^r, \gamma'} D_t J^s - D_t \lrp{R\lrp{J^r, \gamma'} J^s}\\
        \end{alignat*}

        Simplifying the above,
        \begin{alignat*}{1}
            D_t^2 D_r J^s
            =& - D_r \lrp{R\lrp{J^s, \gamma'} \gamma'} - R\lrp{J^r, \gamma'} D_t J^s - D_t \lrp{R\lrp{J^r, \gamma'} J^s}\\
            =& - \lrp{D_r R}\lrp{J^s, \gamma'} \gamma'\\
            &\quad - R\lrp{ D_r J^s, \gamma'} \gamma'\\
            &\quad - R\lrp{J^s, D_r \gamma'} \gamma'\\
            &\quad - R\lrp{J^s, \gamma'} D_r \gamma'\\
            &\quad - R\lrp{J^r, \gamma'} D_t J^s\\
            &\quad - \lrp{D_t R}\lrp{J^r, \gamma'} J^s\\
            &\quad - R\lrp{D_t J^r, \gamma'} J^s\\
            &\quad - R\lrp{J^r, D_t\gamma'} J^s\\
            &\quad -R\lrp{J^r, \gamma'} D_tJ^s
        \end{alignat*}
        Bounding the norms
        \begin{alignat*}{2}
            & \lrn{\lrp{D_r R}\lrp{J^s, \gamma'} \gamma'} 
                && \leq L_R' \lrn{J^r}\lrn{J^s}\lrn{\gamma'}^2\\
                & &&\leq 4 L_R' \lrn{u}^2 \lrn{v}\lrn{w} \\
            & \lrn{R\lrp{ D_r J^s, \gamma'} \gamma'}
                && \leq L_R \lrn{D_r J^s} \lrn{\gamma'}^2 \\
                & && \leq L_R \lrn{u}^2 \lrn{D_r J^s} \\
            & \lrn{R\lrp{J^s, D_r \gamma'} \gamma'}
                && \leq L_R \lrn{J^s} \lrn{D_t J^r} \lrn{u}\\
                & && \leq 4 L_R \lrn{u} \lrn{v}\lrn{w} \\
            & \lrn{R\lrp{J^s, \gamma'} D_r \gamma'}
                && \leq L_R \lrn{J^s}\lrn{\gamma'} \lrn{D_t J^r}\\
                & && \leq 4 L_R \lrn{u} \lrn{v} \lrn{w}\\
            & \lrn{R\lrp{J^r, \gamma'} D_t J^s}
                && \leq L_R \lrn{J^r}\lrn{\gamma'}\lrn{D_t J^s}\\
                & && \leq 4 L_R \lrn{u}\lrn{v}\lrn{w}\\
            & \lrn{\lrp{D_t R}\lrp{J^r, \gamma'} J^s}
                && \leq L_R' \lrn{\gamma'}^2 \lrn{J^r} \lrn{J^s}\\ 
                & && \leq 4 L_R' \lrn{u}^2 \lrn{v}\lrn{w}\\
            & \lrn{R\lrp{D_t J^r, \gamma'} J^s}
                && \leq L_R \lrn{D_t J^r} \lrn{\gamma'} \lrn{J^s}\\
                & && \leq 4 L_R \lrn{u}\lrn{v}\lrn{w}\\
            & \lrn{R\lrp{J^r, D_t\gamma'} J^s} && = 0 \\
            & \lrn{R\lrp{J^r, \gamma'} D_tJ^s} 
                && \leq L_R \lrn{J^r} \lrn{\gamma'} \lrn{D_t J^s}\\
                & && \leq 4 L_R \lrn{u}\lrn{v}\lrn{w}
        \end{alignat*}

        Summing the above, 
        \begin{alignat*}{1}
            \lrn{D_t^2 D_r J^s}
            \leq& L_R \lrn{u}^2 \lrn{D_r J^s}  + 24\lrp{L_R' \lrn{u}^2  + L_R \lrn{u}}\lrn{v}\lrn{w}
        \end{alignat*}

        Integrating wrt $t$, for $t\leq 1$, 
        \begin{alignat*}{1}
            \lrn{D_r J^s} 
            \leq& \int_0^t \int_0^a L_R \lrn{u}^2 \lrn{D_r J^s}  + 24\lrp{L_R' \lrn{u}^2  + L_R \lrn{u}}\lrn{v}\lrn{w} da db\\
            \leq& \int_0^t L_R \lrn{u}^2 \lrn{D_r J^s}  + 24\lrp{L_R' \lrn{u}^2  + L_R \lrn{u}}\lrn{v}\lrn{w} dt
        \end{alignat*}
        
        By Gronwall's Inequality,
        \begin{alignat*}{1}
            \lrn{D_r J^s} 
            \leq& e^{L_R \lrn{u}^2} 24\lrp{ L_R' \lrn{u}^2  + L_R \lrn{u}}\lrn{v}\lrn{w}\\
            \leq& 48 \lrp{ L_R' \lrn{u}^2  + L_R \lrn{u}}\lrn{v}\lrn{w}
        \end{alignat*}

        It then also follows that
        \begin{alignat*}{1}
            \lrn{D_t D_r J^s} 
            \leq& 48 L_R \lrn{u}^2  \lrp{ L_R' \lrn{u}^2  + L_R \lrn{u}}\lrn{v}\lrn{w} + 24\lrp{ L_R' \lrn{u}^2  + L_R \lrn{u}}\lrn{v}\lrn{w}\\
            \leq& 96  \lrp{ L_R' \lrn{u}^2  + L_R \lrn{u}}\lrn{v}\lrn{w}
        \end{alignat*}
\end{proof}

\begin{lemma}\label{l:approx_Dq_Dr_Js}
    Let $\Lambda(q,r,s,t) := \exp_x \lrp{t\lrp{u + qz + rw + sv}}$. Assume that $L_R \lrn{u}^2 \leq \frac{1}{4}$ $L_R$ and $L_R'\lrn{u}^3\leq \frac{1}{4}$. Then for all $t\leq 1$, and for $s=r=q=0$,
    \begin{alignat*}{1}
        \lrn{D_q D_r J^s} \leq& 2^{14} \lrp{ L_R'' \lrn{u}^2 + L_R' \lrn{u} + L_R}\lrn{v}\lrn{w} \lrn{z}
    \end{alignat*}
\end{lemma}

\begin{proof}[Proof of Lemma \ref{l:approx_Dq_Dr_Js}]
    Using the formula for interchanging covariant derivatives,
    \begin{alignat*}{1}
        D_t D_q D_t D_r J^s
        =& D_q D_t D_t D_r J^s - R\lrp{J^q, \gamma'}D_t D_r J^s\\
        D_t D_t D_q D_r J^s
        =& D_t D_q D_t D_r J^s - D_t \lrp{R\lrp{J^q, \gamma'}D_r J^s}\\
        =& D_q D_t D_t D_r J^s - R\lrp{J^q, \gamma'}D_t D_r J^s  - D_t \lrp{R\lrp{J^q, \gamma'}D_r J^s}
    \end{alignat*}

    We will bound the norms of the above terms one by one

    \textbf{Bounding $D_q D_t D_t D_r J^s$:}
    From the proof of Lemma \ref{l:approx_Dr_Js},
    \begin{alignat*}{1}
        D_t D_t D_r J^s
        =& - D_r \lrp{R\lrp{J^s, \gamma'} \gamma'} - R\lrp{J^r, \gamma'} D_t J^s - D_t \lrp{R\lrp{J^r, \gamma'} J^s}\\
        =& - \lrp{D_r R}\lrp{J^s, \gamma'} \gamma'\\
        &\quad - R\lrp{ D_r J^s, \gamma'} \gamma'\\
        &\quad - R\lrp{J^s, D_r \gamma'} \gamma'\\
        &\quad - R\lrp{J^s, \gamma'} D_r \gamma'\\
        &\quad - R\lrp{J^r, \gamma'} D_t J^s\\
        &\quad - \lrp{D_t R}\lrp{J^r, \gamma'} J^s\\
        &\quad - R\lrp{D_t J^r, \gamma'} J^s\\
        &\quad - R\lrp{J^r, D_t\gamma'} J^s\\
        &\quad -R\lrp{J^r, \gamma'} D_tJ^s
    \end{alignat*}
    We will take derivative of each of these terms wrt $q$.
    \begin{alignat*}{1}
        D_q \lrp{\lrp{D_r R}\lrp{J^s, \gamma'} \gamma'}
        =& \lrp{D_q D_r R}\lrp{J^s, \gamma'} \gamma'\\
        &\quad + \lrp{D_r R}\lrp{D_q J^s, \gamma'} \gamma'\\
        &\quad + \lrp{D_r R}\lrp{J^s, D_q\gamma'} \gamma'\\
        &\quad + \lrp{D_r R}\lrp{J^s, \gamma'} D_q \gamma'\\
        \lrn{D_q \lrp{\lrp{D_r R}\lrp{J^s, \gamma'} \gamma'}}
        \leq& L_R'' \lrn{J^q} \lrn{J^r} \lrn{J^s}\lrn{\gamma'}^2\\
        &\quad + L_R' \lrn{J^r} \lrn{D_q J^s} \lrn{\gamma'}^2\\
        &\quad + L_R' \lrn{J^r}\lrn{J^s}\lrn{D_t J^q} \lrn{\gamma'}\\
        &\quad + L_R' \lrn{J^r}\lrn{J^s}\lrn{\gamma'}\lrn{D_t J^q}\\
        \leq& 8 L_R'' \lrn{u}^2 \lrn{v}\lrn{w}\lrn{z}\\
        &\quad + 96 L_R' \lrn{u} \lrp{ L_R' \lrn{u}^3  + L_R \lrn{u}^2}\lrn{v}\lrn{w}\lrn{z}\\
        &\quad + 8 L_R' \lrn{u}\lrn{v}\lrn{w}\lrn{z}\\
        &\quad + 8 L_R' \lrn{u}\lrn{v}\lrn{w}\lrn{z}\\
        \leq& 96 \lrp{L_R''\lrn{u}^2 + L_R'\lrn{u}\lrp{L_R'\lrn{u}^3 + L_R\lrn{u}^2 + 1}}
    \end{alignat*}
    note: $L_R \lrn{u}^2$ is dimensionless, $L_R' \lrn{u}^3$ is dimensionless, $L_R'' \lrn{u}^4$ is dimensionless.
    \begin{alignat*}{1}
        D_q \lrp{R\lrp{ D_r J^s, \gamma'} \gamma'}
        =& \lrp{D_q R} \lrp{ D_r J^s, \gamma'} \gamma'\\
        &\quad + R\lrp{ D_q D_r J^s, \gamma'} \gamma'\\
        &\quad + R\lrp{ D_r J^s, D_q \gamma'} \gamma'\\
        &\quad + R\lrp{ D_r J^s, \gamma'} D_q \gamma'\\
        \lrn{D_q \lrp{R\lrp{ D_r J^s, \gamma'} \gamma'}}
        \leq& L_R' \lrn{J^q} \lrn{D_r J^s} \lrn{\gamma'}^2\\
        &\quad + L_R \lrn{D_q D_r J^s}\lrn{\gamma'}^2\\
        &\quad + L_R \lrn{D_r J^s} \lrn{D_t J^q} \lrn{\gamma'}\\
        &\quad + L_R \lrn{D_r J^s} \lrn{D_t J^q} \lrn{\gamma'}\\
        \leq& 96 L_R' \lrn{u} \lrp{ L_R' \lrn{u}^3  + L_R \lrn{u}^2}\lrn{v}\lrn{w}\lrn{z}\\
        &\quad + L_R \lrn{u}^2 \lrn{D_q D_r J^s}\\
        &\quad + 48 L_R \lrp{ L_R' \lrn{u}^3  + L_R \lrn{u}^2}\lrn{v}\lrn{w}\lrn{z}\\
        &\quad + 48 L_R \lrp{ L_R' \lrn{u}^3  + L_R \lrn{u}^2}\lrn{v}\lrn{w}\lrn{z}\\
        \leq& L_R \lrn{u}^2 \lrn{D_q D_r J^s} + 96 \lrp{L_R + L_R' \lrn{u}}\lrp{ L_R' \lrn{u}^3  + L_R \lrn{u}^2} \lrn{v}\lrn{w}\lrn{z}
    \end{alignat*}

    \begin{alignat*}{1}
        D_q\lrp{R\lrp{J^s, D_r \gamma'} \gamma'}
        =& D_q\lrp{R\lrp{J^s, D_t J^r} \gamma'}\\
        =& \lrp{D_q R}\lrp{J^s, D_t J^r} \gamma'\\
        &\quad + R\lrp{D_q J^s, D_t J^r} \gamma'\\
        &\quad + R\lrp{J^s, D_q D_t J^r} \gamma'\\
        &\quad + R\lrp{J^s, D_t J^r} D_q \gamma'\\
        =& \lrp{D_q R}\lrp{J^s, D_t J^r} \gamma'\\
        &\quad + R\lrp{D_q J^s, D_t J^r} \gamma'\\
        &\quad + R\lrp{J^s, D_t D_q J^r} \gamma' - R\lrp{J^s, R\lrp{\gamma', J^q} J^r} \gamma'\\
        &\quad + R\lrp{J^s, D_t J^r} D_t J^q\\
        \lrn{D_q\lrp{R\lrp{J^s, D_r \gamma'} \gamma'}}
        \leq& L_R' \lrn{J^q} \lrn{J^s} \lrn{D_t J^r} \lrn{\gamma'}\\
        &\quad + L_R \lrn{D_q J^s} \lrn{D_t J^r} \lrn{\gamma'}\\
        &\quad + L_R \lrn{J^s} \lrn{D_t D_q J^r} \lrn{\gamma'} + L_R \lrn{J^s} \cdot L_R \lrn{\gamma'} \lrn{J^q} \lrn{J^r} \lrn{\gamma'}\\
        &\quad + L_R \lrn{J^s} \lrn{D_t J^r} \lrn{D_t J^q}\\
        \leq& 8 L_R' \lrn{u} \lrn{v}\lrn{w}\lrn{z}\\
        &\quad + 96 L_R \lrp{ L_R' \lrn{u}^3  + L_R \lrn{u}^2}\lrn{v}\lrn{w} \lrn{z}\\
        &\quad + 192 L_R \lrp{ L_R' \lrn{u}^3  + L_R \lrn{u}^2}\lrn{v}\lrn{w} \lrn{z} + 8 L_R^2 \lrn{u}^2 \lrn{v}\lrn{w}\lrn{z}\\
        &\quad + 96 L_R \lrp{ L_R' \lrn{u}^3  + L_R \lrn{u}^2}\lrn{v}\lrn{w} \lrn{z}\\
        &\quad + 8L_R \lrn{v}\lrn{w}\lrn{z}\\
        \leq& 300 \lrp{L_R' \lrn{u} + L_R\lrp{ L_R' \lrn{u}^3  + L_R \lrn{u}^2 + 1}}\lrn{v}\lrn{w} \lrn{z}
    \end{alignat*}

    \begin{alignat*}{1}
        D_q \lrp{R\lrp{J^s, \gamma'} D_r \gamma'}
        =& \lrp{D_q R}\lrp{J^s, \gamma'} D_r \gamma'\\
        &\quad + R\lrp{D_q J^s, \gamma'} D_r \gamma'\\
        &\quad + R\lrp{J^s, D_q \gamma'} D_r \gamma'\\
        &\quad + R\lrp{J^s, \gamma'} D_q D_r \gamma'\\
        \lrn{D_q \lrn{R\lrp{J^s, \gamma'} D_r \gamma'}}
        \leq& L_R' \lrn{J^q} \lrn{J^s} \lrn{\gamma'} \lrn{D_t J^r}\\
        &\quad + L_R \lrn{D_q J^s} \lrn{\gamma'} \lrn{D_t J^r}\\
        &\quad + L_R \lrn{J^s} \lrn{D_t J^q} \lrn{D_t J^r}\\
        &\quad + L_R \lrn{J^s} \lrn{\gamma'} \lrn{D_q D_t J^r}\\
        \leq& 8 L_R' \lrn{u}\lrn{v}\lrn{w}\lrn{z}\\
        &\quad + 96 L_R \lrp{ L_R' \lrn{u}^3  + L_R \lrn{u}^2}\lrn{v}\lrn{w} \lrn{z}\\
        &\quad + 8L_R \lrn{v}\lrn{w}\lrn{z}\\
        &\quad + 192 L_R \lrp{ L_R' \lrn{u}^3  + L_R \lrn{u}^2}\lrn{v}\lrn{w} \lrn{z} + 8 L_R^2 \lrn{u}^2 \lrn{v}\lrn{w}\lrn{z}\\
        \leq& 300 \lrp{L_R' \lrn{u} + L_R\lrp{ L_R' \lrn{u}^3  + L_R \lrn{u}^2 + 1}}\lrn{v}\lrn{w} \lrn{z}
    \end{alignat*}
    where we reused some results from the preceding equation block.

    \begin{alignat*}{1}
        \lrn{D_q \lrp{R\lrp{J^r, \gamma'} D_t J^s}}\leq& 300 \lrp{L_R' \lrn{u} + L_R\lrp{ L_R' \lrn{u}^3  + L_R \lrn{u}^2 + 1}}\lrn{v}\lrn{w} \lrn{z}
    \end{alignat*}
    where we use the preceding equation block, and symmetry (specifically, $D_q \lrp{R\lrp{J^r, \gamma'} D_t J^s}$ vs $D_q \lrp{R\lrp{J^s, \gamma'} D_t J^r}$.

    \begin{alignat*}{1}
        D_q \lrp{\lrp{D_t R}\lrp{J^r, \gamma'} J^s}
        =& \lrp{D_q D_t R}\lrp{J^r, \gamma'} J^s\\
        &\quad + \lrp{D_t R}\lrp{D_q J^r, \gamma'} J^s\\
        &\quad + \lrp{D_t R}\lrp{J^r, D_q \gamma'} J^s\\
        &\quad + \lrp{D_t R}\lrp{J^r, \gamma'} D_q J^s\\
        \lrn{D_q \lrp{\lrp{D_t R}\lrp{J^r, \gamma'} J^s}}
        \leq& L_R'' \lrn{\gamma'}^2 \lrn{J^q} \lrn{J^r} \lrn{J^s}\\
        &\quad + L_R' \lrn{\gamma'}^2 \lrn{D_q J^r} \lrn{J^s}\\
        &\quad + L_R' \lrn{\gamma'} \lrn{J^r} \lrn{J^q}\lrn{J^s}\\
        &\quad + L_R' \lrn{\gamma'}^2 \lrn{J^r} \lrn{D_q J^s}\\
        \leq& 8 L_R'' \lrn{u}^2 \lrn{v}\lrn{w}\lrn{z}\\
        &\quad + 96 L_R' \lrn{u} \lrp{ L_R' \lrn{u}^3  + L_R \lrn{u}^2}\lrn{v}\lrn{w}\lrn{z}\\
        &\quad + 8L_R' \lrn{u}\lrn{v}\lrn{w}\lrn{z}\\
        &\quad + 96 L_R' \lrn{u} \lrp{ L_R' \lrn{u}^3  + L_R \lrn{u}^2}\lrn{v}\lrn{w}\lrn{z}\\
        \leq& 192 \lrp{L_R''\lrn{u}^2 + L_R'\lrn{u}\lrp{ L_R' \lrn{u}^3  + L_R \lrn{u}^2+1}}\lrn{v}\lrn{w}\lrn{z}
    \end{alignat*}

    \begin{alignat*}{1}
        D_q\lrp{R\lrp{D_t J^r, \gamma'} J^s}
        =& \lrp{D_q R}\lrp{D_t J^r, \gamma'} J^s\\
        &\quad + R\lrp{D_q D_t J^r, \gamma'} J^s\\
        &\quad + R\lrp{D_t J^r, D_q \gamma'} J^s\\
        &\quad + R\lrp{D_t J^r, \gamma'} D_q J^s\\
        \lrn{D_q\lrp{R\lrp{D_t J^r, \gamma'} J^s}}
        \leq& 300 \lrp{L_R' \lrn{u} + L_R\lrp{ L_R' \lrn{u}^3  + L_R \lrn{u}^2 + 1}}\lrn{v}\lrn{w} \lrn{z}
    \end{alignat*}
    where by symmetry, we use the same upper bound as $\lrn{D_q\lrp{R\lrp{J^s, D_r \gamma'} \gamma'}}$ from earlier.

    \begin{alignat*}{1}
        D_q \lrp{R\lrp{J^r, D_t\gamma'} J^s} = 0
    \end{alignat*}

    \begin{alignat*}{1}
        \lrn{D_q \lrp{R\lrp{J^r, \gamma'} D_t J^s}}\leq& 300 \lrp{L_R' \lrn{u} + L_R\lrp{ L_R' \lrn{u}^3  + L_R \lrn{u}^2 + 1}}\lrn{v}\lrn{w} \lrn{z}
    \end{alignat*}
    (We have bound this previously)

    Summing the above, we can thus bound
    \begin{alignat*}{1}
        \lrn{D_q D_t D_t D_r J^s} 
        \leq& 96 \lrp{L_R''\lrn{u}^2 + L_R'\lrn{u}\lrp{L_R'\lrn{u}^3 + L_R\lrn{u}^2 + 1}}\\
        &\quad + 96 \lrp{L_R + L_R' \lrn{u}}\lrp{ L_R' \lrn{u}^3  + L_R \lrn{u}^2} \lrn{v}\lrn{w}\lrn{z}\\
        &\quad+ 300 \lrp{L_R' \lrn{u} + L_R\lrp{ L_R' \lrn{u}^3  + L_R \lrn{u}^2 + 1}}\lrn{v}\lrn{w} \lrn{z}\\
        &\quad+ 300 \lrp{L_R' \lrn{u} + L_R\lrp{ L_R' \lrn{u}^3  + L_R \lrn{u}^2 + 1}}\lrn{v}\lrn{w} \lrn{z}\\
        &\quad+ 300 \lrp{L_R' \lrn{u} + L_R\lrp{ L_R' \lrn{u}^3  + L_R \lrn{u}^2 + 1}}\lrn{v}\lrn{w} \lrn{z}\\
        &\quad+ 192 \lrp{L_R''\lrn{u}^2 + L_R'\lrn{u}\lrp{ L_R' \lrn{u}^3  + L_R \lrn{u}^2+1}}\lrn{v}\lrn{w}\lrn{z}\\
        &\quad+ 300 \lrp{L_R' \lrn{u} + L_R\lrp{ L_R' \lrn{u}^3  + L_R \lrn{u}^2 + 1}}\lrn{v}\lrn{w} \lrn{z}\\
        &\quad+ 300 \lrp{L_R' \lrn{u} + L_R\lrp{ L_R' \lrn{u}^3  + L_R \lrn{u}^2 + 1}}\lrn{v}\lrn{w} \lrn{z}\\
        &\quad + L_R \lrn{u}^2 \lrn{D_q D_r J^s}\\
        \leq& L_R \lrn{u}^2 \lrn{D_q D_r J^s} + 2500 \lrp{ L_R'' \lrn{u}^2 + \lrp{L_R' \lrn{u} + L_R}\lrp{ L_R' \lrn{u}^3  + L_R \lrn{u}^2 + 1}}\lrn{v}\lrn{w} \lrn{z}
    \end{alignat*}

    Next we bound the last two terms in $D_t D_t D_q D_r J^s$:
    \begin{alignat*}{1}
        \lrn{R\lrp{J^q, \gamma'}D_t D_r J^s}
        \leq& L_R \lrn{J^q} \lrn{\gamma'} \lrn{D_t D_r J^s}\\
        \leq& 192 L_R \lrp{ L_R' \lrn{u}^3  + L_R \lrn{u}^2}\lrn{v}\lrn{w}\lrn{z}
    \end{alignat*}

    \begin{alignat*}{1}
        D_t \lrp{R\lrp{J^q, \gamma'}D_r J^s}
        =& \lrp{D_t R}\lrp{J^q, \gamma'}D_r J^s\\
        &\quad + R\lrp{\lrp{D_t J^q}, \gamma'}D_r J^s\\
        &\quad + R\lrp{J^q, \lrp{D_t \gamma'}}D_r J^s\\
        &\quad + R\lrp{J^q, \gamma'} \lrp{D_t D_r J^s}\\
        \lrn{D_t \lrp{R\lrp{J^q, \gamma'}D_r J^s}}
        \leq& L_R' \lrn{\gamma'}^2\lrn{J^q}\lrn{D_r J^s}\\
        &\quad + L_R \lrn{D_t J^q} \lrn{\gamma'} \lrn{D_r J^s}\\
        &\quad + 0\\
        &\quad + L_R \lrn{J^q}\lrn{\gamma'}\lrn{D_t D_r J^s}\\
        \leq& 96 L_R' \lrn{u} \lrp{ L_R' \lrn{u}^3  + L_R \lrn{u}^2}\lrn{v}\lrn{w}\lrn{z}\\
        &\quad + 96 L_R \lrp{ L_R' \lrn{u}^3  + L_R \lrn{u}^2}\lrn{v}\lrn{w}\lrn{z}\\
        &\quad + 0\\
        &\quad + 192 L_R \lrp{ L_R' \lrn{u}^3  + L_R \lrn{u}^2}\lrn{v}\lrn{w}\lrn{z}\\
        \leq& 200\lrp{L_R + L_R' \lrn{u}}\lrp{ L_R' \lrn{u}^3  + L_R \lrn{u}^2}\lrn{v}\lrn{w}\lrn{z}
    \end{alignat*}

    Plugging all the above into the expression for $D_t D_t D_q D_r J^s$ gives
    \begin{alignat*}{1}
        \lrn{D_t D_t D_q D_r J^s} \leq L_R \lrn{u}^2 \lrn{D_q D_r J^s} + 2^{12} \lrp{ L_R'' \lrn{u}^2 + \lrp{L_R' \lrn{u} + L_R}\lrp{ L_R' \lrn{u}^3  + L_R \lrn{u}^2 + 1}}\lrn{v}\lrn{w} \lrn{z}
    \end{alignat*}
    Integrating wrt $t$,
    \begin{alignat*}{1}
        & \lrn{D_q D_r J^s} \\
        \leq& \int_0^t L_R \lrn{u}^2 \lrn{D_q D_r J^s} + 2^{12} \lrp{ L_R'' \lrn{u}^2 + \lrp{L_R' \lrn{u} + L_R}\lrp{ L_R' \lrn{u}^3  + L_R \lrn{u}^2 + 1}}\lrn{v}\lrn{w} \lrn{z} dt\\
        \leq& e^{L_R\lrn{u}^2} \cdot 2^{12} \lrp{ L_R'' \lrn{u}^2 + \lrp{L_R' \lrn{u} + L_R}\lrp{ L_R' \lrn{u}^3  + L_R \lrn{u}^2 + 1}}\lrn{v}\lrn{w} \lrn{z}\\
        \leq& 2^{12} \lrp{ L_R'' \lrn{u}^2 + \lrp{L_R' \lrn{u} + L_R}\lrp{ L_R' \lrn{u}^3  + L_R \lrn{u}^2 + 1}}\lrn{v}\lrn{w} \lrn{z}\\
        \leq& 2^{14} \lrp{ L_R'' \lrn{u}^2 + L_R' \lrn{u} + L_R}\lrn{v}\lrn{w} \lrn{z}
    \end{alignat*}

\end{proof}

\begin{lemma}\label{l:approx_Ds_Ei}
    \begin{alignat*}{1}
        \lrn{D_s E_i(s,t)}\leq 2L_R \lrn{u}\lrn{v}
    \end{alignat*}
\end{lemma}
\begin{proof}[Proof of Lemma \ref{l:approx_Ds_Ei}]
    stuff
    \begin{alignat*}{1}
        D_t D_s E_i(s,t)
        =& D_s D_t E_i(s,t) - R\lrp{J^s, \gamma'} E_i(s,t)\\
        =& - R\lrp{J^s, \gamma'} E_i(s,t)\\
        \lrn{D_t D_s E_i(s,t)}
        =& \lrn{R\lrp{J^s, \gamma'} E_i(s,t)}\\
        \leq& L_R \lrn{J^s} \lrn{\gamma'}\\
        =& 2L_R \lrn{u}\lrn{v}
    \end{alignat*}
    Where we use the bound on $\lrn{J^s}$ from Lemma \ref{l:jacobi_field_norm_bound} and Lemma \ref{l:sinh_bounds}. Integrating wrt $t$ gives the desired bound.
\end{proof}

\begin{lemma}\label{l:approx_Dr_Ds_Ei}
    \begin{alignat*}{1}
        \lrn{D_r D_s E_i(r, s,t)}
        \leq& 104 \lrp{L_R'\lrn{u} + L_R  \lrp{ L_R' \lrn{u}^3  + L_R \lrn{u}^2 + 1}}\lrn{v}\lrn{w}\\
        \leq& 208 \lrp{L_R'\lrn{u} + L_R}\lrn{v}\lrn{w}
    \end{alignat*}
\end{lemma}
\begin{proof}[Proof of Lemma \ref{l:approx_Dr_Ds_Ei}]
    We mainly need to bound the last quantity
        \begin{alignat*}{1}
            D_t D_r D_s E_i\lrp{r,s,t}
            =& D_r D_t D_s E_i\lrp{r,s,t} - R\lrp{J^r, \gamma'} D_s E_i\lrp{r,s,t}\\
            =& - D_r \lrp{R\lrp{J^s, \gamma'} E_i\lrp{r,s,t}} - R\lrp{J^r, \gamma'} D_s E_i\lrp{r,s,t}\\
            =& - \lrp{D_r R}\lrp{J^s, \gamma'} E_i\lrp{r,s,t}\\
            &\quad - R\lrp{D_r J^s, \gamma'} E_i\lrp{r,s,t}\\
            &\quad - R\lrp{J^s, D_r \gamma'} E_i\lrp{r,s,t}\\
            &\quad - R\lrp{J^s, \gamma'} D_r E_i\lrp{r,s,t}\\
            &\quad - R\lrp{J^r, \gamma'} D_s E_i\lrp{r,s,t}\\
            \lrn{D_t D_r D_s E_i\lrp{r,s,t}}
            \leq& L_R' \lrn{J^r} \lrn{J^s} \lrn{\gamma'}\\
            &\quad + L_R \lrn{D_r J^s} \lrn{\gamma'}\\
            &\quad + L_R \lrn{J^s} \lrn{D_t J^r}\\
            &\quad + L_R \lrn{J^s} \lrn{\gamma'} \lrn{D_r E_i\lrp{r,s,t}}\\
            &\quad + L_R \lrn{J^r} \lrn{\gamma'} \lrn{D_s E_i\lrp{r,s,t}}\\
            \leq& 4L_R' \lrn{u}\lrn{v}\lrn{w}\\
            &\quad + 96 L_R \lrp{ L_R' \lrn{u}^3  + L_R \lrn{u}^2}\lrn{v}\lrn{w}\\
            &\quad + 4L_R \lrn{v}\lrn{w}\\
            &\quad + 4L_R^2 \lrn{u}^2\lrn{v}\lrn{w}\\
            &\quad + 4L_R^2 \lrn{u}^2\lrn{v}\lrn{w}\\
            \leq& 104 \lrp{L_R'\lrn{u} + L_R  \lrp{ L_R' \lrn{u}^3  + L_R \lrn{u}^2 + 1}}\lrn{v}\lrn{w}
        \end{alignat*}
        The above uses Lemma \ref{l:jacobi_field_norm_bound}, Lemma \ref{l:sinh_bounds} and Lemma \ref{l:approx_Dr_Js}.
\end{proof}

\begin{lemma}\label{l:approx_Dq_Dr_Ds_Ei}
    \begin{alignat*}{1}
        \lrn{D_q D_r D_s E_i(q,r,s,t)}\leq 2^{16} \lrp{\lrp{L_R'\lrn{u} + L_R}\lrp{L_R'\lrn{u}^2 + L_R \lrn{u}} + L_R' + L_R L_R'' \lrn{u}^3} \lrn{v}\lrn{w}\lrn{z}
    \end{alignat*}
\end{lemma}
\begin{proof}[Proof of Lemma \ref{l:approx_Dq_Dr_Ds_Ei}]
    \begin{alignat*}{1}
        D_q D_t D_r D_s E_i\lrp{q,r,s,t}
        =& - D_q\lrp{\lrp{D_r R}\lrp{J^s, \gamma'} E_i\lrp{q,r,s,t}}\\
        &\quad - D_q\lrp{R\lrp{D_r J^s, \gamma'} E_i\lrp{q,r,s,t}}\\
        &\quad - D_q\lrp{R\lrp{J^s, D_r \gamma'} E_i\lrp{q,r,s,t}}\\
        &\quad - D_q\lrp{R\lrp{J^s, \gamma'} D_r E_i\lrp{q,r,s,t}}\\
        &\quad - D_q\lrp{R\lrp{J^r, \gamma'} D_s E_i\lrp{q,r,s,t}}
    \end{alignat*}
    We will bound the norm for each of these terms below.
    
    \begin{alignat*}{1}
        D_q\lrp{\lrp{D_r R}\lrp{J^s, \gamma'} E_i\lrp{q,r,s,t}}
        =& \lrp{D_q D_r R}\lrp{J^s, \gamma'} E_i\lrp{q,r,s,t}\\
        &\quad + \lrp{D_r R}\lrp{D_q J^s, \gamma'} E_i\lrp{q,r,s,t}\\
        &\quad + \lrp{D_r R}\lrp{J^s, D_q \gamma'} E_i\lrp{q,r,s,t}\\
        &\quad + \lrp{D_r R}\lrp{J^s, \gamma'} D_q E_i\lrp{q,r,s,t}\\
        \lrn{D_q\lrp{\lrp{D_r R}\lrp{J^s, \gamma'} E_i\lrp{q,r,s,t}}}
        \leq& L_R'' \lrn{J^q}\lrn{J^r} \lrn{J^s} \lrn{\gamma'}\\
        &\quad + L_R' \lrn{J^r} \lrn{D_q J^s} \lrn{\gamma'}\\
        &\quad + L_R' \lrn{J^r} \lrn{J^s} \lrn{D_t J^q}\\
        &\quad + L_R' \lrn{J^r} \lrn{J^s} \lrn{\gamma'} \lrn{D_q E_i\lrp{q,r,s,t}}\\
        \leq& 8 L_R'' \lrn{u} \lrn{v} \lrn{w} \lrn{z}\\
        &\quad + 96 L_R' \lrp{ L_R' \lrn{u}^3  + L_R \lrn{u}^2}\lrn{v}\lrn{w}\lrn{z} \\
        &\quad + 8L_R' \lrn{v}\lrn{w}\lrn{z}\\
        &\quad + 8L_R' L_R \lrn{u}^2\lrn{v}\lrn{w}\lrn{z}\\
        \leq& 104 \lrp{L_R'' \lrn{u} +  L_R' \lrp{ L_R' \lrn{u}^3  + L_R \lrn{u}^2 + 1}}\lrn{v}\lrn{w}\lrn{z}\\
        \leq& 104 \lrp{L_R'' \lrn{u} +  L_R'}\lrn{v}\lrn{w}\lrn{z}
    \end{alignat*}
    
    \begin{alignat*}{1}
        D_q\lrp{R\lrp{D_r J^s, \gamma'} E_i\lrp{q,r,s,t}}
        =& \lrp{D_q R}\lrp{D_r J^s, \gamma'} E_i\lrp{q,r,s,t}\\
        &\quad + R\lrp{D_q D_r J^s, \gamma'} E_i\lrp{q,r,s,t}\\
        &\quad + R\lrp{D_r J^s, D_q \gamma'} E_i\lrp{q,r,s,t}\\
        &\quad + R\lrp{D_r J^s, \gamma'} D_q E_i\lrp{q,r,s,t}\\
        \lrn{D_q\lrp{R\lrp{D_r J^s, \gamma'} E_i\lrp{q,r,s,t}}}
        \leq& L_R ' \lrn{J^q} \lrn{D_r J^s} \lrn{\gamma'}\\
        &\quad + L_R \lrn{D_q D_r J^s} \lrn{\gamma'} \\
        &\quad + L_R \lrn{D_r J^s} \lrn{D_q \gamma'}\\
        &\quad + L_R \lrn{D_r J^s} \lrn{\gamma'} \lrn{D_q E_i\lrp{q,r,s,t}}\\
        \leq& 96 L_R' \lrp{ L_R' \lrn{u}^3  + L_R \lrn{u}^2} \lrn{v}\lrn{w}\lrn{z}\\
        &\quad + 2^{14} L_R \lrn{u} \lrp{ L_R'' \lrn{u}^2 + L_R' \lrn{u} + L_R}\lrn{v}\lrn{w} \lrn{z}\\
        &\quad + 96 L_R \lrp{ L_R' \lrn{u}^3  + L_R \lrn{u}^2} \lrn{v}\lrn{w}\lrn{z}\\
        &\quad + 96 L_R^2 \lrn{u}^2 \lrp{ L_R' \lrn{u}^2  + L_R \lrn{u}} \lrn{v}\lrn{w}\lrn{z}\\
        \leq& 2^{15} \lrp{L_R + L_R' \lrn{u}}\lrp{L_R' \lrn{u}^2 + L_R\lrn{u}}\lrn{v}\lrn{w}\lrn{z}\\
        &\quad + 2^{14} L_R L_R'' \lrn{u}^3\lrn{v}\lrn{w}\lrn{z}
    \end{alignat*}
    Where we use Lemma \ref{l:approx_Dq_Dr_Js}.

    \begin{alignat*}{1}
        D_q\lrp{R\lrp{J^s, D_r \gamma'} E_i\lrp{q,r,s,t}}
        =& \lrp{D_q R}\lrp{J^s, D_r \gamma'} E_i\lrp{q,r,s,t}\\
        &\quad + R\lrp{D_q J^s, D_r \gamma'} E_i\lrp{q,r,s,t}\\
        &\quad + R\lrp{J^s, D_q D_r \gamma'} E_i\lrp{q,r,s,t}\\
        &\quad + R\lrp{J^s, D_r \gamma'} D_q E_i\lrp{q,r,s,t}\\
        \lrn{D_q\lrp{R\lrp{J^s, D_r \gamma'} E_i\lrp{q,r,s,t}}}
        \leq & L_R' \lrn{J^q} \lrn{J^s} \lrn{D_t J^r} \\
        &\quad + L_R \lrn{D_q J^s} \lrn{D_t J^r} \\
        &\quad + L_R \lrn{J^s} \lrn{D_q D_t J^r}\\
        &\quad + L_R \lrn{J^s} \lrn{D_t J^r} \lrn{D_q E_i\lrp{q,r,s,t}}\\
        \leq& 8 L_R' \lrn{v}\lrn{w}\lrn{z}\\
        &\quad + 96 L_R \lrp{ L_R' \lrn{u}^2  + L_R \lrn{u}}\lrn{v}\lrn{w}\lrn{z}\\
        &\quad + 192 L_R \lrp{ L_R' \lrn{u}^2  + L_R \lrn{u}}\lrn{v}\lrn{w}\lrn{z} + 8L_R^2\lrn{u}\lrn{v}\lrn{w}\lrn{z}\\
        &\quad + 8L_R^2 \lrn{u}\lrn{v}\lrn{w}\lrn{s}\\
        \leq& 304 \lrp{L_R' + L_R \lrp{ L_R' \lrn{u}^2  + L_R \lrn{u}}}\lrn{v}\lrn{w}\lrn{z}
    \end{alignat*}
    where we use the fact that 
    \begin{alignat*}{1}
        \lrn{D_q D_t J^r} 
        =& \lrn{D_t D_q J^r - R\lrp{\gamma', J^q} J^r} \\
        \leq& 96  \lrp{ L_R' \lrn{u}^2  + L_R \lrn{u}}\lrn{w}\lrn{z} + 4L_R\lrn{u}\lrn{z}\lrn{w}
    \end{alignat*}
    
    \begin{alignat*}{1}
        D_q\lrp{R\lrp{J^s, \gamma'} D_r E_i\lrp{q,r,s,t}}
        =& \lrp{D_qR}\lrp{J^s, \gamma'} D_r E_i\lrp{q,r,s,t}\\
        &\quad + R\lrp{D_q J^s, \gamma'} D_r E_i\lrp{q,r,s,t}\\
        &\quad + R\lrp{J^s, D_q \gamma'} D_r E_i\lrp{q,r,s,t}\\
        &\quad + R\lrp{J^s, \gamma'} D_q D_r E_i\lrp{q,r,s,t}\\
        \lrn{D_q\lrp{R\lrp{J^s, \gamma'} D_r E_i\lrp{q,r,s,t}}}
        \leq& L_R' \lrn{J^q} \lrn{J^s} \lrn{\gamma'} \lrn{D_r E_i\lrp{q,r,s,t}}\\
        &\quad + L_R \lrn{D_q J^s} \lrn{\gamma'} \lrn{D_r E_i\lrp{q,r,s,t}}\\
        &\quad + L_R \lrn{J^s} \lrn{D_t J^q} \lrn{D_r E_i\lrp{q,r,s,t}}\\
        &\quad + L_R \lrn{J^s}\lrn{\gamma'}\lrn{D_q D_r E_i\lrp{q,r,s,t}}\\
        \leq& 8L_R' L_R \lrn{u}^2 \lrn{v}\lrn{w}\lrn{z}\\
        &\quad + 96 L_R^2 \lrn{u}^2 \lrp{ L_R' \lrn{u}^2  + L_R \lrn{u}} \lrn{v}\lrn{w}\lrn{z}\\
        &\quad + 8 L_R^2 \lrn{u}\lrn{v}\lrn{w}\lrn{z}\\
        &\quad + 208 L_R\lrp{L_R'\lrn{u}^2 + L_R\lrn{u}  \lrp{ L_R' \lrn{u}^3  + L_R \lrn{u}^2 + 1}}\lrn{v}\lrn{w}\lrn{z}\\
        \leq& 300\lrp{L_R'\lrn{u} + L_R}\lrp{L_R'\lrn{u}^2 + L_R \lrn{u}} \lrn{v}\lrn{w}\lrn{z}
    \end{alignat*}

    By symmetry, we can also bound
    \begin{alignat*}{1}
        \lrn{D_q\lrp{R\lrp{J^r, \gamma'} D_s E_i\lrp{q,r,s,t}}}
        \leq 300\lrp{L_R'\lrn{u} + L_R}\lrp{L_R'\lrn{u}^2 + L_R \lrn{u}} \lrn{v}\lrn{w}\lrn{z}
    \end{alignat*}

    Summing,
    \begin{alignat*}{1}
     \lrn{D_q D_t D_r D_s E_i\lrp{q,r,s,t}}
     \leq& 104 \lrp{L_R'' \lrn{u} +  L_R'}\lrn{v}\lrn{w}\lrn{z}\\
     &\quad + 2^{15} \lrp{L_R' \lrn{u} + L_R}\lrp{L_R' \lrn{u}^2 + L_R\lrn{u}}\lrn{v}\lrn{w}\lrn{z}\\
     &\quad + 2^{14} L_R L_R'' \lrn{u}^3\lrn{v}\lrn{w}\lrn{z}\\
     &\quad + 304 \lrp{L_R' + L_R \lrp{ L_R' \lrn{u}^2  + L_R \lrn{u}}}\lrn{v}\lrn{w}\lrn{z}\\
     &\quad + 300\lrp{L_R'\lrn{u} + L_R}\lrp{L_R'\lrn{u}^2 + L_R \lrn{u}} \lrn{v}\lrn{w}\lrn{z}\\
     &\quad + 300\lrp{L_R'\lrn{u} + L_R}\lrp{L_R'\lrn{u}^2 + L_R \lrn{u}} \lrn{v}\lrn{w}\lrn{z}\\
     \leq& 2^{16} \lrp{\lrp{L_R'\lrn{u} + L_R}\lrp{L_R'\lrn{u}^2 + L_R \lrn{u}} + L_R' + L_R L_R'' \lrn{u}^3} \lrn{v}\lrn{w}\lrn{z}
    \end{alignat*}
    
    Finally,
    \begin{alignat*}{1}
        \lrn{D_t D_q D_r D_s E_i\lrp{q,r,s,t}}
        =& \lrn{D_q D_t D_r D_s E_i\lrp{q,r,s,t} - R\lrp{J^q, \gamma'}\lrp{D_r D_s E_i\lrp{q,r,s,t}}}\\
        \leq& \lrn{D_q D_t D_r D_s E_i\lrp{q,r,s,t}} + \lrn{R\lrp{J^q, \gamma'}\lrp{D_r D_s E_i\lrp{q,r,s,t}}}
    \end{alignat*}
    We can bound
    \begin{alignat*}{1}
        \lrn{R\lrp{J^q J^t}\lrp{D_r D_s E_i\lrp{q,r,s,t}}}
        \leq& 4 L_R \lrn{z} \lrn{u} \lrn{D_r D_s E_i\lrp{q,r,s,t}}\\
        \leq& 416 L_R \lrn{u} \lrp{L_R'\lrn{u} + L_R  \lrp{ L_R' \lrn{u}^3  + L_R \lrn{u}^2 + 1}}\lrn{v}\lrn{w}\lrn{z}\\
        \leq& 832 L_R \lrp{L_R'\lrn{u}^2 + L_R \lrn{u}}\lrn{v}\lrn{w}\lrn{z}
    \end{alignat*}

    Therefore, combining the above results and integrating wrt $t\in[0,1]$, 
    \begin{alignat*}{1}
        \lrn{D_q D_r D_s E_i\lrp{q,r,s,t}}
        \leq& 2^{16} \lrp{\lrp{L_R'\lrn{u} + L_R}\lrp{L_R'\lrn{u}^2 + L_R \lrn{u}} + L_R' + L_R L_R'' \lrn{u}^3} \lrn{v}\lrn{w}\lrn{z}
    \end{alignat*}

\end{proof}

\section{Tail Bounds}

\subsection{One-Step Distance Bounds}
\subsubsection{Under Lipschitz Continuity}
\begin{lemma}[One-step distance evolution under Lipschitz Continuity]\label{l:near_tail_bound_one_step}
    Let $\beta$ be a vector field satisfying \ref{ass:beta_lipschitz}. Assume in addition that $\delta \in \Re^+$ satisfies $\delta \leq {\frac{1}{16{L_\beta'}}}$. Let $L_0 := \lrn{\beta(x_0)}$. Let $x_k$ be the following stochastic process:
    \begin{alignat*}{1}
        x_{k+1} = \Exp_{x_k}\lrp{\delta \beta(x_k) + \sqrt{\delta} \xi_k(x_k)}
    \end{alignat*}
    Then for any positive integer $K$, we can bound,
    \begin{alignat*}{1}
        \dist\lrp{x_{k+1}, x_0}^2 
        \leq&  \lrp{1 + 8\delta L_\beta' + \frac{1}{2K} + \delta L_R \lrn{\xi_k(x_k)}^2 + \delta^2 L_R L_0^2}\dist\lrp{x_k, x_0}^2 +  2\delta \lrn{\xi_k(x_k)}^2 + 8K\delta^2 L_0^2\\
        & \quad + \ind{\dist\lrp{x_k,x_0} \leq \frac{1}{\delta \sqrt{L_R} {L_\beta'}}}\lrp{- 2\lin{\sqrt{\delta} \xi_k(x_k), \Exp_{x_k}^{-1}(x_0)}}
    \end{alignat*}
\end{lemma}
\begin{proof}
    We will be using the bound from \cite{zhang2016first} (see Lemma \ref{l:zhang2016}). Let $v := \delta \beta(x_k) + \sqrt{\delta} \xi_k(x_k)$. In case the minimizing geodesic from $x_k$ to $x^*$ is not unique, we let $\Exp_{x_k}^{-1}(x_0)$ be any arbitrary (but consistently fixed) choice. Then Lemma \ref{l:zhang2016} bounds
    \begin{alignat*}{1}
        \dist\lrp{x_{k+1}, x}^2
        \leq& \dist\lrp{x_{k}, x_0}^2 - 2\lin{v, \Exp_{x_k}^{-1}(x_0)} + {\tc\lrp{\sqrt{L_R}\dist\lrp{x_k,x_0}}} \lrn{v}^2\\
        \leq& \dist\lrp{x_{k}, x_0}^2 - 2\lin{v, \Exp_{x_k}^{-1}(x_0)} + \lrp{1+\sqrt{L_R} \dist(x_k,x_0)} \lrn{v}^2
        \elb{e:triangle_lemma_far_copy:1}
    \end{alignat*}
    where $\zeta(r) := \frac{r}{\tanh(r)}$.
    
    We will consider two cases:\\
    \textbf{Case 1: $\dist\lrp{x_k,x_0} \leq \frac{1}{\delta\sqrt{L_R} {L_\beta'}}$}.\\
    From \eqref{e:triangle_lemma_far_copy:1}:
    \begin{alignat*}{1}
        &\dist\lrp{x_{k+1}, x_0}^2 \\
        \leq& \dist\lrp{x_{k}, x_0}^2 - 2\lin{\delta \beta(x_k) + \sqrt{\delta} \xi_k(x_k), \Exp_{x_k}^{-1}(x_0)} + \lrp{1+\sqrt{L_R} \dist(x_k,x_0)} \lrn{\delta \beta(x_k) + \sqrt{\delta} \xi_k(x_k)}^2\\
        \leq& \dist\lrp{x_{k}, x_0}^2 - 2\lin{\delta \beta(x_k) + \sqrt{\delta} \xi_k(x_k), \Exp_{x_k}^{-1}(x_0)} + \delta^2\lrp{{L_0}^2 +  {L_\beta'}^2 \dist\lrp{x_k,x_0}^2} + \delta \lrn{\xi_k(x_k)}^2\\
        &\quad + \delta^2 \sqrt{L_R} \lrp{L_0^2\dist\lrp{x_k,x_0} + {L_\beta'}^2 \dist\lrp{x_k,x_0}^3 }+ \delta \sqrt{L_R} \lrn{{\xi_k(x_k)}}^2 \dist\lrp{x_k,x_0}\\
        \leq& \dist\lrp{x_{k}, x_0}^2 + \delta L_\beta' \dist\lrp{x_{k}, x_0}^2  + K\delta^2 L_0^2 + \frac{1}{4K} \dist\lrp{x_{k}, x_0}^2 - 2\lin{\sqrt{\delta} \xi_k(x_k), \Exp_{x_k}^{-1}(x_0)}\\
        &\quad + \delta^2 L_0^2 + \delta L_\beta' \dist\lrp{x_k,x_0}^2 + \delta \lrn{\xi_k(x_k)}^2\\
        &\quad + \delta^2 L_R L_0^2 \dist\lrp{x_,x_0}^2 + \delta^2 L_0^2 + \delta L_\beta' \dist\lrp{x_k,x_0}^2 + \delta L_R \lrn{\xi_k(x_k)}^2 \dist\lrp{x_k,x_0}^2 + {\delta \lrn{\xi_k(x_k)}^2}\\
        \leq& \lrp{1 + 3\delta L_\beta'+ \delta L_R \lrn{\xi_k(x_k)}^2 + \frac{1}{4K} + \delta^2 L_R L_0^2}\dist\lrp{x_{k}, x_0}^2 - 2\lin{\sqrt{\delta} \xi_k(x_k), \Exp_{x_k}^{-1}(x_0)}\\
        &\quad + \lrp{2K\delta^2 L_0^2 + \delta \lrn{\xi_k(x_k)}^2 + {\delta \lrn{\xi_k(x_k)}^2}}
    \end{alignat*}
    where the third inequality uses the definition of Case 1, and the fourth inequality is by several applications of Young's Inequality.
    
    \textbf{Case 2: $\dist\lrp{x_k,x_0} > \frac{1}{4\delta \sqrt{L_R} {L_\beta'}}$}.\\
    Let us define 
    \begin{alignat*}{1}
        z(t) := \Exp_{x_k}\lrp{t\lrp{\delta \beta(x_k) + \sqrt{\delta}\xi_k(x_k)}}
    \end{alignat*}
    I.e. $z(t)$ interpolates between $x_k$ and $x_{k+1}$. We verify that $z'(t) = \party{z(0)}{z(t)} \lrp{\delta \beta(x_k) + \sqrt{\delta}\xi_k(x_k)}$. We also verify that
    \begin{alignat*}{1}
        \frac{d}{dt} \dist\lrp{z(t), x_0}^2
        \leq& -2\lin{\Exp^{-1}_{z(t)} (x_0), z'(t)}\\
        \leq& \underbrace{- 2 \lin{\delta \beta(z(t)), \Exp_{z(t)}^{-1} (x_0)} }_{\circled{1}} \\
        &\quad + \underbrace{2\lin{\delta \beta(z(t)) - \party{z(0)}{z(t)}\lrp{\delta \beta(x_k) }, \Exp_{z(t)}^{-1} (x_0)}}_{\circled{2}} - \underbrace{2 \lin{\party{z(0)}{z(t)}\lrp{\sqrt{\delta}\xi_k(x_k)}, \Exp_{z(t)}^{-1} (x_0)} }_{\circled{3}}
    \end{alignat*}
    Let's upper bound the terms one by one.
    
    We first bound $\circled{2}$, which represents the "discretization error in drift":
    \begin{alignat*}{1}
        \circled{2} :=& 2 \lin{\delta \beta(z(t)) - \party{z(0)}{z(t)}\lrp{\delta \beta(x_k) }, \Exp_{z(t)}^{-1} (x_0)}\\
        \leq& 2 \lrn{\delta \beta(z(t)) - \party{z(0)}{z(t)}\lrp{\delta \beta(x_k) }} \dist\lrp{z(t), x_0} \\
        \leq& 2 \delta {L_\beta'} \dist\lrp{z(t), x_k} \dist\lrp{z(t), x_0} \\
        \leq& 2\delta L_\beta' \dist\lrp{z(t),x_k}^2 + 2\delta L_\beta' \dist\lrp{z(t),x_0}^2
    \end{alignat*}
    
    By definition of $z(t)$, we know that $\dist\lrp{z(t), x_k} \leq \lrn{\delta \beta(x_k) + \sqrt{\delta} \xi_k(x_k)} \leq \delta L_0 + \delta {L_\beta'} \dist\lrp{x_k,x_0} + \sqrt{\delta} \lrn{\xi_k(x_k)}$, so that $2\delta L_\beta' \dist\lrp{z(t),x_k}^2 \leq 8 \delta^3 {L_\beta'}^3 \dist\lrp{x_k,x_0}^2 + 8\delta^2 {L_\beta'} \lrn{\xi_k(x_k)}^2 + 8 \delta^3 L_\beta' L_0^2\leq \delta L_\beta' \dist\lrp{x_k,x_0}^2 + \delta \lrn{\xi_k(x_k)}^2 + \delta^2 L_0^2$, so that
    \begin{alignat*}{1}
        \circled{2} 
        \leq& \delta L_\beta' \dist\lrp{z(t),x_k}^2 + \delta L_\beta' \dist\lrp{x_k,x_0}^2 + \delta \lrn{\xi_k(x_k)}^2 + \delta^2 L_0^2
    \end{alignat*}
    
    Next, we bound $\circled{3}$, which is the most significant error term. From the definition of Case 2, $\dist\lrp{x_k,x_0} > \frac{1}{\delta\sqrt{L_R} {L_\beta'}}$,
    \begin{alignat*}{1}
        \circled{3} \leq & 2 \lin{\party{z(0)}{z(t)}\lrp{\sqrt{\delta}\xi_k(x_k)}, \Exp_{z(t)}^{-1} (x_0)} \\
        \leq& 2 \sqrt{\delta} \lrn{\xi_k(x_k)} \dist\lrp{z(t), x_0}\\
        \leq& \delta L_\beta' \dist\lrp{z(t),x_0}^2 + \frac{1}{L_\beta'} \lrn{\xi_k(x_k)}^2\\
        \leq& \delta L_\beta' \dist\lrp{z(t),x_0}^2 + {\delta L_R} \lrn{\xi_k(x_k)}^2 \dist\lrp{x_k,x_0}^2
    \end{alignat*}
    where we use our assumption that $\delta \leq \frac{1}{ {L_\beta'}}$.

    Finally, we bound $\circled{1}$ as
    \begin{alignat*}{1}
        - 2 \lin{\delta \beta(z(t)), \Exp_{z(t)}^{-1} (x_0)} 
        \leq& 4\delta L_\beta' \dist\lrp{z(t),x_0}^2 + 4K\delta^2 L_0^2 + \frac{1}{4K} \dist\lrp{z(t),x_0}^2
    \end{alignat*}

    Putting everything together,
    \begin{alignat*}{1}
        \frac{d}{dt} \dist\lrp{z(t), x_0}^2
        \leq& \lrp{6\delta L_\beta' + \frac{1}{4K}} \dist\lrp{z(t), x_0}^2 + \lrp{\delta L_\beta' + \delta L_R \lrn{\xi_k(x_k)}^2} \dist\lrp{x_k, x_0}^2 + \delta \lrn{\xi_k(x_k)}^2 + 4K\delta^2 L_0^2
    \end{alignat*}
    By Gronwall's Lemma (integrating from $t=0$ to $t=1$), 
    \begin{alignat*}{1}
        &\dist\lrp{x_{k+1}, x_0}^2 \\
        =& \dist\lrp{z(1),x_0}^2 \\
        \leq& \exp\lrp{6\delta L_\beta' + \frac{1}{4K} } \dist\lrp{x_k, x_0}^2 + \lrp{2\delta L_\beta' + 2\delta L_R \lrn{\xi_k(x_k)}^2} \dist\lrp{x_k, x_0}^2 + 2\delta \lrn{\xi_k(x_k)}^2 + 8K\delta^2 L_0^2\\
        \leq& \lrp{1 + 8\delta L_\beta' + \frac{1}{2K} + \delta L_R \lrn{\xi_k(x_k)}^2}\dist\lrp{x_k, x_0}^2 +  2\delta \lrn{\xi_k(x_k)}^2 + 8K\delta^2 L_0^2
    \end{alignat*}
    where we use the assumption that $\delta \leq \frac{1}{8 L_\beta'}$.

    \textbf{Combining Case 1 and Case 2:}
    \begin{alignat*}{1}
        \dist\lrp{x_{k+1}, x_0}^2 
        \leq&  \lrp{1 + 8\delta L_\beta' + \frac{1}{2K} + \delta L_R \lrn{\xi_k(x_k)}^2 + \delta^2 L_R L_0^2}\dist\lrp{x_k, x_0}^2 +  2\delta \lrn{\xi_k(x_k)}^2 + 8K\delta^2 L_0^2\\
        & \quad + \ind{\dist\lrp{x_k,x_0} \leq \frac{1}{\delta \sqrt{L_R} {L_\beta'}}}\lrp{- 2\lin{\sqrt{\delta} \xi_k(x_k), \Exp_{x_k}^{-1}(x_0)}}
    \end{alignat*}
\end{proof}
\subsubsection{Under Dissipativity}
\begin{lemma}[One-step distance evolution under Dissipativity]\label{l:far_tail_bound_one_step}
    Assume $\beta$ satisfies \ref{ass:beta_lipschitz}. Let $x^*$ be some point with $\beta(x^*)=0$. Assume that for all $x$ such that $\dist\lrp{x,x^*} \geq \R$, there exists a minimizing geodesic $\gamma : [0,1]\to M$ with $\gamma(0) = x, \gamma(1) = x^*$, and
    \begin{alignat*}{1}
        \lin{\beta(x), \gamma'(0)} \leq - m \dist\lrp{x,x^*}^2
    \end{alignat*}. Assume in addition that $\delta \in \Re^+$ satisfies $\delta \leq {\frac{m}{128 {L_\beta'}^2}}$

    Let $x_k$ be the following stochastic process:
    \begin{alignat*}{1}
        x_{k+1} = \Exp_{x_k}\lrp{\delta \beta(x_k) + \sqrt{\delta} \xi_k(x_k)}
    \end{alignat*}
    Then
    \begin{alignat*}{1}
        \dist\lrp{x_{k+1}, x^*}^2 
        \leq& \lrp{1 - \delta m}\dist\lrp{x_k, x^*}^2 + {\frac{2048 \delta L_R {L_\beta'}^4}{m^5} \lrn{\xi_k(x_k)}^4 + 4\delta {L_\beta'} \R^2}\\
        & \quad + \ind{\dist\lrp{x_k,x^*} \leq \frac{m}{4\delta\sqrt{L_R} {L_\beta'}^2}}\lrp{- 2\lin{\sqrt{\delta} \xi_k(x_k), \Exp_{x_k}^{-1}(x^*)}}
    \end{alignat*}
\end{lemma}
\begin{proof}
    Throughout the proof, it is useful to note that by our assumptions, it must be that $m\leq L_\beta'$. We will be using the bound from \cite{zhang2016first} (see Lemma \ref{l:zhang2016}). Let $v := \delta \beta(x_k) + \sqrt{\delta} \xi_k(x_k)$. In case the minimizing geodesic from $x_k$ to $x^*$ is not unique, we let $\Exp_{x_k}^{-1}(x^*)$ be any arbitrary (but consistently fixed) choice that satisfies $\lin{\beta(x_k), \Exp_{x_k}^{-1}(x^*)} \leq - m \dist\lrp{x_k,x^*}^2$. Then Lemma \ref{l:zhang2016} bounds
    \begin{alignat*}{1}
        \dist\lrp{x_{k+1}, x}^2
        \leq& \dist\lrp{x_{k}, x^*}^2 - 2\lin{v, \Exp_{x_k}^{-1}(x^*)} + {\tc\lrp{\sqrt{L_R}\dist\lrp{x_k,x^*}}} \lrn{v}^2\\
        \leq& \dist\lrp{x_{k}, x^*}^2 - 2\lin{v, \Exp_{x_k}^{-1}(x^*)} + \lrp{1+\sqrt{L_R} \dist(x_k,x^*)} \lrn{v}^2
        \elb{e:triangle_lemma_far_copy}
    \end{alignat*}
    where $\zeta(r) := \frac{r}{\tanh(r)}$.
    
    We will consider two cases:\\
    \textbf{Case 1: $\dist\lrp{x_k,x^*} \leq \frac{m}{4\delta\sqrt{L_R} {L_\beta'}^2}$}.\\
    From \eqref{e:triangle_lemma_far_copy}:
    \begin{alignat*}{1}
        &\dist^2\lrp{x_{k+1}, x^*}^2 \\
        \leq& \dist^2\lrp{x_{k}, x^*} - 2\lin{\delta \beta(x_k) + \sqrt{\delta} \xi_k(x_k), \Exp_{x_k}^{-1}(x^*)} + \lrp{1+\sqrt{L_R} \dist(x_k,x^*)} \lrn{\delta \beta(x_k) + \sqrt{\delta} \xi_k(x_k)}^2\\
        \leq& \dist^2\lrp{x_{k}, x^*} - 2\lin{\delta \beta(x_k) + \sqrt{\delta} \xi_k(x_k), \Exp_{x_k}^{-1}(x^*)} + \delta^2 {L_\beta'}^2 \dist\lrp{x_k,x^*}^2 + \delta \lrn{\xi_k(x_k)}^2\\
        &\quad + \delta^2 \sqrt{L_R} {L_\beta'}^2 \dist\lrp{x_k,x^*}^3 + \delta \sqrt{L_R} \lrn{{\xi_k(x_k)}}^2 \dist\lrp{x_k,x^*}\\
        \leq& \lrp{1 + \delta m/2}\dist^2\lrp{x_{k}, x^*} - 2\lin{\delta \beta(x_k) + \sqrt{\delta} \xi_k(x_k), \Exp_{x_k}^{-1}(x^*)} + \frac{4\delta L_R}{m} \lrn{{\xi_k(x_k)}}^4
        \elb{e:t:dakkmd:1}
    \end{alignat*}
    where we use our assumptions that $\delta \leq m/\lrp{16{L_\beta'}^2}$ and the inequality under Case 1. We used Cauchy Schwarz a few times.

    We can further bound
    \begin{alignat*}{1}
        2\lin{\delta \beta(x_k), \Exp_{x_k}^{-1}(x^*)}
        \leq& \ind{\dist\lrp{x_k,x^*}\geq \R}\lrp{-2m\dist\lrp{x_k,x^*}^2} + \ind{\dist\lrp{x_k,x^*}\leq \R}\lrp{2L_\beta' \dist\lrp{x_k,x^*}^2}\\
        \leq& -2\delta m\dist\lrp{x_k,x^*}^2 + 2\delta \lrp{m + L_\beta'}\R^2
    \end{alignat*}
    Thus
    \begin{alignat*}{1}
        \dist^2\lrp{x_{k+1}, x^*}^2 
        \leq& \lrp{1 -\delta m}\dist^2\lrp{x_{k}, x^*} - 2\lin{\sqrt{\delta} \xi_k(x_k), \Exp_{x_k}^{-1}(x^*)} \\
        &\quad + 2\delta \lrp{m + L_\beta'}\R^2 + \frac{4\delta L_R}{m} \lrn{{\xi_k(x_k)}}^4
        \elb{e:t:omfl:1}
    \end{alignat*}
    where we use the fact that $e^{\delta m} \leq e^{\frac{m^2}{16{L_\beta'}^2}} \leq 2$.

    \textbf{Case 2: $\dist\lrp{x_k,x^*} > \frac{m}{4\delta \sqrt{L_R} {L_\beta'}^2}$}.\\
    Let us define 
    \begin{alignat*}{1}
        z(t) := \Exp_{x_k}\lrp{t\lrp{\delta \beta(x_k) + \sqrt{\delta}\xi_k(x_k)}}
    \end{alignat*}
    I.e. $z(t)$ interpolates between $x_k$ and $x_{k+1}$. We verify that $z'(t) = \party{z(0)}{z(t)} \lrp{\delta \beta(x_k) + \sqrt{\delta}\xi_k(x_k)}$. We also verify that 
    \begin{alignat*}{1}
        \frac{d}{dt} \dist\lrp{z(t), x^*}^2
        \leq& -2\lin{\Exp^{-1}_{z(t)} (x^*), z'(t)}\\
        \leq& \underbrace{- 2 \lin{\delta \beta(z(t)), \Exp_{z(t)}^{-1} (x^*)} }_{\circled{1}} \\
        &\quad + \underbrace{2\lin{\delta \beta(z(t)) - \party{z(0)}{z(t)}\lrp{\delta \beta(x_k) }, \Exp_{z(t)}^{-1} (x^*)}}_{\circled{2}} - \underbrace{2 \lin{\party{z(0)}{z(t)}\lrp{\sqrt{\delta}\xi_k(x_k)}, \Exp_{z(t)}^{-1} (x^*)} }_{\circled{3}}
    \end{alignat*}
    Let's upper bound the terms one by one.
    
    We first bound $\circled{2}$, which represents the "discretization error in drift":
    \begin{alignat*}{1}
        \circled{2} :=& 2 \lin{\delta \beta(z(t)) - \party{z(0)}{z(t)}\lrp{\delta \beta(x_k) }, \Exp_{z(t)}^{-1} (x^*)}\\
        \leq& 2 \lrn{\delta \beta(z(t)) - \party{z(0)}{z(t)}\lrp{\delta \beta(x_k) }} \dist\lrp{z(t), x^*} \\
        \leq& 2 \delta {L_\beta'} \dist\lrp{z(t), x_k} \dist\lrp{z(t), x^*} \\
        \leq& \frac{\delta m}{4} \dist\lrp{z(t),x^*}^2 + \frac{4\delta {L_\beta'}^2}{m} \dist\lrp{z(t),x_k}^2
    \end{alignat*}
    
    By definition of $z(t)$, we know that $\dist\lrp{z(t), x_k} \leq \lrn{\delta \beta(x_k) + \sqrt{\delta} \xi_k(x_k)} \leq \delta {L_\beta'} \dist\lrp{x_k,x^*} + \sqrt{\delta} \lrn{\xi_k(x_k)}$, so that $\frac{4\delta {L_\beta'}^2}{m} \dist\lrp{z(t),x_k}^2 \leq \frac{4\delta^3 {L_\beta'}^4}{m} \dist\lrp{x_k,x^*}^2 + \frac{4\delta^2 {L_\beta'}^2}{m} \lrn{\xi_k(x_k)}^2 \leq \frac{\delta m}{8} \dist\lrp{x_k,x^*}^2 + \delta \lrn{\xi_k(x_k)}^2$, so that
    \begin{alignat*}{1}
        \circled{2} 
        \leq& \frac{\delta m}{4} \dist\lrp{z(t),x^*}^2 + \frac{\delta m}{8} \dist\lrp{x_k,x^*}^2 + \delta \lrn{\xi_k(x_k)}^2
    \end{alignat*}
    
    Next, we bound $\circled{3}$, which is the most significant error term. From the definition of Case 2, $\dist\lrp{x_k,x^*} > \frac{1}{4\delta\sqrt{L_R} {L_\beta'}}$,
    \begin{alignat*}{1}
        \circled{3} \leq & 2 \lin{\party{z(0)}{z(t)}\lrp{\sqrt{\delta}\xi_k(x_k)}, \Exp_{z(t)}^{-1} (x^*)} \\
        \leq& 2 \sqrt{\delta} \lrn{\xi_k(x_k)} \dist\lrp{z(t), x^*}\\
        \leq& \frac{\delta m}{8} \dist\lrp{z(t),x^*}^2 + \frac{8}{m} \lrn{\xi_k(x_k)}^2\\
        \leq& \frac{\delta m}{8} \dist\lrp{z(t),x^*}^2 + \frac{32\delta \sqrt{L_R} {L_\beta'}^2}{m^2} \lrn{\xi_k(x_k)}^2 \dist\lrp{x_k,x^*}\\
        \leq& \frac{\delta m}{8} \dist\lrp{z(t),x^*}^2 + 
        \frac{\delta m}{8} \dist\lrp{x_k,x^*}^2 + \frac{2048 \delta L_R {L_\beta'}^4}{m^5} \lrn{\xi_k(x_k)}^4
    \end{alignat*}
    where we use our assumption that $\delta \leq \frac{m}{128 {L_\beta'}^2}$.

    Finally, we bound $\circled{1}$ as
    \begin{alignat*}{1}
        &- 2 \lin{\delta \beta(z(t)), \Exp_{z(t)}^{-1} (x^*)} \\
        \leq& \ind{\dist\lrp{z(t),x^*}^2 \geq \R}\lrp{-2 \delta m \dist\lrp{z(t),x^*}^2} + \ind{\dist\lrp{z(t),x^*} \leq \R}\lrp{2 \delta L_\beta' \dist\lrp{z(t)x^*}}\\
        \leq& 4 \delta L_\beta' \R^2 - 2\delta m \dist\lrp{\dist\lrp{z(t),x^*}^2}
    \end{alignat*}

    Putting everything together,
    \begin{alignat*}{1}
        \frac{d}{dt} \dist\lrp{z(t), x^*}^2
        \leq& - \frac{3\delta m}{2} \dist\lrp{z(t), x^*}^2 + \frac{\delta m}{4} \dist\lrp{x_k, x^*}^2 +  \frac{2048 \delta L_R {L_\beta'}^4}{m^5} \lrn{\xi_k(x_k)}^4 + 4\delta {L_\beta'} \R^2
    \end{alignat*}
    By Gronwall's Lemma (integrating from $t=0$ to $t=1$), 
    \begin{alignat*}{1}
        \dist\lrp{x_{k+1}, x^*}^2 
        =& \dist\lrp{z(1),x^*}^2 \\
        \leq& \exp\lrp{-{3\delta m/2}} \dist\lrp{x_k, x^*}^2 + \frac{\delta m}{4} \dist\lrp{x_k, x^*}^2 + \frac{2048 \delta L_R {L_\beta'}^4}{m^5} \lrn{\xi_k(x_k)}^4 + 4\delta{L_\beta'} \R^2 \\
        \leq& \lrp{1 - \delta m}\dist\lrp{x_k, x^*}^2 + \frac{2048 \delta L_R {L_\beta'}^4}{m^5} \lrn{\xi_k(x_k)}^4 + 4\delta{L_\beta'} \R^2
        \elb{e:t:omfl:2}
    \end{alignat*}
    where we use the assumption that $\delta \leq \frac{1}{128 m}$ so that $\exp\lrp{-3\delta m/2} \leq 1-5\delta m/4$.

    \textbf{Combining Case 1 and Case 2:}

    Combining \eqref{e:t:omfl:1} and \eqref{e:t:omfl:2},
    \begin{alignat*}{1}
        \dist\lrp{x_{k+1}, x^*}^2 
        \leq& \lrp{1 - \delta m}\dist\lrp{x_k, x^*}^2 + \lrp{\frac{2048 \delta L_R {L_\beta'}^4}{m^5} \lrn{\xi_k(x_k)}^4 + 4\delta {L_\beta'} \R^2}\\
        & \quad + \ind{\dist\lrp{x_k,x^*} \leq \frac{1}{4\delta \sqrt{L_R} {L_\beta'}}}\lrp{- 2\lin{\sqrt{\delta} \xi_k(x_k), \Exp_{x_k}^{-1}(x^*)}}
    \end{alignat*}
\end{proof}

\subsection{$L_p$ Bounds}
\subsubsection{Under Lipschitz Continuity}
\begin{lemma}[L2 Bound and Chevyshev under Lipschitz Continuity]\label{l:near_tail_bound_L2}
    Consider the same setup as Lemma \ref{l:near_tail_bound_one_step}.
    
    Assume in addition that there exists $\sigma_\xi \in \Re^+$ such that for all $x$ and for all $k$, $\E{\lrn{\xi_k(x)}^2} \leq \sigma_\xi^2$. Then for any positive integer $K$, and for all $k\leq K$,
    \begin{alignat*}{1}
        \E{\dist\lrp{x_k,x_0}^2} \leq \exp\lrp{1 + 8K\delta L_\beta' + K\delta L_R \sigma_\xi^2 + K\delta^2 L_R L_0^2} \cdot \lrp{2K\delta \sigma_\xi^2 + 8K^2\delta^2 L_0^2}
    \end{alignat*}
    and
    \begin{alignat*}{1}
        \Pr{\max_{k\leq K} \dist\lrp{x_k,x_0}^2 \geq s}\leq \frac{1}{s}\exp\lrp{2 + 8K\delta L_\beta' + K\delta L_R \sigma_\xi^2 + K\delta^2 L_R L_0^2} \cdot \lrp{2K\delta \sigma_\xi^2 + 4K^2\delta^2 L_0^2}
    \end{alignat*}
\end{lemma}
\begin{proof}
    Let $\F_k$ denote the $\sigma$-field generated by $\xi_0...\xi_{k-1}$.

    To bound the first claim, take expectation of the bound from Lemma \ref{l:near_tail_bound_one_step} wrt $\F_k$:
    \begin{alignat*}{1}
        &\Ep{\F_k}{\dist\lrp{x_{k+1},x_0}^2}\\
        \leq& \Ep{\F_k}{\lrp{1 + 8\delta L_\beta' + \frac{1}{2K} + \delta L_R \lrn{\xi_k(x_k)}^2 + \delta^2 L_R L_0^2}{\dist\lrp{x_k, x_0}^2}} +  2\delta \Ep{\F_k}{\lrn{\xi_k(x_k)}^2} + 8K\delta^2 L_0^2\\
        & \quad + \Ep{\F_k}{\ind{\dist\lrp{x_k,x_0} \leq \frac{1}{\delta \sqrt{L_R} {L_\beta'}}}\lrp{- 2\lin{\sqrt{\delta} \xi_k(x_k), \Exp_{x_k}^{-1}(x_0)}}}\\
        \leq& \lrp{1 + 8\delta L_\beta' + \frac{1}{2K} + \delta L_R \sigma_\xi^2 + \delta^2 L_R L_0^2}{\dist\lrp{x_k, x_0}^2} +  2\delta \sigma_\xi^2 + 8K\delta^2 L_0^2\\
        \leq& \exp\lrp{8\delta L_\beta' + \frac{1}{2K} + \delta L_R \sigma_\xi^2 + \delta^2 L_R L_0^2}{\dist\lrp{x_k, x_0}^2} +  2\delta \sigma_\xi^2 + 8K\delta^2 L_0^2
    \end{alignat*}

    Applying the above recursively,
    \begin{alignat*}{1}
        \E{\dist\lrp{x_K,x_0}^2} \leq \exp\lrp{1 + 8K\delta L_\beta' + K\delta L_R \sigma_\xi^2 + K\delta^2 L_R L_0^2} \cdot \lrp{2K\delta \sigma_\xi^2 + 8K^2\delta^2 L_0^2}
    \end{alignat*}
    The above upper bound clearly also holds for $\E{\dist\lrp{x_k,x_0}^2}$ for all $k\leq K$. This proves our first claim.

    To prove the second claim, let us define
    \begin{alignat*}{1}
        & r_0^2 := 0 \\
        & r_{k+1}^2 := \lrp{1 + 8\delta L_\beta' + \frac{1}{4K} + \delta L_R \lrn{\xi_k(x_k)}^2 + \delta^2 L_R L_0^2}r_k^2 +  2\delta \lrn{\xi_k(x_k)}^2 + 8K\delta^2 L_0^2\\
        & \quad + \ind{\dist\lrp{x_k,x_0} \leq \frac{1}{\delta \sqrt{L_R} {L_\beta'}}}\lrp{- 2\lin{\sqrt{\delta} \xi_k(x_k), \Exp_{x_k}^{-1}(x_0)}}
    \end{alignat*}

    We verify that $r_k$ as defined above is a sub-martingale. Thus by Doob's martingale inequality,
    \begin{alignat*}{1}
        \Pr{\max_{k\leq K} r_k^2 \geq s} \leq \frac{\E{r_K^2}}{s}
    \end{alignat*}
    Furthermore, notice that 
    \begin{alignat*}{1}
        & r_0^2 = \dist\lrp{x_0,x_0}^2 = 0\\
        & r_{k+1}^2 - \dist\lrp{x_{k+1}^2,x_0}^2 \geq \lrp{1 + 8\delta L_\beta' + \frac{2}{K} + \delta L_R \lrn{\xi_k(x_k)}^2 + \delta^2 L_R L_0^2} \lrp{r_k^2 - \dist\lrp{x_k,x_0}^2} \geq 0
    \end{alignat*}
    so that $r_k \geq \dist\lrp{x_k,x_0}$ with probability 1, for all $k$.

    Thus
    \begin{alignat*}{1}
        &\Pr{\max_{k\leq K} \dist\lrp{x_k,x_0}^2 \geq s} \leq \Pr{\max_{k\leq K} r_k^2 \geq s} \leq \frac{\E{r_K^2}}{s} \\
        &\leq \frac{1}{s}\exp\lrp{1 + 8K\delta L_\beta' + K\delta L_R \sigma_\xi^2 + K\delta^2 L_R L_0^2} \cdot \lrp{2K\delta \sigma_\xi^2 + 8K^2\delta^2 L_0^2}
    \end{alignat*}

    The proof for the bound on $r_K^2$ is identical to the proof of the first claim. We conclude our proof of the second claim
\end{proof}
\begin{lemma}[L4 Bound and Chevyshev under Lipschitz Continuity]\label{l:near_tail_bound_L4}
    Let $\beta$ be a vector field satisfying \ref{ass:beta_lipschitz}. Assume in addition that $\delta \in \Re^+$ satisfies $\delta \leq \min\lrbb{\frac{1}{16{L_\beta'}}, \frac{1}{16\sqrt{L_R} L_0}, \frac{1}{L_R d}}$. Let $L_0 := \lrn{\beta(x_0)}$. Let $x_k$ be the following stochastic process:
    \begin{alignat*}{1}
        x_{k+1} = \Exp_{x_k}\lrp{\delta \beta(x_k) + \sqrt{\delta} \xi_k(x_k)}
    \end{alignat*}
   
    Assume in addition that there exists $\sigma_\xi \in \Re^+$ such that for all $x$ and for all $k$, $\E{\lrn{\xi_k(x)}^4} \leq 2d^2$. Then for any positive $K \geq 4$, and for all $k\leq K$,
    \begin{alignat*}{1}
        \E{\dist\lrp{x_k,x_0}^4} \leq \exp\lrp{2 + 8K\delta L_\beta' + K\delta L_R \sigma_\xi^2 + K\delta^2 L_R L_0^2} \cdot \lrp{2K\delta \sigma_\xi^2 + 4K^2\delta^2 L_0^2}
    \end{alignat*}
    and
    \begin{alignat*}{1}
        &\Pr{\max_{k\leq K} \dist\lrp{x_k,x_0} \geq s}\leq \frac{1}{s^4}\exp\lrp{2 + 16 K\delta L_\beta' + 4K\delta L_R d + 3K\delta^2 L_R L_0^2}\lrp{5K^2\delta^2 d^2 + 64K^4\delta^4 L_0^4}
    \end{alignat*}
\end{lemma}
\begin{proof}
    Let $\F_k$ denote the $\sigma$-field generated by $\xi_0...\xi_{k-1}$.

    We will use the following inequality from Lemma \ref{l:near_tail_bound_one_step}:
    \begin{alignat*}{1}
        \dist\lrp{x_{k+1}, x_0}^2 
        \leq&  \lrp{1 + 8\delta L_\beta' + \frac{1}{2K} + \delta L_R \lrn{\xi_k(x_k)}^2 + \delta^2 L_R L_0^2}\dist\lrp{x_k, x_0}^2 +  2\delta \lrn{\xi_k(x_k)}^2 + 8K\delta^2 L_0^2\\
        & \quad + \ind{\dist\lrp{x_k,x_0} \leq \frac{1}{\delta \sqrt{L_R} {L_\beta'}}}\lrp{- 2\lin{\sqrt{\delta} \xi_k(x_k), \Exp_{x_k}^{-1}(x_0)}}
    \end{alignat*}

    Squaring both sides,
    \begin{alignat*}{1}
        & \dist\lrp{x_{k+1},x_0}^4 \\
        \leq& \lrp{1 + 16 \delta L_\beta' + 3\delta L_R \lrn{\xi_k\lrp{x_k}}^2 + 3\delta^2 L_R L_0^2 + 2\delta^2 L_R^2 \lrn{\xi_k(x_k)}^4 + \frac{2}{K}} \dist\lrp{x_k,x_0}^4\\
        &\quad + 5K\delta^2 \lrn{\xi_k(x_k)}^4 + 64 K^3\delta^4L_0^4\\
        &\quad + \circled{*}
        \elb{e:t:alkdsmalsc}
    \end{alignat*}
    where $\circled{*}$ has $0$-mean, and we used a few times Cauchy Schwarz and Young's Inequality.

    Taking expectation wrt $\F_k$,
    \begin{alignat*}{1}
        &\Ep{\F_k}{\dist\lrp{x_{k+1},x_0}^4}\\
        \leq& \lrp{1 + 16 \delta L_\beta' + 3\delta L_R d + 3\delta^2 L_R L_0^2 + 2\delta^2 L_R^2 d^2 + \frac{2}{K}} \dist\lrp{x_k,x_0}^4 + 5K\delta^2 d^2 + 64K^3\delta^4 L_0^4\\
        \leq& \lrp{1 + 16 \delta L_\beta' + 4\delta L_R d + 3\delta^2 L_R L_0^2 + \frac{2}{K}} \dist\lrp{x_k,x_0}^4 + 5K\delta^2 d^2 + 64K^3\delta^4 L_0^4
    \end{alignat*}

    Applying the above recursively,
    \begin{alignat*}{1}
        \E{\dist\lrp{x_{K},x_0}^4} 
        \leq& \exp\lrp{2 + 16 K\delta L_\beta' + 4K\delta L_R d + 3K\delta^2 L_R L_0^2}\lrp{5K^2\delta^2 d^2 + 64K^4\delta^4 L_0^4}
    \end{alignat*}

    To prove the second claim, define
    \begin{alignat*}{1}
        & r_0^4 := 0 \\
        & r_{k+1}^4 := \lrp{1 + 16 \delta L_\beta' + 3\delta L_R \lrn{\xi_k\lrp{x_k}}^2 + 3\delta^2 L_R L_0^2 + 2\delta^2 L_R^2 \lrn{\xi_k(x_k)}^4 + \frac{2}{K}} r_k^4\\
        &\quad + 5K\delta^2 \lrn{\xi_k(x_k)}^4 + 64 K^3\delta^4L_0^4\\
        &\quad + \circled{*}
    \end{alignat*}
    where $\circled{*}$ is the same term as \eqref{e:t:alkdsmalsc}.

    We verify that $r_k$ as defined above is a sub-martingale. Thus by Doob's martingale inequality,
    \begin{alignat*}{1}
        \Pr{\max_{k\leq K} r_k^2 \geq s} \leq \frac{\E{r_K^2}}{s}
    \end{alignat*}
    Furthermore, notice that 
    \begin{alignat*}{1}
        r_0^2 =& \dist\lrp{x_0,x_0}^2 = 0\\
        r_{k+1}^2 - \dist\lrp{x_{k+1}^2,x_0}^2 \geq& \lrp{1 + 16 \delta L_\beta' + 3\delta L_R \lrn{\xi_k\lrp{x_k}}^2 + 3\delta^2 L_R L_0^2 + 2\delta^2 L_R^2 \lrn{\xi_k(x_k)}^4 + \frac{2}{K}} \lrp{r_k^2 - \dist\lrp{x_k,x_0}^2} \\
        \geq& 0
    \end{alignat*}
    so that $r_k \geq \dist\lrp{x_k,x_0}$ with probability 1, for all $k$.

    Thus
    \begin{alignat*}{1}
        &\Pr{\max_{k\leq K} \dist\lrp{x_k,x_0} \geq s} \leq \Pr{\max_{k\leq K} r_k \geq s} \leq \frac{\E{r_K^4}}{s^4} \\
        &\leq \frac{1}{s^4}\exp\lrp{2 + 16 K\delta L_\beta' + 4K\delta L_R d + 3K\delta^2 L_R L_0^2}\lrp{5K^2\delta^2 d^2 + 64K^4\delta^4 L_0^4}
    \end{alignat*}

    The proof for the bound on $r_K^2$ is identical to the proof of the first claim. We conclude our proof of the second claim
\end{proof}
\subsubsection{Under Dissipativity}
\begin{lemma}[L2 Bound and Chevyshev under Dissipativity, Discretized SDE]
    \label{l:far-tail-bound-l2}
    Let $\beta$ be a vector field satisfying \ref{ass:beta_lipschitz}. Let $x^*$ be some point with $\beta(x^*)=0$. Assume that for all $x$ such that $\dist\lrp{x,x^*} \geq \R$, there exists a minimizing geodesic $\gamma : [0,1]\to M$ with $\gamma(0) = x, \gamma(1) = x^*$, and
    \begin{alignat*}{1}
        \lin{\beta(x), \gamma'(0)} \leq - m \dist\lrp{x,x^*}^2
    \end{alignat*}. Assume in addition that $\delta \in \Re^+$ satisfies $\delta \leq {\frac{m}{128 {L_\beta'}^2}}$
    Let $x_k$ be the following stochastic process:
    \begin{alignat*}{1}
        x_{k+1} = \Exp_{x_k}\lrp{\delta \beta(x_k) + \sqrt{\delta} \xi_k(x_k)}
    \end{alignat*}
    where $\xi_k(x_k) \sim \N_{x_k}\lrp{0,I}$

    For any $k$,
    \begin{alignat*}{1}
        \E{\dist\lrp{x_k,x^*}^2} 
        \leq& \exp\lrp{- k\delta m}\E{\dist\lrp{x_0,x^*}^2} + \frac{1 - \lrp{1-\delta m}^k}{\delta m}\lrp{\frac{2048 \delta L_R {L_\beta'}^4d^4}{m^5} + 4\delta {L_\beta'} \R^2}\\
        \leq& \exp\lrp{- k\delta m}\E{\dist\lrp{x_0,x^*}^2} + {\frac{2048 L_R {L_\beta'}^4d^2}{m^6} + \frac{4 {L_\beta'} \R^2}{m}}
    \end{alignat*}
\end{lemma}
\begin{proof}
    In case the minimizing geodesic from $x_k$ to $x^*$ is not unique, we let $\Exp_{x_k}^{-1}(x^*)$ be any arbitrary (but consistently fixed) choice that satisfies $\lin{\beta(x_k), \Exp_{x_k}^{-1}(x^*)} \leq - m \dist\lrp{x_k,x^*}^2$

    Let us define $r_k := \dist\lrp{x_{k},x^*}$. From Lemma \ref{l:far_tail_bound_one_step}, 
    \begin{alignat*}{1}
        \dist\lrp{x_{k+1}, x^*}^2 
        \leq& \lrp{1 - \delta m}\dist\lrp{x_k, x^*}^2 + {\frac{2048 \delta L_R {L_\beta'}^4}{m^5} \lrn{\xi_k(x_k)}^4 + 4\delta {L_\beta'} \R^2}\\
        & \quad + \ind{\dist\lrp{x_k,x^*} \leq \frac{m}{4\delta \sqrt{L_R} {L_\beta'}^2}}\lrp{- 2\lin{\sqrt{\delta} \xi_k(x_k), \Exp_{x_k}^{-1}(x^*)}}
    \end{alignat*}
    Taking expectation,
    \begin{alignat*}{1}
        \E{\dist\lrp{x_{k+1}, x^*}^2 }
        \leq& \lrp{1 - \delta m}\E{\dist\lrp{x_k, x^*}^2} + \frac{2048 \delta L_R {L_\beta'}^4d^2}{m^5} + 4\delta {L_\beta'} \R^2
    \end{alignat*}
    Applying the above recursively,
    \begin{alignat*}{1}
        \E{\dist\lrp{x_{k}, x^*}^2 } 
        \leq& \lrp{1-\delta m}^{k} \E{\dist\lrp{x_{0}, x^*}^2 } + \lrp{\sum_{i=0}^{k-1} \lrp{1-\delta m}^{k-i}}\lrp{\frac{2048 \delta L_R {L_\beta'}^4d^2}{m^5} + 4\delta {L_\beta'} \R^2}\\
        \leq& \exp\lrp{- k\delta m} + \frac{1 - \lrp{1-\delta m}^k}{\delta m}\lrp{\frac{2048 \delta L_R {L_\beta'}^4d^2}{m^5} + 4\delta {L_\beta'} \R^2}
    \end{alignat*}
\end{proof}

\begin{lemma}[L2 Bound under Dissipativity, Exact SDE]
    \label{l:far-tail-bound-l2-brownian}
    Let $\beta$ be a vector field satisfying \ref{ass:beta_lipschitz}. Let $x^*$ be some point with $\beta(x^*)=0$. Assume that for all $x$ such that $\dist\lrp{x,x^*} \geq \R$, there exists a minimizing geodesic $\gamma : [0,1]\to M$ with $\gamma(0) = x, \gamma(1) = x^*$, and
    \begin{alignat*}{1}
        \lin{\beta(x), \gamma'(0)} \leq - m \dist\lrp{x,x^*}^2
    \end{alignat*}. Assume in addition that $\delta \in \Re^+$ satisfies $\delta \leq {\frac{m}{128 {L_\beta'}^2}}$
    Let $x_k$ be the following stochastic process:
    \begin{alignat*}{1}
        x_{k+1} = \Phi\lrp{\delta,x_k,{E^k},\beta,\BB_k}
    \end{alignat*}
    where for each $k$, ${E^k}$ is some orthonormal basis at $x_k$, $\BB_k$ is a $\Re^d$ Brownian Motion, and $\Phi$ is as defined in \eqref{d:x(t)}. 
    For any $k$,
    \begin{alignat*}{1}
        \E{\dist\lrp{x_k,x_0}^2} \leq \exp\lrp{- \delta m}\E{\dist\lrp{x_0,x^*}^2} + \frac{2^{12} L_R {L_\beta'}^4 d^2}{m^{6}} + \frac{8 {L_\beta'} \R^2}{m} 
    \end{alignat*}
\end{lemma}
\begin{proof}
    Consider fixed $k$. Let us define, for $t\in [0,\delta]$,
    \begin{alignat*}{1}
        x^{i,k}(t) :=  \overline{\Phi}\lrp{\delta,x_k,{E^k},\beta,\BB_k,i}
    \end{alignat*}
    Define $x^{i,k}_j := x^{i,k}(j\delta^i)$, where $\delta^i := \delta/2^i$. Notice that by definition, \\
    $x^{i,k}_{j+1} = \Exp_{x^{i,k}_{j+1}}\lrp{\delta^i \beta(x^{i,k}_{j}) + \sqrt{\delta^i} \xi_j\lrp{x^{i,k}_{j}}}$, for some $\xi_j\lrp{x^{i,k}_{j}} \sim \N_{x^{i,k}_{j}}\lrp{0,I}$. By Lemma \ref{l:far-tail-bound-l2}, for all $i$, (note that $2^i \delta^i = \delta$)
    \begin{alignat*}{1}
        \E{\dist\lrp{x^{i,k}_{2^i},x^*}^2} 
        \leq& \exp\lrp{- 2^i \delta^i m}\E{\dist\lrp{x^{i,k}_{0},x^*}^2} + \frac{1 - \lrp{1-\delta^i m}^{2^i}}{\delta^i m}\lrp{\frac{2048 \delta^i L_R {L_\beta'}^4d^4}{m^5} + 4\delta^i {L_\beta'} \R^2}\\
        \leq& \exp\lrp{- \delta m}\E{\dist\lrp{x^{i,k}_{0},x^*}^2} + \frac{2^{i+1}\delta^im}{\delta^i m}\lrp{\frac{2048 \delta^i L_R {L_\beta'}^4d^4}{m^5} + 4\delta^i {L_\beta'} \R^2}\\
        \leq& \exp\lrp{- \delta m}\E{\dist\lrp{x^{i,k}_{0},x^*}^2} + \delta \lrp{\frac{2^{12}L_R {L_\beta'}^4d^4}{m^5} + 8 {L_\beta'} \R^2}
    \end{alignat*}
    From Lemma \ref{l:x(t)_is_brownian_motion}, $x^{i,k}_{2^i} = x^{i,k}(\delta)$ converges almost surely to $x_{k+1} = \Phi\lrp{\delta,x_k,{E^k},\beta,\BB_k}$. Thus
    \begin{alignat*}{1}
        \E{\dist\lrp{x_{k+1},x^*}^2} \leq \exp\lrp{- \delta m}\E{\dist\lrp{x_k,x^*}^2} + \delta \lrp{\frac{2^{12}L_R {L_\beta'}^4d^4}{m^5} + 8 {L_\beta'} \R^2}
    \end{alignat*}
    The conclusion follows by applying the above recursively up to $k=0$.
\end{proof}
\begin{lemma}[L2 Bound under Dissipativity, Nongaussian]
    \label{l:far-tail-bound-l2-nongaussian}
    Let $\beta$ be a vector field satisfying \ref{ass:beta_lipschitz}. Let $x^*$ be some point with $\beta(x^*)=0$. Assume that for all $x$ such that $\dist\lrp{x,x^*} \geq \R$, there exists a minimizing geodesic $\gamma : [0,1]\to M$ with $\gamma(0) = x, \gamma(1) = x^*$, and
    \begin{alignat*}{1}
        \lin{\beta(x), \gamma'(0)} \leq - m \dist\lrp{x,x^*}^2
    \end{alignat*}. Assume in addition that $\delta \in \Re^+$ satisfies $\delta \leq {\frac{m}{128 {L_\beta'}^2}}$
    Let $x_k$ be the following stochastic process:
    \begin{alignat*}{1}
        x_{k+1} = \Exp_{x_k}\lrp{\delta \beta(x_k) + \sqrt{\delta} \xi_k(x_k)}
    \end{alignat*}
    where $\xi_k$ is some random vector field with $\lrn{\xi(x)}\leq L_\xi$ almost surely, for all $x\in M$.

    For any $k$,
    \begin{alignat*}{1}
        \E{\dist\lrp{x_k,x^*}^2} 
        \leq& \exp\lrp{- k\delta m}\E{\dist\lrp{x_0,x^*}^2} + \frac{1 - \lrp{1-\delta m}^k}{\delta m}\lrp{\frac{2048 \delta L_R {L_\beta'}^4L_\xi^4}{m^5} + 4\delta {L_\beta'} \R^2}\\
        \leq& \exp\lrp{- k\delta m}\E{\dist\lrp{x_0,x^*}^2} + {\frac{2048 L_R {L_\beta'}^4L_\xi^2}{m^6} + \frac{4 {L_\beta'} \R^2}{m}}
    \end{alignat*}
\end{lemma}
\begin{proof}
    In case the minimizing geodesic from $x_k$ to $x^*$ is not unique, we let $\Exp_{x_k}^{-1}(x^*)$ be any arbitrary (but consistently fixed) choice that satisfies $\lin{\beta(x_k), \Exp_{x_k}^{-1}(x^*)} \leq - m \dist\lrp{x_k,x^*}^2$

    Let us define $r_k := \dist\lrp{x_{k},x^*}$. From Lemma \ref{l:far_tail_bound_one_step}, 
    \begin{alignat*}{1}
        \dist\lrp{x_{k+1}, x^*}^2 
        \leq& \lrp{1 - \delta m}\dist\lrp{x_k, x^*}^2 + {\frac{2048 \delta L_R {L_\beta'}^4}{m^5} \lrn{\xi_k(x_k)}^4 + 4\delta {L_\beta'} \R^2}\\
        & \quad + \ind{\dist\lrp{x_k,x^*} \leq \frac{m}{4\delta \sqrt{L_R} {L_\beta'}^2}}\lrp{- 2\lin{\sqrt{\delta} \xi_k(x_k), \Exp_{x_k}^{-1}(x^*)}}
    \end{alignat*}
    Taking expectation,
    \begin{alignat*}{1}
        \E{\dist\lrp{x_{k+1}, x^*}^2 }
        \leq& \lrp{1 - \delta m}\E{\dist\lrp{x_k, x^*}^2} + \frac{2048 \delta L_R {L_\beta'}^4L_\xi^2}{m^5} + 4\delta {L_\beta'} \R^2
    \end{alignat*}
    Applying the above recursively,
    \begin{alignat*}{1}
        \E{\dist\lrp{x_{k}, x^*}^2 } 
        \leq& \lrp{1-\delta m}^{k} \E{\dist\lrp{x_{0}, x^*}^2 } + \lrp{\sum_{i=0}^{k-1} \lrp{1-\delta m}^{k-i}}\lrp{\frac{2048 \delta L_R {L_\beta'}^4L_\xi^2}{m^5} + 4\delta {L_\beta'} \R^2}\\
        \leq& \exp\lrp{- k\delta m} + \frac{1 - \lrp{1-\delta m}^k}{\delta m}\lrp{\frac{2048 \delta L_R {L_\beta'}^4L_\xi^2}{m^5} + 4\delta {L_\beta'} \R^2}
    \end{alignat*}
\end{proof}

\subsection{Subgaussian Bounds}

\subsubsection{Under Lipschitz Continuity}
\begin{lemma}[Subgaussian Bound under Lipschitz, Nongaussian Noise]
    \label{l:near_tail_bound_no_stopping}
    Let $\beta$ be a vector field satisfying Assumption \ref{ass:beta_lipschitz}, and assume that there exists a constant $L_\beta$ such that $\lrn{\beta(x)} \leq L_\beta$ for all $x$. Assume $\lrn{\xi} \leq L_\xi$ almost surely. Assume in addition that $K \in \Z^+$ and $\delta \in \Re^+$ satisfy
    \begin{alignat*}{1}
        K\delta\leq\min\lrbb{\frac{1}{8 \sqrt{L_R} L_\beta}, \frac{1}{64 L_R L_\xi^2}, \frac{1}{64 L_\xi^2}}
    \end{alignat*}
    Let $x^*$ be some point with $\beta(x^*)=0$. Let $x_k$ be the following stochastic process:
    \begin{alignat*}{1}
        x_{k+1} = \Exp_{x_k}\lrp{\delta \beta(x_k) + \sqrt{\delta} \xi_k(x_k)}
    \end{alignat*}
    Then for any $t$, and any $k\leq K$,
    \begin{alignat*}{1}
        \Pr{\max_{k\leq K} \dist\lrp{x_k,x_0} \geq t} \leq \exp\lrp{\frac{32K^2\delta^2 {L_\beta}^2 + 8K\delta L_\xi^2 - t^2}{128 K\delta L_\xi^2}}
    \end{alignat*}
\end{lemma}
\begin{proof}[Proof of Lemma \ref{l:near_tail_bound_no_stopping}]
    In case the minimizing geodesic from $x_k$ to $x_0$ is not unique, we let $\Exp_{x_k}^{-1}(x_0)$ be any arbitrary (but consistently fixed) choice that satisfies $\lin{\beta(x_k), \Exp_{x_k}^{-1}(x_0)} \leq - m \dist\lrp{x_k,x_0}^2$. Using the result from Corollary 8 of \cite{zhang2016first} (see Lemma \ref{l:zhang2016}),
    \begin{alignat*}{1}
        \dist\lrp{x_{k+1}, x_0}^2
        \leq& \dist^2\lrp{x_{k}, x_0} - 2\lin{\delta_k \beta(x_k) + \sqrt{\delta_k} \xi_k(x_k), \Exp_{x_k}^{-1}(x_0)}\\
        &\quad + \lrp{1+\sqrt{L_R} \dist(x_k,x_0)} \lrn{\delta_k \beta(x_k) + \sqrt{\delta_k} \xi_k(x_k)}^2
    \end{alignat*}
    where $\zeta(r) := \frac{r}{\tanh(r)}$.
    
    Expanding and simplifying,
    \begin{alignat*}{1}
        & \dist^2\lrp{x_{k+1}, x_0} \\
        \leq& \dist^2\lrp{x_{k}, x_0} - 2\lin{\delta \beta(x_k) + \sqrt{\delta} \xi_k(x_k), \Exp_{x_k}^{-1}(x_0)} + \lrp{1+\sqrt{L_R} \dist(x_k,x_0)} \lrn{\delta \beta(x_k) + \sqrt{\delta} \xi_k(x_k)}^2\\
        \leq& \dist^2\lrp{x_{k}, x_0} - 2\lin{\delta \beta(x_k) + \sqrt{\delta} \xi_k(x_k), \Exp_{x_k}^{-1}(x_0)} + 2\delta^2 {L_\beta}^2 + 2\delta L_\xi^2\\
        &\quad + 2\sqrt{L_R} \dist\lrp{x_k,x_0} \lrp{\delta^2 {L_\beta}^2 + \delta L_\xi^2}\\
        \leq& \lrp{1 + 2L_R \lrp{\delta^2 L_\beta^2 + \delta L_\xi^2}}\dist\lrp{x_{k}, x_0}^2 + 4\delta^2 {L_\beta}^2 + 4\delta L_\xi^2 +2 \delta L_\beta \dist\lrp{x_k.x_0} - 2\sqrt{\delta} \lin{\xi_k(x_k), \Exp_{x_k}^{-1}(x_0)}\\
        \leq& \lrp{1 + \frac{1}{8K}}\dist\lrp{x_{k}, x_0}^2 + 16K\delta^2 {L_\beta}^2 + 4\delta L_\xi^2 \underbrace{- 2\sqrt{\delta} \lin{\xi_k(x_k), \Exp_{x_k}^{-1}(x_0)} }_{:=\theta_k}
        \elb{e:t:oiqmflkq:1}
    \end{alignat*}

    Let us now define $s := \frac{1}{64 K\delta L_\xi^2}$. Using Lemma \ref{l:hoeffding}, we can verify that \\
    $\E{\exp\lrp{s\theta_k}} \leq \E{\exp\lrp{2 s^2 \theta_k^2}} \leq e^{8 \delta s^2 L_\xi^2 \dist\lrp{x_k,x^*}^2} \leq e^{\frac{s^2}{8K} \dist\lrp{x_k,x^*}^2}$. We now apply Lemma \ref{l:doob_maximal}, with
    \begin{alignat*}{1}
        & q_k = s \dist\lrp{x_k,x_0}^2 \qquad \qquad \eta_k := s \theta_k\\
        & a_k = \frac{1}{8 K} \qquad \qquad b_k = s\lrp{16K\delta^2 {L_\beta}^2 + 4\delta L_\xi^2}\\
        & c_k = \frac{s}{8 K} \qquad \qquad d_k = 0
    \end{alignat*}

    Note also that by definition, $q_0 = 0$. Note also that $\sum a_k + c_k = \frac{1}{8}$. Lemma \ref{l:doob_maximal} then guarantees that 
    \begin{alignat*}{1}
        \Pr{\max_{k\leq K} \dist\lrp{x_k,x_0} \geq t} 
        =& \Pr{\max_{k\leq K} s \dist\lrp{x_k,x_0}^2 \geq s t^2} \\
        \leq& \exp\lrp{2q_0 + 2K\lrp{b+d} - \frac{s t^2}{2}}\\
        =& \exp\lrp{s \lrp{16K^2\delta^2 {L_\beta}^2 + 4K \delta L_\xi^2 - \frac{t^2}{2}}} 
    \end{alignat*}
\end{proof}
\begin{lemma}[Subgaussian Bound under Lipschitz, Nongaussian Noise, Euclidean Space]
    \label{l:near_tail_bound_tangent_space_no_stopping}
    Het $\H_k$ be a vector fields satisfying, for all $k$ and for all $\zz$, 
    \begin{alignat*}{1}
        & \E{\H_k(\zz)} = 0 \qquad \E{\H_k(\zz) \H_k(\zz)^T} = I \\
        & \lrn{\H_k(x)} \leq L_\H \qquad \text{almost surely}
    \end{alignat*}
    Let $\vv \in \Re^d$ be some arbitrary vector. Let $\zz_k$ be the following stochastic process:
    \begin{alignat*}{1}
        \zz_{k+1} = \zz_k + \delta \vv + \sqrt{\delta} \H_k(\zz_k)
    \end{alignat*}
    Assume $K\delta \leq \frac{1}{64 L_{\HH}^2}$, then for any $t$, and any $k\leq K$,
    \begin{alignat*}{1}
        \Pr{\max_{k\leq K} \lrn{\zz_k}_2 \geq t} \leq \exp\lrp{ \frac{{32 K^2\delta^2 \lrn{\vv}_2^2 + 4K\delta L_\H^2} - \frac{t^2}{2}}{64K\delta L_\H^2}    }
    \end{alignat*}
\end{lemma}
\begin{proof}[Proof of Lemma \ref{l:near_tail_bound_tangent_space_no_stopping}]
    \begin{alignat*}{1}
        \lrn{\zz_{k+1}}_2^2
        \leq& \lrn{\zz_k}_2^2 + 2\lin{\zz_k, \delta \vv + \sqrt{\delta} \H_k(\zz_k)} + 2 \delta^2 \lrn{\vv}_2^2 + 2\delta \lrn{\H_k(\zz_k)}_2^2\\
        \leq& \lrp{1 + \frac{1}{8K}} \lrn{\zz_k}_2^2 + 16 K\delta^2 \lrn{\vv}_2^2 + 2\delta L_\H^2 + 2\lin{\zz_k, \sqrt{\delta} \H_k(\zz_k)} 
    \end{alignat*}
    Let us now define $s := \frac{1}{64 K\delta L_{\H}^2}$ and $\eta_k := 2s\lin{\zz_k, \sqrt{\delta} \H_k(\zz_k)}$. We verify that $\E{\exp\lrp{s\eta_k}} \leq \E{\exp\lrp{2 s^2 \eta_k^2}} \leq e^{8 \delta s^2 L_{\HH}^2\lrn{\zz_k}_2^2} \leq e^{\frac{s^2}{8K} \lrn{\zz_k}_2^2}$.
    
    We now apply Lemma \ref{l:doob_maximal}, with
    \begin{alignat*}{1}
        & q_k = s \lrn{\zz_k}_2^2\\
        & a_k = \frac{1}{8 K} \qquad \qquad b_k = s\lrp{16 K\delta^2 \lrn{\vv}_2^2 + 2\delta L_\H^2}\\
        & c_k = \frac{s}{8 K} \qquad \qquad d_k = 0
    \end{alignat*}
    we can thus bound
    \begin{alignat*}{1}
        & \Pr{\max_{k\leq K} q_{k} \geq st^2} \leq \exp\lrp{ s\lrp{32 K^2\delta^2 \lrn{\vv}_2^2 + 4K\delta L_\H^2} - \frac{st^2}{2}}\\
        \Rightarrow \qquad 
        & \Pr{\max_{k\leq K} \lrn{\zz_k}_2 \geq t} \leq \exp\lrp{ \frac{{32 K^2\delta^2 \lrn{\vv}_2^2 + 4K\delta L_\H^2} - \frac{t^2}{2}}{64K\delta L_\H^2}    }
    \end{alignat*}
\end{proof}

\subsubsection{Under Dissipativity}
\begin{lemma}[Subgaussian Bound under Dissipativity, Discretized SDE, adaptive stepsize]\label{l:far-tail-bound-gaussian-adaptive}
    Let $\beta$ be a vector field satisfying \ref{ass:beta_lipschitz}. Let $x^*$ be some point with $\beta(x^*)=0$. Assume that for all $x$ such that $\dist\lrp{x,x^*} \geq \R$, there exists a minimizing geodesic $\gamma : [0,1]\to M$ with $\gamma(0) = x, \gamma(1) = x^*$, and
    \begin{alignat*}{1}
        \lin{\beta(x), \gamma'(0)} \leq - m \dist\lrp{x,x^*}^2
    \end{alignat*}. 
    Let $x_k$ be the following stochastic process:
    \begin{alignat*}{1}
        x_{k+1} = \Exp_{x_k}\lrp{\delta_k \beta(x_k) + \sqrt{\delta_k} \xi_k(x_k)}
    \end{alignat*}
    where $\xi_k(x_k) \sim \N_{x_k}\lrp{0,I}$ and for each $k$, $\delta_k$ is a positive stepsize that depends only on $x_k$ and satisfies
    $\delta_k \leq \min\lrbb{\frac{m}{16 {L_\beta'}^2 \lrp{1 + \sqrt{L_R} \dist\lrp{x_k,x^*}}}, \frac{d}{4m \lrp{1 + \sqrt{L_R} \dist\lrp{x_k,x^*}}}}$. Assume that $\dist\lrp{x_0,x^*} \leq 2\R$. Finally, assume that there exists $\delta \in \Re^+$ such that for all $k$, $\delta_k \leq \delta$. Then
    \begin{alignat*}{1}
        \Pr{\max_{i\leq K} \dist\lrp{x_k,x^*} \geq t} \leq 8K\delta m \exp\lrp{1 + \frac{m\R^2}{d} + \frac{2 L_R d}{m} - \frac{m t^2}{32 d}}
    \end{alignat*}
\end{lemma}
\begin{proof}
    In case the minimizing geodesic from $x_k$ to $x^*$ is not unique, we let $\Exp_{x_k}^{-1}(x^*)$ be any arbitrary (but consistently fixed) choice that satisfies $\lin{\beta(x_k), \Exp_{x_k}^{-1}(x^*)} \leq - m \dist\lrp{x_k,x^*}^2$. Using the result from Corollary 8 of \cite{zhang2016first} (see Lemma \ref{l:zhang2016}),
    \begin{alignat*}{1}
        \dist\lrp{x_{k+1}, x^*}^2
        \leq& \dist^2\lrp{x_{k}, x^*} - 2\lin{\delta_k \beta(x_k) + \sqrt{\delta_k} \xi_k(x_k), \Exp_{x_k}^{-1}(x^*)}\\
        &\quad + \lrp{1+\sqrt{L_R} \dist(x_k,x^*)} \lrn{\delta_k \beta(x_k) + \sqrt{\delta_k} \xi_k(x_k)}^2
    \end{alignat*}
    By our assumption, $\lin{\beta(x_k), \Exp_{x_k}^{-1}(x^*)} \leq - m \dist\lrp{x_k,x^*}^2 + m \R^2$. Applying Cauchy Schwarz and simplifying,
    \begin{alignat*}{1}
        & \dist\lrp{x_{k+1}, x^*}^2\\
        \leq& \lrp{-2m \delta_k + 2\delta_k^2 {L_\beta'}^2 +  2 \delta_k^2 \sqrt{L_R} {L_\beta'}^2 \dist\lrp{x_k, x^*}}\dist\lrp{x_k, x^*}^2  + 2\delta_k m \R^2 \\
        &\quad + 2\sqrt{\delta_k} \lin{\xi_k(x_k), \Exp_{x_k}^{-1}(x^*)} + 2 \delta_k \lrp{1+\sqrt{L_R} \dist(x_k,x^*)} \lrn{\xi_k(x_k)}^2\\
        \leq& -\delta_k m \dist\lrp{x_k, x^*}^2 + 2\delta_k m\R^2 + 2\sqrt{\delta_k} \lin{\xi_k(x_k), \Exp_{x_k}^{-1}(x^*)} + 2 \delta_k \lrp{1+\sqrt{L_R} \dist(x_k,x^*)} \lrn{\xi_k(x_k)}^2
    \end{alignat*}
    where we used our assumption that $\delta_k \leq \frac{m}{16 {L_\beta'}^2 \lrp{1 + \sqrt{L_R} \dist\lrp{x_k,x^*}}}$.

    Let $s := \frac{m}{32 d}$ We will now apply Lemma \ref{l:doob_maximal_3} with
    \begin{alignat*}{1}
        & q_k = s \dist\lrp{x_k,x^*}^2 \qquad \xi_k = s\lin{\xi_k(x_k), \Exp_{x_k}^{-1}(x^*)} + 2 \sqrt{\delta_k} s \lrp{1+\sqrt{L_R} \dist(x_k,x^*)} \lrn{\xi_k(x_k)}^2 \\
        & \lambda = m \qquad \gamma  =  2sm\R^2 \qquad \mu = 4 s d + \frac{8 s L_R d^2}{m}
    \end{alignat*}
    Taking expectation conditioned on $\xi_0(x_0)...\xi_{k-1}(x_{k-1})$,
    \begin{alignat*}{1}
        & \E{\exp\lrp{\sqrt{\delta_k} \xi_k}} \\
        =& \E{\exp\lrp{\sqrt{\delta_k} s\lin{\xi_k(x_k), \Exp_{x_k}^{-1}(x^*)} + 2 \delta_k s\lrp{1+\sqrt{L_R} \dist(x_k,x^*)} \lrn{\xi_k(x_k)}^2}} \\
        \leq& \E{\exp\lrp{\sqrt{\delta_k} 2s\lin{\xi_k(x_k), \Exp_{x_k}^{-1}(x^*)} }}^{1/2} \cdot \E{\exp\lrp{ 4 \delta_k s\lrp{1+\sqrt{L_R} \dist(x_k,x^*)} \lrn{\xi_k(x_k)}^2}}^{1/2}\\
        \leq& \exp\lrp{8 \delta_k s^2 \dist\lrp{x_k,x^*}^2 d + 4 \delta_k s\lrp{1 + \sqrt{L_R}\dist\lrp{x_k,x^*}}d}\\
        \leq& \exp\lrp{\frac{\delta_k ms}{2} \dist\lrp{x_k,x^*}^2 + 4\delta_k s d + \frac{8\delta_k s L_R d^2}{m}}
    \end{alignat*} 
    where the second inequality uses the fact that $\lin{\xi_k(x_k), \Exp_{x_k}^{-1}(x^*)} \sim \N(0,\dist\lrp{x_k,x^*})^2$ and Lemma \ref{l:hoeffding} and Lemma \ref{l:subexponential-chi-square} and our assumption that $\delta_k \leq \frac{d}{4m \lrp{1 + \sqrt{L_R} \dist\lrp{x_k,x^*}}}$. The third inequality is by Cauchy Schwarz. We thus verify the requirement for Lemma \ref{l:doob_maximal_3}, which bounds
    \begin{alignat*}{1}
        \Pr{\max_{i\leq K} q_{i} \geq t^2} 
            \leq& 8K\delta m \exp\lrp{q_0 + {16 s \R^2} + \frac{32 s d}{m} + \frac{64 s L_R d^2}{m^2} - \frac{t^2}{2}}
    \end{alignat*}
    Plugging in the definition of $q_k$ and $s$, we get
    \begin{alignat*}{1}
        \Pr{\max_{i\leq K} \dist\lrp{x_k,x^*} \geq t} \leq 32K\delta m \exp\lrp{\frac{m\R^2}{d} + \frac{2 L_R d}{m} - \frac{m t^2}{32 d}}
    \end{alignat*}
\end{proof}

\begin{lemma}[Subgaussian Bound under Dissipativity, Discretized SDE, fixed stepsize]\label{l:far-tail-bound-truncated}
    Let $\beta$ be a vector field satisfying \ref{ass:beta_lipschitz}. Let $x^*$ be some point with $\beta(x^*)=0$. Assume that for all $x$ such that $\dist\lrp{x,x^*} \geq \R$, there exists a minimizing geodesic $\gamma : [0,1]\to M$ with $\gamma(0) = x, \gamma(1) = x^*$, and
    \begin{alignat*}{1}
        \lin{\beta(x), \gamma'(0)} \leq - m \dist\lrp{x,x^*}^2
    \end{alignat*}.
    Let $r \in \Re^+$ denote an arbitrary radius, and assume that $\delta$ is a stepsize satisfying
    \begin{alignat*}{1}
        \delta \leq \min\lrbb{\frac{m}{16 {L_\beta'}^2 \lrp{1 + \sqrt{L_R} r}}, \frac{d}{4m \lrp{1 + \sqrt{L_R} r}}}
    \end{alignat*}. 
    Let $x_k$ be the following stochastic process:
    \begin{alignat*}{1}
        x_{k+1} = \Exp_{x_k}\lrp{\delta \beta(x_k) + \sqrt{\delta} \xi_k(x_k)}
    \end{alignat*}
    where $\xi_k(x_k) \sim \N_{x_k}\lrp{0,I}$. Assume that $\dist\lrp{x_0,x^*} \leq 2\R$. Then
    \begin{alignat*}{1}
        \Pr{\max_{k\leq K} \dist\lrp{x_k,x^*} \geq r} \leq 32K\delta m \exp\lrp{\frac{m\R^2}{d} + \frac{2 L_R d}{m} - \frac{m r^2}{32 d}}
    \end{alignat*}
\end{lemma}
\begin{proof}
    Let us define, for analysis purposes, the following process:
    \begin{alignat*}{1}
        \t{x}_{k+1} = \Exp_{\t{x}_k}\lrp{\delta_k \beta(\t{x}_k) + \sqrt{\delta_k} \xi_k(\t{x}_k)}
    \end{alignat*}
    initialized at $\t{x}_0 = x_0$ and where 
    \begin{alignat*}{1}
        \delta_k := \min\lrbb{\delta, \frac{m}{16 {L_\beta'}^2 \lrp{1 + \sqrt{L_R} \dist\lrp{x_k,x^*}}}, \frac{d}{4m \lrp{1 + \sqrt{L_R} \dist\lrp{x_k,x^*}}}}
    \end{alignat*}
    Define the event $A_k := \lrbb{\max_{i\leq k} \dist\lrp{\t{x}_i,x^*} \leq r}$. Under the event $A_k$, $\delta_i = \delta$ for all $i\leq k$, and consequently, $\t{x}_i = x_i$ for all $i \leq k$. Therefore, $A_k = \lrbb{\max_{i\leq k} \dist\lrp{{x}_i,x^*} \leq r}$. Therefore, 
    \begin{alignat*}{1}
        & \Pr{\max_{k\leq K} \dist\lrp{x_k,x^*} \geq r}\\
        =& \Pr{A_k^c}\\
        =& \Pr{\max_{k\leq K} \dist\lrp{\t{x}_k,x^*} \geq r}\\
        \leq& 32K\delta m \exp\lrp{\frac{m\R^2}{d} + \frac{2 L_R d}{m} - \frac{m r^2}{32 d}}
    \end{alignat*}
    where the last inequality follows from Lemma \ref{l:far-tail-bound-gaussian-adaptive}.
\end{proof}

\begin{lemma}[Subgaussian Bound under Dissipativity, Nongaussian Noise, adaptive stepsize]
    \label{l:far-tail-bound-nongaussian-adaptive}
    Let $\beta$ be a vector field satisfying \ref{ass:beta_lipschitz}. Let $x^*$ be some point with $\beta(x^*)=0$. Assume that for all $x$ such that $\dist\lrp{x,x^*} \geq \R$, there exists a minimizing geodesic $\gamma : [0,1]\to M$ with $\gamma(0) = x, \gamma(1) = x^*$, and
    \begin{alignat*}{1}
        \lin{\beta(x), \gamma'(0)} \leq - m \dist\lrp{x,x^*}^2
    \end{alignat*}. 
    Let $\xi$ be a random vector field satisfying, for all $x\in M$, $\E{\xi(x)}=0$ and $\lrn{\xi(x)}\leq L_\xi$ almost surely. For $k\in \Z^+$, let $\xi_k$ be iid samples of $\xi$. Let $x_k$ be the following stochastic process:
    \begin{alignat*}{1}
        x_{k+1} = \Exp_{x_k}\lrp{\delta_k \beta(x_k) + \sqrt{\delta_k} \xi_k(x_k)}
    \end{alignat*}
    where $\delta_k$ is a positive stepsize that depends only on $x_k$ and satisfies
    $\delta_k \leq \frac{m}{16 {L_\beta'}^2 \lrp{1 + \sqrt{L_R} \dist\lrp{x_k,x^*}}}$. Assume that $\dist\lrp{x_0,x^*} \leq 2\R$. Finally, assume that there exists $\delta \in \Re^+$ such that for all $k$, $\delta_k \leq \delta$. Then
    \begin{alignat*}{1}
        \Pr{\max_{i\leq K} \dist\lrp{x_k,x^*} \geq t} \leq 32K\delta m \exp\lrp{\frac{m\R^2}{L_\xi^2} + \frac{2 L_R L_\xi^2}{m} - \frac{m t^2}{32 L_\xi^2}}
    \end{alignat*}
\end{lemma}
\begin{proof}
    In case the minimizing geodesic from $x_k$ to $x^*$ is not unique, we let $\Exp_{x_k}^{-1}(x^*)$ be any arbitrary (but consistently fixed) choice that satisfies $\lin{\beta(x_k), \Exp_{x_k}^{-1}(x^*)} \leq - m \dist\lrp{x_k,x^*}^2$. Using the result from Corollary 8 of \cite{zhang2016first} (see Lemma \ref{l:zhang2016}),
    \begin{alignat*}{1}
        \dist\lrp{x_{k+1}, x^*}^2
        \leq& \dist^2\lrp{x_{k}, x^*} - 2\lin{\delta_k \beta(x_k) + \sqrt{\delta_k} \xi_k(x_k), \Exp_{x_k}^{-1}(x^*)}\\
        &\quad + \lrp{1+\sqrt{L_R} \dist(x_k,x^*)} \lrn{\delta_k \beta(x_k) + \sqrt{\delta_k} \xi_k(x_k)}^2
    \end{alignat*}
    By our assumption, $\lin{\beta(x_k), \Exp_{x_k}^{-1}(x^*)} \leq - m \dist\lrp{x_k,x^*}^2 + m \R^2$. Applying Cauchy Schwarz and simplifying,
    \begin{alignat*}{1}
        & \dist\lrp{x_{k+1}, x^*}^2\\
        \leq& \lrp{-2m \delta_k + 2\delta_k^2 {L_\beta'}^2 +  2 \delta_k^2 \sqrt{L_R} {L_\beta'}^2 \dist\lrp{x_k, x^*}}\dist\lrp{x_k, x^*}^2  + 2\delta_k m \R^2 \\
        &\quad + 2\sqrt{\delta_k} \lin{\xi_k(x_k), \Exp_{x_k}^{-1}(x^*)} + 2 \delta_k \lrp{1+\sqrt{L_R} \dist(x_k,x^*)} \lrn{\xi_k(x_k)}^2\\
        \leq& -\delta_k m \dist\lrp{x_k, x^*}^2 + 2\delta_k m\R^2 + 2\sqrt{\delta_k} \lin{\xi_k(x_k), \Exp_{x_k}^{-1}(x^*)} + 2 \delta_k \lrp{1+\sqrt{L_R} \dist(x_k,x^*)} \lrn{\xi_k(x_k)}^2
    \end{alignat*}
    where we used our assumption that $\delta_k \leq \frac{m}{16 {L_\beta'}^2 \lrp{1 + \sqrt{L_R} \dist\lrp{x_k,x^*}}}$.

    Let $s := \frac{m}{32 L_\xi^2}$ We will now apply Lemma \ref{l:doob_maximal_3} with
    \begin{alignat*}{1}
        & q_k = s \dist\lrp{x_k,x^*}^2 \qquad \xi_k = s\lin{\xi_k(x_k), \Exp_{x_k}^{-1}(x^*)} + 2 \sqrt{\delta_k} s \lrp{1+\sqrt{L_R} \dist(x_k,x^*)} \lrn{\xi_k(x_k)}^2 \\
        & \lambda = m \qquad \gamma  =  2sm\R^2 \qquad \mu = 4 s d + \frac{8 s L_R d^2}{m}
    \end{alignat*}
    Taking expectation conditioned on $\xi_0(x_0)...\xi_{k-1}(x_{k-1})$,
    \begin{alignat*}{1}
        & \E{\exp\lrp{\sqrt{\delta_k} \xi_k}} \\
        =& \E{\exp\lrp{\sqrt{\delta_k} s\lin{\xi_k(x_k), \Exp_{x_k}^{-1}(x^*)} + 2 \delta_k s\lrp{1+\sqrt{L_R} \dist(x_k,x^*)} \lrn{\xi_k(x_k)}^2}} \\
        \leq& \E{\exp\lrp{\sqrt{\delta_k} 2s\lin{\xi_k(x_k), \Exp_{x_k}^{-1}(x^*)} }}^{1/2} \cdot \E{\exp\lrp{ 4 \delta_k s\lrp{1+\sqrt{L_R} \dist(x_k,x^*)} \lrn{\xi_k(x_k)}^2}}^{1/2}\\
        \leq& \exp\lrp{8 \delta_k s^2 \dist\lrp{x_k,x^*}^2 L_\xi^2 + 4 \delta_k s\lrp{1 + \sqrt{L_R}\dist\lrp{x_k,x^*}}L_\xi^2}\\
        \leq& \exp\lrp{\frac{\delta_k ms}{2} \dist\lrp{x_k,x^*}^2 + 4\delta_k s L_\xi^2 + \frac{8\delta_k s L_R L_\xi^4}{m}}
    \end{alignat*} 
    where the second inequality uses Lemma \ref{l:hoeffding}. The third inequality is by Cauchy Schwarz. We thus verify the requirement for Lemma \ref{l:doob_maximal_3}, which bounds
    \begin{alignat*}{1}
        \Pr{\max_{i\leq K} q_{i} \geq t^2} 
            \leq& 8K\delta m \exp\lrp{q_0 + {16 s \R^2} + \frac{32 s L_\xi^2}{m} + \frac{64 s L_R L_\xi^4}{m^2} - \frac{t^2}{2}}
    \end{alignat*}
    Plugging in the definition of $q_k$ and $s$, we get
    \begin{alignat*}{1}
        \Pr{\max_{i\leq K} \dist\lrp{x_k,x^*} \geq t} \leq 32K\delta m \exp\lrp{\frac{m\R^2}{L_\xi^2} + \frac{2 L_R L_\xi^2}{m} - \frac{m t^2}{32 L_\xi^2}}
    \end{alignat*}
\end{proof}

\begin{lemma}[Subgaussian Bound under Dissipativity, Nongaussian Noise, fixed stepsize]
    \label{l:far-tail-bound-nongaussian-fixed}
    Let $\beta$ be a vector field satisfying \ref{ass:beta_lipschitz}. Let $x^*$ be some point with $\beta(x^*)=0$. Assume that for all $x$ such that $\dist\lrp{x,x^*} \geq \R$, there exists a minimizing geodesic $\gamma : [0,1]\to M$ with $\gamma(0) = x, \gamma(1) = x^*$, and
    \begin{alignat*}{1}
        \lin{\beta(x), \gamma'(0)} \leq - m \dist\lrp{x,x^*}^2
    \end{alignat*}.
    Let $r \in \Re^+$ denote an arbitrary radius, and assume that $\delta$ is a stepsize satisfying
    \begin{alignat*}{1}
        \delta \leq \min\lrbb{\frac{m}{16 {L_\beta'}^2 \lrp{1 + \sqrt{L_R} r}}}
    \end{alignat*}. 
    Let $\xi$ be a random vector field satisfying, for all $x\in M$, $\E{\xi(x)}=0$ and $\lrn{\xi(x)}\leq L_\xi$ almost surely. For $k\in \Z^+$, let $\xi_k$ be iid samples of $\xi$. Let $x_k$ be the following stochastic process:
    \begin{alignat*}{1}
        x_{k+1} = \Exp_{x_k}\lrp{\delta_k \beta(x_k) + \sqrt{\delta_k} \xi_k(x_k)}
    \end{alignat*}
    Assume that $\dist\lrp{x_0,x^*} \leq 2\R$. Then
    \begin{alignat*}{1}
        \Pr{\max_{k\leq K} \dist\lrp{x_k,x^*} \geq r} \leq 32K\delta m \exp\lrp{\frac{m\R^2}{L_\xi^2} + \frac{2 L_R L_\xi^2}{m} - \frac{m t^2}{32 L_\xi^2}}
    \end{alignat*}
\end{lemma}
\begin{proof}
    Let us define, for analysis purposes, the following process:
    \begin{alignat*}{1}
        \t{x}_{k+1} = \Exp_{\t{x}_k}\lrp{\delta_k \beta(\t{x}_k) + \sqrt{\delta_k} \xi_k(\t{x}_k)}
    \end{alignat*}
    initialized at $\t{x}_0 = x_0$ and where 
    \begin{alignat*}{1}
        \delta_k := \min\lrbb{\delta, \frac{m}{16 {L_\beta'}^2 \lrp{1 + \sqrt{L_R} \dist\lrp{x_k,x^*}}}}
    \end{alignat*}
    Define the event $A_k := \lrbb{\max_{i\leq k} \dist\lrp{\t{x}_i,x^*} \leq r}$. Under the event $A_k$, $\delta_i = \delta$ for all $i\leq k$, and consequently, $\t{x}_i = x_i$ for all $i \leq k$. Therefore, $A_k = \lrbb{\max_{i\leq k} \dist\lrp{{x}_i,x^*} \leq r}$. Therefore, 
    \begin{alignat*}{1}
        & \Pr{\max_{k\leq K} \dist\lrp{x_k,x^*} \geq r}\\
        =& \Pr{A_k^c}\\
        =& \Pr{\max_{k\leq K} \dist\lrp{\t{x}_k,x^*} \geq r}\\
        \leq& 32K\delta m \exp\lrp{\frac{m\R^2}{L_\xi^2} + \frac{2 L_R L_\xi^2}{m} - \frac{m t^2}{32 L_\xi^2}}
    \end{alignat*}
    where the last inequality follows from Lemma \ref{l:far-tail-bound-nongaussian-adaptive}.
\end{proof}

\subsection{Near Subgaussian Bounds}

\subsection{Misc}
\begin{lemma}\label{l:subexponential-chi-square}
    For $\lambda \leq \frac{1}{4}$ and $\xi \sim \N(0,I_{d\times d})$,
    \begin{alignat*}{1}
        \E{\exp\lrp{\lambda \lrn{\xi}_2^2}} \leq \exp\lrp{\lambda d + 2 \lambda^2 d} \leq \exp\lrp{2\lambda d}
    \end{alignat*}
\end{lemma}
\begin{proof}
    Consequence of $\chi^2$ distribution being subexponential.
\end{proof}
\begin{lemma}[Hoeffding's Lemma] \label{l:hoeffding}
    Let $\eta_k$ be a $0$-mean random variable. Then for all $\lambda$, 
    \begin{alignat*}{1}
        \Ep{\eta}{\exp\lrp{\lambda \eta}} \leq \Ep{\eta}{\exp\lrp{2\lambda^2 \eta^2}}
    \end{alignat*}
\end{lemma}
\begin{proof}
    \begin{alignat*}{1}
        \Ep{\eta}{\exp\lrp{\lambda \eta}}
        =& \Ep{\eta}{\exp\lrp{\lambda \eta - \Ep{\eta'}{\lambda \eta'}}}\\
        \leq& \Ep{\eta,\eta'}{\exp\lrp{\lambda \lrp{\eta - \eta'}}}\\
        =& \Ep{\eta,\eta',\epsilon}{\exp\lrp{\lambda \epsilon \lrp{\eta - \eta'}}}\\
        \leq& \Ep{\eta,\eta'}{\exp\lrp{\lambda^2 \lrp{\eta - \eta'}^2/2}}\\
        \leq& \Ep{\eta,\eta'}{\exp\lrp{\lambda^2 \lrp{\eta^2 + {\eta'}^2}}}\\
        =& \Ep{\eta}{\exp\lrp{2\lambda^2 \eta^2}}
    \end{alignat*}
    where $\epsilon$ is a Rademacher random variable.
\end{proof}
    
\begin{lemma}[Corollary of Doob's maximal inequality]
    \label{l:doob_maximal}
    Let $K$ be any positive integer. For any $k\leq K$, let $a_k, b_k, c_k, d_k \in \Re^+$ be arbitrary positive constants. Assume that for all $k$, $a_k \leq \frac{1}{4}$ and ${a_k+c_k} \leq \frac{1}{4}$. Let $q_k$ be a semi-martingale of the form
    \begin{alignat*}{1}
        q_{k+1} \leq \lrp{1 + a_k}q_k + b_k + \eta_k
    \end{alignat*}
    where $\eta_k$ are random variables. Assume that for all $k$, $\eta_k$ satisfy
    \begin{alignat*}{1}
        \E{\exp\lrp{\eta_k} | \eta_0...\eta_{k-1}} \leq \exp\lrp{c_k q_{k} + d_k}
    \end{alignat*}

    Assume in addition that $q_k \geq 0$ almost surely, for all $k$. Finally, assume that $\sum_{k=0}^K a_k + c_k \leq\frac{1}{8}$. Then
    \begin{alignat*}{1}
        \Pr{\max_{k\leq K} q_{k} \geq t^2} \leq \exp\lrp{q_0 + \sum_{k=0}^K (b_k+d_k) - \frac{t^2}{2}}
    \end{alignat*}
\end{lemma}
\begin{proof}
    Let us first define
    \begin{alignat*}{1}
        r_0 :=& q_0\\
        r_{k+1} 
        :=& \lrp{1 + a_k} r_k + b_k  + \eta_k
    \end{alignat*}
    i.e. $r_k$ is very similar to $q_k$, only difference being that we replaced $\leq$ by $=$.
    
    We first verify that for all $k\leq K$, $r_k \geq q_k$. For $k=0$, by definition, $r_0 = q_0$. Now assume that $r_k \geq q_k$ for some $k$. Then for $k+1$, 
    \begin{alignat*}{1}
        r_{k+1} 
        :=& \lrp{1 + a_k} r_k + b_k  + \eta_k\\
        \geq& \lrp{1 + a_k} q_k + b_k  + \eta_k\\
        \geq& q_{k+1}
    \end{alignat*}

    We verify below that $\exp\lrp{r_k}$ is a sub-martingale: conditioning on $\eta_0...\eta_{k-1}$, and taking expectation wrt $\eta_k$, 
    \begin{alignat*}{1}
        & \E{\exp\lrp{r_{k+1}} | \eta_0...\eta_k}\\
        =& \E{\exp\lrp{\lrp{1 + a_k} r_k + b_k + \eta_k} | \eta_0...\eta_k}\\
        =& \exp\lrp{(1+a_k)r_k + b_k} \cdot \E{\exp\lrp{\eta_k} | \eta_0...\eta_k}\\
        \geq& \exp\lrp{(1+a_k)r_k + b_k}\\
        \geq& \exp\lrp{ r_k}
    \end{alignat*}
    where the first inequality is by convexity of $\exp$, and $\E{\eta_k}= 0$, and Jensen's inequality. 
    
    Let us now define $s_k := \prod_{i=0}^{k-1} \lrp{1+a_i+c_i}^{-1}$. We can upper bound
    \begin{alignat*}{1}
        & \E{\exp\lrp{s_{k+1} r_{k+1}}}\\
        =& \E{\exp\lrp{s_{k+1} \lrp{(1+a_k) r_{k} + b_k + \eta_k}}}\\
        =& \E{\exp\lrp{s_{k+1}(1+a_k) r_k + s_{k+1} b_k} \cdot \E{\exp\lrp{s_{k+1} \eta_k}| \eta_0...\eta_k}}\\
        \leq& \E{\exp\lrp{s_{k+1}(1+a_k) r_k + s_{k+1} b_k} \cdot \E{\exp\lrp{s_{k+1} \eta_k}| \eta_0...\eta_k}}\\
        \leq& \E{\exp\lrp{s_{k+1}(1+a_k) r_k + s_{k+1} b_k} \cdot \lrp{\E{\exp\lrp{\eta_k}| \eta_0...\eta_k}}^{s_{k+1}}}\\
        \leq& \E{\exp\lrp{s_{k+1}(1+a_k) r_k + s_{k+1} \lrp{b_k + c_kq_k + d_k}}}\\
        =& \E{\exp\lrp{s_{k}r_k}} \cdot \exp\lrp{s_{k+1} \lrp{b_k + d_k}}
    \end{alignat*}
    where the second inequality is by Lemma \ref{l:hoeffding}, the third inequality is by the fact that $s_k \leq 1$ for all $k$ and by Jensen's inequality, the fourth inequality uses our assumption on $\eta_k$ in the Lemma statement, as well as the fact that $s_k \leq 1$ for all $k$. The last equality is by definition of $s_k$ and because $q_k \leq r_k$. Applying this recursively gives
    \begin{alignat*}{1}
        \E{\exp\lrp{s_K r_{K}}} \leq \exp\lrp{r_0} \cdot \exp\lrp{\sum_{k=0}^K s_{k+1} (b_k+d_k)} \leq \exp\lrp{r_0 + \sum_{k=0}^K (b_k+d_k)}
    \end{alignat*}
    
    By Doob's maximal inequality (recall that we $e^{r_k}$ is a sub-martingale),
    \begin{alignat*}{1}
        \Pr{\max_{k\leq K} q_{k} \geq t^2}
        \leq& \Pr{\max_{k\leq K} r_{k} \geq t^2}\\
        =& \Pr{\max_{k\leq K} \exp\lrp{s_K r_{k}} \geq \exp\lrp{s_K t^2}}\\
        \leq& \E{\exp\lrp{s_K r_K } } \cdot \exp\lrp{- \frac{t^2}{2}}\\
        \leq& \exp\lrp{r_0 + \sum_{k=0}^K s_{k+1} (b_k+d_k) - \frac{t^2}{2}}
        \elb{e:t:wkqndasq:4}
    \end{alignat*}
    The first equality uses our assumption that $s_K = \prod_{k=0}^{K-1} \lrp{1+ a_k + c_k}^{-1} \geq e^{- \sum_{k=0}^{K-1} {4(a_k+c_k)}} \geq e^{- \frac{1}{4}} \geq \frac{1}{2}$ and the fact that $r_K \geq q_k \geq 0$.
    \end{proof}

    \begin{lemma}[Uniform Bound]
        \label{l:doob_maximal_3}
        Let $K$ be any positive integer. For $k\leq K$, let $\delta_k \in \Re^+$, let $\lambda, \gamma, \mu \in \Re^+$. Let $q_k$ be a semi-martingale of the form
        \begin{alignat*}{1}
            q_{k+1} \leq \lrp{1 - \delta_k \lambda}q_k + \delta_k \gamma + \sqrt{\delta_k} \xi_k
        \end{alignat*}
        where $\xi_k$ are random variables. Assume that for all $k$, $\xi_k$ satisfy
        \begin{alignat*}{1}
            & \E{\exp\lrp{\sqrt{\delta_k} \xi_k} | \xi_0...\xi_{k-1}} \leq \exp\lrp{\delta_k \lambda q_k/2 + \delta_k \mu}
        \end{alignat*}
        Assume that there is a constant $\delta$ such that for all $k$, $\delta_k \leq \delta \leq \frac{1}{8\lambda}$. Assume in addition that $q_k \geq 0$ almost surely, for all $k$. Then
        \begin{alignat*}{1}
            \Pr{\max_{i\leq K} q_{i} \geq t^2} 
            \leq& 8K\delta \lambda \exp\lrp{q_0 + \frac{8(\gamma + \mu)}{\lambda} - \frac{t^2}{2}}
        \end{alignat*}
    \end{lemma}
    \begin{proof}
        For any $s \leq 1$, 
        \begin{alignat*}{1}
            \E{\exp\lrp{s q_{k+1}}}
            \leq& \E{\exp\lrp{s \lrp{(1-\delta_k \lambda) q_k + \delta_k\gamma + \sqrt{\delta_k} \xi_k}}}\\
            =& \E{\exp\lrp{s \lrp{(1-\delta_k \lambda) q_k + \delta_k\gamma}} \cdot \E{\exp\lrp{s \sqrt{\delta_k} \xi_k}}}\\
            \leq& \E{\exp\lrp{s \lrp{(1-\delta_k \lambda) q_k + \delta_k\gamma}} \cdot \lrp{\E{\exp\lrp{\sqrt{\delta_k} \xi_k}}}^{s}}\\
            \leq& \E{\exp\lrp{s \lrp{(1-\delta_k \lambda/2) q_k + \delta_k(\gamma+\mu)}}}
        \end{alignat*}

        Applying the above recursively, for any $k$, we can bound
        \begin{alignat*}{1}
            \E{\exp\lrp{q_k}}
            \leq& \E{\exp\lrp{(1 - \delta_k \lambda/2) q_{k-1} + \delta_k (\gamma+\mu) }}\\
            \leq& ...\\
            \leq& \E{\exp\lrp{\prod_{i=0}^{k-1}\lrp{1-\delta_i \lambda/2} q_0 + \sum_{i=0}^{k-1} \prod_{j=i}^{k-1} \lrp{1-\delta_j \lambda/2 }\lrp{\delta_i (\gamma+\mu)}}}\\
            \leq& \E{\exp\lrp{e^{-\frac{\lambda}{2} \sum_{i=0}^{k-1}\delta_k } q_0 + (\gamma+\mu) \sum_{i=0}^{k-1} e^{-\frac{\lambda}{2} \sum_{j=i}^{k-1} \delta_j}\delta_i}}
            \elb{e:t:qokdsda:1}
        \end{alignat*}
        Let us define $t_k := \sum_{i=0}^{k} \delta_i$. By our assumption that $\delta_i \leq \frac{1}{4\lambda}$, we can verify that $\sum_{i=0}^{k-1} e^{\frac{\lambda}{2} \sum_{j=i}^{k-1} \delta_j}\delta_i \leq 2 \int_0^{t_k} e^{-\frac{\lambda(t_k - t)}{2}} dt \leq \frac{4}{\lambda}$. Therefore, for all $k$,
        \begin{alignat*}{1}
            \E{\exp\lrp{q_{k}}} \leq \exp\lrp{q_0 + \frac{4(\gamma+\mu)}{\lambda}}
        \end{alignat*}
        Let us now define $N := \left\lceil \frac{1}{4\delta \lambda}\right\rceil \geq 1$ (inequality is because $\delta \leq \delta_k \leq \frac{1}{8\lambda}$). We verify that $q_{k+1} \leq q_{k-1} + \delta_k (\gamma+\mu) + \eta_k$. Let us now apply Lemma \ref{l:doob_maximal} with $\eta_k = \sqrt{\delta_k} \xi_k$, $a_k=0, b_k = \delta_k \gamma, c_k = \lambda/2, d_k=\mu$ and the fact that $\delta_k \lambda \leq 1/4$ to bound, for any $k$,
        \begin{alignat*}{1}
            \Pr{\max_{i\leq N} q_{k+i} \geq t^2} 
            \leq& \E{\exp\lrp{q_k + (\gamma+\mu) \sum_{i=0}^N \delta_{i+N} - \frac{t^2}{2}}} 
            \leq \exp\lrp{q_0 + \frac{8(\gamma+\mu)}{\lambda} - \frac{t^2}{2}}
        \end{alignat*}
        where we use the fact that $N \leq \frac{1}{2\delta \lambda}$

        Applying union bound over the events $\lrbb{\max_{i\leq N} q_{k+i} \geq t^2}$ for $k = 0, N, 2N...$, we can bound, for any positive integer $M$,
        \begin{alignat*}{1}
            \Pr{\max_{i\leq MN} q_{i} \geq t^2} 
            \leq& \sum_{j=0}^{M-1} \Pr{\max_{i\leq N} q_{jN + i} \geq t^2}\\
            \leq& \sum_{j=0}^{M-1} \exp\lrp{q_0 + \frac{8(\gamma+\mu)}{\lambda} - \frac{t^2}{2}}\\
            \leq& M \exp\lrp{q_0 + \frac{8(\gamma+\mu)}{\lambda} - \frac{t^2}{2}}
        \end{alignat*}
        Plugging in $M = \frac{K}{N} \leq 8K\delta \lambda$, it follows that for any $K$, 
        \begin{alignat*}{1}
            \Pr{\max_{i\leq K} q_{i} \geq t^2} 
            \leq& 8K\delta \lambda \exp\lrp{q_0 + \frac{8(\gamma+\mu)}{\lambda} - \frac{t^2}{2}}
        \end{alignat*}

    \end{proof}

    \begin{lemma}\label{l:zhang2016}
        Let $M$ satisfy Assumption \ref{ass:sectional_curvature_regularity}. For any 3 points $x,y,z \in M$, let $u, v \in T_y M$ be such that $z = \Exp_y(v)$ and $x = \Exp_y(u)$. Assume in addition that $\lrn{u} = \dist\lrp{x,y}$. Then
        \begin{alignat*}{1}
            \dist\lrp{z, x}^2
            \leq& \dist\lrp{y, x}^2 - 2\lin{v, u} + {\tc\lrp{\sqrt{L_R}\dist\lrp{y,x}}} \lrn{v}^2
        \end{alignat*}    
        where $\zeta(r) := \frac{r}{\tanh(r)}$.
    \end{lemma}
    The above lemma is a restatement of Corollary 8 from \cite{zhang2016first}.

\section{CLT}\label{s:clt}

    \subsection{Main CLT Result and Proof}\label{ss:clt_main_outline}

    \begin{lemma}\label{l:clt:main}
        Let $\H: \Re^d \to \Re^d$ be a random function, assume there exist constants $L_\H, L_\H', L_\H''$ such that
        \begin{alignat*}{2}
            & \text{For all $\lrn{\uu}_2$:}&&\qquad \lrn{\H(\uu)}_2 \leq L_\H \\
            & \text{For all $\lrn{\uu}_2 \leq \frac{1}{\sqrt{L_R}}$:}&& \qquad  \lrn{\nabla \H(\uu)}_2 \leq L_\H'\qquad  \lrn{\nabla^2 \H(\uu)}_2 \leq L_\H''
        \end{alignat*}
        Assume also that for all $\zz \in \Re^d$, $\E{\H(\zz)}=0$ and $\E{\H(\zz)\H(\zz)^T} = I$. For $k\in \Z^+$, let $\H_k$ denote independent samples of $\H$, and let $\zb_k$ denote independent samples from $\N(0,I)$. Let $T,\delta,K$ satisfy
        \begin{alignat*}{1}
            &1.\ T \leq \min\lrbb{\frac{1}{16 L_\H^2}, \frac{1}{16 {L_\H'}^2}, \frac{1}{32 \sqrt{L_R} L_\beta}, \frac{1}{128 L_R L_\H^2},\frac{1}{256 L_RL_\H^2 \log \lrp{2^{10} L_R}}}\\
            &2.\ \delta := T^3, K := T/\delta =  T^{-2}
            \elb{ass:clt_constants}
        \end{alignat*}
        For $k=0...K$, define two discrete stochastic processes
        \begin{alignat*}{1}
            & \t{\zz}_{k + 1} = \t{\zz}_k + \delta \bbeta + \sqrt{\delta} \H_k(\t{\zz}_{k})\\
            & {\zz}_{k + 1} = {\zz}_{k}  + \delta \bbeta + \sqrt{\delta} \zb_{k + 1}
            \elb{d:zz_main_theorem}
        \end{alignat*}
        initialized at the same point $\zz_{0}=\t{\zz}_{0}$. 

        Then there exists $\lambda_6 = poly\lrp{L_\H, L_\H', L_\H'', L_R}$ such that
        \begin{alignat*}{1}
            W_2\lrp{\t{\zz}_{K}, \zz_K}
            \leq& \lambda_6 \cdot \log\lrp{\frac{1}{T L_\H^2}}^4 \cdot T^{3/2}
        \end{alignat*}
    \end{lemma}

    \begin{proof}[Proof of Lemma \ref{l:clt:main}]
        Let $r :=  64 L_\H^2 \log\lrp{\frac{2^7}{T^2 L_\H^4}} \geq 1$. Let us define $s(r) := \lceil 256 r^2 L_{\H}^2 \rceil$. It will be useful later on to note that $s(r) \leq 512 r^2 L_\H^2$.

        For any $k$, let us define
        \begin{alignat*}{1}
            \h{\zz}_{k+1} := \zz_k + \delta \bbeta + \sqrt{\delta} \H_k (\zz_k)
        \end{alignat*}
        \textbf{A recursive decomposition:}\\    
        We first write the $W_2$ distance at step $k+1$ in terms of the $W_2$ distance at step $k$:
        \begin{alignat*}{1}
            W_2\lrp{\t{\zz}_{k+1}, \zz_{k+1}}
            \leq& W_2\lrp{\t{\zz}_{k+1}, \h{\zz}_{k+1}} + W_2\lrp{\h{\zz}_{k+1}, {\zz}_{k+1}}
        \end{alignat*}

        Let $A_k$ denote the event that $\lrn{\t{\zz}_k}_2 \leq \frac{1}{4\sqrt{L_R}} \text{ AND } \lrn{{\zz}_k}_2 \leq \frac{1}{4\sqrt{L_R}}$, i.e. \\$\ind{A_k} := \ind{\lrn{\t{\zz}_k}_2 \leq \frac{1}{4\sqrt{L_R}}} \cdot \ind{\lrn{\t{\zz}_k}_2 \leq \frac{1}{4\sqrt{L_R}}}$.
    
        We can verify that
        \begin{alignat*}{1}
            & \E{\lrn{\t{\zz}_{k+1} - \h{\zz}_{k+1}}_2^2}\\
            =& \E{\lrn{\t{\zz}_{k} - {\zz}_{k} + \sqrt{\delta}\lrp{\H_k(\t{\zz}_k) - \H_k(\zz_k)}}_2^2}\\
            =& \E{\lrn{\t{\zz}_{k} - {\zz}_{k}}_2^2} + \delta \E{\lrn{\H_k(\t{\zz}_k) - \H_k(\zz_k)}_2^2}\\
            \leq& \E{\lrn{\t{\zz}_{k} - {\zz}_{k}}_2^2} + \delta \E{\ind{A_k^c} L_\H^2} + \delta \E{\ind{A_k}{L_\H'}^2 \lrn{\t{\zz}_k - \zz_k}_2^2}\\
            \leq& \lrp{1 + \delta {L_{\H}'}^2} \E{\lrn{\t{\zz}_{k} - {\zz}_{k}}_2^2} + \delta L_\H^2 \Pr{A_k^c}
        \end{alignat*}
        the second equality uses the fact that $\H_k$ is zero mean conditioned on $(\t{\zz}_k, \zz_k)$.

        Taking square roots,
        \begin{alignat*}{1}
            W_2(\t{\zz}_{k+1}, \t{\zz}_{k+1}) \leq \lrp{1 + \delta {L_{\H}'}^2}W_2(\t{\zz}_{k}, {\zz}_{k}) + \sqrt{\delta} L_\H \sqrt{\Pr{A_k^c}}
        \end{alignat*}
    
        where we use Assumption \ref{ass:clt_H_lipschitz}.
    
        \begin{alignat*}{1}
            W_2\lrp{\t{\zz}_{k+1}, \zz_{k+1}}
            \leq& W_2\lrp{\t{\zz}_{k+1}, \h{\zz}_{k+1}} + W_2\lrp{\h{\zz}_{k+1}, {\zz}_{k+1}} \\
            \leq& \lrp{1 + \delta {L_{\H}'}^2}W_2(\t{\zz}_{k}, {\zz}_{k}) + \sqrt{\delta} L_\H \sqrt{\Pr{A_k^c}} + W_2(\h{\zz}_{k+1},\zz_{k+1})
        \end{alignat*}
    
        Applying the above recursively for $k=s(r)...K$, we get
        \begin{alignat*}{1}
            W_2\lrp{\t{\zz}_{K}, \zz_{K}}
            \leq& \lrp{1+\delta {L_{\H}'}^2}^{K-s(r)} W_2 \lrp{\t{\zz}_{s(r)}, \zz_{s(r)}} + \sum_{k=s(r)}^K \lrp{1+\delta {L_{\H}'}^2}^{K-k}\lrp{\sqrt{\delta} L_\H \sqrt{\Pr{A_k^c}} + W_2(\t{\zz}_{k},\zz_{k})}\\
            \leq& 2W_2 \lrp{\t{\zz}_{s(r)}, \zz_{s(r)}}  + 2 \sum_{k=s(r)}^K \lrp{\sqrt{\delta} L_\H \sqrt{\Pr{A_k^c}} + W_2(\h{\zz}_{k},\zz_{k})}
        \end{alignat*}
        where we use the fact that $1 + \delta {L_\H'}^2 \leq \exp\lrp{\delta {L_\H'}^2}$ and that $K\delta \leq \frac{1}{16 {L_\H'}^2}$ from \eqref{ass:clt_constants}.
    
        \textbf{Bounding $W_2\lrp{\t{\zz}_{s(r)}, \zz_{s(r)}}$}\\
        We can verify that 
        \begin{alignat*}{1}
            & \E{\lrn{\zz_{k}}_2^2} \leq 2\lrp{k\delta}^2 \lrn{\bbeta}_2^2 + 2k\delta d \leq 2 k^2 \delta^2 L_\beta^2 + 2k\delta d\\
            & \E{\lrn{\t{\zz}_{k}}_2^2} \leq 2\lrp{k\delta}^2 \lrn{\bbeta}_2^2 + 2k\delta d \leq 2 k^2 \delta^2 L_\beta^2 + 2k\delta d
        \end{alignat*}
        so that by simple triangle inequality,
        \begin{alignat*}{1}
            W_2\lrp{\t{\zz}_{s(r)}, \zz_{s(r)}} \leq 4s(r) \delta L_\beta + 2 \sqrt{s(r) \delta d} \leq 1024 L_\H^2 r^2 \delta L_\beta + 64 L_\H r \sqrt{\delta d}
        \end{alignat*}
    
        \textbf{Bounding $\sum_{k=s(r)}^K W_2(\h{\zz}_{k+1}, {\zz}_{k+1})$:}\\
        We will now bound the term $\sum_{k=s(r)}^K W_2(\h{\zz}_{k},\zz_{k})$. By Lemma \ref{l:clt:one-step-bound}
        \begin{alignat*}{1}
            W_2^2\lrp{\zz_{k+1}, \h{\zz}_{k+1}} 
            \leq& 2^{18} k\delta \lrp{\C_{k}^2(1+r^8) + \frac{d^2}{{k}^4}} \\
            &\quad + 2^{13} L_\H^2 k\delta  \cdot \lrp{\exp\lrp{-\frac{r^2}{32 L_\H^2}} + \exp\lrp{- \frac{1}{32 k\delta L_R L_\H^2}}}
        \end{alignat*}
    
        From Lemma \ref{l:sum_Ck}, 
        \begin{alignat*}{1}
            \sum_{k= {s(r)}}^K \sqrt{k\delta} \lrp{\C_k + \frac{d}{k^2}}
            \leq \delta^{3/2} K \lambda_1 + \delta \sqrt{K} \lambda_2 + \sqrt{\delta} \log(K) \lambda_3 + \sqrt{\delta} \lambda_4 + \delta^2 K \lambda_5
        \end{alignat*}
        for $\lambda_1,\lambda_2,\lambda_3,\lambda_4,\lambda_5$ are $poly(d,L_\H, L_\H', L_\H'')$ defined in \eqref{d:clt_lambdas}.
        (While the above looks complicated, the dominating term is really just $\sqrt{\delta} \log(K) \lambda_3$.) Thus
        \begin{alignat*}{1}
            \sum_{k=s(r)}^K W_2(\h{\zz}_{k},\zz_{k})
            \leq& \lrp{1+r^4}\lrp{\delta^{3/2} K \lambda_1 + \delta \sqrt{K} \lambda_2 + \sqrt{\delta} \log(K) \lambda_3 + \sqrt{\delta} \lambda_4+ \delta^2 K \lambda_5} \\
            &\quad  + 2^{7} L_\H K^{3/2} \sqrt{\delta}  \cdot \lrp{\exp\lrp{-\frac{r^2}{32 L_\H^2}} + \exp\lrp{- \frac{1}{32 K\delta L_R L_\H^2}}}
        \end{alignat*}
    
        \textbf{Combining all the terms,} we get
        \begin{alignat*}{1}
            W_2\lrp{\t{\zz}_{K}, \zz_K}
            \leq& 1024 L_\H^2 r^2 \delta L_\beta + 64 L_\H r \sqrt{\delta d}\\
            &\quad+ \lrp{1+r^4}\lrp{\delta^{3/2} K \lambda_1 + \delta \sqrt{K} \lambda_2 + \sqrt{\delta} \log(K) \lambda_3 + \sqrt{\delta} \lambda_4+ \delta^2 K \lambda_5} \\
            &\quad  + 2^{7} L_\H K^{3/2} \sqrt{\delta}  \cdot \lrp{\exp\lrp{-\frac{r^2}{32 L_\H^2}} + \exp\lrp{- \frac{1}{32 K\delta L_R L_\H^2}}}
        \end{alignat*}

        For the rest of this proof, it is more convenient to work with $T$ instead of $K$ and $\delta$. Recall that $\delta = (K\delta)^3 = T^3$ and $K = T/\delta = \frac{1}{T^2}$. We can then rewrite the above upperbound as
        \begin{alignat*}{1}
            W_2\lrp{\t{\zz}_{K}, \zz_K}
            \leq& 1024 L_\H^2 r^2 T^3 L_\beta + 64 L_\H r T^{3/2} \sqrt{d}\\
            &\quad+ \lrp{1+r^4}\lrp{T^{9/2} \lambda_1 + T^{2} \lambda_2 + T^{3/2} \log(1/T) \lambda_3 + T^{3/2} \lambda_4+ T^6 K \lambda_5} \\
            &\quad  + 2^{7} L_\H \cdot T^{-3/2}  \cdot \lrp{\exp\lrp{-\frac{r^2}{32 L_\H^2}} + \exp\lrp{- \frac{1}{32 T L_R L_\H^2}}}
        \end{alignat*}

        By our definition of $\delta = T^2$ and Lemma \ref{l:useful_xlogx},
        \begin{alignat*}{1}
            & 2^{7} L_\H T^{-3/2}  \cdot \exp\lrp{- \frac{1}{32 K\delta L_R L_\H^3}}
            \leq T^{3/2} L_\H^3\\
            \Leftarrow \qquad 
            & \frac{1}{64 L_R T L_\H^2} \geq \log\lrp{\frac{16}{T L_\H^2}}\\
            \Leftarrow \qquad 
            & \frac{16}{TL_\H^2 \log\lrp{16/(TL_\H^2))}} \geq 1024 L_R \\
            \Leftarrow \qquad 
            & T \leq \frac{1}{256 L_RL_\H^2 \log \lrp{2^{10} L_R}}
        \end{alignat*}
        The last line is satisfied by assumption in our lemma statement.

        Next, by definition of $r = 64 L_\H^2 \log\lrp{\frac{2^7}{T^2 L_\H^4}}$ (see start of proof), we verify that 
        \begin{alignat*}{1}
            2^{7} L_\H \cdot \frac{1}{\sqrt{T}} \exp\lrp{-\frac{r^2}{32 L_\H^2}} \leq T^{3/2} L_\H^3
        \end{alignat*}

        Finally, recall from Lemma \ref{l:sum_Ck} that $\lambda_1...\lambda_5$ are $poly\lrp{L_\H,L_\H',L_\H''}$. Combined with the definition of $r$, as well as the fact that $T \leq 1/d \leq 1$ (from assumption in Lemma statement), we verify that there exists $\lambda_6 = poly\lrp{L_R, L_\H, L_\H', L_\H''}$ such that
        \begin{alignat*}{1}
            W_2\lrp{\t{\zz}_{K}, \zz_K}
            \leq& 1024 L_\H^2 r^2 T^3 L_\beta + 64 L_\H r T^{3/2} \sqrt{d}\\
            &\quad+ \lrp{1+r^4}\lrp{T^{9/2} \lambda_1 + T^{2} \lambda_2 + T^{3/2} \log(1/T) \lambda_3 + T^{3/2} \lambda_4+ T^6 K \lambda_5} \\
            &\quad  + 2^{7} L_\H \cdot T^{-3/2}  \cdot \lrp{\exp\lrp{-\frac{r^2}{32 L_\H^2}} + \exp\lrp{- \frac{1}{32 T L_R L_\H^2}}}\\
            \leq& \lambda_6 \cdot \log\lrp{\frac{1}{T}}^4 \cdot T^{3/2}
        \end{alignat*}
           
    \end{proof}
    \subsubsection{One-Step-Wasserstein-Bound}
    The goal of this section is to bound $W_2(\zz_{k+1},\h{\zz}_{k+1})$.

    \begin{lemma}\label{l:clt:one-step-bound}
        Consider the same setup as Lemma \ref{l:clt:main}. Let $r\in \Re^+$ be an arbitrary positive constant, let $k \geq \max\lrbb{256 r^2 L_{\H}^2, 16 L_\H^2}$, let
        \begin{alignat*}{1}
            \h{\zz}_{k+1} := \zz_k + \sqrt{\delta} \H(\zz_k)
        \end{alignat*}
        Then
        \begin{alignat*}{1}
            W_2^2\lrp{\zz_{k+1}, \h{\zz}_{k+1}} 
            \leq& 2^{18} k\delta \lrp{\C_{k}^2(1+r^8) + \frac{d^2}{{k}^4}} \\
            &\quad + 2^{13} L_\H^2 k\delta  \cdot \lrp{\exp\lrp{-\frac{r^2}{32 L_\H^2}} + \exp\lrp{- \frac{1}{32 k\delta L_R L_\H^2}}}
        \end{alignat*}
        where $\C_k$ is defined in \eqref{d:clt_C}.
    \end{lemma}
    \begin{proof}
        The main idea of the proof is via a recursion over $k$. Let us consider a fixed $k$, we will bound $W_2^2\lrp{\zz_{k+1}, \h{\zz}_{k+1}}$ in terms of $W_2^2\lrp{\zz_k, \h{\zz}_k}$. 

        \textbf{Step 1: Transform $\zz$ to $\uu$:} \\
        To simplify notation, Let us also shift the coordinates, so that $\uu_0 := \zz_k - k\delta \bbeta = \sqrt{\delta} \sum_{i=0}^{k} \zb_i$ $, \uu_\delta := \zz_{k+1}- (k+1)\delta \bbeta = \sum_{i=1}^{k} \zb_i$, $\h{\uu}_\delta := \h{\zz}_{k+1} - (k+1) \delta \bbeta = \sqrt{\delta} \sum_{i=0}^{k-1} \zb_i + \sqrt{\delta} \H(\zz_k)$. To further simplify notation, we let $\zb := \zb_k$. We also let $\H(\uu):= \H_k(\uu + k\delta \bbeta) = \H_k(\zz_k)$. Given these definitions, we verify that
        \begin{alignat*}{1}
            & \uu_0 \sim \N(0,k\delta I)\\
            & \uu_\delta := \uu_0 + \sqrt{\delta} \zb, \qquad \zb \sim \N(0,I)\\
            & \h{\uu}_\delta := \uu_0 + \sqrt{\delta} \H(\uu_0)
            \numberthis \label{d:uu}
        \end{alignat*}

        Under our assumption that $k\delta \leq \frac{1}{4 L_\beta \sqrt{L_R}}$, we verify that $\lrn{\uu}_2 \leq \frac{1}{2\sqrt{L_R}} \Rightarrow \lrn{\zz_k}_2 = \lrn{\uu + k\delta \bbeta}_2 \leq \lrn{\uu}_2 + k\delta L_\beta \leq \frac{1}{\sqrt{L_R}}$, so that \eqref{ass:clt_H_lipschitz} implies
        \begin{alignat*}{2}
            & \text{For all $\lrn{\uu}_2$:}&&\qquad \lrn{\H(\uu)}_2 \leq L_\H \\
            & \text{For all $\lrn{\uu}_2 \leq \frac{1}{2\sqrt{L_R}}$:}&& \qquad  \lrn{\nabla \H(\uu)}_2 \leq L_\H'\qquad  \lrn{\nabla^2 \H(\uu)}_2 \leq L_\H'' 
            \numberthis \label{ass:clt_H_lipschitz_2}
        \end{alignat*}

        One can also verify that $W_2(\zz_{k},\h{\zz}_{k}) = W_2(\uu_\delta, \h{\uu}_\delta)$. Therefore, for the rest of this section, we aim to bound $W_2(\uu_\delta,\h{\uu}_\delta)$.
    
        \textbf{Step 2: Talagrand's inequality with truncated tail}
    
        Let $p_{\delta}$ denote the density of $\uu_{\delta}$ and $\h{p}_\delta$ denote the density of $\h{\uu}_\delta$. Then $\log\lrp{p_{\delta} (\uu)} \propto - \frac{1}{2(k+1)\delta} \lrn{\uu}_2^2$. Thus $p_{\delta}$ is strongly log-concave, with parameter $\frac{1}{(k+1)\delta}$. Talagrand's inequality (see e.g. \cite{otto2000generalization}) then implies that
        \begin{alignat*}{1}
            W_2^2 (p_{\delta}, \h{p}_{\delta}) \leq (k+1)\delta \KL{\h{p}_{\delta}}{p_{\delta}} \leq \chi^2(\h{p}_{\delta} || p_{\delta})
        \end{alignat*}
        (the second inequality always holds, see \cite{zhai2018high}).
    
        \emph{However, we cannot apply Talagrand's inequality directly. Under the assumptions of Lemma \ref{l:clt:main}, $\H$ does not have Lipschitz derivatives when $\lrn{\uu}_2$ is large (this is due to distortion from curvature being more significant for large $\lrn{\uu}_2$.} Instead, we consider a "truncated" version of Talagrand's inequality:
    
        \begin{lemma}\label{l:talagrand_truncated}
            Let $p$ be a density such that $\nabla^2 p(\uu) \succ m I$. Let $q$ be any density. Let $c$ be any positive radius, let $B_c$ be a ball of radius $c$ centered at $0$, and let $a := p(B_c)$ and $b := q(B_c)$. Finally, assume that $1-a \leq \frac{1}{16}$ and $1- b \leq \frac{1}{16}$.
            
            Then
            \begin{alignat*}{1}
                W_2^2(p,q) \leq& \frac{3}{m} \int_{B_c} \lrp{\frac{q(x)}{p(x)} - 1}^2 p(x) dx\\
                &\quad + 2 \max\lrbb{(1-a),(1-b)} \cdot \lrp{\sqrt{\Ep{x\sim p(x)}{\lrn{x}_2^4}} + \sqrt{\Ep{y\sim q(y)}{\lrn{y}_2^4}}}
            \end{alignat*}
        \end{lemma}
        We defer the proof to Section \ref{ss:clt_auxiliary}. We will apply Lemma \ref{l:talagrand_truncated} with 
        \begin{alignat*}{1}
            & p := p_\delta \qquad q := \h{p}_\delta \qquad m := (k+1)\delta \qquad c := \min\lrbb{\frac{1}{4\sqrt{L_R}}, r \sqrt{k\delta}}
        \end{alignat*}
        To bound the $2 \max\lrbb{(1-a),(1-b)} \cdot \lrp{\sqrt{\Ep{x\sim p(x)}{\lrn{x}_2^4}} + \sqrt{\Ep{y\sim q(y)}{\lrn{y}_2^4}} + 8}$ term, we apply Lemma \ref{l:talagrand_truncated}, with

        We now bound $a,b, \E{\lrn{x}_2^4}, \E{\lrn{y}_2^4}$: By Lemma \ref{l:clt_gaussian_tail}, and recalling our definition of $\uu_\delta$ and $\h{\uu}_\delta$, 
        \begin{alignat*}{1}
            & 1-a \leq \Pr{\lrn{\uu_\delta}_2^2 \geq r^2 k\delta} + \Pr{\lrn{\uu_\delta}_2^2 \geq \frac{1}{4L_R}}\\
            & 1-b \leq \Pr{\lrn{\h{\uu}_\delta}_2^2 \geq r^2 k\delta} + \Pr{\lrn{\h{\uu}_\delta}_2^2 \geq \frac{1}{4L_R}}\\
            & \Pr{\lrn{\uu_\delta}_2^2 \geq r^2 k\delta} \leq \exp\lrp{\frac{-r^2k\delta + k\delta d}{8k\delta}} \leq 2\exp\lrp{- \frac{r^2}{8}}\\
            & \Pr{\lrn{\h{\uu}_\delta}_2^2 \geq r^2 k\delta} \leq 2\exp\lrp{- \frac{r^2 k\delta}{8k\delta L_\H^2}} \leq 2\exp\lrp{\frac{-r^2}{8L_\H^2}}\\
            & \Pr{\lrn{\uu_\delta}_2^2 \geq \frac{1}{4\sqrt{L_R}}} \leq 2\exp\lrp{- \frac{1}{32 k\delta L_R}}\\
            & \Pr{\lrn{\h{\uu}_\delta}_2^2 \geq \frac{1}{4\sqrt{L_R}}} \leq 2\exp\lrp{- \frac{1}{32 k\delta L_R L_\H^2}}
        \end{alignat*}

        \textbf{Step 2.1: bounding $4th$ moments}
        To bound the fourth moment terms, 
        \begin{alignat*}{1}
            \sqrt{\E{\lrn{\uu_\delta}_2^4}} 
            \leq& 2^{8} L_\H^2 k\delta + \sqrt{\int_0^\infty s \exp\lrp{- \frac{\sqrt{s}}{16k\delta}} ds} 
            \leq 2^{10} L_\H^2 k \delta
        \end{alignat*}
        where we use Assumption \ref{ass:clt_constants}: $k\delta \leq \frac{1}{L_\H^2}$.
    
        And
        \begin{alignat*}{1}
            \sqrt{\E{\lrn{\h{\uu}_\delta}_2^4}} 
            \leq& 2^{10} L_\H^2 k\delta + \sqrt{\int_0^\infty s \exp\lrp{- \frac{\sqrt{s}}{32k\delta}} ds} 
            \leq 2^{12} L_\H^2 k \delta
        \end{alignat*}
        where we use Assumption \ref{ass:clt_constants}: $k\delta \leq \frac{1}{L_\H^2}$.
    
        Combining all the results up to this point, we have
        \begin{alignat*}{1}
            W_2^2\lrp{p_\delta, \h{p}_\delta} \leq 3 k\delta \int_{c} \lrp{\frac{\h{p}_\delta(x)}{p_\delta(x)} - 1}^2 p_\delta (x) dx + 2^{13} L_\H^2 k\delta  \cdot \lrp{\exp\lrp{-\frac{r^2}{32 L_\H^2}} + \exp\lrp{- \frac{1}{32 k\delta L_R L_\H^2}}}
        \end{alignat*}
    
        \textbf{Step 2.2: Bounding the $\chi^2$ term}
    
        It remains to bound the $ \lrp{\frac{\h{p}_\delta(x)}{p_\delta(x)} - 1}^2 $ term. This is the most difficult part of this section. 
    
        \begin{lemma}\label{l:chi-square-bound-preview}
            Consider the same setup as Lemma \ref{l:clt:one-step-bound} with $\uu_0, \uu_\delta, \h{\uu}_\delta$ as defined in \eqref{d:uu}, with densities $p_0, p_\delta, \h{p}_\delta$ respectively. Let $\G(\uu):= \uu + \sqrt{\delta} \H(\uu)$. Assume $\delta \leq \frac{1}{16d^2 {L_\H'}^2}$. For any $r\in \Re^+$, for any $\lrn{\h{\uu}}_2 \leq \min\lrbb{\frac{1}{4\sqrt{L_R}}, r \sqrt{k\delta}}$, and for any $k \geq \max\lrbb{256 r^2 L_{\H}^2, 16 L_\H^2}$,
            \begin{alignat*}{1}
                \lrabs{\frac{\h{p}_\delta \lrp{\h{\uu}}}{p_\delta \lrp{\h{\uu}}}  -1}^2\leq 2^{12}\C_{k}^2(1+r^8) + \frac{16d^2}{{k}^4}
            \end{alignat*}
            where $\C_k$ is defined in \eqref{d:clt_C}.
        \end{lemma}
        We restate this lemma as Lemma \ref{l:chi-square-bound}, as well as provide its proof, in the next section.
    
        Assuming $k$ satisfies the condition of Lemma \ref{l:chi-square-bound-preview}, then we can plug in the bound from Lemma \ref{l:chi-square-bound-preview} into our upper bound on $W_2^2(p_\delta, \h{p}_\delta)$ to get
        \begin{alignat*}{1}
            W_2^2\lrp{p_\delta, \h{p}_\delta} 
            \leq& 2^{18} k\delta \lrp{\C_{k}^2(1+r^8) + \frac{d^2}{{k}^4}} \\
            &\quad + 2^{13} L_\H^2 k\delta  \cdot \lrp{\exp\lrp{-\frac{r^2}{32 L_\H^2}} + \exp\lrp{- \frac{1}{32 k\delta L_R L_\H^2}}}
            \numberthis \label{e:clt-one-step-bound-final}
        \end{alignat*}

        The lemma statement follows by using $u_\delta := \zz_k - k\delta \bbeta $.

    \end{proof}

    \subsection{$\chi^2$ bound}\label{ss:one_step_clt}
        Following the setup in the proof of Lemma \ref{l:clt:one-step-bound}, let us define
        \begin{alignat*}{1}
            p_\delta := \text{ density of $\uu_\delta$}\\
            \h{p}_\delta := \text{ density of $\h{\uu}_\delta$}\\
            \h{p}_0 := \text{ density of $\uu_0$}
        \end{alignat*}
        for $\uu_\delta$ and $\h{\uu}_\delta$ as defined in \eqref{d:uu}.

        The goal of this section is to bound the term
        \begin{alignat*}{1}
            \lrp{\frac{\h{p}_\delta(\uu)}{p_\delta(\uu)} -1}^2
        \end{alignat*}
        which is needed for the proof of Lemma \ref{l:clt:one-step-bound}.

        Our main result is Lemma \ref{l:chi-square-bound}. Most of the work in this section goes into establishing a suitable approximatin of $\h{p}_\delta$, which is given in lemma \ref{l:clt_chisq_main}.
        
        \begin{lemma}\label{l:expression_for_p}
            Consider the same setup as Lemma \ref{l:clt:one-step-bound} with $\uu_0, \uu_\delta, \h{\uu}_\delta$ as defined in \eqref{d:uu}, with densities $p_0, p_\delta, \h{p}_\delta$ respectively. Let $\G(\uu):= \uu + \sqrt{\delta} \H(\uu)$, then for all $\bar{\uu}$ satisfying $\lrn{\bar{\uu}}_2 \leq \frac{1}{4\sqrt{L_R}}$,
            \begin{alignat*}{1}
                \pt_\delta(\h{\uu}) 
                =& \E{ p_0\lrp{\G^{-1}(\h{\uu})} \lrp{\det\lrp{\nabla \G \lrp{\G^{-1} \lrp{\h{\uu}}}}}^{-1}}
            \end{alignat*}
            where the expectation is wrt randomness in $\H$.
        \end{lemma}
        \begin{proof}[Proof of Lemma \ref{l:expression_for_p}]
            In this proof, we need to bound the smoothness of $\H$ for all points on the line $\ell(t) = t(\h{\uu}) + (1-t) \G^{-1} (\h{\uu})$, $t\in[0,1]$. From the assumption in \eqref{ass:clt_H_lipschitz}, it suffices to ensure that for all $t\in[0,1]$, $\lrn{\ell(t)}_2 \leq \frac{1}{\sqrt{L_R}}$. We already assumed that $\lrn{\bar{\uu}}_2 \leq \frac{1}{4\sqrt{L_R}}$. By Lemma \ref{l:G_is_invertible}, $\lrn{\G^{-1}(\uu)}_2 \leq \frac{1}{2\sqrt{L_R}}$, so $\ell(t)$ being an interpolation, satisfies $\lrn{\ell(t)}_2 \leq \frac{1}{2\sqrt{L_R}}$ for all $t\in[0,1]$. 
            
            From the change of variable formula,
            \begin{alignat*}{1}
                \pt_\delta(\h{\uu}) 
                =& \E{ p_0\lrp{\G^{-1}(\h{\uu})} \det\lrp{\nabla \lrp{\G}^{-1} \lrp{\h{\uu}}} }
            \end{alignat*}
    
            By the inverse function theorem,
            \begin{alignat*}{1}
                \nabla \G^{-1} \lrp{\h{\uu}}
                =& \lrp{\nabla \G \lrp{\G^{-1} \lrp{\h{\uu}}}}^{-1}
            \end{alignat*}
            Plugging in, and using the fact that $\det\lrp{A^{-1}} = \frac{1}{\det\lrp{A}}$,
            \begin{alignat*}{1}
                \pt_\delta(\h{\uu}) 
                =& \E{p_0\lrp{\G^{-1}(\h{\uu})} \lrp{\det\lrp{\nabla \G \lrp{\G^{-1} \lrp{\h{\uu}}}}}^{-1}}
            \end{alignat*}
        \end{proof}
        
        \begin{lemma}\label{l:taylor_u_ut}
            Let $\H$ satisfy \eqref{ass:clt_H_lipschitz_2}. Let $\delta \leq \min\lrbb{\frac{1}{16 L_R L_\H^2}, \frac{1}{16 {L_\H'}^2}}$. Let $\G(\uu):= \uu + \sqrt{\delta} \H(\uu)$. Assume $\bar{\uu}$ satisfies $\lrn{\bar{\uu}}_2 \leq \frac{1}{4\sqrt{L_R}}$, and let $\u := \G^{-1} \lrp{\h{\uu}}$. Then
            \begin{alignat*}{2}
                &\lrn{\u - \h{\uu}}_2
                &&\leq \sqrt{\delta} L_{\H}\\
                &\lrn{\u - \lrp{\h{\uu} - \sqrt{\delta} \H\lrp{\h{\uu}}}}_2
                &&\leq \delta L_{\H}' L_{\H}\\
                &\lrn{\u - \lrp{\h{\uu} - \sqrt{\delta} \H(\h{\uu}) + {\delta} \nabla \H\lrp{\h{\uu}} \H\lrp{\h{\uu}}}}_2 
                &&\leq \delta^{3/2} \lrp{L_{\H}''  L_{\H}^2 + \lrp{L_{\H}'}^2 L_{\H}}
            \end{alignat*}
        \end{lemma}
        \begin{proof}[Proof of Lemma \ref{l:taylor_u_ut}]
            We define the line $\ell(t) = t(\h{\uu}) + (1-t) \G^{-1} (\h{\uu})$. Using identical steps as the proof of Lemma \ref{l:expression_for_p}, we can show that $\lrn{\ell(t)}_2 \leq \frac{1}{2\sqrt{L_R}}$ for all $t\in[0,1]$. Therefore, the Lipschitz derivative bounds for $\H$ in \eqref{ass:clt_H_lipschitz} hold along $\ell(t)$. Note that $\h{\uu} = \ell(1)$ and $\uu = \ell(0)$.

            By definition of $\G$,
            \begin{alignat*}{1}
                \h{\uu} = \u + \sqrt{\delta} \H(\u)
            \end{alignat*}

            By Assumption \ref{ass:clt_H_lipschitz}, $\lrn{\sqrt{\delta} \H \lrp{\u}} \leq \sqrt{\delta} L_{\H}$. Thus
            \begin{alignat*}{1}
                \lrn{\h{\uu} - \u}_2 \leq \sqrt{\delta} L_{\H}
            \end{alignat*}

            Next, by Taylor's Theorem,
            \begin{alignat*}{1}
                \H\lrp{\h{\uu}} = \H\lrp{\u} + \int_0^{1} \nabla \H\lrp{\u + t\lrp{\h{\uu} - \u}} \lrp{\h{\uu} - \u}dt
            \end{alignat*}
            Therefore,
            \begin{alignat*}{1}
                \lrn{\H\lrp{\h{\uu}} - \H\lrp{\u}}_2
                \leq& \lrn{\int_0^{1} \nabla \H\lrp{\u + t\lrp{\h{\uu} - \u}} \lrp{\h{\uu} - \u}dt}_2\\
                \leq& L_{\H}' \lrn{\h{\uu} - \u}_2\\
                \leq& \sqrt{\delta} L_{\H}' L_{\H}
            \end{alignat*}

            Thus
            \begin{alignat*}{1}
                \lrn{\u - \h{\uu} + \sqrt{\delta} \H\lrp{\h{\uu}}}_2
                =& \lrn{\u - \h{\uu} + \sqrt{\delta} \H\lrp{\u} + \sqrt{\delta} \lrp{\H\lrp{\h{\uu}}-\H\lrp{\u}}}_2\\
                =& \lrn{\sqrt{\delta} \lrp{\H\lrp{\u}-\H\lrp{\h{\uu}}}}_2\\
                \leq& \delta L_{\H}' L_{\H}
            \end{alignat*}

            Finally,
            \begin{alignat*}{1}
                \nabla \H\lrp{\h{\uu} + t\lrp{\u - \h{\uu}}} = \nabla \H\lrp{\u} + \int_0^{t} \nabla^2 \H\lrp{\u + s\lrp{\h{\uu} - \u}} \lrp{\h{\uu} - \u}ds
            \end{alignat*}

            Therefore,
            \begin{alignat*}{1}
                \lrn{\nabla \H\lrp{\h{\uu} + t\lrp{\u - \h{\uu}}} - \nabla \H\lrp{\u}}_2
                \leq& \lrn{\int_0^{t} \nabla^2 \H\lrp{\u + s\lrp{\h{\uu} - \u}} \lrp{\h{\uu} - \u}ds}_2\\
                \leq& tL_{\H}'' \lrn{\h{\uu} - \u}_2\\
                \leq& t\sqrt{\delta} L_{\H}'' L_{\H}
            \end{alignat*}

            By Taylor's Theorem,
            \begin{alignat*}{1}
                \H(\u) 
                =& \H(\h{\uu}) + \int_0^{1} \nabla \H\lrp{\h{\uu} + t\lrp{\u - \h{\uu}}} \lrp{\u - \h{\uu}}dt\\
                =& \H(\h{\uu}) + \nabla \H\lrp{\h{\uu}} \lrp{\u - \h{\uu}}\\
                &\quad + \int_0^{1} \lrp{\nabla \H\lrp{\h{\uu} + t\lrp{\u - \h{\uu}}} - \nabla \H\lrp{\h{\uu}}}\lrp{\u - \h{\uu}}dt\\
                =& \H(\h{\uu}) - \sqrt{\delta} \nabla \H\lrp{\h{\uu}} \H\lrp{\h{\uu}}\\
                &\quad + \int_0^{1} \lrp{\nabla \H\lrp{\h{\uu} + t\lrp{\u - \h{\uu}}} - \nabla \H\lrp{\h{\uu}}}\lrp{\u - \h{\uu}}dt\\
                &\quad + \nabla \H\lrp{\h{\uu}} \lrp{\u - \h{\uu} + \sqrt{\delta} \H\lrp{\u}}
            \end{alignat*}

            Therefore,
            \begin{alignat*}{1}
                & \lrn{\H(\u) - \H(\h{\uu}) + \sqrt{\delta} \nabla \H\lrp{\h{\uu}} \H\lrp{\h{\uu}}}_2\\
                \leq& \lrn{\int_0^{1} \lrp{\nabla \H\lrp{\h{\uu} + t\lrp{\u - \h{\uu}}} - \nabla \H\lrp{\h{\uu}}}\lrp{\u - \h{\uu}}dt}_2\\
                &\quad + \lrn{\nabla \H\lrp{\h{\uu}} \lrp{\u - \h{\uu} + \sqrt{\delta} \H\lrp{\u}}}_2\\
                \leq& \delta L_{\H}''  L_{\H}^2 + \delta \lrp{L_{\H}'}^2 L_{\H}
            \end{alignat*}

            Finally,
            \begin{alignat*}{1}
                & \lrn{\u - \lrp{\h{\uu} - \sqrt{\delta} \H(\h{\uu}) + {\delta} \nabla \H\lrp{\h{\uu}} \H\lrp{\h{\uu}}}}_2\\
                =& \lrn{\u - \h{\uu} + \sqrt{\delta} \H\lrp{\u} - \lrp{\sqrt{\delta} \H\lrp{\u} - \sqrt{\delta}\H(\h{\uu}) + \delta \nabla \H\lrp{\h{\uu}} \H\lrp{\h{\uu}}}}_2\\
                =& \lrn{\sqrt{\delta} \H\lrp{\u} - \sqrt{\delta}\H(\h{\uu}) + \delta \nabla \H\lrp{\h{\uu}} \H\lrp{\h{\uu}}}_2\\
                \leq& \delta^{3/2} L_{\H}''  L_{\H}^2 + \delta^{3/2} \lrp{L_{\H}'}^2 L_{\H}
            \end{alignat*}
        \end{proof}

        \begin{lemma}\label{l:taylor_quadratic_Ginv}
            Let $\H$ satisfy \eqref{ass:clt_H_lipschitz_2}. Let $\delta \leq \min\lrbb{\frac{1}{16 L_R L_\H^2}, \frac{1}{16 {L_\H'}^2}}$. Let $\G(\uu):= \uu + \sqrt{\delta} \H(\uu)$. Assume $\bar{\uu}$ satisfies $\lrn{\bar{\uu}}_2 \leq \frac{1}{4\sqrt{L_R}}$, and let $\u := \G^{-1} \lrp{\h{\uu}}$. Then
            \begin{alignat*}{2}
                &\lrabs{\lrn{\G^{-1}\lrp{\h{\uu}}}_2^2 - \lrn{\h{\uu}}_2^2} 
                &&\leq  \sqrt{\delta} \lrp{2L_{\H} \lrn{\h{\uu}}_2 + \sqrt{\delta} L_{\H}^2}\\
                &\lrabs{\lrn{\G^{-1}\lrp{\h{\uu}}}_2^2 - \lrp{\lrn{\h{\uu}}_2^2 -2\sqrt{\delta} \lin{\h{\uu}, \H\lrp{\h{\uu}}}}} 
                &&\leq \delta \lrp{2L_{\H}' L_{\H} \lrn{\h{\uu}}_2 + L_{\H}^2}
            \end{alignat*}
            and
            \begin{alignat*}{1}
                & \lrabs{\lrn{\G^{-1}\lrp{\h{\uu}}}_2^2 - \lrp{\lrn{\h{\uu}}_2^2 - 2\sqrt{\delta}\lin{\h{\uu}, \H(\h{\uu})} + \delta \lin{\h{\uu},\nabla \H\lrp{\h{\uu}} \H\lrp{\h{\uu}}} + \delta \lrn{\H(\h{\uu})}_2^2  }}\\
                \leq&  \delta^{3/2} \lrp{\lrp{L_{\H}''  L_{\H}^2 + \lrp{L_{\H}'}^2 L_{\H}} \lrn{\h{\uu}}_2  + 2 L_{\H}' L_{\H}^2 + \sqrt{\delta} \lrp{L_{\H}' L_{\H}}^2 }
            \end{alignat*}

        \end{lemma}

        \begin{proof}[Proof of Lemma \ref{l:taylor_quadratic_Ginv}]
            From Lemma \ref{l:taylor_u_ut},
            \begin{alignat*}{1}
                \G^{-1}\lrp{\h{\uu}} = \h{\uu} - \sqrt{\delta} \H(\h{\uu}) + {\delta} \nabla \H\lrp{\h{\uu}} \H\lrp{\h{\uu}} + \gamma_3
            \end{alignat*}
            for some $\gamma_3$ satisfying
            \begin{alignat*}{2}
                &\lrn{\gamma_3 - \sqrt{\delta} \H(\h{\uu}) + {\delta} \nabla \H\lrp{\h{\uu}} \H\lrp{\h{\uu}}}_2
                &&\leq \sqrt{\delta} L_{\H}\\
                &\lrn{\gamma_3 + {\delta} \nabla \H\lrp{\h{\uu}} \H\lrp{\h{\uu}}}_2
                &&\leq \delta L_{\H}' L_{\H}\\
                &\lrn{\gamma_3}_2 
                &&\leq \delta^{3/2} \lrp{L_{\H}''  L_{\H}^2 + \lrp{L_{\H}'}^2 L_{\H}}
            \end{alignat*}

            Expanding the quadratic, we get
            \begin{alignat*}{1}
                \lrn{\G^{-1}\lrp{\h{\uu}}}_2^2
                =& \lrn{\h{\uu} - \sqrt{\delta} \H(\h{\uu}) + {\delta} \nabla \H\lrp{\h{\uu}} \H\lrp{\h{\uu}} + \gamma_3}_2^2\\
                =& \lrn{\h{\uu}}_2^2 + 2\lin{\h{\uu}, - \sqrt{\delta} \H(\h{\uu}) + {\delta} \nabla \H\lrp{\h{\uu}} \H\lrp{\h{\uu}} + \gamma_3} \\
                &\quad + \lrn{- \sqrt{\delta} \H(\h{\uu}) + {\delta} \nabla \H\lrp{\h{\uu}} \H\lrp{\h{\uu}} + \gamma_3}_2^2
                \numberthis \label{e:t:l:taylor_quadratic_Ginv:1}
            \end{alignat*}
            Using the preceding bounds,
            \begin{alignat*}{1}
                \lrabs{2\lin{\h{\uu}, - \sqrt{\delta} \H(\h{\uu}) + {\delta} \nabla \H\lrp{\h{\uu}} \H\lrp{\h{\uu}} + \gamma_3} }
                \leq& 2 \sqrt{\delta} L_{\H} \lrn{\h{\uu}}_2
            \end{alignat*}
            so that
            \begin{alignat*}{1}
                \lrabs{\lrn{\G^{-1}\lrp{\h{\uu}}}_2^2 - \lrn{\h{\uu}}_2^2} \leq  \sqrt{\delta} \lrp{2L_{\H} \lrn{\h{\uu}}_2 + \sqrt{\delta} L_{\H}^2}
            \end{alignat*}
            This gives our first bound.

            Next, using the bound on $\lrn{\gamma_3 + {\delta} \nabla \H\lrp{\h{\uu}} \H\lrp{\h{\uu}}}_2
            \leq \delta L_{\H}' L_{\H}$, and splitting the inner product term in \eqref{e:t:l:taylor_quadratic_Ginv:1}, we get our second bound
            \begin{alignat*}{1}
                \lrabs{\lrn{\G^{-1}\lrp{\h{\uu}}}_2^2 - \lrp{\lrn{\h{\uu}}_2^2 -2\sqrt{\delta} \lin{\h{\uu}, \H\lrp{\h{\uu}}}}} \leq \delta \lrp{2L_{\H}' L_{\H} \lrn{\h{\uu}}_2 + L_{\H}^2}
            \end{alignat*}

            Finally, 
            \begin{alignat*}{1}
                &2\lin{\h{\uu}, - \sqrt{\delta} \H(\h{\uu}) + {\delta} \nabla \H\lrp{\h{\uu}} \H\lrp{\h{\uu}} + \gamma_3}\\
                =& 2\lin{\h{\uu}, - \sqrt{\delta} \H(\h{\uu}) + {\delta} \nabla \H\lrp{\h{\uu}} \H\lrp{\h{\uu}}} + 2\lin{\h{\uu}, \gamma_3}
            \end{alignat*}
            where 
            \begin{alignat*}{1}
                \lin{\h{\uu}, \gamma_3}
                \leq \lrn{\h{\uu}}_2 \lrn{\gamma_3}_2
                = \delta^{3/2} \lrp{L_{\H}''  L_{\H}^2 + \lrp{L_{\H}'}^2 L_{\H}} \lrn{\h{\uu}}_2
            \end{alignat*}
            Next, we consider the last term in \eqref{e:t:l:taylor_quadratic_Ginv:1},
            \begin{alignat*}{1}
                &\lrn{- \sqrt{\delta} \H(\h{\uu}) + {\delta} \nabla \H\lrp{\h{\uu}} \H\lrp{\h{\uu}} + \gamma_3}_2^2\\
                =& \delta \lrn{\H(\h{\uu})}_2^2 - 2\lin{\sqrt{\delta} \H\lrp{\h{\uu}}, \delta \nabla \H\lrp{\h{\uu}} \H\lrp{\h{\uu}} + \gamma_3} + \lrn{\delta \nabla \H\lrp{\h{\uu}} \H\lrp{\h{\uu}} + \gamma_3}_2^2
            \end{alignat*}
            Using the upper bounds on $\gamma_3$, we verify that
            \begin{alignat*}{1}
                \lrabs{2\lin{\sqrt{\delta} \H\lrp{\h{\uu}}, \delta \nabla \H\lrp{\h{\uu}} \H\lrp{\h{\uu}} + \gamma_3}}
                \leq& 2\sqrt{\delta} L_{\H} \cdot \lrn{\delta \nabla \H\lrp{\h{\uu}} \H\lrp{\h{\uu}} + \gamma_3}_2\\
                \leq& \delta^{3/2} \lrp{2 L_{\H}' L_{\H}^2}\\
                \lrn{\delta \nabla \H\lrp{\h{\uu}} \H\lrp{\h{\uu}} + \gamma_3}_2^2
                \leq& \delta^2 \lrp{L_{\H}' L_{\H}}^2
            \end{alignat*}

            Plugging the above into \eqref{e:t:l:taylor_quadratic_Ginv:1} gives our third bound
            \begin{alignat*}{1}
            & \lrabs{\lrn{\G^{-1}\lrp{\h{\uu}}}_2^2 - \lrp{\lrn{\h{\uu}}_2^2 + 2\lin{\h{\uu}, - \sqrt{\delta} \H(\h{\uu}) + {\delta} \nabla \H\lrp{\h{\uu}} \H\lrp{\h{\uu}}} + \delta \lrn{\H(\h{\uu})}_2^2  }}\\
            \leq&  \delta^{3/2} \lrp{\lrp{L_{\H}''  L_{\H}^2 + \lrp{L_{\H}'}^2 L_{\H}} \lrn{\h{\uu}}_2  + 2 L_{\H}' L_{\H}^2 + \sqrt{\delta} \lrp{L_{\H}' L_{\H}}^2 }
            \end{alignat*}
            
        \end{proof}
        \begin{lemma}\label{l:taylor_nablaG_Ginv}
            Let $\H$ satisfy \eqref{ass:clt_H_lipschitz_2}. Let $\delta \leq \min\lrbb{\frac{1}{16 L_R L_\H^2}, \frac{1}{16 {L_\H'}^2}}$. Let $\G(\uu):= \uu + \sqrt{\delta} \H(\uu)$. Assume $\bar{\uu}$ satisfies $\lrn{\bar{\uu}}_2 \leq \frac{1}{4\sqrt{L_R}}$, and let $\u := \G^{-1} \lrp{\h{\uu}}$. Then
            \begin{alignat*}{1}
                & \lrn{\nabla \G\lrp{\G^{-1}\lrp{\h{\uu}}}_2 - I} \leq \sqrt{\delta} L_{\H}'\\
                & \lrn{\nabla \G\lrp{\G^{-1}\lrp{\h{\uu}}} - \lrp{I + \sqrt{\delta} \nabla \H\lrp{\h{\uu}}}}_2 \leq \delta L_{\H}'' L_{\H} \\
                &\lrn{\nabla \G\lrp{\G^{-1}\lrp{\h{\uu}}} - \lrp{I + \sqrt{\delta} \nabla \H\lrp{\h{\uu}} - \delta \lin{\nabla^2 \H\lrp{\h{\uu}}, \H\lrp{\h{\uu}}} }}_2 \leq \delta^{3/2}\lrp{ L_{\H} L_{\H}' L_{\H}'' + L_{\H}''' L_{\H}^2 }
            \end{alignat*}
            
        \end{lemma}
        
        \begin{proof}[Proof of Lemma \ref{l:taylor_nablaG_Ginv}]
            We define the line $\ell(t) = t(\h{\uu}) + (1-t) \G^{-1} (\h{\uu})$. Using identical steps as the proof of Lemma \ref{l:expression_for_p}, we can show that $\lrn{\ell(t)}_2 \leq \frac{1}{2\sqrt{L_R}}$ for all $t\in[0,1]$. Therefore, the Lipschitz derivative bounds for $\H$ in \eqref{ass:clt_H_lipschitz} hold along $\ell(t)$. Note that $\h{\uu} = \ell(1)$ and $\uu = \ell(0)$.

            We will define three levels of increasingly finer approximations for $\nabla \G\lrp{\G^{-1}\lrp{\h{\uu}}}$.
            
            By definition of $\G$,
            \begin{alignat*}{1}
                \nabla \G\lrp{\G^{-1}\lrp{\h{\uu}}}
                =& I + \sqrt{\delta} \nabla \H\lrp{\G^{-1}\lrp{\h{\uu}}}
            \end{alignat*}

            The first level of approximation is thus
            \begin{alignat*}{1}
                \lrn{\nabla \G\lrp{\G^{-1}\lrp{\h{\uu}}} - I}_2 \leq \sqrt{\delta} \lrn{\nabla \H\lrp{\G^{-1} \lrp{\h{\uu}}}}_2 \leq \sqrt{\delta} L_{\H}'
            \end{alignat*}
            Next, let us define
            \begin{alignat*}{1}
                \nu :=& \G^{-1}\lrp{\h{\uu}} - \h{\uu}
            \end{alignat*}
            From Lemma \ref{l:taylor_u_ut}, $\lrn{\nu}_2\leq \sqrt{\delta} L_{\H}$. Using a first-order expansion of $\nabla \H$,
            \begin{alignat*}{1}
                & \nabla \H\lrp{\h{\uu} + \nu}
                = \nabla \H\lrp{\h{\uu}} + \int_0^1 \lin{\nabla^2 \H\lrp{\h{\uu} + t\nu}, \nu} dt
            \end{alignat*}
            Therefore,
            \begin{alignat*}{1}
                \lrn{\nabla \G\lrp{\G^{-1}\lrp{\h{\uu}}} - \lrp{I + \sqrt{\delta} \nabla \H\lrp{\h{\uu}}}}_2
                =& \lrn{\sqrt{\delta} \nabla \H\lrp{\G^{-1}\lrp{\h{\uu}}} - \sqrt{\delta} \nabla \H\lrp{\h{\uu}}}_2\\
                =& \lrn{\sqrt{\delta} \int_0^1 \lin{\nabla^2 \H\lrp{\h{\uu} + t\nu}, \nu} dt}_2 \\
                \leq& \delta L_{\H}'' L_{\H}
            \end{alignat*}

            Finally, let us define
            \begin{alignat*}{1}
                \epsilon :=& \G^{-1}\lrp{\h{\uu}} - \lrp{\h{\uu} - \sqrt{\delta} \H(\h{\uu})}
            \end{alignat*}
            From Lemma \ref{l:taylor_u_ut}, $\lrn{\epsilon}_2\leq \delta L_{\H}' L_{\H}$. Note that by our definition, $\nu =  - \sqrt{\delta} \H(\h{\uu}) + \epsilon$.

            Performing a second-order Taylor expansion of $\nabla \H$,
            \begin{alignat*}{1}
                & \nabla \H\lrp{\h{\uu} + \nu}\\
                =& \nabla \H\lrp{\h{\uu}} + \int_0^1 \lin{\nabla^2 \H\lrp{\h{\uu} + t\nu}, \nu} dt\\
                =& \nabla \H\lrp{\h{\uu}} + \lin{\nabla^2 \H\lrp{\h{\uu}}, \nu} + \int_0^1 \lin{\int_0^t \nabla^3 \H\lrp{\h{\uu} + s\nu} \nu ds, \nu} dt\\
                =& \nabla \H\lrp{\h{\uu}} - \sqrt{\delta} \lin{\nabla^2 \H\lrp{\h{\uu}}, \H\lrp{\h{\uu}}} \\
                &\quad + \lin{\nabla^2 \H\lrp{\h{\uu}}, \epsilon} + \int_0^1 \lin{\int_0^t \nabla^3 \H\lrp{\h{\uu} + s\nu} \nu ds, \nu} dt
            \end{alignat*}

            Using our earlier bound on $\lrn{\epsilon}_2$ and $\lrn{\nu}_2$,
            \begin{alignat*}{1}
                \lrn{\lin{\nabla^2 \H\lrp{\h{\uu}}, \epsilon}}_2
                \leq& \lrn{\nabla^2 \H\lrp{\h{\uu}}}_2\lrn{\epsilon}_2 \\
                \leq& \delta L_{\H} L_{\H}' L_{\H}''\\
                \lrn{\int_0^1 \lin{\int_0^t \nabla^3 \H\lrp{\h{\uu} + s\nu} \nu ds, \nu} dt}_2
                \leq& \lrn{\nabla^3 \H\lrp{\h{\uu}}}_2 \lrn{\nu}_2^2\\
                \leq& \delta L_{\H}''' L_{\H}^2
            \end{alignat*}
            Thus
            \begin{alignat*}{1}
                & \lrn{\nabla \G\lrp{\G^{-1}\lrp{\h{\uu}}} - \lrp{I + \sqrt{\delta} \nabla \H\lrp{\h{\uu}} - \delta \lin{\nabla^2 \H\lrp{\h{\uu}}, \H\lrp{\h{\uu}}} }}_2\\
                =& \lrn{\sqrt{\delta} \nabla \H\lrp{\G^{-1}\lrp{\h{\uu}}} - \lrp{\sqrt{\delta} \nabla \H\lrp{\h{\uu}} - \delta \lin{\nabla^2 \H\lrp{\h{\uu}}, \H\lrp{\h{\uu}}} }}_2\\
                =& \sqrt{\delta} \lrn{\lin{\nabla^2 \H\lrp{\h{\uu}}, \epsilon} + \int_0^1 \lin{\int_0^t \nabla^3 \H\lrp{\h{\uu} + s\nu} \nu ds, \nu} dt}_2\\
                \leq& \delta^{3/2}\lrp{ L_{\H} L_{\H}' L_{\H}'' + L_{\H}''' L_{\H}^2 }
            \end{alignat*}

        \end{proof}
        \begin{lemma}\label{l:taylor_logdet}
            Let $\H$ satisfy \eqref{ass:clt_H_lipschitz_2}. Let $\delta \leq \min\lrbb{\frac{1}{16 L_R L_\H^2}, \frac{1}{16 {L_\H'}^2}}$. Let $\G(\uu):= \uu + \sqrt{\delta} \H(\uu)$. Assume $\bar{\uu}$ satisfies $\lrn{\bar{\uu}}_2 \leq \frac{1}{4\sqrt{L_R}}$, and let $\u := \G^{-1} \lrp{\h{\uu}}$. Then
            \begin{alignat*}{1}
                & \lrabs{\log \det\lrp{\nabla \G\lrp{\G^{-1}\lrp{\h{\uu}}}} - \lrp{\sqrt{\delta} \tr\lrp{\nabla\H\lrp{\h{\uu}}} - \delta \tr\lrp{\lin{\nabla^2 \H\lrp{\h{\uu}}, \H\lrp{\h{\uu}}}} - \frac{\delta}{2} \tr\lrp{\lrp{\nabla\H\lrp{\h{\uu}}}^2 }}}\\
                \leq& \delta^{3/2} d\lrp{L_{\H} L_{\H}' L_{\H}'' + L_{\H}''' L_{\H}^2 + L_{\H}''L_{\H}' L_{\H} + \sqrt{\delta} \lrp{L_{\H}'' L_{\H}}^2 + 2\lrp{L_{\H}'}^3}
            \end{alignat*}

            \begin{alignat*}{1}
                \lrabs{ \log \det\lrp{\nabla \G\lrp{\G^{-1}\lrp{\h{\uu}}}} - \sqrt{\delta} \tr\lrp{\nabla\H\lrp{\h{\uu}}}} 
                \leq& \delta d \lrp{L_{\H}'' L_{\H} + 2\lrp{L_{\H}}^2}
            \end{alignat*}

            \begin{alignat*}{1}
                \lrabs{ \log \det\lrp{\nabla \G\lrp{\G^{-1}\lrp{\h{\uu}}}}} \leq \sqrt{\delta} \cdot 2d L_{\H}'
            \end{alignat*}
        \end{lemma}
        \begin{proof}
            From Lemma \ref{l:taylor_nablaG_Ginv}, 
            \begin{alignat*}{1}
                \nabla \G\lrp{\G^{-1}\lrp{\h{\uu}}} = I + \sqrt{\delta} \nabla \H\lrp{\h{\uu}} - \delta \lin{\nabla^2 \H\lrp{\h{\uu}}, \H\lrp{\h{\uu}}} + \mu_3
            \end{alignat*}
            for some $\mu_3$ satisfying
            \begin{alignat*}{2}
                &\lrn{ \mu_3  + \sqrt{\delta} \nabla \H\lrp{\h{\uu}} - \delta \lin{\nabla^2 \H\lrp{\h{\uu}}, \H\lrp{\h{\uu}}} }_2
                &&\leq \sqrt{\delta} L_{\H}'\\
                &\lrn{\mu_3 - \delta \lin{\nabla^2 \H\lrp{\h{\uu}}, \H\lrp{\h{\uu}}}}_2
                &&\leq \delta L_{\H}'' L_{\H}\\
                &\lrn{\mu_3}_2 
                &&\leq \delta^{3/2}\lrp{ L_{\H} L_{\H}' L_{\H}'' + L_{\H}''' L_{\H}^2 }
            \end{alignat*}

            Let $A := \nabla \G\lrp{\G^{-1}\lrp{\h{\uu}}} - I = \sqrt{\delta} \nabla \H\lrp{\h{\uu}} - \delta \lin{\nabla^2 \H\lrp{\h{\uu}}, \H\lrp{\h{\uu}}} + \mu_3$. 

            We verify that
            \begin{alignat*}{1}
                \tr\lrp{A} = \sqrt{\delta} \tr\lrp{\nabla\H\lrp{\h{\uu}}} - \delta \tr\lrp{\lin{\nabla^2 \H\lrp{\h{\uu}}, \H\lrp{\h{\uu}}}} + \tr\lrp{\mu_3}
            \end{alignat*}
            and
            \begin{alignat*}{1}
                \tr\lrp{A^2 } 
                =& \delta \tr\lrp{\lrp{\nabla\H\lrp{\h{\uu}}}^2 } + \tr\lrp{\lrp{\sqrt{\delta} \nabla\H\lrp{\h{\uu}}} \lrp{- \delta \lin{\nabla^2 \H\lrp{\h{\uu}}, \H\lrp{\h{\uu}}} + \mu_3} }\\
                &\quad + \tr\lrp{ \lrp{- \delta \lin{\nabla^2 \H\lrp{\h{\uu}}, \H\lrp{\h{\uu}}} + \mu_3}^2 }
            \end{alignat*}

            By Cauchy Schwarz,
            \begin{alignat*}{1}
                &\lrabs{\tr\lrp{\lrp{\sqrt{\delta} \nabla\H\lrp{\h{\uu}}} \lrp{- \delta \lin{\nabla^2 \H\lrp{\h{\uu}}, \H\lrp{\h{\uu}}} + \mu_3} }  }\\
                \leq& \sqrt{\tr\lrp{\lrp{\sqrt{\delta} \nabla\H\lrp{\h{\uu}}} \lrp{\sqrt{\delta} \nabla\H\lrp{\h{\uu}}}^T}} \cdot \sqrt{\tr\lrp{ \lrp{- \delta \lin{\nabla^2 \H\lrp{\h{\uu}}, \H\lrp{\h{\uu}}} + \mu_3} \lrp{- \delta \lin{\nabla^2 \H\lrp{\h{\uu}}, \H\lrp{\h{\uu}}} + \mu_3}^T }}\\
                \leq& d \lrn{\sqrt{\delta} \nabla\H\lrp{\h{\uu}}}_2 \lrn{- \delta \lin{\nabla^2 \H\lrp{\h{\uu}}, \H\lrp{\h{\uu}}} + \mu_3}_2\\
                \leq& \delta^{3/2} d L_{\H}''L_{\H}' L_{\H}
            \end{alignat*}

            Again by Cauchy Schwarz,
            \begin{alignat*}{1}
                \tr\lrp{ \lrp{- \delta \lin{\nabla^2 \H\lrp{\h{\uu}}, \H\lrp{\h{\uu}}} + \mu_3}^2 }
                \leq& \tr\lrp{ \lrp{- \delta \lin{\nabla^2 \H\lrp{\h{\uu}}, \H\lrp{\h{\uu}}} + \mu_3} \lrp{- \delta \lin{\nabla^2 \H\lrp{\h{\uu}}, \H\lrp{\h{\uu}}} + \mu_3}^T}\\
                \leq& \delta^2 d \lrp{L_{\H}'' L_{\H}}^2
            \end{alignat*}

            Finally, $\lrabs{\tr\lrp{\mu_3}} \leq d \lrn{\mu_3}_2 = \delta^{3/2} d \lrp{ L_{\H} L_{\H}' L_{\H}'' + L_{\H}''' L_{\H}^2 }$.

            From Lemma \ref{l:taylor_determinant_inverse}, 
            \begin{alignat*}{1}
                \lrabs{ \log \det\lrp{I + A} - \lrp{\tr\lrp{A} - \frac{\tr\lrp{A^2}}{2}}} \leq 2 d \lrn{A}_2^3
                \numberthis \label{e:t:l:taylor_nablaG_Ginv:1}
            \end{alignat*}

            By plugging the previous bounds into \eqref{e:t:l:taylor_nablaG_Ginv:1}, and moving terms around,
            \begin{alignat*}{1}
                & \lrabs{\log \det\lrp{I + A} - \lrp{\sqrt{\delta} \tr\lrp{\nabla\H\lrp{\h{\uu}}} - \delta \tr\lrp{\lin{\nabla^2 \H\lrp{\h{\uu}}, \H\lrp{\h{\uu}}}} - \frac{\delta}{2} \tr\lrp{\lrp{\nabla\H\lrp{\h{\uu}}}^2 }}}\\
                \leq& \lrabs{\tr\lrp{\mu_3}}  + \frac{1}{2}\lrabs{\tr\lrp{\lrp{\sqrt{\delta} \nabla\H\lrp{\h{\uu}}} \lrp{- \delta \lin{\nabla^2 \H\lrp{\h{\uu}}, \H\lrp{\h{\uu}}} + \mu_3} } + \tr\lrp{ \lrp{- \delta \lin{\nabla^2 \H\lrp{\h{\uu}}, \H\lrp{\h{\uu}}} + \mu_3}^2 }}\\
                \leq& 2 d \lrn{A}_2^2 + \delta^{3/2} d\lrp{L_{\H} L_{\H}' L_{\H}'' + L_{\H}''' L_{\H}^2 + L_{\H}''L_{\H}' L_{\H} + \sqrt{\delta} \lrp{L_{\H}'' L_{\H}}^2 }\\
                \leq& \delta^{3/2} d\lrp{L_{\H} L_{\H}' L_{\H}'' + L_{\H}''' L_{\H}^2 + L_{\H}''L_{\H}' L_{\H} + \sqrt{\delta} \lrp{L_{\H}'' L_{\H}}^2 + 2\lrp{L_{\H}'}^3}
            \end{alignat*}
            This proves our first claim.

            We note that
            \begin{alignat*}{1}
                \lrabs{\tr\lrp{A} - \sqrt{\delta} \tr\lrp{\nabla\H\lrp{\h{\uu}}}} \leq d\lrn{\mu_3 - \delta \lin{\nabla^2 \H\lrp{\h{\uu}}, \H\lrp{\h{\uu}}}}_2 \leq \delta d L_{\H}'' L_{\H}
            \end{alignat*}
            Combining the above with the bound $\lrabs{ \log \det\lrp{I + A} - \tr\lrp{A}} \leq 2 d \lrn{A}_2^2$ from Lemma \ref{l:taylor_determinant_inverse}, we get
            \begin{alignat*}{1}
                \lrabs{ \log \det\lrp{I + A} - \sqrt{\delta} \tr\lrp{\nabla\H\lrp{\h{\uu}}}} 
                \leq& 2d\lrn{A}_2^2 + \delta L_{\H}'' L_{\H}\\
                \leq& \delta d \lrp{L_{\H}'' L_{\H} + 2\lrp{L_{\H}'}^2}
            \end{alignat*}
            This proves our second claim.

            Finally, using the third inequality of Lemma \ref{l:taylor_determinant_inverse},
            \begin{alignat*}{1}
                \lrabs{ \log \det\lrp{I + A}} \leq 2 d \lrn{A}_2 = 2d \sqrt{\delta} L_{\H}'
            \end{alignat*}

        \end{proof}

        \begin{lemma}\label{l:symmetrization}
            Let $x\in \Re$ be a random variable satisfying $\E{x} = 0$ and $\lrabs{x} \leq \alpha \leq \frac{1}{2}$ almost surely. Then
            \begin{alignat*}{1}
                \lrabs{\E{e^x} - e^{1/2\E{x^2}}} \leq 8\alpha^3
            \end{alignat*}
        \end{lemma}
        \begin{proof}[Proof of Lemma \ref{l:symmetrization}]
            By Taylor's Theorem
            \begin{alignat*}{1}
                \E{e^x}
                =& \E{1 + x + \frac{x^2}{2} + \sum_{i=3}^\infty \frac{x^i}{i!}}\\
                =& \E{1 + \frac{x^2}{2} + \sum_{i=3}^\infty \frac{x^i}{i!}}
            \end{alignat*}

            On the other hand,
            \begin{alignat*}{1}
                \exp\lrp{\frac{1}{2}\E{x^2}}
                =& \E{1 + \frac{\E{x^2}}{2} + \sum_{i=2}^\infty \frac{\lrp{\E{x^2}/2}^i}{i!}}
            \end{alignat*}

            We verify that 
            \begin{alignat*}{1}
                \lrabs{\E{\sum_{i=3}^\infty \frac{x^i}{i!}}}
                =& \lrabs{\E{x^3 \sum_{i=0}^\infty \frac{x^i}{\lrp{i+3}!}}} \leq \alpha^3 \sum_{i=0}^\infty \frac{\alpha^i}{{i}!}
                \leq \alpha^3 e^\alpha 
            \end{alignat*}
            and that
            \begin{alignat*}{1}
                \lrabs{\sum_{i=2}^\infty \frac{\lrp{\frac{1}{2}\E{x^2}}^i}{i!}}
                =& \lrabs{\sum_{i=2}^\infty \frac{\lrp{\E{x^2}/2}^i}{i!}}\\
                \leq& \frac{\alpha^4}{8} \lrabs{\sum_{i=0}^\infty \frac{\lrp{\E{x^2}/2}^i}{\lrp{i+2}!} }\\
                \leq& \frac{\alpha^4}{8}e^{\alpha^2/2}
            \end{alignat*}
            Using our assumption that $\alpha \leq 1$, we know that $\alpha^4 \leq \alpha^3$ and $\alpha^2 \leq \alpha \leq 1$, so that
            \begin{alignat*}{1}
                \alpha^3 e^{\alpha} + \frac{\alpha^4}{8}e^{\alpha^2/2} \leq 8 \alpha^3
            \end{alignat*}
            
        \end{proof}

        \begin{lemma}\label{l:clt_chisq_main}
            Consider the same setup as Lemma \ref{l:clt:one-step-bound} with $\uu_0, \uu_\delta, \h{\uu}_\delta$ as defined in \eqref{d:uu}, with densities $p_0, p_\delta, \h{p}_\delta$ respectively. Let $\G(\uu):= \uu + \sqrt{\delta} \H(\uu)$. Assume $\delta \leq \frac{1}{16d^2 {L_\H'}^2}$. For any $r\in \Re^+$, for any $\lrn{\h{\uu}}_2 \leq \min\lrbb{\frac{1}{4\sqrt{L_R}}, r \sqrt{k\delta}}$, and for any $k \geq \max\lrbb{256 r^2 L_{\H}^2, 16 L_\H^2}$,
            \begin{alignat*}{1}
                \h{p}_\delta \lrp{\h{\uu}} = p_0\lrp{\h{\uu}}\lrp{\exp\lrp{-\frac{d}{2k} + \frac{\lrn{\h{\uu}}_2^2}{2k^2\delta}} + 32\epsilon}
            \end{alignat*}
            for some $\lrabs{\epsilon} \leq 32\C_k$, where $\C_k$ is as defined in \eqref{d:clt_C}.
        \end{lemma}
        \begin{proof}[Proof of Lemma \ref{l:clt_chisq_main}]
            From Lemma \ref{l:expression_for_p}, and noting that $p_0(\uu) \propto \exp\lrp{- \frac{1}{2} \lrn{\uu}_2^2}$,
            \begin{alignat*}{1}
                \h{p}_\delta\lrp{\h{\uu}}
                =& p_0\lrp{\h{\uu}} \E{\exp\lrp{-\frac{1}{2k\delta}\lrp{\lrn{\G^{-1}\lrp{\h{\uu}}}_2^2 - \lrn{\h{\uu}}_2^2} - \log\det \lrp{\nabla \G\lrp{\G^{-1} \lrp{\h{\uu}}}}}}
            \end{alignat*}

            For convenience, define $\alpha := -\frac{1}{2}\lrp{\lrn{\G^{-1}\lrp{\h{\uu}}}_2^2 - \lrn{\h{\uu}}_2^2}$ and $\beta:= - \log\det \lrp{\nabla \G\lrp{\G^{-1} \lrp{\h{\uu}}}}$.

            Using series expansion of $e^x$,
            \begin{alignat*}{1}
                \frac{\h{p}_\delta(\h{\uu})}{p_0\lrp{\h{\uu}}} = \E{\exp\lrp{\frac{\alpha}{k\delta} + \beta}}
                =& \E{\sum_{i=0}^\infty \frac{\lrp{\frac{\alpha}{k\delta} + \beta}^i}{i!} }\\
                =& \E{1 + \frac{\alpha}{k\delta} + \beta + \frac{\alpha^2}{2k^2\delta^2} + \frac{1}{2} \beta^2 + \frac{\alpha}{k\delta} \beta} + \E{\sum_{i=3}^\infty \frac{\lrp{\frac{\alpha}{k\delta} + \beta}^i}{i!} }
                \numberthis \label{e:l:clt_chisq_main:full_term}
            \end{alignat*}

            \textbf{We first consider $\E{\alpha}$. }
            Let us define
            \begin{alignat*}{1}
                \epsilon_3 := -\frac{1}{2}\lrn{\G^{-1}\lrp{\h{\uu}}}_2^2 + \frac{1}{2}\lrp{\lrn{\h{\uu}}_2^2 - 2\sqrt{\delta}\lin{\h{\uu}, \H(\h{\uu})} + \delta \lin{\h{\uu},\nabla \H\lrp{\h{\uu}} \H\lrp{\h{\uu}}} + \delta \lrn{\H(\h{\uu})}_2^2 }
            \end{alignat*}
            From Lemma \ref{l:taylor_quadratic_Ginv}, $\lrabs{\epsilon_3} \leq \delta^{3/2} \lrp{\lrp{L_{\H}''  L_{\H}^2 + \lrp{L_{\H}'}^2 L_{\H}} \lrn{\h{\uu}}_2  + 2 L_{\H}' L_{\H}^2 + \sqrt{\delta} \lrp{L_{\H}' L_{\H}}^2 }$

            Using the fact that $\E{\H\lrp{\h{\uu}}} = 0$, we get
            \begin{alignat*}{1}
                \E{\alpha} 
                =& - \frac{1}{2} \E{\delta \lin{\h{\uu},\nabla \H\lrp{\h{\uu}} \H\lrp{\h{\uu}}} + \delta \lrn{\H(\h{\uu})}_2^2} + \frac{1}{2} \E{\epsilon_3  }\\
                =& -\frac{\delta}{2}\lin{\h{\uu},\E{ \nabla \H\lrp{\h{\uu}} \H\lrp{\h{\uu}}}} - \frac{\delta d}{2} + \frac{1}{2} \E{\epsilon_3  }
            \end{alignat*}
            Where we use the fact that $\E{\H(\h{\uu}) \H(\h{\uu})^T} = I$.

            \textbf{We next consider $\E{\frac{1}{2}\alpha^2}$.} 
            
            Let us define
            \begin{alignat*}{1}
                \epsilon_2 := -\frac{1}{2} \lrn{\G^{-1}\lrp{\h{\uu}}}_2^2 + \frac{1}{2} \lrp{\lrn{\h{\uu}}_2^2 -2\sqrt{\delta} \lin{\h{\uu}, \H\lrp{\h{\uu}}}}
            \end{alignat*}
            
            From Lemma \ref{l:taylor_quadratic_Ginv}, $\lrabs{\epsilon_2} \leq \delta \lrp{2L_{\H}' L_{\H} \lrn{\h{\uu}}_2 + L_{\H}^2}$.
            Note that $\alpha = \sqrt{\delta} \H\lrp{\h{\uu}} + \epsilon_2$. Therefore,
            \begin{alignat*}{1}
                \E{\alpha^2}
                =& \delta \E{\lin{\h{\uu}, \H\lrp{\h{\uu}}}^2} + \E{\sqrt{\delta} \lin{\h{\uu}, \H\lrp{\h{\uu}}} \cdot \epsilon_2 + \epsilon_2^2}\\
                =& \delta \lrn{\h{\uu}}_2^2  + \E{\sqrt{\delta} \lin{\h{\uu}, \H\lrp{\h{\uu}}} \cdot \epsilon_2 + \epsilon_2^2}
            \end{alignat*}
            where we use the fact that $\E{\H\lrp{\h{\uu}} \H\lrp{\h{\uu}}^T} = I$. The last two terms can be bounded as 
            \begin{alignat*}{1}
                \lrabs{\E{\sqrt{\delta} \lin{\h{\uu}, \H\lrp{\h{\uu}}} \cdot \epsilon_2 + \epsilon_2^2}}
                \leq& \delta^{3/2} \lrp{8L_{\H}' L_{\H}^2 \lrn{\h{\uu}}_2^2 + 4 L_{\H}^3\lrn{\h{\uu}}_2 + \sqrt{\delta} \lrp{2L_{\H}' L_{\H} \lrn{\h{\uu}}_2 + L_{\H}^2}^2 }
            \end{alignat*}

            \textbf{Next, we consider $\E{\beta}$.}
            
            Define
            \begin{alignat*}{1}
                \tau_3 := -\log \det\lrp{\nabla \G\lrp{\G^{-1} \lrp{\h{\uu}}}} + \lrp{\sqrt{\delta} \tr\lrp{\nabla\H\lrp{\h{\uu}}} - \delta \tr\lrp{\lin{\nabla^2 \H\lrp{\h{\uu}}, \H\lrp{\h{\uu}}}} - \frac{\delta}{2} \tr\lrp{\lrp{\nabla\H\lrp{\h{\uu}}}^2 }}
            \end{alignat*}
            so that
            \begin{alignat*}{1}
                \E{\beta} 
                =& \E{- \sqrt{\delta} \tr\lrp{\nabla\H\lrp{\h{\uu}}} + \delta \tr\lrp{\lin{\nabla^2 \H\lrp{\h{\uu}}, \H\lrp{\h{\uu}}}} + \frac{\delta}{2} \tr\lrp{\lrp{\nabla\H\lrp{\h{\uu}}}^2 }} +\E{\tau_3}\\
                =& \E{ \delta \tr\lrp{\lin{\nabla^2 \H\lrp{\h{\uu}}, \H\lrp{\h{\uu}}}} + \frac{\delta}{2} \tr\lrp{\lrp{\nabla\H\lrp{\h{\uu}}}^2 }} +\E{\tau_3}
            \end{alignat*}
            where we use the fact that $\E {-\sqrt{\delta} \tr\lrp{\nabla\H\lrp{\h{\uu}}}} = -\sqrt{\delta} \tr\lrp{ \nabla\E {\H\lrp{\h{\uu}}}} = 0$.

            From the first inequality of Lemma \ref{l:taylor_logdet},
            \begin{alignat*}{1}
                \lrabs{\tau_3} \leq& \delta^{3/2} d\lrp{L_{\H} L_{\H}' L_{\H}'' + L_{\H}''' L_{\H}^2 + L_{\H}''L_{\H}' L_{\H} + \sqrt{\delta} \lrp{L_{\H}'' L_{\H}}^2 + 2\lrp{L_{\H}'}^3}
            \end{alignat*}

            \textbf{Next, we consider $\frac{1}{2}\E{\beta^2}$}
            We define 
            \begin{alignat*}{1}
                \tau_2:= - \log \det\lrp{\nabla \G\lrp{\G^{-1}\lrp{\h{\uu}}}} + \sqrt{\delta} \tr\lrp{\nabla\H\lrp{\h{\uu}}}
            \end{alignat*}
            Expanding the quadratic,
            \begin{alignat*}{1}
                \E{\beta^2} = \delta \E{\tr\lrp{\nabla\H\lrp{\h{\uu}}}^2}  + \E{ - \sqrt{\delta} \tr\lrp{\nabla\H\lrp{\h{\uu}}} \cdot \tau_2 + \tau_2^2}
            \end{alignat*}

            The second term can be bounded with probability 1 as
            \begin{alignat*}{1}
                \lrabs{- \sqrt{\delta} \tr\lrp{\nabla\H\lrp{\h{\uu}}} \cdot \tau_2 + \tau_2^2}
                \leq& \lrabs{\sqrt{\delta} \tr\lrp{\nabla\H\lrp{\h{\uu}}} \cdot \tau_2} + \lrabs{\tau_2}^2\\
                \leq& \lrabs{\beta}{\tau_2} + \lrabs{\tau_2}^2\\
                \leq& 2 \sqrt{\delta} \cdot 2 d L_{\H}' \cdot \delta d \lrp{L_{\H}'' L_{\H} + 2\lrp{L_{\H}'}^2} + \lrp{\delta d \lrp{L_{\H}'' L_{\H} + 2\lrp{L_{\H}'}^2}}^2\\
                \leq& \delta^{3/2} d^2 \lrp{4 L_{\H}''L_{\H}' L_{\H} + 8\lrp{L_{\H}'}^3 + 8\sqrt{\delta}\lrp{L_{\H}'' L_{\H}}^2 + 32\sqrt{\delta} \lrp{L_{\H}'}^4 }
            \end{alignat*}
            where we bound $\lrabs{\tau_2}\leq \delta d \lrp{L_{\H}'' L_{\H} + 2\lrp{L_{\H}'}^2}$ using the second inequality of Lemma \ref{l:taylor_logdet} and $\lrabs{\beta} \leq \sqrt{\delta} \cdot 2 d L_{\H}' $ using the third inequality of Lemma \ref{l:taylor_logdet}.

            \textbf{Finally, we consider $\E{\alpha \beta}$.}
            We will use
            \begin{alignat*}{1}
                \alpha 
                =& \sqrt{\delta} \lin{\h{\uu}, \H\lrp{\h{\uu}}} + \epsilon_2\\
                \beta 
                =& - \sqrt{\delta} \tr\lrp{\nabla\H\lrp{\h{\uu}}} + \tau_2
            \end{alignat*}
            Thus
            \begin{alignat*}{1}
                \E{\alpha \beta}
                =& - \delta \E{\lin{\h{\uu}, \H\lrp{\h{\uu}}}\tr\lrp{\nabla\H\lrp{\h{\uu}}}} + \E{\epsilon_2 \cdot \beta + \tau_2 \cdot \alpha - \tau_2 \cdot \epsilon_2}\\
                =& - \delta \lin{\h{\uu}, \E{\H\lrp{\h{\uu}}\tr\lrp{\nabla\H\lrp{\h{\uu}}}}} + \E{\epsilon_2 \cdot \beta + \tau_2 \cdot \alpha - \tau_2 \cdot \epsilon_2}
            \end{alignat*}
            We can bound, with probability 1, the last term as
            \begin{alignat*}{1}
                & \lrabs{\epsilon_2 \cdot \beta} + \lrabs{\tau_2 \cdot \alpha} + \lrabs{\tau_2 \cdot \epsilon_2}\\
                \leq& \lrp{\delta \lrp{2L_{\H}' L_{\H} \lrn{\h{\uu}}_2 + L_{\H}^2}} \cdot \lrp{\sqrt{\delta} \cdot 2 d L_{\H}'}\\
                &\quad + \lrp{\delta d \lrp{L_{\H}'' L_{\H} + 2\lrp{L_{\H}'}^2}} \cdot \lrp{\sqrt{\delta} \lrp{2L_{\H} \lrn{\h{\uu}}_2 + \sqrt{\delta} L_{\H}^2}}\\
                &\quad + \lrp{\delta \lrp{2L_{\H}' L_{\H} \lrn{\h{\uu}}_2 + L_{\H}^2}} \cdot \lrp{\delta d \lrp{L_{\H}'' L_{\H} + 2\lrp{L_{\H}'}^2}}\\
                \leq& \delta^{3/2} d \lrp{8\lrp{L_{\H}'}^2 L_{\H} + 2L_{\H}''L_{\H}^2 + 2\sqrt{\delta} L_{\H}''L_{\H}'L_{\H}^2 + 4\sqrt{\delta} (L_{\H}')^3 L_{\H}} \lrn{\h{\uu}}_2\\
                &\quad + \delta^{3/2} d \lrp{2 L_{\H}' L_{\H}^2 + 4\sqrt{\delta} \lrp{L_{\H}''L_{\H}^3 + (L_{\H}')^2 L_{\H}^2}}
            \end{alignat*}

            \textbf{We now simplify $\E{\alpha + \beta + \frac{1}{2} \alpha^2 + \frac{1}{2} \beta^2 + \alpha \beta}$}

            From the first claim in Lemma \ref{l:cancellations_due_to_constant_covariance}, two terms from $\E{\frac{\alpha}{k\delta}}$ and $\E{\frac{\alpha}{k\delta} \beta}$ respectively cancel out:
            \begin{alignat*}{1}
                - \frac{1}{k\delta}\lin{\h{\uu}, \E{ \nabla \H\lrp{\h{\uu}} \H\lrp{\h{\uu}}}} - \frac{1}{k\delta}\lin{\h{\uu}, \E{\H\lrp{\h{\uu}}\tr\lrp{\nabla\H\lrp{\h{\uu}}}}}= 0
            \end{alignat*}

            From the second claim from Lemma \ref{l:cancellations_due_to_constant_covariance}, three terms from $\E{\beta}$ and $\E{\frac{1}{2} \beta^2}$ cancel out:
            \begin{alignat*}{1}
                \E{2\tr\lrp{\lin{\nabla^2 \H\lrp{\h{\uu}}, \H\lrp{\h{\uu}}}} + \tr\lrp{\lrp{\nabla\H\lrp{\h{\uu}}}^2 } + \tr\lrp{\nabla\H\lrp{\h{\uu}}}^2} = 0
            \end{alignat*}

            Put together,
            \begin{alignat*}{1}
                \E{\frac{\alpha}{k\delta} + \beta + \frac{1}{2k^2\delta^2} \alpha^2 + \frac{1}{2} \beta^2 + \frac{\alpha}{k\delta} \beta}
                = - \frac{d}{2k} + \frac{\lrn{\h{\uu}}_2^2}{2k^2 \delta} + \epsilon
                \numberthis \label{e:l:clt_chisq_main:most_important_term}
            \end{alignat*}
            for some $\lrabs{\epsilon} \leq 32\C'$, where
            \begin{alignat*}{1}
                \C' :=&
                \frac{\sqrt{\delta}}{k} \lrp{L_{\H}''  L_{\H}^2 + \lrp{L_{\H}'}^2 L_{\H}} \lrn{\h{\uu}}_2 + \frac{\sqrt{\delta}}{k} L_{\H}' L_{\H}^2 + \frac{\delta}{k} \lrp{L_{\H}' L_{\H}}^2\\
                &\quad  +  \frac{1}{k^2\sqrt{\delta}} \lrp{L_{\H}' L_{\H}^2 \lrn{\h{\uu}}_2^2 + L_{\H}^3\lrn{\h{\uu}}_2} + \frac{1}{k^2} {\lrp{L_{\H}' L_{\H} \lrn{\h{\uu}}_2 + L_{\H}^2}^2 }\\
                &\quad + \delta^{3/2} d\lrp{L_{\H} L_{\H}' L_{\H}'' + L_{\H}''' L_{\H}^2 + L_{\H}''L_{\H}' L_{\H} + \lrp{L_{\H}'}^3} + \delta^2 \lrp{L_{\H}'' L_{\H}}^2 \\
                &\quad + \delta^{3/2} d^2 \lrp{L_{\H}''L_{\H}' L_{\H} + \lrp{L_{\H}'}^3} + \delta^2\lrp{\lrp{L_{\H}'' L_{\H}}^2 + \lrp{L_{\H}'}^4 }\\
                &\quad + \frac{\sqrt{\delta}}{k} d L_{\H}' L_{\H}^2 + \frac{\delta}{k} \lrp{L_{\H}''L_{\H}^3 + (L_{\H}')^2 L_{\H}^2}\\
                &\quad +  \frac{\sqrt{\delta}}{k} d \lrp{\lrp{L_{\H}'}^2 L_{\H} + L_{\H}''L_{\H}^2}\lrn{\h{\uu}}_2 + \frac{\delta}{k}\lrp{L_{\H}''L_{\H}'L_{\H}^2 + (L_{\H}')^3 L_{\H}} \lrn{\h{\uu}}_2\\
                \leq& \frac{r\delta}{\sqrt{k}} \lrp{L_{\H}''  L_{\H}^2 + \lrp{L_{\H}'}^2 L_{\H}} + \frac{\sqrt{\delta}}{k} L_{\H}' L_{\H}^2 + \frac{\delta}{k} \lrp{L_{\H}' L_{\H}}^2\\
                &\quad + \frac{r^2\sqrt{\delta}}{k}L_{\H}' L_{\H}^2 + \frac{r}{k^{3/2}} L_{\H}^3 + \frac{\delta}{k} \lrp{L_{\H}' L_{\H}}^2 + \frac{1}{k^2}L_{\H}^4\\
                &\quad + \delta^{3/2} d\lrp{L_{\H} L_{\H}' L_{\H}'' + L_{\H}''' L_{\H}^2 + L_{\H}''L_{\H}' L_{\H} + \lrp{L_{\H}'}^3} + \delta^2 \lrp{L_{\H}'' L_{\H}}^2 \\
                &\quad + \delta^{3/2} d^2 \lrp{L_{\H}''L_{\H}' L_{\H} + \lrp{L_{\H}'}^3} + \delta^2\lrp{\lrp{L_{\H}'' L_{\H}}^2 + \lrp{L_{\H}'}^4 }\\
                &\quad + \frac{\sqrt{\delta}}{k} d L_{\H}' L_{\H}^2 + \frac{\delta}{k} \lrp{L_{\H}''L_{\H}^3 + (L_{\H}')^2 L_{\H}^2}\\
                &\quad + \frac{r\delta}{\sqrt{k}} d \lrp{\lrp{L_{\H}'}^2 L_{\H} + L_{\H}''L_{\H}^2} + \frac{r\delta^{3/2}}{\sqrt{k}}\lrp{L_{\H}''L_{\H}'L_{\H}^2 + (L_{\H}')^3 L_{\H}}
            \end{alignat*}
            From Lemma \ref{l:taylor_quadratic_Ginv},
            \begin{alignat*}{1}
                \lrabs{\frac{\alpha}{k\delta}} 
                \leq& \frac{1}{k\sqrt{\delta}} \lrp{2L_{\H} \lrn{\h{\uu}}_2 + \sqrt{\delta} L_{\H}^2}\\
                \leq& \frac{2 r L_{\H} }{\sqrt{k}} + \frac{L_{\H}^2 }{k} \\
                \leq& \frac{1}{8}
            \end{alignat*}
            where the second inequality uses our bound that $\lrn{\h{\uu}}_2 \leq r\sqrt{k\delta L_\H}$, and the last inequality uses the assumption that $k \geq 16 L_\H^2$ and $k \geq 256 r^2 L_\H^2$.

            From Lemma \ref{l:taylor_nablaG_Ginv} and Lemma \ref{l:taylor_determinant_inverse}, $\lrabs{\beta} \leq 2\sqrt{\delta} dL_{\H}' \leq \frac{1}{8}$ with probability 1. Therefore, 
            \begin{alignat*}{1}
                \lrabs{\frac{\alpha}{k\delta} + \beta} \leq \frac{1}{4}
            \end{alignat*}
            and
            \begin{alignat*}{1}
                \lrabs{\frac{\alpha}{k\delta} + \beta + \frac{1}{2}\lrp{\frac{\alpha}{k\delta} + \beta}^2} \leq \frac{1}{2}
            \end{alignat*}

            From our assumptions, $\frac{d}{2k} \leq \frac{1}{8}$ and $\frac{\lrn{\h{\uu}}_2^2}{2k^2 \delta} \leq \frac{r^2k\delta}{2k^2 \delta} \leq \frac{r^2}{2k} \leq \frac{1}{8}$, so that by Taylor expansion,
            \begin{alignat*}{1}
                \lrabs{\lrp{1- \frac{d}{2k} + \frac{\lrn{\h{\uu}}_2^2}{2k^2 \delta}} - \exp\lrp{- \frac{d}{2k} + \frac{\lrn{\h{\uu}}_2^2}{2k^2 \delta}}} \leq \frac{d^2}{k^2} + \frac{r^4}{k^2}
                \numberthis \label{e:t:back_to_exp}
            \end{alignat*}

            We can bound the term $\E{\sum_{i=3}^\infty \frac{\lrp{\frac{\alpha}{k\delta} + \beta}^i}{i!} }$ as
            \begin{alignat*}{1}
                \lrabs{\E{\sum_{i=3}^\infty \frac{\lrp{\frac{\alpha}{k\delta} + \beta}^i}{i!} }}
                \leq& \E{\sum_{i=3}^\infty \frac{\lrabs{\frac{\alpha}{k\delta} + \beta}^i}{i!} }\\
                =& \E{\lrabs{\frac{\alpha}{k\delta} + \beta}^3 \cdot \sum_{i=3}^\infty \frac{\lrabs{\frac{\alpha}{k\delta} + \beta}^{i-3}}{i!}}\\
                \leq& \E{\lrabs{\frac{\alpha}{k\delta} + \beta}^3 \cdot \sum_{i=0}^\infty \frac{\lrabs{\frac{\alpha}{k\delta} + \beta}^{i}}{i!}}\\
                \leq& \E{\lrabs{\frac{\alpha}{k\delta} + \beta}^3 \cdot e^{\lrabs{\frac{\alpha}{k\delta} + \beta}}}
            \end{alignat*}
            Using our earlier bound that $e^{\lrabs{\frac{\alpha}{k\delta} + \beta}}\leq \exp\lrp{\frac{1}{4}}$,
            \begin{alignat*}{1}
                \lrabs{\frac{\alpha}{k\delta} + \beta}^3 \cdot e^{\lrabs{\frac{\alpha}{k\delta} + \beta}}
                \leq& \delta^{3/2} \lrp{\frac{2L_{\H} \lrn{\h{\uu}}_2 + \sqrt{\delta} L_{\H}^2}{k\delta} + 2 dL_{\H}'}^3 \cdot \exp\lrp{\frac{1}{4}}\\
                \leq& 8 \delta^{3/2}  \lrp{\frac{L_{\H}^3 \lrn{\h{\uu}}_2^3 + \delta^{3/2} L_{\H}^6}{k^3\delta^3} +  d^3 {L_{\H}'}^3} \\
                \leq& 8r^3 \frac{L_\H^3}{k^{3/2}} + 8\frac{L_\H^6}{k^3} + 8\delta^{3/2} d^3 {L_\H'}^3
            \end{alignat*}

            Plugging the above and \eqref{e:t:back_to_exp} into \eqref{e:l:clt_chisq_main:full_term},
            \begin{alignat*}{1}
            \h{p}_\delta \lrp{\h{\uu}} = p_0\lrp{\h{\uu}}\lrp{\exp\lrp{\frac{d}{2k} + \frac{\lrn{\h{\uu}}_2^2}{2k} }+ \epsilon}
            \end{alignat*}
            for some $\lrabs{\epsilon} \leq 32\C_k\lrp{1+ r^4}$, where
            \begin{alignat*}{1}
                \C_k
                :=& \frac{\delta}{\sqrt{k}} \lrp{L_{\H}''  L_{\H}^2 + \lrp{L_{\H}'}^2 L_{\H}} + \frac{\sqrt{\delta}}{k} L_{\H}' L_{\H}^2 + \frac{\delta}{k} \lrp{L_{\H}' L_{\H}}^2\\
                &\quad + \frac{\sqrt{\delta}}{k}L_{\H}' L_{\H}^2 + \frac{1}{k^{3/2}} L_{\H}^3 + \frac{\delta}{k} \lrp{L_{\H}' L_{\H}}^2 + \frac{1}{k^2}L_{\H}^4\\
                &\quad + \delta^{3/2} d\lrp{L_{\H} L_{\H}' L_{\H}'' + L_{\H}''' L_{\H}^2 + L_{\H}''L_{\H}' L_{\H} + \lrp{L_{\H}'}^3} + \delta^2 \lrp{L_{\H}'' L_{\H}}^2 \\
                &\quad + \delta^{3/2} d^2 \lrp{L_{\H}''L_{\H}' L_{\H} + \lrp{L_{\H}'}^3} + \delta^2\lrp{\lrp{L_{\H}'' L_{\H}}^2 + \lrp{L_{\H}'}^4 }\\
                &\quad + \frac{\sqrt{\delta}}{k} d L_{\H}' L_{\H}^2 + \frac{\delta}{k} \lrp{L_{\H}''L_{\H}^3 + (L_{\H}')^2 L_{\H}^2}\\
                &\quad + \frac{\delta}{\sqrt{k}} d \lrp{\lrp{L_{\H}'}^2 L_{\H} + L_{\H}''L_{\H}^2} + \frac{\delta^{3/2}}{\sqrt{k}}\lrp{L_{\H}''L_{\H}'L_{\H}^2 + (L_{\H}')^3 L_{\H}}\\
                &\quad + \frac{d^2}{k^2} + \frac{1}{k^2}\\
                &\quad +  \frac{L_\H^3}{k^{3/2}} + \frac{L_\H^6}{k^3} + \delta^{3/2} d^3 {L_\H'}^3
                \numberthis \label{d:clt_C}
            \end{alignat*}

        \end{proof}

        \begin{lemma}\label{l:p*_delta_expansion}
            Consider the same setup as Lemma \ref{l:clt:one-step-bound} with $\uu_0, \uu_\delta, \h{\uu}_\delta$ as defined in \eqref{d:uu}, with densities $p_0, p_\delta, \h{p}_\delta$ respectively. Let $\G(\uu):= \uu + \sqrt{\delta} \H(\uu)$. Assume $\delta \leq \frac{1}{16d^2 {L_\H'}^2}$. For any $r\in \Re^+$, for any $\lrn{\h{\uu}}_2 \leq \min\lrbb{\frac{1}{4\sqrt{L_R}}, r \sqrt{k\delta}}$, and for any $k \geq \max\lrbb{256 r^2 L_{\H}^2, 16 L_\H^2}$,
            \begin{alignat*}{1}
                p_\delta(\h{\uu})
                =& p_0\lrp{\h{\uu}} \cdot \exp\lrp{\frac{1}{2k^2\delta} \lrn{\h{\uu}}_2^2 - \frac{d}{2k} + \epsilon}
            \end{alignat*}

            for some $\lrabs{\epsilon} \leq \frac{2d}{k^2}$.   
            
        \end{lemma}
        \begin{proof}[Proof of Lemma \ref{l:p*_delta_expansion}]
            Note that            
            \begin{alignat*}{1}
                &p_0(\h{\uu}) \propto \exp\lrp{- \frac{1}{2k\delta} \lrn{\h{\uu}}_2^2 - \frac{d}{2} \log \lrp{k\delta}}\\
                &p_\delta(\h{\uu}) \propto \exp\lrp{- \frac{1}{2(k+1)\delta} \lrn{\h{\uu}}_2^2 - \frac{d}{2} \log \lrp{(k+1)\delta}}
            \end{alignat*}
            
            Reorganizing,
            \begin{alignat*}{1}
                p_\delta(\h{\uu})
                =& p_0\lrp{\h{\uu}} \cdot \exp\lrp{\frac{1}{2k\delta}\lrp{1-\frac{k}{k+1}}\lrn{\h{\uu}}_2^2 - \frac{d}{2} \log\lrp{\frac{k+1}{k}}}\\
                =& p_0\lrp{\h{\uu}} \cdot \exp\lrp{\frac{1}{2k^2\delta} \lrn{\h{\uu}}_2^2 - \frac{d}{2k} + \epsilon}
            \end{alignat*}

            where $\lrabs{\epsilon} \leq \frac{2d}{k^2}$ (Taylor expansion).            
        \end{proof}

        The following Lemma is a restatement of Lemma \ref{l:chi-square-bound-preview}, for ease of reference
        \begin{lemma}\label{l:chi-square-bound}
            Consider the same setup as Lemma \ref{l:clt:one-step-bound} with $\uu_0, \uu_\delta, \h{\uu}_\delta$ as defined in \eqref{d:uu}, with densities $p_0, p_\delta, \h{p}_\delta$ respectively. Let $\G(\uu):= \uu + \sqrt{\delta} \H(\uu)$. Assume $\delta \leq \frac{1}{16d^2 {L_\H'}^2}$. For any $r\in \Re^+$, for any $\lrn{\h{\uu}}_2 \leq \min\lrbb{\frac{1}{4\sqrt{L_R}}, r \sqrt{k\delta}}$, and for any $k \geq \max\lrbb{256 r^2 L_{\H}^2, 16 L_\H^2}$,
            \begin{alignat*}{1}
                \lrabs{\frac{\h{p}_\delta \lrp{\h{\uu}}}{p_\delta \lrp{\h{\uu}}}  -1}^2\leq 2^{12}\C_k^2\lrp{1+r^8} + \frac{16d^2}{k^4}
            \end{alignat*}
            where $\C_k$ is defined in \eqref{d:clt_C}.
        \end{lemma}
        \begin{proof}[Proof of Lemma \ref{l:chi-square-bound}]
            Combining Lemma \ref{l:clt_chisq_main} and \ref{l:p*_delta_expansion},
            \begin{alignat*}{1}
                \frac{\h{p}_\delta \lrp{\h{\uu}}}{p_\delta \lrp{\h{\uu}}}  
                =& \frac{p_0\lrp{\h{\uu}}\lrp{\exp\lrp{-\frac{d}{2k} + \frac{\lrn{\h{\uu}}_2^2}{2k}} + \epsilon_1}}
                {p_0\lrp{\h{\uu}} \exp\lrp{ - \frac{d}{2k}+ \frac{1}{2k^2\delta} \lrn{\h{\uu}}_2^2 + \epsilon_2}}\\
                =& \exp\lrp{\epsilon_2} + \frac{ \epsilon_1}{\exp\lrp{ - \frac{d}{2k} + \frac{1}{2k^2\delta} \lrn{\h{\uu}}_2^2+ \epsilon_2}}
            \end{alignat*}
            where $\lrabs{\epsilon_1}\leq 32\C_k(1+r^4)$ for $\C_k$ as defined in \eqref{d:clt_C} and $\lrabs{\epsilon_2} \leq \frac{2d}{k^2}$. 

            By Taylor Expansion, 
            \begin{alignat*}{1}
                \lrabs{\frac{\h{p}_\delta \lrp{\h{\uu}}}{p_\delta \lrp{\h{\uu}}}  -1}^2
                \leq& \lrp{\frac{4d}{k^2} + \frac{32 \C_k(1+r^4)}{\exp\lrp{- \frac{d}{2k} + \frac{1}{2k^2\delta} \lrn{\h{\uu}}_2^2 - \frac{2d}{k^2}}}  }^2\\
                \leq& 2^{12} \C_k^2\lrp{1+r^8} + \frac{16d^2}{k^4}
            \end{alignat*}
            where we use the fact that $\frac{d}{2k} + \frac{2d}{k^2} \leq \frac{1}{4}$.

        \end{proof}

    \subsubsection{Truncated Talagrand}
    \label{ss:clt_auxiliary}
    \begin{proof}[Proof of Lemma \ref{l:talagrand_truncated}]
        We first define $p'$ and $q'$ to be the density of $p$ and $q$ conditional over $B_c$, i.e. $p':= p/a$, $q':=q/b$. Think of $B_r$ as a manifold equipped with the Euclidean distance (note this has nothing todo with the manifold $M$ that we are considering). It is convex and connected. We verify that $p'$ satisfies the Bakry Emery condition with parameter $m$, and hence Talagrand's inequality, with parameter $m$ (see \cite{otto2000generalization}). Thus
        \begin{alignat*}{1}
            W_2^2(p',q') \leq \frac{1}{m} KL(q'||p')
        \end{alignat*}
        Furthermore, using the inequality $t \log(t) \leq t^2 - t$ (this step is almost exactly copied from \cite{zhai2018high}
        \begin{alignat*}{1}
            KL(q'||p')
            =& \int_{B_c} \frac{q'(x)}{p'(x)}  \log \lrp{\frac{q'(x)}{p'(x)}} dp'(x)\\
            \leq& \int_{B_c} \lrp{\frac{q'(x)}{p'(x)}}^2 - \frac{q'(x)}{p'(x)} dp'(x)\\
            =& \int_{B_c} \lrp{\frac{q'(x)}{p'(x)} - 1}^2 + \frac{q'(x)}{p'(x)} - 1 dp'(x)\\
            =& \int_{B_c} \lrp{\frac{q'(x)}{p'(x)} - 1}^2 dp'(x)
        \end{alignat*}

        By definition of $p'$ and $q'$, we can further upper bound
        \begin{alignat*}{1}
            \int_{B_c} \lrp{\frac{q'(x)}{p'(x)} - 1}^2 p'(x) dx
            =&  \frac{1}{a} \int_{B_c} \lrp{ \frac{a}{b} \frac{q(x)}{p(x)} - 1}^2 p(x) dx\\
            =& \frac{a}{b^2} \int_{B_c} \lrp{\frac{q(x)}{p(x)} - \frac{b}{a}}^2 p(x) dx\\
            \leq& \frac{2a}{b^2} \int_{B_c} \lrp{\frac{q(x)}{p(x)} - 1}^2 p(x) dx + \frac{2a}{b^2} \int_{B_c} \lrp{1 - \frac{b}{a}}^2 p(x) dx\\
            \leq& \frac{2a}{b^2} \int_{B_c} \lrp{\frac{q(x)}{p(x)} - 1}^2 p(x) dx + \frac{2}{ab^2} (a-b)^2\\
            \leq& 3 \int_{B_c} \lrp{\frac{q(x)}{p(x)} - 1}^2 p(x) dx + 8 (1-a)^2 + 8 (1-b)^2
        \end{alignat*}
        where the above use Young's inequality, and our upper bound on $1-a$, $1-b$, and the fact that $a\leq 1$ and $b\leq 1$.

        Finally, to bound $W_2(p,q)$ in terms of $W_2(p',q')$, we construct the following simple coupling: Let $x',y'$ be drawn from the optimal coupling of $p'$ and $q'$ (i.e. $x\sim p'$, $y\sim q'$, $\E{\lrn{x-y}_2^2} = W_2^2 (p',q')$). Let $x^c \sim \at{p}{B_c^c}$ and $y^c \sim \at{q}{B_c^c}$ independently. Assume for now that $a \leq b$. We now let $x,y$ be drawn from the following mixture process:
        \begin{itemize}
            \item With probability $a$: $x := x'$, $y:= y'$
            \item With probability $b-a$: $x := x^c$, $y := y'$
            \item with probability $1-b$: $x := x^c$, $y := y^c$
        \end{itemize}
        We first verify that the density of $x$ is $a \cdot p' + (b-a) \cdot p^c + (1-b) \cdot p^c = a \cdot p' + (1-a) \cdot p^c = p$. We smilarly verify that the density of $y$ is $q$. Therefore, $x,y$ is a valid coupling between $p$ and $q$. Under this coupling, the expected square distance is given by
        \begin{alignat*}{1}
            \E{\lrn{x-y}_2^2}
            =& a \cdot \E{\lrn{x'-y'}_2^2} + (b-a) \cdot \E{\lrn{x^c-y'}_2^2} + (1-b) \E{\lrn{x^c - y^c}_2^2}
        \end{alignat*}
        We bound each of the three terms above separately. First, from definition, $\E{\lrn{x' - y'}_2^2} = W_2^2(p',q')$. Next,
        \begin{alignat*}{1}
            (b-a) \cdot \E{\lrn{x^c-y'}_2^2}
            \leq& 2(b-a) \E{\lrn{x^c}_2^2} + 2(b-a) \E{\lrn{y'}_2^2}\\
            \leq& 2(b-a) \int \ind{\lrn{x}_2 \geq r} \lrn{x}_2^2 \frac{p(x)}{1-a} dx + 2(b-a) \int \lrn{y}_2^2 p(y) dy\\
            \leq& 2 \int \ind{\lrn{x}_2 \geq r} \lrn{x}_2^2 p(x) dx + 2(b-a) \int \lrn{y}_2^2 p(y) dy\\
            \leq& 2 (1-a) \cdot \sqrt{\E{\lrn{x}_2^4}} + 2 (b-a) \E{\lrn{y}_2^2}
        \end{alignat*}
        where the first line is by Young's inequality, the second line is by definition of $x^c$, $y^c$. The third line uses $b-a \leq 1-a$. The last line is by Cauchy Schwarz.

        We now bound the third term:
        \begin{alignat*}{1}
            (1-b) \E{\lrn{x^c - y^c}_2^2}
            \leq& 2(1-b) \E{\lrn{x^c}_2^2} + 2(1-b) \E{\lrn{y^c}_2^2}\\
            =& 2(1-b) \int \ind{\lrn{x}_2 \geq r} \lrn{x}_2^2 \frac{p(x)}{1-a} dx + 2(1-b) \int \ind{\lrn{y}_2 \geq r} \lrn{y}_2^2 \frac{p(y)}{1-b} dy\\
            \leq& 2 \int \ind{\lrn{x}_2 \geq r} \lrn{x}_2^2 p(x) dx + 2 \int \ind{\lrn{y}_2 \geq r} \lrn{y}_2^2 q(y) dy\\
            \leq& 2 (1-a) \cdot \sqrt{\E{\lrn{x}_2^4}} + 2 (1-b) \cdot \sqrt{\E{\lrn{y}_2^4}}
        \end{alignat*}
        where the steps have very similar justifications as the preceding equation block. Note that we use $\frac{1-b}{1-a} \leq 1$ as $b\geq a$. Putting everything together,
        \begin{alignat*}{1}
            W_2^2(p,q) 
            \leq& \E{\lrn{x-y}_2^2} \\
            \leq& W_2^2(p',q') + 4 (1-a) \cdot \sqrt{\E{\lrn{x}_2^4}} + 2 (b-a) \E{\lrn{y}_2^2} + 2 (1-b) \cdot \sqrt{\E{\lrn{y}_2^4}}\\
            \leq& \frac{3}{m} \int_{B_c} \lrp{\frac{q(x)}{p(x)} - 1}^2 p(x) dx + 4 (1-a) \cdot \lrp{\sqrt{\E{\lrn{x}_2^4}} + \sqrt{\E{\lrn{y}_2^4}}}
        \end{alignat*}

        Recall that we assumed that $a\leq b$. If instead $a > b$, then by similar steps, we have
        \begin{alignat*}{1}
            W_2^2(p,q) \leq \frac{3}{m} \int_{B_c} \lrp{\frac{q(x)}{p(x)} - 1}^2 p(x) dx + 4 (1-b) \cdot \lrp{\sqrt{\E{\lrn{x}_2^4}} + \sqrt{\E{\lrn{y}_2^4}}}
        \end{alignat*}

        This concludes our proof.
        
    \end{proof}
    \subsubsection{Invertibility}
    \begin{lemma}\label{l:G_is_invertible}
        Let $\H$ satisfy
        \begin{alignat*}{2}
            & \forall \uu \qquad && \lrn{\H(\uu)}_2 \leq L_\H\\
            & \forall \uu : \lrn{\uu}_2 \leq \frac{1}{2\sqrt{L_R}} \qquad && \lrn{\nabla \H(\uu)}_2 \leq L_\H'
        \end{alignat*}
        Let $\G (\uu) := \uu + \sqrt{\delta} \H(\uu)$. 
        Assume that $\delta \leq \min\lrbb{\frac{1}{16 L_R L_\H^2}, \frac{1}{16 {L_\H'}^2}}$. 
        
        Then for any $\xx$ satisfying $\lrn{\xx}_2 \leq \frac{1}{4\sqrt{L_R}}$, $\G^{-1}(\xx)$ exists and is unique. Furthermore, $\lrn{\G^{-1}(\xx)}_2 \leq \frac{1}{2\sqrt{L_R}}$
    \end{lemma}
    \begin{proof}[Proof of Lemma \ref{l:G_is_invertible}]
        Define the map $T(\uu) := \uu - \frac{1}{2} \lrp{\uu + \sqrt{\delta} \H(\uu) - \xx}$. Let $\uu_0$ be arbitary, and let $\uu_{k+1} = T(\uu_k)$. We first verify that for any initial $\uu_0$, there exists a $i$ such that for all $k\geq i$, $\lrn{\uu_k}_2 \leq \frac{1}{2\sqrt{L_R}}$. To see this,
        \begin{alignat*}{1}
            \lrn{\uu_{k+1} - \xx}_2
            =& \lrn{T(\uu_k) - \xx}_2\\
            =& \lrn{\frac{1}{2} \lrp{\uu_k - x} - \frac{\sqrt{\delta}}{2} {\H(\uu_k)}}_2\\
            \leq& \frac{1}{2} \lrn{\uu_{k} - \xx}_2 + \frac{\sqrt{\delta}}{2} L_\H\\
            \leq& \frac{1}{2} \lrn{\uu_{k} - \xx}_2 + \frac{1}{8\sqrt{L_R}} 
        \end{alignat*}

        By induction, for all $k$,
        \begin{alignat*}{1}
            \lrn{\uu_k - \xx}_2 \leq \frac{1}{2^k} \lrn{\uu_0 - \xx}_2 + \frac{1}{4\sqrt{L_R}} 
        \end{alignat*}
        our claim then follows by triangle inequality and our bound on $\lrn{\xx}_2$.

        We now verify that $T(\cdot)$ is a contractive map. Consider any initial $\uu_0$ and $\vv_0$. Let $k$ be the minimum index such that $\lrn{\uu_j}_2 \leq \frac{1}{2\sqrt{L_R}}$ and $\lrn{\vv_j}_2 \leq \frac{1}{2\sqrt{L_R}}$ for all $j\geq k$. By our assumption, for all $s\in[0,1]$, $\lrn{\nabla \H(s \uu_k + (1-s) \vv_k)}_2 \leq L_\H'$. Integrating, we get $\lrn{\H(\uu_j) - \H(\vv_j)}_2 \leq L_\H \lrn{\uu_j - \vv_j}_2$ for all $j \geq k$. Therefore, 
        \begin{alignat*}{1}
            \lrn{\uu_{k+1} - \vv_{k+1}}_2
            =&\lrn{T(\uu_k) - T(\vv_k)}_2\\
            =& \lrn{ \lrp{\uu_k - \frac{1}{2} \lrp{\uu_k + \sqrt{\delta} \H(\uu_k) - \xx}} - \lrp{\vv_k - \frac{1}{2} \lrp{\vv_k + \sqrt{\delta} \H(\vv_k) - \xx}}}_2\\
            \leq& \frac{1}{2} \lrn{\uu_k - \vv_k}_2 + \frac{1}{2} \sqrt{\delta} \lrn{\H(\uu_k) - \H(\vv_k)}_2\\
            \leq& \frac{1}{2} \lrn{\uu_k - \vv_k}_2 + \frac{1}{4} \lrn{\uu_k - \vv_k}_2\\
            \leq& \frac{3}{4} \lrn{\uu_k - \vv_k}_2
        \end{alignat*}
        
        Using essentially the proof of Banach Fixed Point Theorem, there is a unique fixed point, which we denote $\uu^*$. From our definitions, we verify that $\G^{-1}(\xx) = \uu$ if and only if $T(\uu) = \uu$.
    \end{proof}

    \subsubsection{Tail Bound}

    Consider the processes defined in \eqref{d:zz_main_theorem}. Both $\h{\zz}_k$ and $\zz_k$ are sub-Gaussian:
    \begin{lemma}\label{l:clt_gaussian_tail}
        \begin{alignat*}{1}
            & \E{\exp\lrp{\frac{\lrn{\h{\zz}_{k} - k \delta \bbeta}_2^2}{8k\delta L_{\H}^2}}}
            \leq 2\\
            & \E{\exp\lrp{\frac{\lrn{\zz_k - k\delta \bbeta}_2^2}{8k\delta}}} \leq \exp(d/8)
        \end{alignat*}

        Consequently,
        \begin{alignat*}{1}
            & \Pr{\lrn{\h{\zz}_{k} - k \delta \bbeta}_2^2 \geq t} \leq 2 \exp\lrp{\frac{-t}{8k\delta L_{\H}^2}}\\
            & \Pr{\lrn{{\zz}_{k} - k \delta \bbeta}_2^2 \geq t} \leq \exp\lrp{\frac{-t + k\delta d}{8k\delta}} 
        \end{alignat*}
    \end{lemma}
    \begin{proof}
        The bound for $\h{\zz}_k$ follows from Lemma \ref{l:clt_submartingale_tail}, by taking $\uu_k := \zz_k - k\delta \beta(\h{y_0})$ and $\H_k := \H_k(\h{\zz}_k)$. We verify from Cauchy Schwarz that it suffices to take $L_\H:= L_{\H}$.

        The bound for $\zz_k$ is because $\zz_k - k\delta \bbeta\sim \N(0, k\delta I)$. Thus $a := \frac{1}{k\delta}\lrn{\zz_k - k\delta \bbeta}_2^2$ is a Chi-square random variable with mean $d$, this it is sub-exponential with parameters $(2\sqrt{d}, 4)$. Thus $\E{e^{(a - d)/8 }} \leq e^{d/32}$. Plugging in the definition for $a$,
        \begin{alignat*}{1}
            \E{\exp\lrp{\frac{\lrn{\zz_k - k\delta \bbeta}_2^2}{8K\delta}}} \leq \exp(d/8)
        \end{alignat*}
        
        The probability bounds are by Markov's Inequality.

    \end{proof}

    \begin{lemma}\label{l:clt_submartingale_tail}
        Let $\uu_{k+1} = \uu_k + \sqrt{\delta} \H_k$ be a martingale. Assume that for all $c\in \Re^d$, $\E{\exp\lrp{\lin{c,\H_k}}} \leq \lrn{c}_2 \cdot L_\H$ a.s. Then for all $k$, 
        \begin{alignat*}{1}
            \E{\exp\lrp{\frac{\lrn{\uu_{k}}_2^2}{8k\delta L_\H^2}}}
            \leq& 2
        \end{alignat*}
    \end{lemma}
    \begin{proof}
        Consider an arbitrary $i < k$. Then
        \begin{alignat*}{1}
            \lrn{\uu_{i+1}}_2^2 
            =& \lrn{\uu_i}_2^2 + \sqrt{\delta} \lin{\uu_i, \H_i} + \delta \lrn{\H_i}_2^2\\
            \leq& \lrn{\uu_k}_2^2 + \sqrt{\delta} \lin{\uu_i, \H_i} + \delta L_\H^2 
        \end{alignat*}

        Taking expectation, 
        \begin{alignat*}{1}
            & \Ep{\F_i}{\exp\lrp{s\lrn{\uu_{i+1}}_2^2}}\\
            \leq& \exp\lrp{s\lrn{\uu_i}_2^2 + s\delta L_\H^2}\cdot \Ep{\F_i}{\exp\lrp{- s\sqrt{\delta}\lin{\uu_i,\H_i}}}\\
            \leq& \exp\lrp{s\lrn{\uu_i}_2^2 + s\delta L_\H^2} \cdot {\exp\lrp{s^2 \delta L_\H^2 \lrn{\uu_i}_2^2}}\\
            =& \exp\lrp{s \lrp{1 + s\delta L_\H^2} \lrn{\uu_i}_2^2 + s\delta L_\H^2} 
        \end{alignat*}

        Applying the above recursively, and choosing $s:= \frac{1}{8k\delta L_\H^2}$,
        \begin{alignat*}{1}
            \E{\exp\lrp{\frac{\lrn{\uu_{i}}_2^2}{8k\delta L_\H^2}}}
            \leq& 2
        \end{alignat*}
    \end{proof}

    \subsubsection{Other Algebra}

        \begin{lemma}\label{l:cancellations_due_to_constant_covariance}
            Let $\H(\uu)$ satisfy, for all $\u \in \Re^d$, 
            \begin{alignat*}{1}
                \E{\H\lrp{\u}} = 0\qquad 
                \E{\H\lrp{\u} \H\lrp{\u}^T} = I
            \end{alignat*}
            Then for any $\h{\uu}$ such that $\nabla \H(\h{\uu})$ exists, 
            \begin{alignat*}{1}
                \E{\nabla \H\lrp{\h{\uu}} \H\lrp{\h{\uu}} + H\lrp{\h{\uu}}\tr\lrp{\nabla \H\lrp{\h{\uu}}}}  = 0
            \end{alignat*}
            and for any $\h{\uu}$ such that $\nabla^2 \H(\h{\uu})$ exists, 
            \begin{alignat*}{1}
                \E{2\tr\lrp{\lin{\nabla^2 \H\lrp{\h{\uu}}, \H\lrp{\h{\uu}}}} + \tr\lrp{\lrp{\nabla\H\lrp{\h{\uu}}}^2 } + \tr\lrp{\nabla\H\lrp{\h{\uu}}}^2} = 0
            \end{alignat*}
        \end{lemma}
        \begin{proof}[Proof of Lemma \ref{l:cancellations_due_to_constant_covariance}]
            Let us define $D(\u) := \E{\H\lrp{\u} \H\lrp{\u}^T}$. Clearly $D(\u) = I$ for all $\u$. Thus
            \begin{alignat*}{1}
                \sum_j \del_j D_{i,j}(\u) = 0
            \end{alignat*}
    
            For simplicity, we omit the explicit dependence on $\u$ below, i.e. $\H := \H(\u)$ and $D:= D(\u)$
    
            With some algebra, we verify that
            \begin{alignat*}{1}
                \sum_j \del_j D_{i,j}
                =& \sum_j \del_j \E{\H \H^T}_{i,j}\\
                =& \sum_j \del_j \E{\H^i \H^j}\\
                =& \sum_j \E{\lrp{\del_j \H^i}\H^j + \lrp{\del_j \H^j} \H^i}
            \end{alignat*}
    
            We verify from algebra that
            \begin{alignat*}{1}
                \lrb{\nabla \H\lrp{\h{\uu}} \H\lrp{\h{\uu}} + H\lrp{\h{\uu}}\tr\lrp{\nabla \H\lrp{\h{\uu}}}}_i = \sum_{j} \del_j \H^i\lrp{\h{\uu}} \H^j\lrp{\h{\uu}} + \sum_{j} \H^i\lrp{\h{\uu}} \del_j \H^j\lrp{\h{\uu}} 
            \end{alignat*}

            Thus we get our first conclusion: for all $i$,
            \begin{alignat*}{1}
                \E{\lrb{\nabla \H\lrp{\h{\uu}} \H\lrp{\h{\uu}} + H\lrp{\h{\uu}}\tr\lrp{\nabla \H\lrp{\h{\uu}}}}_i} 
                = \sum_j \del_j D_{i,j} = 0
            \end{alignat*}

            Taking another derivative of $\sum_j \del_j D_{i,j}(\u) = 0$, we get
            \begin{alignat*}{1}
                \sum_{i,j} \del_i \del_j D_{i,j} = 0
            \end{alignat*}
    
            With some algebra, we verify that
            \begin{alignat*}{1}
                \sum_{i,j} \del_i \del_j D_{i,j}
                =& \sum_{i,j} \del_i \del_j \E{\H \H^T}_{i,j}\\
                =& \sum_{i,j} \E{\lrp{\del_i \del_j \H^i}\H^j + \lrp{\del_j \H^i}\lrp{\del_i \H^j} + \lrp{\del_i \del_j \H^j} \H^i + \lrp{\del_j \H^j} \lrp{\del_i \H^i}}\\
                =& \sum_{i,j} \E{2\lrp{\del_i \del_j \H^i}\H^j + \lrp{\del_j \H^i}\lrp{\del_i \H^j} + \lrp{\del_j \H^j} \lrp{\del_i \H^i}}
            \end{alignat*}

            We also verify that
            \begin{alignat*}{1}
                \tr\lrp{\lin{\nabla^2 \H\lrp{\h{\uu}}, \H\lrp{\h{\uu}}}} 
                =& \sum_{i,i} \del_j \del_i \H^i\lrp{\h{\uu}} \cdot H^j\lrp{\h{\uu}}\\
                \tr\lrp{\lrp{\nabla\H\lrp{\h{\uu}}}^2 }
                =& \sum_{i,j} \del_j \H^i \lrp{\h{\uu}} \cdot \del_i \H^j \lrp{\h{\uu}}\\
                \tr\lrp{\nabla\H\lrp{\h{\uu}}}^2
                =& \sum_{i,j} \del_i \H^i\lrp{\h{\uu}} \cdot \del_j \H^j\lrp{\h{\uu}}
            \end{alignat*}
            Thus combining the above,

            \begin{alignat*}{1}
                \E{2\tr\lrp{\lin{\nabla^2 \H\lrp{\h{\uu}}, \H\lrp{\h{\uu}}}} + \tr\lrp{\lrp{\nabla\H\lrp{\h{\uu}}}^2 } + \tr\lrp{\nabla\H\lrp{\h{\uu}}}^2}
                =& \sum_{i,j} \del_i \del_j D_{i,j} = 0
            \end{alignat*}
            This proves our second claim.
            

        \end{proof}

        \begin{lemma}\label{l:taylor_determinant_inverse}
            Let $A\in \Re^{d\times d}$, and let $\lrn{A}_2 < \frac{1}{4}$. Then
            \begin{alignat*}{1}
                \lrabs{ \log \det\lrp{I + A} - \lrp{\tr\lrp{A} - \frac{\tr\lrp{A^2}}{2}}} \leq 2 d \lrn{A}_2^3\\
                \lrabs{ \log \det\lrp{I + A} - \tr\lrp{A}} \leq 2 d \lrn{A}_2^2\\
                \lrabs{ \log \det\lrp{I + A}} \leq 2 d \lrn{A}_2
            \end{alignat*}
            
        \end{lemma}
        \begin{proof}
            By assumption on $A$, $I + A$ in invertible. Thus using (a corollary of) the Jacobi Formula,
            \begin{alignat*}{1}
                \log \det\lrp{e^{I + A}} 
                =& \tr\log\lrp{I + A}
            \end{alignat*}
            By series expansion of matrix $\log$,
            \begin{alignat*}{1}
                \log\lrp{I + A}
                =& \sum_{i=1}^\infty \lrp{-1}^{i+1} \frac{\lrp{ A}^i}{i}\\
                =&  A - \frac{A^2}{2} + \sum_{i=3}^d \lrp{-1}^{i+1} \frac{\lrp{ A}^i}{i}
            \end{alignat*}
            The last term can be bounded as
            
            \begin{alignat*}{1}
                \lrabs{\tr\lrp{\sum_{i=3}^d \lrp{-1}^{i+1} \frac{\lrp{ A}^i}{i}}}
                \leq& \sum_{i=3}^d \frac{1}{i} \lrabs{\tr\lrp{A^i}}\\
                \leq& \sum_{i=3}^d \frac{1}{i} \lrp{\sqrt{\tr\lrp{\lrp{AA^T}^i}}\cdot \sqrt{d}}\\
                \leq& \sum_{i=3}^d \frac{1}{i} d \lrn{A}_2^i\\
                \leq& d \lrn{A}_2^3 \sum_{i=3}^d \lrn{A}_2^{i-3}\\
                \leq& d \lrn{A}_2^3 \sum_{i=3}^d \lrp{\frac{1}{4}}^{i-3}\\
                \leq& 2 d \lrn{A}_2^3
            \end{alignat*}

            Using similar arguments, but truncating at $i=2$ instead of $i=3$,
            \begin{alignat*}{1}
                \log \lrp{I+A} = A + \sum_{i=2}^d \lrp{-1}^{i+1} \frac{\lrp{ A}^i}{i}
            \end{alignat*}
            and the last term can be bounded as
            \begin{alignat*}{1}
                \lrabs{\sum_{i=2}^d \lrp{-1}^{i+1} \frac{\lrp{ A}^i}{i}} \leq 2d \lrn{A}_2^2
            \end{alignat*}
            we skip the steps as they are identical to the preceding proof.

            Finally, the last desired inequality involves bounding the entire sum
            \begin{alignat*}{1}
                \lrabs{\sum_{i=3}^d \lrp{-1}^{i+1} \frac{\lrp{ A}^i}{i}} \leq 2d \lrn{A}_2
            \end{alignat*}

        \end{proof}

        \begin{lemma}\label{l:sum_Ck}
            Let $\C_k$ be as defined in \eqref{d:clt_C}. Then for any $K,r$, 
            \begin{alignat*}{1}
                \sum_{k=1}^K \sqrt{k\delta}\lrp{ \C_k + \frac{d}{k^2}} =  \lrp{\delta^{3/2} K \lambda_1 + \delta \sqrt{K} \lambda_2 + \sqrt{\delta} \log(K) \lambda_3 + \sqrt{\delta} \lambda_4 + \delta^2 K \lambda_5}
            \end{alignat*}
            for $\lambda_1,\lambda_2,\lambda_3,\lambda_4,\lambda_5 = poly\lrp{L_\H,L_\H',L_\H''}$ defined in \eqref{d:clt_lambdas}.
        \end{lemma}
    
        \begin{proof}
            Recall from \eqref{d:clt_C} that
            \begin{alignat*}{1}
                \C_k
                :=& \frac{\delta}{\sqrt{k}} \lrp{L_{\H}''  L_{\H}^2 + \lrp{L_{\H}'}^2 L_{\H}} + \frac{\sqrt{\delta}}{k} L_{\H}' L_{\H}^2 + \frac{\delta}{k} \lrp{L_{\H}' L_{\H}}^2\\
                &\quad + \frac{\sqrt{\delta}}{k}L_{\H}' L_{\H}^2 + \frac{1}{k^{3/2}} L_{\H}^3 + \frac{\delta}{k} \lrp{L_{\H}' L_{\H}}^2 + \frac{1}{k^2}L_{\H}^4\\
                &\quad + \delta^{3/2} d\lrp{L_{\H} L_{\H}' L_{\H}'' + L_{\H}''' L_{\H}^2 + L_{\H}''L_{\H}' L_{\H} + \lrp{L_{\H}'}^3} + \delta^2 \lrp{L_{\H}'' L_{\H}}^2 \\
                &\quad + \delta^{3/2} d^2 \lrp{L_{\H}''L_{\H}' L_{\H} + \lrp{L_{\H}'}^3} + \delta^2\lrp{\lrp{L_{\H}'' L_{\H}}^2 + \lrp{L_{\H}'}^4 }\\
                &\quad + \frac{\sqrt{\delta}}{k} d L_{\H}' L_{\H}^2 + \frac{\delta}{k} \lrp{L_{\H}''L_{\H}^3 + (L_{\H}')^2 L_{\H}^2}\\
                &\quad + \frac{\delta}{\sqrt{k}} d \lrp{\lrp{L_{\H}'}^2 L_{\H} + L_{\H}''L_{\H}^2} + \frac{\delta^{3/2}}{\sqrt{k}}\lrp{L_{\H}''L_{\H}'L_{\H}^2 + (L_{\H}')^3 L_{\H}}\\
                &\quad + \frac{d^2}{k^2} + \frac{1}{k^2}\\
                &\quad +  \frac{L_\H^3}{k^{3/2}} + \frac{L_\H^6}{k^3} + \delta^{3/2} d^3 {L_\H'}^3
            \end{alignat*}
            Let us reorganize the terms above, and group them by their dependence on $\delta$ and $k$:
            \begin{alignat*}{1}
                & 1.\ \frac{\delta}{\sqrt{k}}\lrp{\lrp{L_{\H}''  L_{\H}^2 + \lrp{L_{\H}'}^2 L_{\H}} +  d \lrp{\lrp{L_{\H}'}^2 L_{\H} + L_{\H}''L_{\H}^2} + \sqrt{\delta} \lrp{L_{\H}''L_{\H}'L_{\H}^2 + (L_{\H}')^3 L_{\H}}}\\
                &2.\  \frac{\sqrt{\delta}}{k}\lrp{ (2 + d)L_{\H}' L_{\H}^2  + \sqrt{\delta}\lrp{L_{\H}' L_{\H}}^2 + \sqrt{\delta} \lrp{L_{\H}''L_{\H}^3}}\\
                &3.\  \frac{1}{k^{3/2}} \lrp{2 L_{\H}^3} + \frac{1}{k^2} \lrp{L_{\H}^4 + d^2 + 1} + \frac{1}{k^3} L_{\H}^6\\
                &4.\ \delta^{3/2} \lrp{d\lrp{L_{\H}''' L_{\H}^2 + L_{\H}''L_{\H}' L_{\H}} + d^2 \lrp{L_{\H} L_{\H}' L_{\H}'' + \lrp{L_{\H}'}^3} + d^3 {L_\H'}^3 + \sqrt{\delta} \lrp{\lrp{L_{\H}'' L_{\H}}^2 + \lrp{L_{\H}'}^4 }} 
            \end{alignat*}
            \eqref{d:clt_C}
    
            Then
            \begin{alignat*}{1}
                & \sum_{k=1}^K \sqrt{k\delta} \lrp{\C_k + \frac{d}{k^2}}\\
                =& \delta^{3/2} K \cdot \lrp{\lrp{L_{\H}''  L_{\H}^2 + \lrp{L_{\H}'}^2 L_{\H}} +  d \lrp{\lrp{L_{\H}'}^2 L_{\H} + L_{\H}''L_{\H}^2} + \sqrt{\delta} \lrp{L_{\H}''L_{\H}'L_{\H}^2 + (L_{\H}')^3 L_{\H}}}\\
                &\quad + \delta \sqrt{K} \cdot \lrp{ (2 + d)L_{\H}' L_{\H}^2  + \sqrt{\delta}\lrp{L_{\H}' L_{\H}}^2 + \sqrt{\delta} \lrp{L_{\H}''L_{\H}^3}}\\
                &\quad + \sqrt{\delta} \log(K)  \cdot \lrp{2 L_{\H}^3} \\
                &\quad + \sqrt{\delta} \cdot \lrp{L_{\H}^4 + d^2 + 1 + L_{\H}^6}\\
                &\quad + \delta^2 K \cdot \lrp{d\lrp{L_{\H}''' L_{\H}^2 + L_{\H}''L_{\H}' L_{\H}} + d^2 \lrp{L_{\H} L_{\H}' L_{\H}'' + \lrp{L_{\H}'}^3} + d^3 {L_\H'}^3 + \sqrt{\delta} \lrp{\lrp{L_{\H}'' L_{\H}}^2 + \lrp{L_{\H}'}^4 }} \\
                =:& \delta^{3/2} K \lambda_1 + \delta \sqrt{K} \lambda_2 + \sqrt{\delta} \log(K) \lambda_3 + \sqrt{\delta} \lambda_4 + \delta^2 K \lambda_5
                \numberthis \label{d:clt_lambdas}
            \end{alignat*}
        \end{proof}

    \section{Fundamental Manifold Results}
    
    In this section, we provide Taylor expansion style inequalities for the evolution of geodesics on manifold. By making use of tools from Matrix ODE, we can bound the distance between two points along geodesics under various conditions. To do so, we start by bounding the variation of Jacobi field by Riemannian curvature.
    
    \subsection{Jacobi Field Approximations}

    \begin{lemma}\label{l:jacobi_field_norm_bound}
        Let $\Lambda(s,t) : [0,1]\times[0,1]\to M$ be a field of variations, where for each fixed $s$, $t \to \Lambda(s,t)$ is a geodesic. Let us define $\C := \sqrt{L_R \lrn{\del_t \Lambda(s,0)}^2}$. Then for all $s,t\in[0,1]$, 
        \begin{alignat*}{1}
            &\lrn{\del_s \Lambda(s,t)}
            \leq \cosh\lrp{{\C}} \lrn{\del_s \Lambda(s,0)} + \frac{\sinh\lrp{{\C}}}{{\C}} \lrn{D_t \del_s \Lambda(s,0)}\\
            &\lrn{\del_s \Lambda(s,t) - \party{\Lambda(s,0)}{\Lambda(s,t)}\lrp{\del_s \Lambda(s,0) + t D_t \del_s \Lambda(s,0)}}
            \leq \lrp{\cosh\lrp{{\C}} - 1} \lrn{\del_s \Lambda(s,0)} + \lrp{\frac{\sinh\lrp{{\C}}}{{\C}}  - 1} \lrn{D_t \del_s \Lambda(s,0)}\\
            & \lrn{\del_s \Lambda(s,t) - \party{\Lambda(s,0)}{\Lambda(s,t)}\lrp{\del_s \Lambda(s,0)}}
            \leq \lrp{\cosh\lrp{{\C}} - 1} \lrn{\del_s \Lambda(s,0)} + \frac{\sinh\lrp{{\C} }}{{\C}}\lrn{D_t \del_s \Lambda(s,0)}\\
            & \lrn{D_t \del_s \Lambda(s,t)}
            \leq {\C}  \sinh\lrp{{\C}} \lrn{\del_s \Lambda(s,0)} + \cosh\lrp{{\C}} \lrn{D_t \del_s \Lambda(s,0)}\\
            &\lrn{D_t \del_s \Lambda(s,t) - \party{\Lambda(s,0)}{\Lambda(s,t)}\lrp{D_t \del_s \Lambda(s,0)}} 
            \leq {\C}  \sinh\lrp{{\C}}\lrn{\del_s \Lambda(s,0)}  + \lrp{\cosh\lrp{{\C}} - 1} \lrn{D_t \del_s \Lambda(s,0)}
        \end{alignat*}

        If in addition, the derivative of the Riemannian curvature tensor is globally bounded by $L_R'$, then
        \begin{alignat*}{1}
            & \lrn{D_t \del_s \Lambda(s,t) - \party{\Lambda(s,0)}{\Lambda(s,t)}\lrp{D_t \del_s \Lambda(s,0)} - t\party{\Lambda(s,0)}{\Lambda(s,t)}\lrp{D_t D_t \del_s \Lambda(s,0)}}\\
            \leq& \lrp{L_R'\lrn{\del_t\Lambda(s,0)}^3 + \C^4} e^{\C} \lrn{\del_s \Lambda(s,0)} + \lrp{L_R'\lrn{\del_t\Lambda(s,0)}^3 + \C^2}e^{\C} \lrn{D_t \del_s \Lambda(s,0)}
        \end{alignat*}
    \end{lemma}
    \begin{proof}
        For any fixed $s$, let $E_i(s,0)$ be a basis of $T_{\Lambda(s,0)} M$.
        
        Let $E_i(s,t)$ denote an orthonormal frame along $\gamma_s(t) := \Lambda(s,t)$, by parallel transporting $E_i(s,0)$.
        
        Let $\JJ(s,t) \in \Re^d$ denote the coordinates of $\del_s \Lambda(s,t)$ wrt $E_i(s,t)$. Let $\KK(s,t) \in \Re^d$ denote the coordinates of $D_t \del_s \Lambda(s,t)$ wrt $E_i(s,t)$. Let $\aa(s,t) \in \Re^d$ denote the coordinates of $\del_t \Lambda(s,t)$ wrt $E_i(s,t)$ (this is constant for fixed $s$, for all $t$). Let $\RR(s,t) \in \Re^{4d}$ be such that $\RR^i_{j,k} = \lin{R(E_j(s,t), E_k(s,t)) E_l(s,t), E_i(s,t)}$. Let $\MM(s,t)$ denote the matrix with \\
        $\MM_{i,j}(s,t) := -\sum_{k,l} \RR^i_{jkl}(s,t) \aa_k(s,t) \aa_l(s,t)$. We verify that $\lrn{\MM(s,t)}_2 \leq L_R \lrn{\del_t \Lambda(s,t)}^2 = L_R \lrn{\del_t \Lambda(s,0)}^2$.
        
        The Jacobi Equation states that $D_t D_t \del_s \Lambda(s,t) = - R(\del_s \Lambda(s,t), \del_t \Lambda(s,t)) \del_t \Lambda(s,t)$. We verify that $-\lin{R(\del_s , \del_t) \del_t, E_i} = -\sum_{j,k,l} \RR^i_{jkl} \JJ_j \aa_k \aa_l = \lrb{\MM(s,t)\JJ(s,t)}_i$, thus $D_t D_t \del_s \Lambda(s,t) = \sum_{i=1}^d \lrb{\MM(s,t) \JJ(s,t)}_{i} \cdot E_i(s,t)$.

        We verify that $\frac{d}{dt} \JJ_i(s,t) = D_t \lin{\del_s \Lambda(s,t), E_i(s,t)} = \lin{D_t \del_s \Lambda(s,t), E_i(s,t)} = \KK_i(s,t)$. We also verify that $\frac{d}{dt} \KK_i(s,t) = D_t \lin{D_t \del_s \Lambda(s,t), E_i(s,t)} = \lin{D_t D_t \del_s \Lambda(s,t), E_i(s,t)} = \lrb{\MM(s,t)\JJ(s,t)}_i$.

        Let us now consider a fixed $s$. To simplify notation, we drop the $s$ dependence. The Jacobi Equation, in coordinate form, corresponds to the following second-order ODE:
        \begin{alignat*}{1}
            & \frac{d}{dt} \JJ(t) = \KK(t)\\
            & \frac{d}{dt} \KK(t) = \MM(t) \JJ(t) dt
        \end{alignat*}

        Define $L_{\MM} := L_R \lrn{\del_t \Lambda(s,0)}^2 = \C^2$. We verify that $L_{\MM} \geq \max_{t\in[0,1]} \lrn{\MM(t)}_2$. Then from \ref{l:formal-matrix-exponent}, we see that
        \begin{alignat*}{1}
            \cvec{\JJ(t)}{\KK(t)} = \emat\lrp{t; \MM} \cvec{\JJ(0)}{\KK(0)}
        \end{alignat*}
        From Lemma \ref{l:matrix-exponent-block-bounds},
        \begin{alignat*}{1}
            \emat\lrp{t; \bmat{0 & I \\ \MM(t) & 0}} = \bmat{\AA(t) & \BB(t) \\ \CC(t) & \DD(t)} 
        \end{alignat*}
        where each block is $\Re^{2d}$, and can be bounded as
        \begin{alignat*}{1}
            & \lrn{\AA(t)}_2 \leq \cosh\lrp{{\C} t} \leq \cosh\lrp{{\C}}\\
            & \lrn{\BB(t)}_2 \leq \frac{\sinh\lrp{{\C} t}}{{\C}} \leq \frac{\sinh\lrp{{\C} }}{{\C}} \\
            & \lrn{\CC(t)}_2 \leq {\C}  \sinh\lrp{{\C} t} \leq {\C}  \sinh\lrp{{\C}}\\
            & \lrn{\DD(t)}_2 \leq \cosh\lrp{{\C}t} \leq \cosh\lrp{{\C}}\\
            & \lrn{\AA(t) - I}_2 \leq \cosh\lrp{{\C} t} - 1 \leq \cosh\lrp{{\C}} - 1\\
            & \lrn{\BB(t) - tI}_2 \leq \frac{\sinh\lrp{{\C} t}}{{\C}}  - t \leq \frac{\sinh\lrp{{\C}}}{{\C}}  - 1\\
            & \lrn{\DD(t) - I}_2 \leq \cosh\lrp{{\C} t} - 1 \leq \cosh\lrp{{\C}} - 1
        \end{alignat*}
        where we use the fact that $\cosh(r)$, $\sinh(r)$ and $\frac{\sinh(r)}{r}$ are monotonically increasing and that $\frac{\sinh(r)}{r} - 1 \geq 0$ for positive $r$.

        It follows that
        \begin{alignat*}{1}
            \JJ(t) =& \AA(t) \JJ(0) + \BB(t) \KK(0)\\
            \KK(t) =& \CC(t) \JJ(0) + \DD(t) \KK(0)
            \elb{e:t:kjqnmd:1}
        \end{alignat*}

        Thus
        \begin{alignat*}{1}
            & \lrn{\del_s \Lambda(s,t)}\\
            =& \lrn{\JJ(t)}\\
            =& \lrn{\AA(t) \JJ(0) + \BB(t) \KK(0)}\\
            \leq& \cosh\lrp{{\C}} \lrn{\del_s \Lambda(s,0)} + \frac{\sinh\lrp{{\C}}}{{\C}} \lrn{D_t \del_s \Lambda(s,0)}
        \end{alignat*}
        and
        \begin{alignat*}{1}
            & \lrn{\del_s \Lambda(s,t) - \lrbb{\del_s \Lambda(s,0) + t D_t \del_s \Lambda(s,0)}^{\to\Lambda(s,t)}}\\
            =& \lrn{\JJ(t) - \JJ(0) - t \KK(0)}_2\\
            =& \lrn{\lrp{\AA(t) - I} \JJ(0) + \lrp{\BB(t) - tI} \KK(0)}_2\\
            \leq& \lrp{\cosh\lrp{{\C}} - 1} \lrn{\del_s \Lambda(s,0)} + \lrp{\frac{\sinh\lrp{{\C}}}{{\C}}  - 1} \lrn{D_t \del_s \Lambda(s,0)}
        \end{alignat*}
        and
        \begin{alignat*}{1}
            & \lrn{\del_s \Lambda(s,t) - \lrbb{\del_s \Lambda(s,0)}^{\to\Lambda(s,t)}}\\
            \leq& \lrn{\lrp{\AA(t) - I} \JJ(0) + \BB(t) \KK(0)}_2\\
            \leq& \lrp{\cosh\lrp{{\C}} - 1} \lrn{\del_s \Lambda(s,0)} + \frac{\sinh\lrp{{\C} }}{{\C}}\lrn{D_t \del_s \Lambda(s,0)}
        \end{alignat*}

        Similarly,
        \begin{alignat*}{1}
            & \lrn{D_t \del_s \Lambda(s,t)}\\
            =& \lrn{\KK(t)}_2\\
            \leq& \lrn{\CC(t) \JJ(0)}_2 + \lrn{\DD(t)\KK(0)}_2\\
            \leq& {\C}  \sinh\lrp{{\C}} \lrn{\del_s \Lambda(s,0)} + \cosh\lrp{{\C}} \lrn{D_t \del_s \Lambda(s,0)}
        \end{alignat*}
        and
        \begin{alignat*}{1}
            & \lrn{D_t \del_s \Lambda(s,t) - \lrbb{D_t \del_s \Lambda(s,0)}^{\to \Lambda(s,t)}}\\
            =& \lrn{\KK(t) - \KK(0)}_2\\
            =& \lrn{\CC(t) \JJ(0) + \lrp{\DD(t) - I} \KK(0)}_2\\
            \leq& {\C}  \sinh\lrp{{\C}}\lrn{\del_s \Lambda(s,0)}  + \lrp{\cosh\lrp{{\C}} - 1} \lrn{D_t \del_s \Lambda(s,0)}
        \end{alignat*}

        To prove the last bound, let us define $L_{\MM}' := L_R' \lrn{\del_t \Lambda(s,0)}^3$. We verify that \\
        $L_{\MM}' \geq \max_{t\in[0,1]} \lrn{\MM(t) - \MM(0)}_2$.
        
        we know that
        \begin{alignat*}{1}
            & \lrn{D_t \del_s \Lambda(s,t) - \lrbb{D_t \del_s \Lambda(s,0)}^{\to \Lambda(s,t)} - t\lrbb{D_t D_t \del_s \Lambda(s,0)}^{\to \Lambda(s,t)}}\\
            =& \lrn{\KK(t) - \KK(0) - t \MM(0) \JJ(0)}_2\\
            =& \lrn{\int_0^t \MM(r) \JJ(r) - \MM(0) \JJ(0) dr}_2\\
            \leq& \int_0^t \lrn{\MM(r) - \MM(0)}_2 \lrn{\JJ(0)}_2 dr + \int_0^t \lrn{\MM(r)}_2 \lrn{\JJ(r) - \JJ(0)}_2 dr\\
            \leq& \int_0^t L_\MM' \lrn{\del_s \Lambda(s,r)} + L_\MM \lrn{\del_s \Lambda(s,r) - \party{\Lambda(s,0)}{\Lambda(s,r)}\lrp{\del_s \Lambda(s,0)}} dr
        \end{alignat*}
        where the last line follows from \eqref{e:t:kjqnmd:1}. From our earlier results in this lemma,
        \begin{alignat*}{2}
            & \lrn{\del_s \Lambda(s,r)} 
            &&\leq \cosh\lrp{{\C}} \lrn{\del_s \Lambda(s,0)} + \frac{\sinh\lrp{{\C}}}{{\C}} \lrn{D_t \del_s \Lambda(s,0)}\\
            & &&\leq e^{\C} \lrn{\del_s \Lambda(s,0)} + e^{\C} \lrn{D_t \del_s \Lambda(s,0)}\\
            & \lrn{\del_s \Lambda(s,t) - \party{\Lambda(s,0)}{\Lambda(s,t)}\lrp{\del_s \Lambda(s,0)}}
            &&\leq \lrp{\cosh\lrp{{\C}} - 1} \lrn{\del_s \Lambda(s,0)} + \frac{\sinh\lrp{{\C} }}{{\C}}\lrn{D_t \del_s \Lambda(s,0)}\\
            & &&\leq {\C}^2 e^{\C} \lrn{\del_s \Lambda(s,0)} + e^{\C} \lrn{D_t \del_s \Lambda(s,0)}
        \end{alignat*}
        where the simplifications are from Lemma \ref{l:sinh_bounds}. Put together,
        \begin{alignat*}{1}
            & \lrn{D_t \del_s \Lambda(s,t) - \lrbb{D_t \del_s \Lambda(s,0)}^{\to \Lambda(s,t)} - t\lrbb{D_t D_t \del_s \Lambda(s,0)}^{\to \Lambda(s,t)}}\\
            \leq& \lrp{L_R'\lrn{\del_t\Lambda(s,0)}^3 + \C^4} e^{\C} \lrn{\del_s \Lambda(s,0)} + \lrp{L_R'\lrn{\del_t\Lambda(s,0)}^3 + \C^2}e^{\C} \lrn{D_t \del_s \Lambda(s,0)}
        \end{alignat*}
    \end{proof}

    \subsection{Discrete Coupling}

    \begin{lemma}\label{l:jacobi_field_divergence}
        Let $x,y,z \in M$, with $x = \Exp_z(u)$, $y = \Exp_z(v)$. 
        \begin{alignat*}{1}
            \dist(x,y) \leq  \frac{\sinh\lrp{\sqrt{L_R}\lrp{\lrn{u}+\lrn{v}} t}}{\sqrt{L_R}\lrp{\lrn{u}+\lrn{v}}} \lrn{v - u}
        \end{alignat*}
        
    \end{lemma}
    \begin{proof}
        Let us define the variational field
        \begin{alignat*}{1}
            \Lambda(s,t) = \Exp_z\lrp{t \lrp{u + s (v-u)}}
        \end{alignat*}
        We verify that
        \begin{alignat*}{1}
            \del_s \Lambda(s,0) =& 0\\
            \del_t \Lambda(s,0) =& u + s (v-u)\\
            D_t \del_s \Lambda(s,0) =& v-u
        \end{alignat*}

        Lemma \ref{l:jacobi_field_norm_bound} then immediately gives
        \begin{alignat*}{1}
            & \lrn{\del_s \Lambda(s,t)}
            \leq \frac{\sinh\lrp{\sqrt{L_R}\lrp{\lrn{u}+\lrn{v}} t}}{\sqrt{L_R}\lrp{\lrn{u}+\lrn{v}}} \lrn{v - u}
        \end{alignat*}

    \end{proof}

    \begin{lemma}\label{l:discrete-approximate-synchronous-coupling}
        Let $x,y\in M$. Let $\gamma(s):[0,1] \to M$ be a minimizing geodesic between $x$ and $y$ with $\gamma(0) = x$ and $\gamma(1) = y$. Let $u\in T_x M$ and $v\in T_y M$. Let $u(s)$ and $v(s)$ be the parallel transport of $u$ and $v$ along $\gamma$, with $u(0) = u$ and $v(1) = v$. 
        
        Then
        \begin{alignat*}{1}
            \dist\lrp{\Exp_{x}(u), \Exp_y(v)}^2 \leq& \lrp{1+ 4 \C^2 e^{4\C}} \dist\lrp{x,y}^2 + 32 e^{\C} \lrn{v(0) - u(0)}^2 + 2\lin{\gamma'(0), v(0) - u(0)} 
        \end{alignat*}
        where $\C := \sqrt{L_R} \lrp{\lrn{u} + \lrn{v}}$.

    \end{lemma}
    \begin{proof}
        Let us consider the length function $E(\gamma) = \int_0^1 \lrn{\gamma'(s)}^2 ds$. We define a variation of geodesics $\Lambda(s,t)$:
        \begin{alignat*}{1}
            \Lambda(s,t):= \Exp_{\gamma(s)}\lrp{t\lrp{u(s) + s (v(s) - u(s))}}
        \end{alignat*}
        We verify that
        \begin{alignat*}{1}
            \del_s \Lambda(s,0) =& \gamma'(s)\\
            \del_t \Lambda(s,0) =& u(s) + s (v(s)-u(s))\\
            D_t \del_s \Lambda(s,0) =& v(s)-u(s)
        \end{alignat*}
        
        Consider a fixed $t$, and let $\gamma_t(s):= \Lambda(s,t)$ (so $\gamma_t'(s)$ is the velocity wrt $s$). 
        \begin{alignat*}{1}
            & \frac{d}{dt} E(\gamma_t)\\
            =& \frac{d}{dt}\int_0^1 \lrn{\gamma_t'(s)}^2 ds\\
            =& \int_0^1 2\lin{\gamma_t'(s), D_t \gamma_t'(s)} ds\\
            =& \int_0^1 2\lin{\del_s \Lambda(s,t), D_t \del_s \Lambda(s,t)} ds\\
            =& \int_0^1 2\lin{\del_s \Lambda(s,0), D_t \del_s \Lambda(s,0)} ds \\
            &\quad + \int_0^1 2\lin{\del_s \Lambda(s,0), \lrbb{D_t \del_s \Lambda(s,t)}^{\to \Lambda(s,0)} - D_t \del_s \Lambda(s,0)} ds \\
            &\quad +\int_0^1 2\lin{\del_s \Lambda(s,t) - \lrbb{\del_s \Lambda(s,0)}^{\to \Lambda(s,t)}, D_t \del_s \Lambda(s,t)} ds 
            \numberthis \label{e:t:ddtE}
        \end{alignat*}

        For any $s$, and for $t=0$, $\del_s \Lambda(s,0) = \gamma'(s)$ and $D_t \del_s \Lambda(s,0) = v(s) - u(s)$. Using the fact that norms and inner products are preserved under parallel transport, the first term can be simplified as
        \begin{alignat*}{1}
            \int_0^1 2\lin{\del_s \Lambda(s,0), D_t \del_s \Lambda(s,0)} ds = 2\lin{\gamma'(0), v(0) - u(0)}
        \end{alignat*}

        To bound the second and third term, we use Lemma \ref{l:jacobi_field_norm_bound}:
        \begin{alignat*}{1}
            \lrn{\del_s \Lambda(s,0)}
            =& \lrn{\gamma'(0)}
        \end{alignat*}
        and
        \begin{alignat*}{1}
            &\lrn{D_t \del_s \Lambda(s,t) - \party{\Lambda(s,0)}{\Lambda(s,t)}\lrp{D_t \del_s \Lambda(s,0)}} \\
            \leq& \sqrt{L_R \lrn{\del_t \Lambda(s,0)}^2}  \sinh\lrp{\sqrt{L_R \lrn{\del_t \Lambda(s,0)}^2}}\lrn{\del_s \Lambda(s,0)}  + \lrp{\cosh\lrp{\sqrt{L_R \lrn{\del_t \Lambda(s,0)}^2}} - 1} \lrn{D_t \del_s \Lambda(s,0)}\\
            \leq& \C  \sinh\lrp{\C}\lrn{\gamma'(0)}  + \lrp{\cosh\lrp{\C} - 1} \lrn{v(0) - u(0)}
        \end{alignat*}
        where we use the fact that $\sqrt{L_R \lrn{\del_t \Lambda(s,0)}^2}\leq \C$.

        We can thus bound the second term of \eqref{e:t:ddtE} as
        \begin{alignat*}{1}
            &\lrabs{\int_0^1 2\lin{\del_s \Lambda(s,0), \lrbb{D_t \del_s \Lambda(s,t)}^{\to \Lambda(s,0)} - D_t \del_s \Lambda(s,0)} ds}\\
            \leq& 2\lrn{\gamma'(0)} \cdot \lrp{\C  \sinh\lrp{\C}\lrn{\gamma'(0)}  + \lrp{\cosh\lrp{\C} - 1} \lrn{v(0) - u(0)}}\\
            \leq& 4 \lrn{\gamma'(0)}^2 \lrp{\C \sinh\lrp{\C} + \lrp{\cosh\lrp{\C}-1}^2} + 4 \lrn{v(0) - u(0)}^2
        \end{alignat*}
        
        Finally, to bound the third term of \eqref{e:t:ddtE}, we again apply Lemma \ref{l:jacobi_field_norm_bound}:
        \begin{alignat*}{1}
            &\lrn{\del_s \Lambda(s,t) - \party{\Lambda(s,0)}{\Lambda(s,t)}\lrp{\del_s \Lambda(s,0)}}\\
            \leq& \lrp{\cosh\lrp{\sqrt{L_R \lrn{\del_t \Lambda(s,0)}^2} } - 1} \lrn{\del_s \Lambda(s,0)} + \lrp{\frac{\sinh\lrp{\sqrt{L_R \lrn{\del_t \Lambda(s,0)}^2}}}{\sqrt{L_R \lrn{\del_t \Lambda(s,0)}^2}}} \lrn{D_t \del_s \Lambda(s,0)}\\
            \leq& \lrp{\cosh\lrp{\sqrt{L_R} \lrp{\lrn{u(s)} + \lrn{v(s)}}} - 1} \lrn{\gamma'(s)} + \lrp{\frac{\sinh\lrp{\sqrt{L_R} \lrp{\lrn{u(s)} + \lrn{v(s)}}}}{\sqrt{L_R} \lrp{\lrn{u(s)} + \lrn{v(s)}}}} \lrn{v(s)-u(s)}\\
            =& \lrp{\cosh\lrp{\C} - 1} \lrn{\gamma'(0)} + \lrp{\frac{\sinh\lrp{\C}}{\C}} \lrn{v(0) - u(0)}
        \end{alignat*}
        and
        \begin{alignat*}{1}
            & \lrn{D_t \del_s \Lambda(s,t)}\\
            \leq& \sqrt{L_R \lrn{\del_t \Lambda(s,0)}^2} \sinh\lrp{\sqrt{L_R \lrn{\del_t \Lambda(s,0)}^2}} \lrn{\del_s \Lambda(s,0)} + \cosh\lrp{\sqrt{L_R \lrn{\del_t \Lambda(s,0)}^2}} \lrn{D_t \del_s \Lambda(s,0)}\\
            \leq& \sqrt{L_R}\lrp{\lrn{u} + \lrn{v}}  \sinh\lrp{\sqrt{L_R}\lrp{\lrn{u} + \lrn{v}}} \lrn{\gamma'(0)} + \cosh\lrp{\sqrt{L_R}\lrp{\lrn{u} + \lrn{v}}} \lrn{v(0) - u(0)}\\
            =& \C  \sinh\lrp{\C} \lrn{\gamma'(0)} + \cosh\lrp{\C} \lrn{v(0) - u(0)}
        \end{alignat*}
        for $0 \leq t \leq 1$, where we usse the fact that $\cosh(r)$ and $\frac{\sinh(r)}{r}$ are monotonically increasing in $r$.
        Put together, the third term of \eqref{e:t:ddtE} is bounded as
        \begin{alignat*}{1}
            & \lrabs{\int_0^1 2\lin{\del_s \Lambda(s,t) - \lrbb{\del_s \Lambda(s,0)}^{\to \Lambda(s,t)}, D_t \del_s \Lambda(s,t)} ds} \\
            \leq& 2 \lrp{\lrp{\cosh\lrp{\C} - 1} \lrn{\gamma'(0)} + \lrp{\frac{\sinh\lrp{\C}}{\C}} \lrn{v(0) - u(0)} }\cdot \lrp{\C  \sinh\lrp{\C} \lrn{\gamma'(0)} + \cosh\lrp{\C} \lrn{v(0) - u(0)}}\\
            \leq& 8 \lrn{v(0) - u(0)}^2 \lrp{\cosh\lrp{\C}^2 + \frac{\sinh\lrp{\C}^2}{\C^2}} + 8\lrn{\gamma'(0)}^2 \lrp{\lrp{\cosh\lrp{\C} - 1}^2 + \C^2 \sinh\lrp{\C}^2}
        \end{alignat*}

        Put together, we get
        \begin{alignat*}{1}
            &\lrabs{\frac{d}{dt} E(\gamma_t) - 2\lin{\gamma'(0), v(0) - u(0)}} \\
            \leq& 8 \lrn{v(0) - u(0)}^2 \lrp{\cosh\lrp{\C}^2 + \frac{\sinh\lrp{\C}^2}{\C^2}} + 8\lrn{\gamma'(0)}^2 \lrp{\lrp{\cosh\lrp{\C} - 1}^2 + \C^2 \sinh\lrp{\C}^2}\\
            &\quad + 4 \lrn{v(0) - u(0)}^2 + 4 \lrn{\gamma'(0)}^2 \lrp{\C \sinh\lrp{\C} + \lrp{\cosh\lrp{\C}-1}^2}\\
            \leq& 8 \lrn{v(0) - u(0)}^2 \lrp{\cosh\lrp{\C}^2 + \frac{\sinh\lrp{\C}^2}{\C^2} + 1} \\
            &\quad + 8\lrn{\gamma'(0)}^2 \lrp{2\lrp{\cosh\lrp{\C} - 1}^2 + \C^2 \sinh\lrp{\C}^2 + \C \sinh\lrp{\C}}
        \end{alignat*}

        Integrating for $t=\in[0,1]$, and noting that $E(\gamma_0) = \lrn{\gamma'(0)}$,
        \begin{alignat*}{1}
            E\lrp{\gamma_1} 
            \leq& \lrp{1+8 \lrp{2\lrp{\cosh\lrp{\C} - 1}^2 + \C^2 \sinh\lrp{\C}^2 + \C \sinh\lrp{\C}}} E\lrp{\gamma_0} \\
            &\qquad + 8 \lrn{v(0) - u(0)}^2 \lrp{\cosh\lrp{\C}^2 + \frac{\sinh\lrp{\C}^2}{\C^2} + 1}\\
            &\qquad + 2\lin{\gamma'(0), v(0) - u(0)} 
        \end{alignat*}

        From Lemma \ref{l:sinh_bounds}, we can upper bound
        \begin{alignat*}{1}
            8 \lrp{2\lrp{\cosh\lrp{\C} - 1}^2 + \C^2 \sinh\lrp{\C}^2 + \C \sinh\lrp{\C}}
            \leq& 8 r^4 e^{2r} + 8 r^4 e^{2r} + r^2 e^r\\
            \leq& 4 r^2 e^{4r}\\
            \cosh\lrp{\C}^2 + \frac{\sinh\lrp{\C}^2}{\C^2} + 1
            \leq& 4e^{2r}
        \end{alignat*}
        where we use the fact that $r^2 \leq e^{2r}/6$ for all $r \geq 0$. 

        The conclusion follows by noting that $\dist(x,y) = \sqrt{E\lrp{\gamma_0}}$ and $\dist\lrp{\Exp_{x}(u), \Exp_y(v)} \leq \sqrt{E\lrp{\gamma_1}}$.

    \end{proof}

    \begin{lemma}\label{l:discrete-approximate-synchronous-coupling-ricci}
        Let $x,y\in M$. Let $\gamma(s):[0,1] \to M$ be a minimizing geodesic between $x$ and $y$ with $\gamma(0) = x$ and $\gamma(1) = y$. Let $u\in T_x M$ and $v\in T_y M$. Let $u(s)$ and $v(s)$ be the parallel transport of $u$ and $v$ along $\gamma$. Let $u = u_1 + u_2$ and $v = v_1 + v_2$ be a decomposition such that $v_2 = \party{x}{y} u_2$, where the parallel transport is along $\gamma(s)$. 

        Let us define $u_1(s), u_2(s), v_1(s)$, all mapping from $[0,1] \to T_{\gamma(s)} M$, such that they are the parallel transport of $u_1, u_2, v_1$ along $\gamma(s)$ respectively ($u_1(0) = u_1$, $u_2(0) = u_2$, $v_1(1) = v_1$, $u_2(1) = v_2$)
        
        Then
        \begin{alignat*}{1}
            &\dist\lrp{\Exp_{x}(u), \Exp_y(v)}^2 - \dist\lrp{x,y}^2\\
            \leq& 2\lin{\gamma'(0), v(0) - u(0)} + \lrn{v(0) - u(0)}^2 \\
            &\quad -2\int_0^1 \lin{R\lrp{\gamma'(s),(1-s) u(s) + s v(s)}(1-s) u(s) + s v(s),\gamma'(s)} ds \\
            &\quad + \lrp{2\C^2 e^{\C} + 18\C^4 e^{2\C}} \lrn{v(0) - u(0)}^2 + \lrp{18\C^4 e^{2\C} + 4\C'} \dist\lrp{x,y}^2 + 4 \C^2 e^{2\C} \dist\lrp{x,y} \lrn{v(0) - u(0)}
        \end{alignat*}
        where $\C := \sqrt{L_R} \lrp{\lrn{u} + \lrn{v}}$ and $\C' := L_R' \lrp{\lrn{u} + \lrn{v}}^3$.
    \end{lemma}
    \begin{proof}

        The proof is similar to Lemma \ref{l:discrete-approximate-synchronous-coupling}. Let us consider the length function $E(\gamma) = \int_0^1 \lrn{\gamma'(s)}^2 ds$. We define a variation of geodesics $\Lambda(s,t)$:
        \begin{alignat*}{1}
            \Lambda(s,t):= \Exp_{\gamma(s)}\lrp{t\lrp{u(s) + s (v(s) - u(s))}}
        \end{alignat*}
        We verify that
        \begin{alignat*}{1}
            \del_s \Lambda(s,0) =& \gamma'(s)\\
            \del_t \Lambda(s,0) =& u(s) + s (v(s)-u(s))\\
            D_t \del_s \Lambda(s,0) =& v(s)-u(s)
        \end{alignat*}

        Consider a fixed $t$, and let $\gamma_t(s):= \Lambda(s,t)$ (so $\gamma_t'(s)$ is the velocity wrt $s$). 

        \begin{alignat*}{1}
            & \frac{d}{dt} E(\gamma_t)\\
            =& \frac{d}{dt}\int_0^1 \lrn{\gamma_t'(s)}^2 ds\\
            =& \int_0^1 2\lin{\gamma_t'(s), D_t \gamma_t'(s)} ds\\
            =& \int_0^1 2\lin{\del_s \Lambda(s,t), D_t \del_s \Lambda(s,t)} ds
        \end{alignat*}
        and
        \begin{alignat*}{1}
            & \frac{d^2}{dt^2} E(\gamma_t)\\
            =& \int_0^1 2\lin{D_t \del_s \Lambda(s,t), D_t \del_s \Lambda(s,t)} ds + \int_0^1 2\lin{\del_s \Lambda(s,t), D_t D_t \del_s \Lambda(s,t)} ds\\
            =& \int_0^1 2\lrn{D_t \del_s \Lambda(s,t)}^2 ds -\int_0^1 2\lin{R\lrp{\del_s(\Lambda(s,t)), \del_t \Lambda(s,t)}\del_t \Lambda(s,t),\del_s \Lambda(s,t)} ds\\
            =& \int_0^1 2\lrn{D_t \del_s \Lambda(s,t)}^2 ds - \int_0^1 2\lin{R\lrp{\del_s(\Lambda(s,t)), \party{\Lambda(s,0)}{\Lambda(s,t)}\del_t \Lambda(s,0)}\party{\Lambda(s,0)}{\Lambda(s,t)}\del_t \Lambda(s,t),\del_s \Lambda(s,t)} ds\\
            \leq& \int_0^1 2\lrn{D_t \del_s \Lambda(s,t)}^2 ds - \int_0^1 2\lin{R\lrp{\del_s(\Lambda(s,0)),\del_t \Lambda(s,0)}\del_t \Lambda(s,t),\del_s \Lambda(s,0)} ds \\
            &\quad + \int_0^1 4 L_R \lrn{\del_t \Lambda(s,0)}^2 \lrn{\del_s \Lambda(s,t)} \lrn{\del_s \Lambda(s,t) - \party{\Lambda(s,0)}{\Lambda(s,t)}\del_s \Lambda(s,0)} + 4 L_R' \lrn{\del_t \Lambda(s,0)}^3 \lrn{\del_s \Lambda(s,0)}^2 ds
        \end{alignat*}
        where the second equality uses the Jacobi equation.

        The Riemannian curvature tensor term can be simplified as
        \begin{alignat*}{1}
            & -\int_0^1 2\lin{R\lrp{\del_s(\Lambda(s,0)),\del_t \Lambda(s,0)}\del_t \Lambda(s,t),\del_s \Lambda(s,0)} ds\\
            =& -2\int_0^1 \lin{R\lrp{\gamma'(s),(1-s) u(s) + s v(s)}(1-s) u(s) + s v(s),\gamma'(s)} ds
        \end{alignat*}

        We further bound
        \begin{alignat*}{1}
            & \int_0^1 2\lrn{D_t \del_s \Lambda(s,t)}^2 ds \\
            \leq& \int_0^1 2\lrp{{\C}  \sinh\lrp{{\C}} \lrn{\del_s \Lambda(s,0)} + \cosh\lrp{{\C}} \lrn{D_t \del_s \Lambda(s,0)}}^2 ds \\
            \leq& 2\cosh\lrp{\C}^2 \int_0^1 \lrn{D_t \del_s \Lambda(s,0)}^2 ds\\
            &\quad + 2{\C}^2  \sinh\lrp{{\C}}^2 \int_0^1 \lrn{\del_s \Lambda(s,0)}^2 ds +  4 \C \int_0^1 \sinh\lrp{\C} \cosh\lrp{\C} \lrn{ \del_s \Lambda(s,0)}\lrn{D_t \del_s \Lambda(s,0)} ds\\
            =& 2\cosh\lrp{\C}^2 \lrn{v(0)-u(0)}^2\\
            &\quad + 2{\C}^2  \sinh\lrp{{\C}}^2 \dist\lrp{x,y}^2 +  4 \C \int_0^1 \sinh\lrp{\C} \cosh\lrp{\C} \lrn{ \del_s \Lambda(s,0)}\lrn{D_t \del_s \Lambda(s,0)} ds\\
            \leq& 2 \lrn{v(0) - u(0)}^2 + 2 \lrp{\C^2 e^{\C} + \C^4 e^{2\C}} \lrn{v(0) - u(0)}^2\\
            &\quad + 2\C^4 e^{2\C} \dist\lrp{x,y}^2 + 4 \C^2 e^{2\C} \dist\lrp{x,y} \lrn{v(0) - u(0)}
        \end{alignat*}
        where we use Lemma \ref{l:jacobi_field_norm_bound} and Lemma \ref{l:sinh_bounds}.

        We also bound
        \begin{alignat*}{1}
            & \int_0^1 4 L_R \lrn{\del_t \Lambda(s,0)}^2 \lrn{\del_s \Lambda(s,t)} \lrn{\del_s \Lambda(s,t) - \party{\Lambda(s,0)}{\Lambda(s,t)}\del_s \Lambda(s,0)} ds\\
            \leq& 4\C^2 \lrp{\cosh\lrp{\C} \dist\lrp{x,y} + \frac{\sinh\lrp{\C}}{\C} \lrn{u(0) - v(0)}}\lrp{\lrp{\cosh\lrp{\C}-1}\dist\lrp{x,y} + \lrp{\frac{\sinh\lrp{\C}}{\C}-1} \lrn{u(0)-v(0)}}\\
            \leq& 16\C^4e^{2\C} \dist\lrp{x,y}^2 + 16\C^4e^{2\C} \lrn{u(0) - v(0)}^2
        \end{alignat*}
        where we use Lemma \ref{l:jacobi_field_norm_bound} and Lemma \ref{l:sinh_bounds}.

        We finally bound
        \begin{alignat*}{1}
            & \int_0^1 4 L_R' \lrn{\del_t \Lambda(s,0)}^3 \lrn{\del_s \Lambda(s,0)}^2 ds
            \leq 4\C' \dist\lrp{x,y}^2
        \end{alignat*}
        by definition of $\C'$.

        Combining the above bounds,
        \begin{alignat*}{1}
          &E\lrp{\gamma_1}
          = \at{\frac{d}{dt} E\lrp{\gamma_t}}{t=0} + \int_0^1 \int_0^r \frac{d^2}{dt^2} E\lrp{\gamma_t} dt dr\\
          &\quad\leq\  2\lin{\gamma'(0), v(0) - u(0)} + \lrn{v(0) - u(0)}^2\\
          &\quad-2\int_0^1 \lin{R\lrp{\gamma'(s),(1-s) u(s) + s v(s)}(1-s) u(s) + s v(s),\gamma'(s)} ds \\
          &\quad + 2 \lrp{\C^2 e^{\C} + \C^4 e^{2\C}} \lrn{v(0) - u(0)}^2 + 2\C^4 e^{2\C} \dist\lrp{x,y}^2 + 4 \C^2 e^{2\C} \dist\lrp{x,y} \lrn{v(0) - u(0)}\\
          &\quad + 16\C^4e^{2\C} \dist\lrp{x,y}^2 + 16\C^4e^{2\C} \lrn{u(0) - v(0)}^2 + 4\C' \dist\lrp{x,y}^2\\
          &\quad= 2\lin{\gamma'(0), v(0) - u(0)} + \lrn{v(0) - u(0)}^2\\
          &\quad-2\int_0^1 \lin{R\lrp{\gamma'(s),(1-s) u(s) + s v(s)}(1-s) u(s) + s v(s),\gamma'(s)} ds \\
          &\quad + \lrp{2\C^2 e^{\C} + 18\C^4 e^{2\C}} \lrn{v(0) - u(0)}^2 + \lrp{18\C^4 e^{2\C} + 4\C'} \dist\lrp{x,y}^2 + 4 \C^2 e^{2\C} \dist\lrp{x,y} \lrn{v(0) - u(0)}
        \end{alignat*}
        Our conclusion follows as $\dist\lrp{\Exp_x(u), \Exp_y(v)}^2 \leq E(\gamma_1)$.
    \end{proof}

    \begin{lemma}\label{l:sinh_ode}
        Let $a_t, b_t : t \to \Re^+$ satisfy
        \begin{alignat*}{1}
            & \frac{d}{dt} a_t = b_t\\
            & \frac{d}{dt} b_t \leq C a_t
        \end{alignat*}
        with initial conditions $a_0, b_0$, then for all $t$,
        \begin{alignat*}{1}
            & a_t \leq {a_0} \cosh\lrp{\sqrt{C} t} + \frac{b_0}{\sqrt{C}} \sinh\lrp{\sqrt{C} t}\\
            & b_t \leq \sqrt{C} \lrp{{a_0} \sinh\lrp{\sqrt{C} t} + \frac{b_0}{\sqrt{C}} \cosh\lrp{\sqrt{C} t}}
        \end{alignat*}
    \end{lemma}
    \begin{proof}
        Let $x_t := {a_0} \cosh\lrp{\sqrt{C} t} + \frac{b_0}{\sqrt{C}} \sinh\lrp{\sqrt{C} t}$ and $y_t := \sqrt{C} \lrp{{a_0} \sinh\lrp{\sqrt{C} t} + \frac{b_0}{\sqrt{C}} \cosh\lrp{\sqrt{C} t}}$. We verify that
        \begin{alignat*}{1}
            & \frac{d}{dt} x_t = \sqrt{C} \lrp{{a_0} \sinh\lrp{\sqrt{C} t} + \frac{b_0}{\sqrt{C}} \cosh\lrp{\sqrt{C} t}} = y_t\\
            & \frac{d}{dt} y_t = C \lrp{{a_0} \cosh\lrp{\sqrt{C} t} + \frac{b_0}{\sqrt{C}} \sinh\lrp{\sqrt{C} t}} = C x_t
        \end{alignat*}
        We further verify the initial conditions. Note that $\sinh(0) = 0$ and $cosh(0) = 1$. Thus
        \begin{alignat*}{1}
            & x_0 = a_0\\
            & y_0 = b_0
        \end{alignat*}
        Finally, we verify that $a_t \leq x_t$ and $b_t \leq y_t$ for all $t$:
        \begin{alignat*}{1}
            & \frac{d}{dt} x_t - a_t = y_t - b_t\\
            & \frac{d}{dt} y_t - b_t \geq C \lrp{x_t - a_t}
        \end{alignat*}
        
    \end{proof}

    \begin{lemma}\label{l:sinh_ode_with_offset}
        Let $a_t, b_t : t \to \Re^+$ satisfy
        \begin{alignat*}{1}
            & \frac{d}{dt} a_t = b_t\\
            & \frac{d}{dt} b_t \leq C a_t + Dt + E
        \end{alignat*}
        with initial conditions $a_0 = 0, b_0 = 0$, then for all $t$,
        \begin{alignat*}{1}
            & a_t \leq \frac{E}{C} \cosh(\sqrt{C} t) + \frac{D}{C^{3/2}} \sinh\lrp{\sqrt{C} t} - \frac{D}{C}t  - \frac{E}{C}\\
            & b_t \leq \frac{E}{\sqrt{C}} \sinh(\sqrt{C} t) + \frac{D}{C} \cosh\lrp{\sqrt{C} t} - \frac{D}{C}
        \end{alignat*}
    \end{lemma}
    \begin{proof}
        Let $x_t := \frac{E}{C} \cosh(\sqrt{C} t) + \frac{D}{C^{3/2}} \sinh\lrp{\sqrt{C} t} - \frac{D}{C}t  - \frac{E}{C}$ and \\
        $y_t :=  \frac{E}{\sqrt{C}} \sinh(\sqrt{C} t) + \frac{D}{C} \cosh\lrp{\sqrt{C} t} - \frac{D}{C}$.
        
        We verify that
        \begin{alignat*}{1}
            & \frac{d}{dt} x_t = \frac{E}{\sqrt{C}} \sinh(\sqrt{C} t) + \frac{D}{C} \cosh\lrp{\sqrt{C} t} - \frac{D}{C} = y_t\\
            & \frac{d}{dt} y_t = E \cosh(\sqrt{C} t) + \frac{D}{\sqrt{C}} \sinh\lrp{\sqrt{C} t} = C x_t + D t + E
        \end{alignat*}
        we also verify the initial conditions that $x_0 = 0$ and $y_0 = 0$.

    \end{proof}

    \begin{lemma}\label{l:identity_covariance_parallel_transport}
		Let $x,y\in M$, and let $E_1...E_d$ be an orthonormal basis at $T_xM$.
		Let $v\in T_x M$ be a random vector with $\E{\lin{E_i, v}\lin{E_j,v}} = \ind{i=j}$.
		Let $\gamma:[0,1] \to M$ be any smooth path between $x$ and $y$. Let $v(t)$ be the parallel transport of $v$ along $\gamma$. Then for any basis $E'_1...E'_d$ at $T_y M$,
		\begin{alignat*}{1}
			\E{\lin{v(t), E_i'}\lin{v(t), E_j'}} = \ind{i=j}
		\end{alignat*}
		In other words, if $v$ has identity covariance, then the parallel transport of $v$ has identity covariance.
	\end{lemma}
	\begin{proof}
		Let $E_i(t)$ be an orthonormal frame along $\gamma$ with $E_i(0) = E_i$.
		Under parallel transport, $\frac{d}{dt}\lin{v(t), E_i(t)}=0$. Thus for all $t$,
		\begin{alignat*}{1}
			\E{\lin{E_i(t), v}\lin{E_j(t),v}}
			= \E{\lin{E_i(0), v}\lin{E_j(0),v}} = \ind{i=j}
		\end{alignat*}
		Finally, consider any basis $E_i'$. Let $E_i' = \sum_k \alpha^i_k E_i(1)$, i.e. $\alpha^i_k = \lin{E_i', E_k(1)}$ Then
		\begin{alignat*}{1}
			& \lin{v, E_i'}\lin{v, E_j'}\\
			=& \sum_{k,l} \alpha^i_j \alpha^j_l \lin{v, E_k}\lin{v, E_l}\\
			=& \sum_{k,l} \alpha^i_j \alpha^j_l \ind{k=l}\\
			=& \sum_{k} \alpha^i_k \alpha^j_k\\
			=& \lin{E_i', E_j'}\\
			=& \ind{i=j}
		\end{alignat*}
	\end{proof}

    \begin{lemma}
        \label{l:symmetric_distribution_parallel_transport}
        Let $x,y\in M$, and let $E_1...E_d$ be an orthonormal basis at $T_xM$.

        Let $\alpha$ denote a spherically symmetric random variable in $\Re^d$, i.e. for any orthogonal matrix $G \in \Re^{d\times d}$
        \begin{alignat*}{1}
            \alpha \overset{d}{=} G \alpha
        \end{alignat*}
        Then for any $x\in M$, let $E_1... E_d$ and $E_1'...E_d'$ be two sets of orthonormal bases of $T_x M$. then
        \begin{alignat*}{1}
            \sum_{i=1}^d \alpha_i E_i \overset{d}{=} \sum_{i=1}^d \alpha_i E_i'
        \end{alignat*}

        Consequently, let $v\in T_x M:= \sum_{i=1}^d \alpha_i E_i$. Let $\gamma:[0,1] \to M$ be any differentiable path between $x$ and $y$. Let $v(t)$ be the parallel transport of $v$ along $\gamma$. Then for any orthogonal basis $E'_1...E'_d$ at $T_y M$,
		\begin{alignat*}{1}
			v(1) \overset{d}{=} \sum_{i=1}^d \alpha_i E_i'
		\end{alignat*}
        
    \end{lemma}

    \begin{proof}
        First, we verify that if $\alpha$ is spherically symmetric, and $E_1.. E_d$, $E'_1...E'_d$ are two sets of orthonormal basis at some point $x$, then
        \begin{alignat*}{1}
            \sum_{i=1}^d \alpha_i E_i \overset{d}{=} \sum_{i=1}^d \alpha_i E'_i
        \end{alignat*}
    
        To see this, notice that there exists an orthogonal matrix $G$, with $G_{i,j} = \lin{E_i, E'_j}$, such that
        \begin{alignat*}{1}
            E_i = \sum_{j=1}^d G_{j,i} E'_j
        \end{alignat*}
        We further verify that $G_{i,j}$ is orthogonal. It suffices to verify that $G G^T = I$.
        \begin{alignat*}{1}
            \ind{j=k} = \lin{E'_i, E'_j}
            =& \lin{\sum_{k=1}^d \lin{E'_i, E_k} E_j, \sum_{\ell=1}^d \lin{E'_j, E_\ell} E_\ell}\\
            =& \sum_{k,\ell} \lin{E'_i, E_k} \lin{E'_j, E_\ell} \lin{E_j, E_\ell}\\
            =& \sum_{k} \lin{E'_i, E_k} \lin{E'_j, E_k} \\
            =& \lin{G_{i,\cdot}, G_{j,\cdot}}
        \end{alignat*}
        Note that the inner product on the last line is dot product over $\Re^d$, and the inner product on preceding lines are over $T_x M$. The above implies that
        \begin{alignat*}{1}
            G G^T = I
        \end{alignat*}
        i.e. $G$ is orthogonal.
        
        Now consider any arbitrary function $f: T_x M \to \Re$, then
        \begin{alignat*}{1}
            \E{f\lrp{\sum_i \alpha_i E_i}}
            =& \E{f\lrp{\sum_i \sum_j \alpha_i G_{i,j} E'_j}}\\
            :=& \E{f\lrp{\sum_j \beta_j E'_j}}
        \end{alignat*}
        where we defined $\beta_j := \sum_i \alpha_i G_{i,j}$. We finally verify that $\beta \overset{d}{=} \alpha$. This follows from the fact that $\beta = G^T \alpha$, where $G$ is an orthogonal matrix, and the definition of spherical symmetry for $\alpha$.

        Consider an arbitrary line $\gamma(t) : [0,1] \to M$, with $x := \gamma(0)$, $y := \gamma(1)$. Let $E_i$ be an orthonormal basis at $T_x M$, and $E_i(t)$ be an orthonormal basis at $T_{\gamma(t)} M$ obtained from parallel transport of $E_i$. This proves the first claim.

        To verify the second claim, let $v \in T_x M$ be a random vector, given by 
        \begin{alignat*}{1}
            v = \sum_{i=1}^d \alpha_i E_i
        \end{alignat*}
        where $\alpha$ is some spherically random vector in $\Re^d$. Let $v(t)$ be the parallel transport of $v$ along $\gamma$. Let $\alpha(t):= \lin{v(t), E_i(t)}$. Then by definition of parallel transport, for all $i$,
        \begin{alignat*}{1}
            \frac{d}{dt} \lin{v(t), E_i(t)} = 0
        \end{alignat*}
        so that for all $t\in[0,1]$,
        \begin{alignat*}{1}
            \alpha(t) := \alpha
        \end{alignat*}
        the second claim then follows from the first claim.

    \end{proof}

    \section{First and higher order trivializations}
    \label{s:trivialization}
    In this section, we consider a "base point" $x\in M$, and show that for any $u,v\in T_x M$ such that $\lrn{u}$ and $\lrn{v}$ are sufficiently small, $\mu(t) :=\Exp_{\Exp_x(u)}(t v)$ can be described by $\mu(t) = \Exp_x (a(t))$, where $a(t)$ is a second order ODE in $T_x M$. Subsequently, by bounding various moments of $a(t)$, we obtain different orders of approximations of the geodesic $\mu(t)$. The main result of this section is Lemma \ref{n:l:w(s)}, which we state in Section \ref{n:s:main_aa_result}. 
    
    \subsection{Setup}
    \label{sss:proof_of_triangle_distortion_G:setup}
    Before stating our main result, we will first need to set up some notation, and introduce a few key quantities. The definitions in this subsection will be used throughout the entirety of Section \ref{s:trivialization}.

    Let $x\in M$, let $E = \lrbb{E_1...E_d}$ denote an orthonormal basis of $T_x M$. Throughout this section, $x$ and $E$ are arbitrary, but fixed.

    For any $u\in T_x M$, define
    \begin{alignat*}{1}
        \gamma(t;u) := \Exp_{x}(t u)
    \end{alignat*}
    For $t\in[0,1]$, we define the parallel orthonormal frame along $\gamma(t;u)$ as
    \begin{alignat*}{1}
        E(t;u) := \lrbb{E_1(t;u)...E_d(t;u)} := \lrbb{\party{x}{\gamma(t;u)}E_1 ... \party{x}{\gamma(t;u)}E_d}
        \elb{n:e:d:etu}
    \end{alignat*}
    where $\party{x}{\gamma(t;u)}$ is along the geodesic $\gamma(t;u)$.
    
    We define \textbf{the coordinates of the Riemannian curvature tensor at $\gamma(t;u)$, wrt $E_i(t;u)$}:
    \begin{alignat*}{1}
        \RR^i_{jkl}(t;u) := \lin{R(E_j(t;u), E_k(t;u))E_l(t;u), E_i(t;u)}
        \elb{n:d:RR(t)}
    \end{alignat*}
    
    Finally, we define
    \begin{alignat*}{1}
        \MM_{i,j}(t;u) := -\sum_{k,l} \RR^i_{jkl}(t;u) \uu_k \uu_l
        \elb{n:d:M(t)} 
    \end{alignat*}
    We verify that for all $t\in[0,1]$, $\lrn{\MM(t;u)}_2 \leq L_R \lrn{\uu}_2^2$. 

    Given any $u,v\in T_x M$, let $J(t;u,v)$ denote the unique Jacobi field along $\gamma(t;u)$, satisfying
    \begin{alignat*}{1}
        & J(0;u,v) = 0 \qquad D_t J(0;u,v) = v
        \elb{n:d:J(t)}
    \end{alignat*}
    To simplify notation later on, let us also define
    \begin{alignat*}{1}
        K(t;u,v) := D_t J(t;u,v)
    \end{alignat*}
    
    Next, we define a few quantities whose meaning will become clear later on. First, for $u,v\in T_x M$, let $p(t;u,v)$ denote the vector field along $\gamma(t;u)$, given by 
    \begin{alignat*}{1}
        p(t;u,v) 
        :=& - \lrp{\nabla_J R}\lrp{J, \gamma'} \gamma' - \lrp{\nabla_{\gamma'} R}(J, \gamma') J - 2 R\lrp{J, \gamma'} K 
        \numberthis \label{n:d:p(t)}
    \end{alignat*}
    where $\gamma:= \gamma(t;u,v), J := J(t;u,v), K := K(t;u,v)$. We also let $\pp(t;u,v) : [0,1] \to \Re^d$ as the coordinates of $p(t;u,v)$ with respect to $E(t;u)$, i.e. $\pp_i(t;u,v) := \lin{p(t;u,v) , E_i(t;u)}$.

    Let $\emat(t;\MM)$ denote the solution to the matrix ODE, as defined in Lemma \ref{l:formal-matrix-exponent}. For $u\in T_x M$, define
    \begin{alignat*}{1}
        \GG(u) = \bmat{I_{d\times d} & 0} \emat\lrp{1; \bmat{0 & I \\ \MM(\cdot;u) & 0}} \bmat{0 \\ I_{d\times d}}
        \elb{n:d:GG}
    \end{alignat*}
    where $\MM(t;u)$ is defined in \eqref{n:d:M(t)}. We also define, for $u\in T_x M$, $\BB(t;u)$ and $\bar{\BB}(t;u)$ as
    \begin{alignat*}{1}
        & \BB(t;u) := \bmat{I_{d\times d} & 0} \emat\lrp{t ;\bmat{0 & I_{d\times d} \\ \MM(\cdot;u) & 0}} \bmat{0 \\ I_{d\times d}}\\
        & \bar{\BB}(t;u) := \bmat{I_{d\times d} & 0} \emat\lrp{1-r;\bmat{0 & I \\ \NN_r(\cdot;u) & 0}} \bmat{0 \\ I_{d\times d}}
        \elb{n:d:BB}
    \end{alignat*}
    where $\NN_s(t;u)  := \MM(s+t;u)$. We will now define an important quantity $\FF(u,v)$: for any $u\in T_x M$ such that $\BB(1;u)$ is invertible, and for any $v\in T_x M$,
    \begin{alignat*}{1}
        \FF(u,v):= - \BB(1;u)^{-1} \lrp{\int_0^1 \bar{\BB}(r;u) {\pp(r;u,v)} dr}
        \numberthis \label{n:d:FF}
    \end{alignat*}

    We define, for $u,v\in T_x M$, $\aa(t;u,v)$ as the solution to the following second-order ODE:
    \begin{alignat*}{1}
        &\aa(0;u,v) = \uu \qquad \aa'(0;u,v) = \vv\\
        & \frac{d}{ds} \aa(s;u,v) = \aa'(s;u,v)\\
        & \frac{d}{ds} \aa'(s;u,v) = \FF(a(s;u,v), a'(s;u,v))
        \numberthis \label{n:d:w(s)}
    \end{alignat*}
    where $\FF$ is as defined in \eqref{n:d:FF}.

    Finally, we will be using the constant $\C_r$ defined in \eqref{d:c_r} throughout this section. We reproduce it below for ease of reference:
    \begin{alignat*}{1}
        & \C_r := \frac{1}{16} \min\lrbb{{L_R'}^{-1/3}, \frac{1}{8 \sqrt{L_R}}}
    \end{alignat*}

    \subsection{A Tangent Space Curve for the $\Exp^{-1}$ of a Geodesic}
    \label{n:s:main_aa_result}
    
    We now state the main result of Section \ref{s:trivialization}:
    \begin{lemma}\label{n:l:w(s)}
        Let $x \in M$, let $u,v \in T_x M$. Assume that $\lrn{u} \leq \C_r$ and $\lrn{u} \leq \C_r$. Let $x' := \Exp_x (u)$.
        
        Let $\aa(s;u,v)$ be as defined in \eqref{n:d:w(s)} and $\GG(u)$ be as defined in \eqref{n:d:GG}. Let $a(s;u,v) := \aa(s;u,v) \circ E$ and $G(v;u)$ denote the tensor given by $G(v;u) = \lrp{\GG(u) \vv} \circ E$.
        \begin{alignat*}{1}
            \Exp_{x}\lrp{a(1;u,v)} = \Exp_{x'} \lrp{\party{x}{x'} G(u;v)}
        \end{alignat*}
        where $\party{x}{x'}$ denotes parallel transport along $\Exp_x (t u)$.
    \end{lemma}
    Intuitively, Lemma \ref{n:l:w(s)} is related to the following question: if $a(s;u,v) : s \to T_x M$ is a curve in the tangent space of $x$, with initial condition $a(0) = u$, $a'(0)=v$, how should $a(s;u,v)$ curve (i.e. what is $a''(s;u,v)$), such that $\mu(s) := \Exp_x(a(s;u,v))$ is a geodesic? We will see that $a''(s;u,v) := \FF(a(s;u,v),a'(s;u,v))$ is exactly the curvature that makes $\mu(s)$ a geodesic. We also verify that $\mu'(0)$ is given by $\party{x}{x'} G(u;v)$.

    \begin{proof}
        To simplify notation, we drop the explicit dependence of $a$ on $u,v$, i.e. for the rest of this proof, let $a(s) := a(s;u,v)$. 
        
        We first verify an uniform bound $\lrn{a(s)} \leq \frac{1}{2 \sqrt{L_R}}$ for all $s$. This follows from Lemma \ref{n:l:initial_regularity_suffices}, which guarantees that under our assumed bounds on $\lrn{u}$ and $\lrn{v}$, for all $s$,
        \begin{alignat*}{1}
            \lrn{{a}(s)}_2 \leq 2\lrn{{a}'(0)}_2 \leq 4 \C_r \leq \frac{1}{2\sqrt{L_R}}.
            \elb{n:e:t:lflsadm}
        \end{alignat*}
        We can thus apply Lemma \ref{n:l:inverse_map} pointwise for $s\in[0,1]$, which guarantees the existence of a solutioin to \eqref{n:d:w(s)}.

        In order to verify that $\mu(s) := \Exp_{x}\lrp{a(s;u,v)}$ is a geodesic, it suffices to verify that \\
        $D_s \del_s \Exp_{x}\lrp{a(s;u,v)} = 0$ for all $s$. Define $\Lambda(s,t) := \Exp_{x}\lrp{t \cdot a(s)}$, so that $\mu(s) = \Lambda(s,1)$. In Lemma \ref{n:l:meaning_of_gg}, we show that there exists a function $g$ (explicitly defined in \eqref{n:d:g(t)_and_h(t)}) such that $D_s \del_s \Lambda(s,1) = g(1;a(s),a'(s),a''(s))$. In Lemma \ref{n:l:inverse_map}, we verify that $g(1;u,v,\FF(u,v) \circ E) = 0$ for all $u,v$ satisfying the assumed norm bounds (which we verify in \eqref{n:e:t:lflsadm}). Combining these two lemmas, we see that the definition of $\aa(s)$ with $\aa''(s) = \FF(a'(s),a''(s))$ satisfies, for all $s$,
        \begin{alignat*}{1}
            D_s \Lambda(s,1) = 0
        \end{alignat*}
        This shows that $\mu(s) = \Exp_x(a(s;u,v))$ is a geodesic. By definition, $\mu(0) = \Exp_x(a(0;u,v)) = \Exp_x(u) = x'$.

        We will now verify the direction of the geodesic, given by $\del_s \Lambda(0,1)$, i.e. the Jacobi field at $s=0,t=1$. By Lemma \ref{n:l:meaning_of_JK}, this Jacobi field is given by $J(1;a(0),a'(0)) = \JJ(1;a(0),a'(0)) \circ E(1;a(0))$, with $\JJ$ defined in \eqref{n:d:JJ(t)_and_KK(t)}. From point 1. of Lemma \ref{l:formal-matrix-exponent} and from the definition of $\JJ$ and $\KK$ in \eqref{n:d:JJ(t)_and_KK(t)},
        \begin{alignat*}{1}
            \cvec{\JJ(1;u,v)}{\KK(1;u,v)} = \emat\lrp{1; \bmat{0 & I \\ \MM(\cdot;u) & 0}} \cvec{0}{\vv}
        \end{alignat*}
        It follows from algebra that $\JJ(t;u,v) = \bmat{I & 0} \emat\lrp{1; \bmat{0 & I \\ \MM(\cdot;u) & 0}} \cvec{0}{I} \vv =: \GG(u) \vv$. Therefore, $D_s \Lambda(0,1) = \GG(a(0)) \aa'(0) \circ E(1;a(0)) = \party{x}{x'} \lrp{\GG(a(0)) \aa'(0) \circ E}$

        The conclusion follows by observing that $a(0) = u, a'(0)=v$ by definition in \eqref{n:d:w(s)}.
    \end{proof}

    \subsection{The Jacobi Field along $\gamma(t;u)$}
    Let $v \in T_x M$ and let $\vv$ be the coordinates of $v$ wrt $E$, i.e. $v = \vv \circ E$. Let $\JJ(t;u,v) : [0,1] \to \Re^d$ and $\KK(t;u,v) : [0,1] \to \Re^d$ be the solution to the following second-order ODE:
    \begin{alignat*}{1}
        & \JJ(0;u,v) = 0\\
        & \KK(0;u,v) = \vv \\
        &\frac{d}{dt} \JJ(t;u,v) = \KK(t;u,v)\\
        &\frac{d}{dt} \KK(t;u,v) = \MM(t;u) \JJ(t;u,v)
        \numberthis \label{n:d:JJ(t)_and_KK(t)}
    \end{alignat*}
    For existence and uniqueness of $\JJ$ and $\KK$, see Theorem 4.31 of \cite{lee2018introduction}. In Lemma \ref{n:l:meaning_of_JK} below, we see that $\JJ$ and $\KK$ as defined in \eqref{n:d:JJ(t)_and_KK(t)} are the coordinates of a Jacobi field deflined along $\gamma(t;u)$. 

    \begin{lemma}\label{n:l:meaning_of_JK}
        For any $u,v\in T_x M$, let $\gamma(t;u), E(t;u)$ be as defined in Section \ref{sss:proof_of_triangle_distortion_G:setup}. Let $J(t;u,v)$ denote the unique Jacobi field along $\gamma(t;u)$ with initial condition $J(0;u,v) = u$ and let $D_t J(0;u,v) = v$. Let $K(t;u,v) := D_t J(t;u,v)$. Let $\JJ(t;u,v)$ and $\KK(t;u,v)$ be as defined in \eqref{n:d:JJ(t)_and_KK(t)}. Then $\JJ(t;u,v)$ and $\KK(t;u,v)$ are the coordinates, wrt $E(t;u)$, of $J(t;u,v)$ and $K(t;u,v)$ respectively. I.e. for $i = 1...d$, $J(t;u,v) = \JJ(t;u,v) \circ E(t;u)$ and $K(t;u,v) = \KK(t;u,v) \circ E(t;u)$
    \end{lemma}
    The proof is identical to that of Proposition 10.2 of \cite{lee2018introduction}, and we omit it. The main idea is to notice that the ODE in \eqref{n:d:JJ(t)_and_KK(t)} is exactly the Jacobi Equation (written in coordinates wrt $E_i(t;u)$).

    The following lemma, which bounds the difference between $\GG(u)$ and $I$, will come in useful later:
    \begin{lemma}
        \label{n:l:GG_norm_bound}
        Let $\GG(u)$ be as defined in \eqref{n:d:GG}. Then
        \begin{alignat*}{1}
            \lrn{\GG(u) - I}_2 \leq \frac{1}{6} L_R \lrn{u}_2^2 e^{\sqrt{L_R}\lrn{u}_2}
        \end{alignat*}
    \end{lemma}
    \begin{proof}
        Let 
        \begin{alignat*}{1}
            \bmat{\AA& \BB\\ \CC & \DD} := \emat\lrp{1; \bmat{0 & I \\ \MM(\cdot;u) & 0}}
        \end{alignat*}
        so that $\GG(u) = \BB$ . From, Lemma \ref{l:matrix-exponent-block-bounds} and Lemma \ref{l:sinh_bounds}
        \begin{alignat*}{1}
            \lrn{\BB - I}_2 \leq \frac{1}{\sqrt{L_{\MM}}} \sinh(\sqrt{L_{\MM}}) - 1 \leq \frac{1}{6} L_{\MM} e^{\sqrt{L_{\MM}}}
        \end{alignat*}
        where $L_{\MM}$ is an upper bound on $\MM(t;u)$, for all $t$. By definition of $\MM$ in \eqref{n:d:M(t)}, we can take $L_{\MM} = L_R \lrn{u}^2$. 
    \end{proof}

    \subsection{Second Variation ODE}
    Let $\MM(t;u)$ be as defined in \eqref{n:d:M(t)}. Let $p(t;u,v)$ be as defined in \eqref{n:d:p(t)}, and let $\pp(t;u,v)$ be its coordinates wrt $E(t;u)$.
    
    For $t\in[0,1]$, let $\gg(t;u,v,w)$ and $\hh(t;u,v,w)$ denote the solution to the following ODE:
    \begin{alignat*}{1}
        &\gg(0;u,v,w) = 0\\
        &\hh(0;u,v,w) = \ww\\
        &\ddt \gg(t;u,v,w) = \hh(t)\\
        &\ddt \hh(t;u,v,w) = \pp(t;u,v) + \MM(t;u) \gg(t;u,v,w)
        \numberthis \label{n:d:g(t)_and_h(t)}
    \end{alignat*}
    where $\ww$ denotes the coordinates of $w$ wrt $E$, i.e. $\ww_i = \lin{w,E_i}$ and $\MM(t;u)$ is as defined in \eqref{n:d:M(t)}.
    
    \begin{lemma}\label{n:l:meaning_of_gg}
        Let $\aa(s) : [0,1] \to \Re^d$ be any twice differentiable curve. Let $a(s;u,v) := \aa(s;u,v) \circ E$. Let $\Lambda(s,t) = \Exp_x\lrp{t\cdot a(s)}$. Let $\gg$ and $\hh$ be as defined in \eqref{n:d:g(t)_and_h(t)}, let $g(t;\a(s),\a'(s),\a''(s)):= \gg(t;a(s),a'(s),a''(s))\circ E(t;a(s)) $ and $h(t;\a(s),\a'(s),\a''(s)):= \hh(t;a(s),a'(s),a''(s)) \circ E(t;a(s))$. \\
        Then
        \begin{alignat*}{1}
            &  D_s \del_s \Lambda(s,t) = g(t;\a(s),\a'(s),\a''(s))\\
            &  D_t D_s \del_s \Lambda(s,t) = h(t;\a(s),\a'(s),\a''(s))
        \end{alignat*} 
    \end{lemma}
    \begin{proof}[Proof of Lemma \ref{n:l:meaning_of_gg}]
    In the rest of this proof, we will consider a fixed but arbitrary $s\in[0,1]$. Let $E_i(t;a(s))$ denote the orthonormal frame along $\Lambda(s,t) = \gamma(t;a(s))$.

    To simplify notation, we drop the explicit dependence on $s$, $t$, $\aa(s)$ and $\aa'(s)$ when unambiguous, and we use $\del_s$ to denote $\del_s \Lambda$ and $\del_t$ to denote $\del_t \Lambda$ when unambiguous.
    \begin{alignat*}{1}
        & D_t D_t D_s \del_s \Lambda(s,t)\\
        =& D_t D_s D_t \del_s - D_t \lrp{R(\del_s, \del_t) \del_s}\\
        =& D_s D_t D_t \del_s - R(\del_s, \del_t)  \lrp{D_t \del_s} - D_t \lrp{R(\del_s, \del_t) \del_s}\\
        =& - D_s \lrp{R\lrp{\del_s, \del_t} \del_t} - R(\del_s, \del_t)  \lrp{D_t \del_s} - D_t \lrp{R(\del_s, \del_t) \del_s}\\
        =& - \lrp{D_s R}\lrp{\del_s, \del_t} \del_t - R\lrp{D_s \del_s, \del_t} \del_t - R\lrp{\del_s, D_t \del_s} \del_t - R\lrp{\del_s, \del_t} D_t \del_s\\
        &\quad - R\lrp{\del_s, \del_t} D_t \del_s\\
        &\quad - \lrp{D_t R}(\del_s, \del_t) \del_s - R(D_t \del_s, \del_t) \del_s - R(\del_s, \del_t) D_t \del_s\\
        =& - R(D_s \del_s,\del_t) \del_t \\
        & - \lrp{D_s R}\lrp{\del_s, \del_t} \del_t - \lrp{D_t R}(\del_s, \del_t) \del_s \\
        &  - R\lrp{\del_s, D_t \del_s} \del_t - 3 R\lrp{\del_s, \del_t} D_t \del_s - R(D_t \del_s, \del_t) \del_s\\
        =& - R(D_s \del_s,\del_t) \del_t \\
        & - \lrp{D_s R}\lrp{\del_s, \del_t} \del_t - \lrp{D_t R}(\del_s, \del_t) \del_s  - 2 R\lrp{\del_s, \del_t} D_t \del_s
        \numberthis \label{n:e:Ds_dels}
    \end{alignat*}

    where we used multiple times the equation $D_s D_t V - D_t D_s V = R\lrp{\del_s \Lambda, \del_t \Lambda} V$ (see proposition 7.5 of \cite{lee2018introduction}), and the fact that $D_t \del_s \Lambda = D_s \del_t \Lambda$. The last line uses the identity $R(a,b)c + R(b,c)a + R(c,a)b = 0$ for all $a,b,c$.

    Recall that $\del_t = \del_t \Lambda(s,t) = \gamma'(t;a(s))$ and $\del_s = \del_s \Lambda(s,t) = J(t;a(s),a'(s))$, so that the last line, $- \lrp{D_s R}\lrp{\del_s, \del_t} \del_t - \lrp{D_t R}(\del_s, \del_t) \del_s  - 4 R\lrp{\del_s, \del_t} D_t \del_s$ is exactly equal to $p(t;a(s),a'(s))$ as defined in \eqref{n:d:p(t)}. 

    If we let $g(t) := D_s \del_s \Lambda(s,t)$ and $h(t) := D_t g(t) = D_t D_s \del_s \Lambda(s,t)$, then \eqref{n:e:Ds_dels} can be written as a second-order ODE
    \begin{alignat*}{1}
        & g(0) = 0\\
        & h(0) = a''(0)\\
        & D_t g(t) = h(t)\\
        & D_t h(t) = - R(g(t), \del_t \Lambda(s,t) ) \del_t \Lambda(s,t) + p(t;a(s), a'(s))
    \end{alignat*}
    The initial condition $g(0) = 0$ is because $\Lambda(s,0) = x$ for all $s$, and thus $\at{D_s \del_s \Lambda(s,t)}{t=0} = 0$.The initial condition $h(0) = a''(s)$ is because,from the definition of $\Lambda$, $\at{\del_t \Lambda(s,t)}{t=0} = a(s)$ and $\at{D_s D_s \del_t \Lambda(s,t)}{t=0} = a''(s)$, thus $\at{D_t D_s \del_s \Lambda(s,t)}{t=0} = a''(s) + \at{R(\del_s \Lambda(s,t), \del_t \Lambda(s,t))\del_s \Lambda(s,t)}{t=0} = a''(s)$ since $\at{\del_s \Lambda(s,t)}{t=0} = 0$.

    Letting $\gg(t)$ and $\hh(t)$ denote the coordinates of $g(t)$ and $h(t)$ wrt $E(t;a(s))$, we verify via the definition of $\MM$ in \eqref{n:d:M(t)}  and $\pp$ in \eqref{n:d:p(t)}, that
    \begin{alignat*}{1}
        & \gg(0) = 0\\
        & \hh(0) = \aa''(0)\\
        & \ddt \gg(t) = \hh(t)\\
        & \ddt \hh(t) = \pp(t;a(s),a'(s)) + \MM(t;a(s)) \gg(t)
    \end{alignat*}
    where we use the fact that $\gg_i(t) = \lin{g(t), E_i(t;a(s))}$, and $\del_t \gg_i(t) = \lin{D_t g(t), E_i(t;a(s))}$ as $D_t E_i(t;a(s)) = 0$. We see that the above is exactly \eqref{n:d:g(t)_and_h(t)} with $u = a(s), v= a'(s), w= a''(s)$.        
    \end{proof}

    \begin{lemma}\label{n:l:inverse_map}
        Let $\FF$ be as defined in \eqref{n:d:FF}. Let $\gg$ be as defined in \eqref{n:d:g(t)_and_h(t)}. For any $u\in T_x M$ satisfying $L_R \lrn{u}_2^2 \leq \frac{1}{4}$ and for any $v\in T_x M$, $\FF(u,v)$ is well defined and
        \begin{alignat*}{1}
            \gg(1;u,v,\FF(u,v) \circ E) = 0
        \end{alignat*} 
        
    \end{lemma}
    \begin{proof}[Proof of Lemma \ref{n:l:inverse_map}]
        We apply the second claim from Lemma \ref{l:formal-matrix-exponent} with $\zz(t) = \cvec{\gg(t;u,v,w)}{\hh(t;u,v,w)}$ and  $\MM(t) := \bmat{0 & I \\ \MM(t;u) & 0}$ and $\vv(t) = \cvec{0}{\pp(t;u,v)}$, which gives
        \begin{alignat*}{1}
            \cvec{\gg(t;u,v,w)}{\hh(t;u,v,w)} = \int_0^t \emat\lrp{t-r;\bmat{0 & I \\ \NN_r(\cdot;u) & 0}} \cvec{0}{\pp(r;u,v)} dr + \emat\lrp{t;\bmat{0 & I \\ \MM(\cdot;u) & 0}} \cvec{0}{\ww}
        \end{alignat*}
        where $\ww$ is the coordinates of $w$ wrt $E$, and for any $r$, $\NN_r(t;u) := \MM(r + t;u)$.

        Define the four blocks
        \begin{alignat*}{1}
            \bmat{\AA & \BB \\ \CC & \DD} := \emat\lrp{1;\bmat{0 & I \\ \MM(\cdot;u) & 0}}
        \end{alignat*}
        and the four blocks
        \begin{alignat*}{1}
            \bmat{\bar{\AA}(r) & \bar{\BB}(r) \\ \bar{\CC}(r) & \bar{\DD}(r)} := \emat\lrp{1-r;\bmat{0 & I \\ \NN_r(\cdot;u) & 0}}
        \end{alignat*}
        Then by algebra,
        \begin{alignat*}{1}
            \gg(1;u,v,w) = \int_0^1 \bar{\BB}(r) {\pp(r;u,v)} dr  + \BB \ww
        \end{alignat*}

        Recall the definition of $\MM(t;u)$ from \eqref{n:d:M(t)}. Under our assumption that $L_R \lrn{u}^2 \leq \frac{1}{4}$, we verify that for all $t$, $\lrn{\MM(t;u)}_2 \leq L_R \lrn{u}^2 \leq \frac{1}{4}$ and for all $r,t$, $\lrn{\NN_r(t;u)}_2 \leq L_R \lrn{u}^2 \leq \frac{1}{4}$. From Lemma \ref{l:matrix-exponent-block-bounds} and Lemma \ref{l:sinh_bounds}, $\lrn{\BB- I}_2 \leq \frac{1}{\sqrt{L_{\MM}}} \sinh\lrp{\sqrt{L_{\MM}}} - 1 \leq \frac{1}{6} L_{\MM} e^{\sqrt{L_{\MM}}} \leq \frac{1}{3} L_{\MM} \leq \frac{1}{12}$, thus $\BB$ is invertible, so that $\FF$ as defined in \eqref{n:d:FF} is well defined.
        
        Let $\FF(u,v)$ be as defined in \eqref{n:d:FF}, and notice that $\BB$ is the same as $\BB(1;u)$ from \eqref{n:d:BB}, and $\bar{\BB}(t)$ is the same as $\bar{\BB}(t;u)$ from \eqref{n:d:BB}, so that
        \begin{alignat*}{1}
            \FF(u,v):= - \BB^{-1} \lrp{\int_0^1 \bar{\BB}(r) {\pp(r;u,v)} dr}
        \end{alignat*}
        Plugging into the expression for $\gg(1;u,v,w)$ above, we verify that
        \begin{alignat*}{1}
            \gg(1;u,v,\FF(u,v) \circ E) = 0.
        \end{alignat*}
        This proves the lemma.
    \end{proof}

    \begin{lemma}
        \label{n:l:FF_norm_bound}
        Let $\FF$ be as defined in \eqref{n:d:FF}. For any $u\in T_x M$ satisfying $L_R \lrn{u}_2^2 \leq \frac{1}{4}$ and for any $v\in T_x M$, 
        \begin{alignat*}{1}
            \lrn{\FF(u,v)}_2 \leq 16\lrp{L_R' \lrn{u}^2 + L_R \lrn{u}} \lrn{v}^2
        \end{alignat*}
    \end{lemma}
    \begin{proof}
        The rest of the proof will be devoted to bounding the norm of $\lrn{\FF(u,v)}_2$. Using the fact that  $\lrn{\NN_r(t;u)}_2 \leq L_R \lrn{\uu}_2^2 \leq \frac{1}{4}$ and applying Lemma \ref{l:matrix-exponent-block-bounds}, we guarantee that for all $r$, $\lrn{\bar{\BB}(r) - I}_2 \leq \frac{1}{6} L_{\MM} e^{\sqrt{L_{\MM}}}$. We can thus bound $\lrn{\BB^{-1}}_2 \leq \frac{1}{1 - \frac{1}{6}L_\MM e^{\sqrt{L_\MM}}}$ and $\lrn{\bar{\BB}}_2 \leq 1 + \frac{1}{6}L_\MM e^{\sqrt{L_\MM}}$. Thus
        \begin{alignat*}{1}
            \lrn{\FF(u,v)}_2
            =& \lrn{\BB^{-1} \lrp{\int_0^1 \bar{\BB}(r) {\pp(r;u,v)} dr}}_2 \\
            \leq& \frac{1 + \frac{1}{6}L_\MM e^{\sqrt{L_\MM}}}{1 - \frac{1}{6}L_\MM e^{\sqrt{L_\MM}}} \int_0^1\lrn{\pp(t;u,v)}_2 dt
        \end{alignat*}

        Again using $L_\MM = L_R \lrn{\uu}_2^2 \leq \frac{1}{4}$, we can bound $\frac{1 + \frac{1}{6}L_\MM e^{\sqrt{L_\MM}}}{1 - \frac{1}{6}L_\MM e^{\sqrt{L_\MM}}} \leq 1 + L_\MM \leq 2$. It remains to bound $\lrn{\pp(t;u,v)}$ for all $t$. From the definition of $\pp$ in \eqref{n:d:p(t)}, 
        \begin{alignat*}{1}
            & \lrn{\pp(t;u,v)}_2\\
            =& \lrn{p(t;u,v)}\\
            \leq& \lrn{\lrp{\nabla_J R}\lrp{J, \gamma'} \gamma'} + \lrn{\lrp{\nabla_{\gamma'} R}(J, \gamma') J} + \lrn{2 R\lrp{J, \gamma'} K }\\
            \leq& 2L_R' \lrn{J}^2 \lrn{\gamma'}^2 + 2 L_R \lrn{J}\lrn{\gamma'}\lrn{K}\\
            =:& 2L_R' \lrn{J(t;u,v)}^2 \lrn{\gamma'(t;u)}^2 + 2 L_R \lrn{J(t;u,v)}\lrn{\gamma'(t;u)}\lrn{K(t;u,v)}
        \end{alignat*} 
        By definition, $\gamma'(t;u) = u$. By Lemma \ref{l:jacobi_field_norm_bound}, we can bound for all $t\in [0,1]$,
        \begin{alignat*}{1}
            \lrn{J(t;u,v)}
            \leq& \cosh\lrp{{\C}} \lrn{J (0;u,v)} + \frac{\sinh\lrp{{\C}}}{{\C}} \lrn{{K (0;u,v)}} 
            = \frac{\sinh\lrp{{\C}}}{{\C}} \lrn{v}\\
            \lrn{K(t;u,v)}
            \leq& {\C}  \sinh\lrp{{\C}} \lrn{J (0;u,v)} + \cosh\lrp{{\C}} \lrn{K (0;u,v)}
            = \cosh\lrp{{\C}} \lrn{v}
        \end{alignat*}
        where $\C = \sqrt{L_R} \lrn{u}$ by our assumption in the lemma statement. We use the fact that $J(0;u,v)=0$ and $K(0;u,v) = v$ by definition in \eqref{n:d:JJ(t)_and_KK(t)}. By Lemma \ref{l:sinh_bounds} and our assumption on $u$, we can bound $\cosh(\C) \leq 1 + L_R \lrn{u}^2 e^{\sqrt{L_R} \lrn{u}} \leq 2$ and $\sinh(\C)/\C \leq e^{\sqrt{L_R} \lrn{u}} \leq 2$, so that
        \begin{alignat*}{1}
            \lrn{\pp(t;u,v)}_2 \leq 8\lrp{L_R' \lrn{u}^2 + L_R \lrn{u}} \lrn{v}^2
        \end{alignat*}
        Plugging into the earlier bound on $\lrn{\FF(u,v)}_2$, we get 
        \begin{alignat*}{1}
            \lrn{\FF(u,v)}_2 \leq 16\lrp{L_R' \lrn{u}^2 + L_R \lrn{u}} \lrn{v}^2
        \end{alignat*}
    \end{proof}

    \begin{lemma}\label{n:l:initial_regularity_suffices}
        Let $\aa(s;u,v)$ be as defined in \eqref{n:d:w(s)}. Assume that $\lrn{u} \leq \C_r$ and $\lrn{v} \leq \C_r$, where $\C_r$ is the constant defined in \eqref{d:c_r}. Then  for all $s\in[0,1]$,
        \begin{alignat*}{1}
            & \lrn{{\aa}(s;u,v)}_2 \leq \lrn{\aa(0;u,v)} + 2\lrn{{\aa}'(0;u,v)}_2 \leq 4\C_r\\
            & \lrn{{\aa}'(s;u,v)}_2 \leq 2 s\lrn{{\aa}'(0;u,v)}_2 \leq 2\C_r
        \end{alignat*}
    \end{lemma}
    \begin{proof}
        To reduce notation we drop the explicit dependence on $u$ and $v$, i.e. $a(s) := a(s;u,v)$.

        Let us define a "capped" version of \eqref{n:d:w(s)}:
        \begin{alignat*}{1}
            &\bar{\aa}(0) = \aa(0) = \uu \qquad \bar{\aa}'(0) = \aa'(0) = \vv\\
            & \frac{d}{ds} \bar{\aa}(s) = \bar{\aa}'(s)\\
            & \frac{d}{ds} \bar{\aa}'(s) = \FF(\Pi_{4\C_r}\lrp{\bar{\aa}(s)}, \Pi_{4\C_r}\lrp{\bar{\aa}'(s)}, 0)
        \end{alignat*}
        where $\Pi_{4\C_r}(\cc) := \twocase{\cc}{\lrn{\cc}_2 \leq {4\C_r}}{\frac{{4\C_r}}{\lrn{\cc}_2} \cdot \cc}{\lrn{\cc}_2 \geq {4\C_r}}$. Thus for any $\cc$, $L_R \lrn{\Pi_{4\C_r} (\cc)}_2^2 \leq L_R \lrp{4\C_r}^2 \leq \frac{1}{4}$, so that $L_R\lrn{\Pi_{4\C_r}\lrp{\bar{\aa}(s)}}_2^2 \leq \frac{1}{4}$ for all $s$. We can then apply Lemma \ref{n:l:FF_norm_bound}, which guarantees that for all $s$, 
        \begin{alignat*}{1}
            \lrn{\FF(\Pi_{4\C_r} \lrp{\bar{\aa}(s)}, \Pi_{4\C_r}\lrp{\bar{\aa}'(s)})}_2 
            \leq& 16 \lrp{L_R' \lrn{\Pi_{4\C_r}\lrp{\bar{\aa}(s)}}_2^2 + L_R \lrn{\Pi_{4\C_r}\lrp{\bar{\aa}'(s)}}_2} \lrn{\Pi_{4\C_r}\lrp{\bar{\aa}'(s)}}_2^2\\
            \leq& 16 \lrp{L_R' \C_r^3 + L_R \C_r^2} \lrn{\Pi_{4\C_r}\lrp{\bar{\aa}'(s)}}_2\\
            \leq& \lrn{\Pi_{4\C_r}\lrp{\bar{\aa}'(s)}}_2^2 \leq \frac{1}{4} \lrn{\bar{\aa}'(s)}_2
        \end{alignat*}
        where the second inequality follows by definition of $\Pi$, the third inequality follows from \eqref{d:c_r}.
        
        Therefore,
        \begin{alignat*}{1}
            \frac{d}{ds} \lrn{\bar{\aa}'(s)}_2 \leq \frac{1}{4} \lrn{\bar{\aa}'(s)}_2
        \end{alignat*}
        By Gronwall's Lemma, $\lrn{\bar{\aa}'(s)}_2 \leq e^{1/4}\lrn{\bar{\aa}'(0)}_2$ for all $s \in[0,1]$.

        It then follows that
        \begin{alignat*}{1}
            \frac{d}{ds} \lrn{\bar{\aa}(s)}_2 \leq \lrn{\bar{\aa}'(s)}_2 \leq e^{1/4} \lrn{\bar{\aa}'(0)}_2
        \end{alignat*}
        Thus $\lrn{\bar{\aa}'(s)}_2 \leq \lrn{\bar{\aa}(0)}_2 + 2 s\lrn{\bar{\aa}'(0)}_2 \leq 4\C_r$ for all $s$.

        This implies that for all $s\in[0,1]$, $\Pi_\alpha(\bar{\aa}(s)) = \bar{\aa}(s)$ and $\Pi_\alpha(\bar{\aa}'(s)) = \bar{\aa}'(s)$ and \\
        $\FF(\Pi_{4\C_r}\lrp{\bar{\aa}(s)}, \Pi_{4\C_r}\lrp{\bar{\aa}'(s)}, 0) =\FF({\bar{\aa}(s)}, {\bar{\aa}'(s)}, 0)$. Thus $\cvec{\bar{\aa}(s)}{\bar{\aa}'(s)}$ has identical initialization and dynamics compared to $\cvec{\aa(s)}{\aa'(s)}$. This in turn implies that $\cvec{\bar{\aa}(s)}{\bar{\aa}'(s)} = \cvec{\aa(s)}{\aa'(s)}$ for all $s\in[0,1]$.
    \end{proof}

    \section{Approximation Bounds}
    The following Lemma is taken from Lemma 3 of \cite{sun2019escaping}, and is used at multiple places in our proof:
    \begin{lemma}
        Consider any $x\in M$, $u,v \in T_x M$. Define $x':= \Exp_{x}(u)$. Let $\party{x}{x'}$ denote parallel transport along the geodesic $\Exp_{x}(u)$ (in case the geodesic is not unique). Then
        \begin{alignat*}{1}
            \dist\lrp{\Exp_{x}(u+v), \Exp_{x'}(\party{x}{x'}v)}
            \leq& L_R \lrn{u}\lrn{v}\lrp{\lrn{u} + \lrn{v}} e^{\sqrt{L_R} \lrp{\lrn{u}+\lrn{v}}} 
        \end{alignat*}
        \label{n:l:triangle_distortion}
    \end{lemma}

    \subsection{First order approximation of $\aa(s;u,v)$}
    We first show in Lemma \ref{n:l:one-step-retraction-bound-ww} that $\aa(1;u,v)$ can be well approximated by $\aa(0;u,v) + \aa'(0;u,v)$. Using this bound on $\lrn{\aa(1;u,v) - \aa(0;u,v) - \aa'(0;u,v)}$, we can in turn bound \\ 
    $\dist\lrp{\Exp_{x}(u+v), \Exp_{x}(a(1;u,v)}$

    \begin{lemma}\label{n:l:one-step-retraction-bound-ww}
        Let $x \in M$, let $u,v \in T_x M$, assume that $\lrn{u} \leq \C_r$ and $\lrn{u} \leq \C_r$ ($\C_r$ is defined in \eqref{d:c_r}). Let $\aa$ be as defined in \eqref{n:d:w(s)}, then
        \begin{alignat*}{1}
            \lrn{\aa(1;u,v) - \uu - \vv}_2 \leq  2^{10}\lrp{L_R' \lrp{\lrn{\uu}_2 + \lrn{\vv}_2}^2 + L_R \lrp{\lrn{\uu}_2 + \lrn{\vv}_2}}\lrn{\vv}_2^2
        \end{alignat*}
    \end{lemma}

    \begin{proof}
        To reduce notation we drop the explicit dependence on $u$ and $v$, i.e. $a(s) := a(s;u,v)$.
        
        From Lemma \ref{n:l:initial_regularity_suffices}, given our bassumed bounds on $\lrn{u}$ and $\lrn{v}$, we can uniformly bound $L_R \lrn{\aa(s)}_2^2 \leq \frac{1}{4}$ for all $s$. We can thus apply Lemma \ref{n:l:FF_norm_bound} to bound, for all $s$,
        \begin{alignat*}{1}
            \lrn{\aa''(s)}_2 
            =& \lrn{\FF(\aa(s),\aa'(s))}_2 
            \leq 16 \lrp{L_R' \lrn{\aa(s)}_2^2 + L_R \lrn{\aa(s)}_2} \lrn{\aa'(s)}_2^2
        \end{alignat*}
        Applying second order Taylor expansion,
        \begin{alignat*}{1}
            \lrn{\aa(1) - \uu - \vv}_2
            =& \lrn{\aa(1) - \aa(0) - \aa'(0)}_2\\
            \leq& \int_0^1 \lrn{\aa'(s) - \aa'(0)}_2 ds\\
            \leq& \int_0^1 \lrn{\aa''(s)}_2 ds\\
            \leq& \int_0^1 16 \lrp{L_R' \lrn{\aa(s)}_2^2 + L_R \lrn{\aa(s)}_2} \lrn{\aa'(s)}_2^2 ds\\
            \leq& 2^{10} \lrp{L_R' \lrp{\lrn{\aa(0)}_2 + \lrn{\aa'(0)}_2}^2 + L_R \lrp{\lrn{\aa(0)}_2 + \lrn{\aa'(0)}_2}}\lrn{\aa'(0)}_2^2
        \end{alignat*}
        where the fourth inequality is by Lemma \ref{n:l:initial_regularity_suffices}.

        Finally, it follows by definition in \eqref{n:d:w(s)} that $\lrn{\aa(0)}_2 = \lrn{\uu}_2$ and $\lrn{\aa'(0)}_2 = \lrn{\vv}_2$. Therefore,
        \begin{alignat*}{1}
            \lrn{\aa(1) - \uu - \vv}_2 \leq 2^{10}\lrp{L_R' \lrp{\lrn{\uu}_2 + \lrn{\vv}_2}^2 + L_R \lrp{\lrn{\uu}_2 + \lrn{\vv}_2}}\lrn{\vv}_2^2
        \end{alignat*}
    \end{proof}

    \begin{lemma}
        \label{l:triangle_distortion_G}
        Let $x \in M$, let $u,v \in T_x M$. Assume that $\lrn{u} \leq \C_r$ and $\lrn{u} \leq \C_r$. Let $x' := \Exp_x (u)$. Let $\GG(u)$ be as defined in \eqref{n:d:GG}, and let $G(v;u) = \lrp{\GG(u) \vv} \circ E$. Let Let $\party{x}{x'}$ denote parallel transport along the geodesic $\Exp_{x}(u)$. Then
        \begin{alignat*}{1}
            \dist\lrp{\Exp_{x}(u+v), \Exp_{x'}(\party{x}{x'}G(v))}
            \leq& 2^{12}\lrp{L_R' \lrp{\lrn{u} + \lrn{v}}^2 + L_R \lrp{\lrn{u} + \lrn{v}}}\lrn{v}^2
        \end{alignat*}
        and
        \begin{alignat*}{1}
            \lrn{G - Id} \leq \frac{1}{3} L_R \lrn{u}^2
        \end{alignat*}
    \end{lemma}

    \begin{proof}
        Let $\aa(s;u,v)$ be as defined in \eqref{n:d:w(s)}. To simplify notation, for the rest of the proof we drop the explicit dependence on $u,v$ and let $a(s):= a(s;u,v)$.
        
        Let $\GG(u)$ be as defined in \eqref{n:d:GG}. Let $G(v;u) := \lrp{\GG(u) \vv} \circ E$. By Lemma \ref{n:l:w(s)},
        \begin{alignat*}{1}
            \Exp_x(a(1;u,v)) = \Exp_{x'}\lrp{\party{x}{x'} G(v;u)}.
        \end{alignat*}
        By Lemma \ref{n:l:one-step-retraction-bound-ww}, $\lrn{\aa(1;u,v) - \uu - \vv} \leq 2^{11}\lrp{L_R' \lrp{\lrn{u} + \lrn{v}}^2 + L_R \lrp{\lrn{u} + \lrn{v}}}\lrn{v}^2$. Together with Lemma \ref{l:jacobi_field_divergence},
        \begin{alignat*}{1}
            \dist\lrp{\Exp_{x}(u + v),\Exp_x(a(1;u,v))} 
            \leq& \frac{\sinh\lrp{\sqrt{L_R}\lrp{\lrn{a(1;u,v)}+\lrn{u+v}}}}{\sqrt{L_R}\lrp{\lrn{a(1;u,v)}+\lrn{u+v}}} \lrn{a(1;u,v) - u - v}\\
            \leq& e^{\sqrt{L_R}\lrp{\lrn{a(1;u,v)}+\lrn{u+v}}} \lrn{a(1;u,v) - u - v}
        \end{alignat*}
        where the second inequality uses Lemma \ref{l:sinh_bounds}. By our assumption that $\lrn{u} \leq \C_r$, $\lrn{v} \leq \C_r$, we can use Lemma \ref{n:l:initial_regularity_suffices} to bound $\lrn{a(1;u,v)} \leq 4\C_r$, so that altogether, $e^{\sqrt{L_R}\lrp{\lrn{a(1;u,v)}+\lrn{u+v}}} \leq 2$. Combining all ove the above,
        \begin{alignat*}{1}
            \dist\lrp{\Exp_{x}(u + v), \Exp_{x'}\lrp{\party{x}{x'} G(v;u)}}
            =&  \dist\lrp{\Exp_{x}(u + v),\Exp_x(a(1;u,v))} \\
            \leq& 2^{12}\lrp{L_R' \lrp{\lrn{u} + \lrn{v}}^2 + L_R \lrp{\lrn{u} + \lrn{v}}}\lrn{v}^2
        \end{alignat*}
    
        Finally, by Lemma \ref{n:l:GG_norm_bound}, $\lrn{\GG(u) - I}_2 \leq \frac{1}{6} L_R \lrn{u}^2 e^{\sqrt{L_R}\lrn{u}_2} \leq \frac{1}{3} L_R \lrn{u}^2$. This concludes the proof.
        
      \end{proof}


\section{Matrix ODE}\label{s:matrix_exponent}

In this section we provide Gronwall-style inequality for matrix ODE. This section does not make explicit use of Riemannian geometry but it provides a natural way to study ODEs on manifold.

\begin{lemma}[Formal Matrix Exponent]\label{l:formal-matrix-exponent}
    Given $\MM(t): \Re^+ \to \Re^{d\times d}$, define $\emat\lrp{t;\MM}: \Re^+ \to \Re^{d\times d}$ as the solution to the matrix ODE
    \begin{alignat*}{1}
        \emat\lrp{0;\MM} =& I\\
        \ddt \emat\lrp{t;\MM} =& \MM(t) \emat\lrp{t;\MM}
    \end{alignat*}

    Then
    \begin{enumerate}
        \item Let $\xx(t)$ be the solution to the ODE $\ddt \xx(t) = \MM(t) \xx(t)$, for some $\MM$, then
        \begin{alignat*}{1}
            \xx(t) = \emat\lrp{t;\MM} \xx(0)
        \end{alignat*}
        \item Let $\zz(t)$ be the solution to $\ddt \zz(t) = \MM(t) \zz(t) + \vv(t)$, for some $\MM$, $\vv$, then
        \begin{alignat*}{1}
            \zz(T) = \int_0^T \emat\lrp{T-s;\NN_s} \vv(s) ds + \emat\lrp{T;\MM} \zz(0)
        \end{alignat*}
        where for any $s,t$, $\NN_s(t) := \MM(s+t)$.
    \end{enumerate}
\end{lemma}
\begin{proof}[Proof of Lemma \ref{l:formal-matrix-exponent}]
    Let $\yy_t := \emat\lrp{t;\MM} \xx(0)$. We verify that
    \begin{alignat*}{1}
        \yy(0) =& 0\\
        \ddt \yy(t) =& \lrp{\ddt \emat\lrp{t;\MM}} \xx(0) = \MM(t) \yy(t)
    \end{alignat*}
    Given the same dynamics and initial conditions, we conclude that $\xx(t) = \yy(t)$ for all $t$.

    To verify the second claim, note that 
    \begin{alignat*}{1}
        & \ddt  \int_0^t \emat\lrp{t-s;\NN_s} \vv(s) ds\\
        =& \emat\lrp{0;\NN_t} \vv(s) + \int_0^t \lrp{\frac{d}{dt} \emat\lrp{t-s;\NN_s}} \vv(s) ds\\
        =& \vv(s) + \int_0^t \NN_s(t-s) \emat\lrp{t-s;\NN_s} \vv(s) dt\\
        =& \vv(s) + \MM(t) \int_0^t \emat\lrp{t-s;\NN_s} \vv(s) ds
    \end{alignat*}

    Additionally, $\ddt  \emat\lrp{t;\MM} \zz(0) = \MM(t)  \emat\lrp{t;\MM} \zz(0)$, summing,
    \begin{alignat*}{1}
        \ddt \zz(t)
        =& \ddt \int_0^t \emat\lrp{t-s;\NN_s} \vv(s) ds + \emat\lrp{t;\MM} \zz(0) \\
        =& \vv(s) + \MM(t) \int_0^t \emat\lrp{t-s;\NN_s} \vv(s) ds + \MM(t) \emat\lrp{t;\MM} \zz(0)\\
        =& \vv(s) + \MM(t) \zz(t)
    \end{alignat*}
\end{proof}

\begin{lemma}\label{l:matrix-exponent-block-bounds}
    Let $\emat$ be as defined in Lemma \ref{l:formal-matrix-exponent}. Let 
    \begin{alignat*}{1}
        \bmat{\AA(t) & \BB(t) \\ \CC(t) & \DD(t)} := \emat\lrp{t; \bmat{0 & I \\ \MM(t) & 0}}
    \end{alignat*}
    for some $\MM(t)$. Assume $\lrn{\MM(t)}_2 \leq L_{\MM}$ for all $t$. Then for all $t$,
    \begin{alignat*}{1}
        & \lrn{\AA(t)}_2 \leq \cosh\lrp{\sqrt{L_{\MM}} t}\\
        & \lrn{\BB(t)}_2 \leq \frac{1}{\sqrt{L_{\MM}}} \sinh\lrp{\sqrt{L_{\MM}} t}\\
        & \lrn{\CC(t)}_2 \leq \sqrt{L_{\MM}} \sinh\lrp{\sqrt{L_{\MM}} t}\\
        & \lrn{\DD(t)}_2 \leq \cosh\lrp{\sqrt{L_{\MM}}t}
    \end{alignat*}

    and
    \begin{alignat*}{1}
        & \lrn{\AA(t) - I}_2 \leq \cosh(\sqrt{L_{\MM}} t) - 1\\
        & \lrn{\BB(t) - tI}_2 \leq \frac{1}{\sqrt{L_{\MM}}} \sinh\lrp{\sqrt{L_{\MM}} t} - t\\
        & \lrn{\DD(t) - I}_2 \leq \cosh\lrp{\sqrt{L_{\MM}} t} - 1
    \end{alignat*}

    \begin{alignat*}{1}
        \lrn{\AA(t) - I}_2 \leq& \frac{1}{2} L_{\MM}e^{L_{\MM}}\\
        \lrn{\BB(t) - tI}_2 \leq& \frac{1}{6} L_{\MM} e^{L{\MM}}\\
        \lrn{\DD(t) - I}_2 \leq& \frac{1}{2} L_{\MM}e^{L_{\MM}}\\
        \lrn{\CC(t) }_2 \leq& L_{\MM}e^{L_{\MM}}\\
    \end{alignat*}
\end{lemma}
\begin{proof}[Proof of Lemma \ref{l:matrix-exponent-block-bounds}]
    We first verify the first part of the lemma. Consider the ODE given by
    \begin{alignat*}{1}
        \ddt \cvec{\xx}{\yy}(t) = \bmat{0 & I \\ \MM(t) & 0} \cvec{\xx(t)}{\yy(t)}
    \end{alignat*}
    By Lemma \ref{l:formal-matrix-exponent}, $\bmat{\AA(t) & \BB(t) \\ \CC(t) & \DD(t)}$ satisfies
    \begin{alignat*}{1}
        \cvec{\xx(t)}{\yy(t)} = \bmat{\AA(t) & \BB(t) \\ \CC(t) & \DD(t)} \cvec{\xx(0)}{\yy(0)}
    \end{alignat*}
    By Cauchy Schwarz,
    \begin{alignat*}{1}
        & \frac{d}{dt} \lrn{\xx(t)}_2 \leq \lrn{\yy(t)}_2\\
        & \frac{d}{dt} \lrn{\yy(t)}_2 \leq L_{\MM}\lrn{\xx(t)}_2
    \end{alignat*}
    We apply Lemma \ref{l:sinh_ode}, with $a_t := \lrn{\xx(t)}_2$ and $b_t := \lrn{\yy(t)}_2$, $C := L_{\MM}$. Then
    \begin{alignat*}{1}
        & \lrn{\xx_t}_2 \leq {\lrn{\xx_0}_2} \cosh\lrp{\sqrt{L_{\MM}} t} + \frac{\lrn{\yy_0}_2}{\sqrt{L_{\MM}}} \sinh\lrp{\sqrt{L_{\MM}} t}\\
        & \lrn{\yy_t}_2 \leq \sqrt{L_{\MM}} \lrp{{\lrn{\xx_0}_2} \sinh\lrp{\sqrt{L_{\MM}} t} + \frac{\lrn{\yy_0}_2}{\sqrt{L_{\MM}}} \cosh\lrp{\sqrt{L_{\MM}} t}}
    \end{alignat*}
    This immediately implies that
    \begin{alignat*}{1}
        & \lrn{\AA(t)}_2 \leq \cosh\lrp{\sqrt{L_{\MM}} t}\\
        & \lrn{\BB(t)}_2 \leq \frac{1}{\sqrt{L_{\MM}}} \sinh\lrp{\sqrt{L_{\MM}} t}\\
        & \lrn{\CC(t)}_2 \leq \sqrt{L_{\MM}} \sinh\lrp{\sqrt{L_{\MM}} t}\\
        & \lrn{\DD(t)}_2 \leq \cosh\lrp{\sqrt{L_{\MM}}t}
    \end{alignat*}
    This proves the first claim of the Lemma.

    We now prove the second claim. We verify that 
    \begin{alignat*}{1}
        \ddt \cvec{\xx(t) - \xx(0) - t \yy(0)}{\yy(t) - \yy(0)} = \cvec{\yy(t) - \yy(0)}{\MM(t) \xx(t)}
    \end{alignat*}

    \begin{alignat*}{1}
        \ddt \cvec{\xx(t) - \xx(0) - t \yy(0)}{\yy(t) - \yy(0)} = \cvec{\yy(t) - \yy(0)}{\MM(t) \xx(t)}
        =& \cvec{\yy(t) - \yy(0)}{\MM(t) \xx(t)}\\
        =& \cvec{\yy(t) - \yy(0)}{\MM(t) \lrp{\xx(t) - \xx(0) - t\yy(0)}} + \cvec{0}{\MM(t) \lrp{\xx(0) + t\yy(0)}}
    \end{alignat*}
    Thus
    \begin{alignat*}{1}
        \ddt \lrn{\xx(t) - \xx(0) - t \yy(0)}_2 
        \leq& \lrn{\yy(t) - \yy(0)}_2\\
        \ddt \lrn{\yy(t) - \yy(0)}_2 
        \leq& L_{\MM} \lrn{\xx(t) - \xx(0) - t \yy(0)}_2 + L_{\MM} \lrp{\lrn{\xx(0)}_2 + t \lrn{\yy(0)}_2}
    \end{alignat*}

    We verify that
    \begin{alignat*}{1}
        \lrn{\yy(t) - \yy(0)}_2 \leq L_{\MM} \int_0^t \lrn{\xx(s) - s\yy(0)}_2 ds + \frac{t^2}{2} L_{\MM}
    \end{alignat*}
    Let us apply Lemma \ref{l:sinh_ode_with_offset} with $a_t = \lrn{\xx_t - \xx(0) - t\yy(0)}_2$, $b_t = \lrn{\yy(t) - \yy(0)}_2$, $C = L_{\MM}$, $D = L_{\MM}\lrn{\yy(0)}_2$ and $E = L_{\MM}\lrn{\xx(0)}_2$
    \begin{alignat*}{1}
        \lrn{\xx_t - \xx(0) - t\yy(0)}_2 
        \leq& \frac{L_{\MM}\lrn{\xx(0)}_2}{L_{\MM}} \cosh(\sqrt{L_{\MM}} t) + \frac{L_{\MM}\lrn{\yy(0)}_2}{L_{\MM}^{3/2}} \sinh\lrp{\sqrt{L_{\MM}} t}\\
        &\qquad  - \frac{L_{\MM}\lrn{\yy(0)}_2}{L_{\MM}}t  - \frac{L_{\MM}\lrn{\xx(0)}_2}{L_{\MM}}\\
        =& \lrn{\xx(0)}_2 \lrp{\cosh(\sqrt{L_{\MM}} t) - 1}+ \lrn{\yy(0)}_2 \lrp{\frac{1}{\sqrt{L_{\MM}}} \sinh\lrp{\sqrt{L_{\MM}} t} - t}\\
        \lrn{\yy(t) - \yy(0)}_2 
        \leq& \frac{L_{\MM}\lrn{\xx(0)}_2}{\sqrt{L_{\MM}}} \sinh(\sqrt{L_{\MM}} t) + \frac{L_{\MM}\lrn{\yy(0)}_2}{L_{\MM}} \cosh\lrp{\sqrt{L_{\MM}} t} - \frac{L_{\MM}\lrn{\yy(0)}_2}{L_{\MM}}\\
        =& \lrn{\xx(0)}_2 \sqrt{L_{\MM}}  \sinh(\sqrt{L_{\MM}} t) + \lrn{\yy(0)}_2 \lrp{\cosh\lrp{\sqrt{L_{\MM}} t} - 1}
        \elb{e:t:ksnlqdk}
    \end{alignat*}

    Again from Lemma \ref{l:formal-matrix-exponent}, we know that
    \begin{alignat*}{1}
        \cvec{\xx(t)}{\yy(t)} = \bmat{\AA(t) & \BB(t) \\ \CC(t) & \DD(t)} \cvec{\xx(0)}{\yy(0)}
    \end{alignat*}
    thus
    \begin{alignat*}{1}
        \cvec{\xx(t) - \xx(0) - t\yy(0)}{\yy(t) - \yy(0)} = \bmat{\AA(t) - I & \BB(t) - t I \\ \CC(t) & \DD(t) - I} \cvec{\xx(0)}{\yy(0)}
    \end{alignat*}
    combined with \eqref{e:t:ksnlqdk}, and using the fact that the above hold for all $\yy(0)$ and $\xx(0)$, we can bound
    \begin{alignat*}{1}
        & \lrn{\AA(t) - I}_2 \leq \cosh(\sqrt{L_{\MM}} t) - 1\\
        & \lrn{\BB(t) - tI}_2 \leq \frac{1}{\sqrt{L_{\MM}}} \sinh\lrp{\sqrt{L_{\MM}} t} - t\\
        & \lrn{\DD(t) - I}_2 \leq \cosh\lrp{\sqrt{L_{\MM}} t} - 1
    \end{alignat*}

    Finally, to prove the third claim, 
    \begin{alignat*}{1}
        & \frac{d}{dt} \cvec{\xx(t) - t \yy(0) - \frac{t^2}{2} \MM(0) \xx(0)}{\yy(t) - \yy(0) - t \MM(0) \xx(0)} \\
        =& \bmat{0 & I \\ \MM(t) & 0} \cvec{\xx(t)}{\yy(t)} - \bmat{0 & I \\ \MM(0) & 0} \cvec{\yy(0)}{\vv}\\
    \end{alignat*}

\end{proof}

\begin{lemma}\label{l:matrix-exponent-block-bounds-second-order}
    Let $\emat$ be as defined in Lemma \ref{l:formal-matrix-exponent}. Let 
    \begin{alignat*}{1}
        \bmat{\AA(t) & \BB(t) \\ \CC(t) & \DD(t)} := \emat\lrp{t; \bmat{0 & I \\ \MM(t) & 0}}
    \end{alignat*}
    for some $\MM(t)$. Assume $\lrn{\MM(t)}_2 \leq L_{\MM}$ for all $t$. Then for all $t$,
    \begin{alignat*}{1}
            \lrn{\CC(t) - t \MM(0)}_2 \leq \frac{\lrp{L_\MM'+\frac{1}{2} L_\MM^2}}{\sqrt{L_\MM}} \sinh(\sqrt{L_\MM} t)
    \end{alignat*}
\end{lemma}
\begin{proof}
    The proof is similar to Lemma \ref{l:matrix-exponent-block-bounds}. Consider the ODE given by
    \begin{alignat*}{1}
        \ddt \cvec{\xx}{\yy}(t) = \bmat{0 & I \\ \MM(t) & 0} \cvec{\xx(t)}{\yy(t)}
    \end{alignat*}
    with initial condition $\yy(0) = 0$.

    By Lemma \ref{l:formal-matrix-exponent}, $\bmat{\AA(t) & \BB(t) \\ \CC(t) & \DD(t)}$ satisfies
    \begin{alignat*}{1}
        \cvec{\xx(t)}{\yy(t)} = \bmat{\AA(t) & \BB(t) \\ \CC(t) & \DD(t)} \cvec{\xx(0)}{\yy(0)}
    \end{alignat*}
    \begin{alignat*}{1}
        & \ddt \cvec{\xx(t) - \xx(0) - \frac{t^2}{2} \MM(0) \xx(0)}{\yy(t) - t\MM(0) \xx(0)}\\
        =& \cvec{y(t) - t\MM(0) \xx(0)}{\MM(t) \xx(t) - \MM(0) \xx(0)}\\
        =& \cvec{y(t) - t\MM(0) \xx(0)}{\MM(t) \xx(t) - \MM(0) \xx(0)}\\
        =& \cvec{\yy(t)- t \MM(0) \xx(0)}{\MM(t) \lrp{\xx(t) - \xx(0) - \frac{t^2}{2} \MM(0)}} + \cvec{0}{\lrp{\MM(t) - \MM(0)} \xx(0)} + \cvec{0}{\frac{t^2}{2}\MM(t) \MM(0) \xx(0)}
    \end{alignat*}

    By Cauchy Schwarz, for all $t\leq 1$,
    \begin{alignat*}{1}
        & \frac{d}{dt} \lrn{\xx(t) - \xx(0) - \frac{t^2}{2} \MM(0) \xx(0)}_2 \leq \lrn{\yy(t)- t \MM(0) \xx(0)}_2\\
        & \frac{d}{dt} \lrn{\yy(t) - t\MM(0) \xx(0)}_2 \leq L_\MM \lrn{\xx(t) - \xx(0) - \frac{t^2}{2} \MM(0) \xx(0)}_2 + \lrp{L_\MM'+\frac{1}{2} L_\MM^2}\lrn{\xx(0)}_2
    \end{alignat*}

    Apply Lemma \ref{l:sinh_ode_with_offset} with $a_t = \lrn{\xx(t) - \xx(0) - \frac{t^2}{2} \MM(0) \xx(0)}_2$, $b_t = \lrn{\yy(t) - t\MM(0) \xx(0)}_2$, $C = L_{\MM}$, $D = 0$ and $E = \lrp{L_\MM'+\frac{1}{2} L_\MM^2}\lrn{\xx(0)}_2$ to get
    \begin{alignat*}{1}
        & a_t \leq \frac{\lrp{L_\MM'+\frac{1}{2} L_\MM^2}\lrn{\xx(0)}_2}{L_\MM} \lrp{\cosh(\sqrt{L_\MM} t) - 1}\\
        & b_t \leq \frac{\lrp{L_\MM'+\frac{1}{2} L_\MM^2}\lrn{\xx(0)}_2}{\sqrt{L_\MM}} \sinh(\sqrt{L_\MM} t)
    \end{alignat*}

    Finally, recall that
    \begin{alignat*}{1}
        \yy(t) - t \MM(0) \xx(0) = \lrp{\CC(t) - t \MM(0)} \xx(0)
    \end{alignat*}
    Since we have shown that $\lrn{\yy(t) - t \MM(0) \xx(0)}_2 \leq \frac{\lrp{L_\MM'+\frac{1}{2} L_\MM^2}\lrn{\xx(0)}_2}{\sqrt{L_\MM}} \sinh(\sqrt{L_\MM} t)$ for all $\xx(0)$, it follows that
    \begin{alignat*}{1}
        \lrn{\CC(t) - t \MM(0)}_2 \leq \frac{\lrp{L_\MM'+\frac{1}{2} L_\MM^2}}{\sqrt{L_\MM}} \sinh(\sqrt{L_\MM} t)
    \end{alignat*}
\end{proof}

\section{Main Theorem and Proofs for SGD Generalization}
\label{s:euclidean_embedding}

\subsection{Additional Notation and Assumptions}
\label{ss:euclidean_setup}
In this section, let us $A: \Re^d\to \Sym^{d\times d}$ be a matrix-valued function, and let $g(x):=A^{-1}(x)$. Let us consider the Riemannian manifold $\lrp{\Re^d, g}$, and let $\dist\lrp{x,y}$ be the Riemannian distance for this manifold.

\begin{assumption}
    \label{ass:g_A_regularity}
    Exist strictly positive constants $\lambda_A, L_A$, such that for all $x,y\in \Re^d$, 
    \begin{alignat*}{1}
        & \lambda_A I \prec A(x) \prec L_A I\\
        & \tr\lrp{A(x) - A(y)}^2 \leq {L_A'}^2 \lrn{x-y}_2^2
    \end{alignat*}
\end{assumption}
\begin{assumption}
    \label{ass:g_A_Lipschitz}
    Exist strictly positive constant $L_A'$, such that for all $x,y\in \Re^d$, 
    \begin{alignat*}{1}
        & \tr\lrp{A(x) - A(y)}^2 \leq {L_A'}^2 \lrn{x-y}_2^2
    \end{alignat*}
\end{assumption}
Note that under Assumption \ref{ass:g_A_Lipschitz}, $\lambda_A \lrn{x-y}_2 \leq \dist\lrp{x,y} \leq L_A \lrn{x-y}_2$. It is useful to note that combining Assumption \ref{ass:g_A_Lipschitz} with Lemma 1 from \cite{eldan2020clt} gives the following bound:
\begin{alignat*}{1}
    \tr\lrp{A^{1/2}(x) - A^{1/2}(y)}^2 \leq \frac{{L_A'}^2}{\lambda_A} \lrn{x-y}_2^2\\
    \elb{e:sqrt_A_regularity}
\end{alignat*}

We use $\lin{\cdot,\cdot}$ to denote the inner product wrt $g$, i.e. given $u,v \in T_x M$, $\lin{u,v} = \sum_{i,j=1}^d u_i v_j g_{i,j}(x)$. We use $\linb{u,v}$ to denote the \emph{Euclidean dot product}, i.e. $\linb{u,v} = \sum_{i=1}^d u_i v_i$. It follows that $\lin{u,v} = \linb{u, g(x) v}$. We will use $\nabla$ to denote the Levi Civita connection. We use the bold version $\nablab$ to denote the \emph{Euclidean derivative} of a function $h: \Re^d \to \Re$, i.e.
\begin{alignat*}{1}
    \frac{d}{dt} f(x + t v) = \linb{\nablab f(x), v}
\end{alignat*}
We will use $e_i$ to denote the basis vectors, e.g. $e_1 = [0,1,0,0...0]$. We will also use $\partial_i f(x) := \lin{\nablab f(x), e_i}$. Similarly, let $\nabla^2 f$ denote the Hessian tensor wrt the Riemannian metric, and $\nablab^2 f$ denote the Hessian wrt Euclidean metric.

We let $\Gamma$ denote the Christoffel symbols of $g$. Specifically,for $i,j,k\in\lrbb{1...d}$,
\begin{alignat*}{1}
    \Gamma^k_{ij}(x) := \frac{1}{2} \sum_{l=1}^d \lrp{\lrp{\lrp{g(x)}^{-1}}_{kl} \lrp{\del_i g_{jl}(x) + \del_j g_{il}(x) - \del_l g_{ij}(x)}}
    \elb{d:christoffel_symbol}
\end{alignat*}
It will also be convenient to define $\Gamma^k(x)$ as the matrix whose $(i,j)^{th}$ entry is $\Gamma^k_{i,j}$. With slight overloading of notation, we define $\phi(x): \Re^d \to \Re^d$ and $\phi(x,v): \Re^d \times \Re^d \to \Re^d$ as follows:
\begin{alignat*}{1}
    & \phi_k(x):= \tr\lrp{\lrp{g(x)}^{-1} \Gamma^k(x)}\\
    & \phi_k(x,v):= v^T \Gamma^k(x) v
    \elb{d:phi_for_christoffel_symbol}
\end{alignat*}

\begin{assumption}
    \label{ass:phi_Gamma_regularity}
    Exist $L_\phi$, $L_{\phi}'$, such that for all $x,y,v\in \Re^d$,
    \begin{alignat*}{1}
        & \lrn{\phi(x,v)}_2 \leq L_{\phi} \lrn{v}_2^2\\
        & \lrn{\phi(x,v) - \phi(y,v)}_2 \leq L_{\phi}' \lrn{x-y}_2 \lrn{v}_2^2\\
        & \lrn{\phi(x,v) - \phi(x,w)}_2 \leq L_{\phi}\lrn{v-w}_2 \lrp{\lrn{v}+\lrn{w}}
    \end{alignat*}
\end{assumption}
\begin{remark}
    The form of assumption \ref{ass:phi_Gamma_regularity} really arises from its definition in \eqref{d:phi_for_christoffel_symbol}. $L_\phi$ and $L_{\phi}'$ can be bounded, for example, if $\Gamma^i_{jk}$ is coordinate-wise bounded and smooth.
\end{remark}

We note a useful characterization of geodesics in terms of coordinates below (see, e.g. Eqn 4.16 of \cite{lee2018introduction}). A curve $x(t) : \Re \to \Re^d$ is a geodesic if and only if, for all $k\in \lrbb{1...d}$,
\begin{alignat*}{1}
    & x''_k(t) = - \sum_{i,j=1}^d \lrp{x'_i(t) x'_j(t) \Gamma^k_{ij} \lrp{x(t)}} = - \phi(x(t),x'(t))\\
    \Leftrightarrow \qquad 
    & x''(t) = \phi(x(t),x'(t))
    \elb{e:geodesic_equation_coordinate_christoffel}
\end{alignat*}

\begin{assumption}
    \label{ass:euclidean_xi_assumption}
    We assume that the random vector field satisfies, for all $x\in \Re^d$,
    \begin{alignat*}{1}
        \E{\xi(x)} = 0 \qquad \E{\xi(x) \xi(x)^T} = A(x)
    \end{alignat*}
\end{assumption}

\begin{assumption}
    \label{ass:beta_lipschitz_euclidean}
    Assume that for all $x,y$, with probability $1$,
    \begin{alignat*}{1}
        & \lrn{\beta(x) - \beta(y)}_2 \leq L_{\beta,2}' \lrn{x-y}_2\\
        & \lrn{\xi(x) - \xi(y)}_2 \leq L_{\xi,2}' \lrn{x-y}_2\\
        & \lrn{\xi(x)}_2 \leq L_{\xi,2}
    \end{alignat*}
\end{assumption}

\subsection{Euclidean Walk to Manifold Walk}
    \label{ss:euclidean_walk_to_manifold_walk}
\begin{lemma}[One Step Error On Euclidean Space]
    \label{l:euclidean_step_to_manifold_step}
    Let $\beta$ and $\xi$ satisfy Assumption \ref{ass:beta_lipschitz_euclidean} and Assumption \ref{ass:euclidean_xi_assumption}. Assume \ref{ass:phi_Gamma_regularity}, \ref{ass:g_A_regularity} and \ref{ass:g_A_Lipschitz}. Let $x\in \Re^d$ be arbitrary, and define
    \begin{alignat*}{1}
        & x' = \Exp_{x} \lrp{\delta \beta(x)+ \sqrt{\delta} \xi(x) + \frac{1}{2} \phi(x) }\\
        & \t{x}' = x + \delta \beta(x) + \sqrt{\delta} \xi(x)
    \end{alignat*}
    Then
    \begin{alignat*}{2}
        & \lrn{x' - \t{x}'}_2^2
        &&\leq O\lrp{\delta^6 \lrn{\beta(x)}_2^6 + \delta^2}\\
        &\lrn{x' - \t{x}' + \frac{1}{2} \phi(x,\sqrt{\delta} \xi(x)) - \frac{\delta}{2}\phi(x)}_2
        &&\leq O\lrp{\delta^3 \lrn{\beta(x)}_2^3 + \delta^{3/2}}
    \end{alignat*}
    where $O\lrp{}$ hides polynomial dependence on $L_A, \frac{1}{\lambda_A}, L_\phi, L_{\phi}',d,L_\xi$.
\end{lemma}

\begin{proof}[Proof of Lemma \ref{l:euclidean_step_to_manifold_step}]
    Let us define the geodesic $y(t) = \Exp_{x} \lrp{t\lrp{\delta \beta(x) + \delta \phi(x) + \sqrt{\delta} \xi(x)}}$ for $t\in[0,1]$. From \eqref{e:geodesic_equation_coordinate_christoffel}, we verify that $y(t)$ satisfies the following second order ODE:
    \begin{alignat*}{1}
        & y(0) = x\\
        & v(0) = \delta \beta(x) + \sqrt{\delta} \xi(x) + \frac{\delta}{2} \phi(x)\\
        & d y(t) = v(t) dt\\
        & d v(t) = - \phi(y(t),v(t))
    \end{alignat*}
    Thus $y(1) = x'$. We will now bound the distance between $\t{x}'$ and $y(1)$.
    \begin{alignat*}{1}
        & \lrn{y(1) - \t{x}'}_2^2\\
        =& \lrn{\int_0^1 v(t) - \delta \beta(x) - \sqrt{\delta} \xi(x) dt}_2^2\\
        =& \lrn{\int_0^1 v(t) - v(0) + \frac{\delta}{2} \phi\lrp{x} dt}_2^2\\
        =& \lrn{\int_0^1 \int_0^t -\phi\lrp{y(s),v(s)} + \delta \phi\lrp{x} ds dt}_2^2\\
        \leq& \int_0^1 \int_0^t \lrn{-\phi\lrp{y(s),v(s)} + \delta \phi\lrp{x}}_2^2 ds dt
        \elb{e:t:woiemffk:1}\\
        \leq& 8 \int_0^1 \int_0^t \lrn{\phi\lrp{y(s),v(s)} - \phi\lrp{y(0),v(s)} }_2^2 + \lrn{\phi\lrp{y(0),v(s)} - \phi\lrp{y(0),v(0)}}_2^2 \\
        &\quad + \lrn{\phi\lrp{y(0),v(0)} - \delta \phi\lrp{x} }_2^2 ds dt
    \end{alignat*}
    We bound each of the terms above. It will first be convenient to establish a uniform bound on $\lrn{v(t)}$, $\lrn{v(t) - v(0)}$ and $\lrn{y(s) - y(0)}$ for all $t$. Notice by definition of the geodesic that $\lrn{v(t)} = \lrn{v(0)}$ (note: this is $g$-norm) for all $t$, equivalently, $v(t)^T g(y(t)) v(t) = g(0)^T g(y(0)) v(0)$ for all $t$. Let us denote, by $L_v := \sup_{t\in[0,1]} \lrn{v(t)}_2$
    \begin{alignat*}{1}
        & L_v \leq \frac{L_A}{\lambda_A} \lrn{v(0)}_2 \leq \frac{L_A}{\lambda_A} \cdot \lrp{\lrn{\delta\beta(x)}_2 + \sqrt{\delta} L_\xi}\\
        & \lrn{v(t) - v(0)}_2 \leq L_\phi L_v^2 \\
        & \lrn{y(s) - y(0)}_2 \leq L_v
    \end{alignat*}
    By Assumption \ref{ass:phi_Gamma_regularity},
    \begin{alignat*}{1}
        & \lrn{\phi\lrp{y(s),v(s)} - \phi\lrp{y(0),v(s)} }_2
        \leq L_\phi' \lrn{y(s) - y(0)}_2 \lrn{v(s)}_2^2 \leq L_\phi' L_v^3\\
        & \lrn{\phi\lrp{y(0),v(s)} - \phi\lrp{y(0),v(0)}}_2
        \leq L_\phi^2 \lrp{\lrn{v(s)}_2 + \lrn{v(0)}_2} \lrn{v(s) - v(0)}_2 \leq 2 L_\phi L_v^3 \\
        & \lrn{\phi\lrp{y(0),v(0)} - \delta \phi\lrp{x} }_2
        \leq L_{\phi} \lrp{\delta d L_A + \delta^2 \lrn{\beta(x)}^2 + \delta L_\xi^2}
    \end{alignat*}

    Plugging and simplifying, 
    \begin{alignat*}{1}
        \lrn{y(1) - \t{x}'}_2^2
        \leq& O\lrp{\delta^6 \lrn{\beta(x)}_2^6 + \delta^4 \lrn{\beta(x)}_2^4 + \delta^2}\\
        =& O\lrp{\delta^6 \lrn{\beta(x)}_2^6 + \delta^2}
    \end{alignat*}
    where $O\lrp{}$ hides polynomial dependence on $L_A, \frac{1}{\lambda_A}, L_\phi, L_{\phi}',d,L_\xi$.

    We can also bound
    \begin{alignat*}{1}
        & \lrn{\phi\lrp{y(0),v(0)} - \phi\lrp{x,\sqrt{\delta}\xi(x)} }_2
        \leq L_\phi \lrp{L_v + \sqrt{\delta} \lrn{\xi(x)}_2} \lrp{\lrn{\delta \beta(x)}_2 + \delta L_\phi}
    \end{alignat*}

    Next, following the same steps leading up to \eqref{e:t:woiemffk:1},
    \begin{alignat*}{1}
        &\lrn{y(1) - \t{x}' + \frac{1}{2} \phi(y(0),\sqrt{\delta} \xi(x)) - \frac{\delta}{2}\phi(x)}_2\\
        \leq& \int_0^1 \int_0^t \lrn{-\phi\lrp{y(s),v(s)} + \delta \phi\lrp{x}  + \phi(y(0),\sqrt{\delta} \xi(x)) - \delta \phi(x)}_2 ds dt\\
        \leq& \int_0^1 \int_0^t \lrn{\phi\lrp{y(s),v(s)} - \phi\lrp{y(0),v(s)} }_2 + \lrn{\phi\lrp{y(0),v(s)} - \phi\lrp{y(0),v(0)}}_2 \\
        &\quad + \lrn{\phi\lrp{y(0),v(0)} - \phi\lrp{y(0),\sqrt{\delta} \xi(x) }}_2 ds dt\\
        \leq& O\lrp{\delta^3 \lrn{\beta(x)}_2^3 + \delta^2 \lrn{\beta(x)}_2^2 + \delta^{3/2} + \delta^{3/2} \lrn{\beta(x)}_2 + \delta^{3/2}}\\
        =& O\lrp{\delta^3 \lrn{\beta(x)}_2^3 + \delta^{3/2}}
    \end{alignat*}
    where $O\lrp{}$ hides polynomial dependence on $L_A, \frac{1}{\lambda_A}, L_\phi, L_{\phi}',d,L_\xi$.

\end{proof}

\begin{lemma}
    \label{l:euclidean_walk_to_manifold_walk}
    Assume $\beta$ and $\xi$ satisfy Assumption \ref{ass:beta_lipschitz_euclidean} and Assumption \ref{ass:euclidean_xi_assumption}. Assume \ref{ass:phi_Gamma_regularity}, \ref{ass:g_A_regularity} and \ref{ass:g_A_Lipschitz}.
    \begin{alignat*}{1}
        & x_{k+1} = \Exp_{x_k} \lrp{\delta \beta(x_k)+ \sqrt{\delta} \xi_k(x_k) + \frac{\delta}{2} \phi(x_k) }\\
        & z_{k+1} = x_k + \delta \beta(z_k) + \sqrt{\delta} \xi_k(z_k)
    \end{alignat*}
    with $x_0 = z_0$. Assume that $K\delta \leq \min\lrbb{\frac{1}{L_{\beta,2}'}, \frac{1}{L_{\xi,2}'}}$, then
    \begin{alignat*}{1}
        \E{\lrn{x_{K} - z_{K}}_2^2} 
        \leq& O\lrp{K\delta^6 \lrp{\sum_{k=0}^{K-1}\E{\lrn{\beta(x_k)}_2^6}}+ K\delta^2}
    \end{alignat*}
    where $O\lrp{}$ hides polynomial dependence on $L_A, \frac{1}{\lambda_A}, L_\phi, L_{\phi}',d,L_{\xi,2}$. and where expectation is wrt the randomness in $\xi_0...\xi_{K-1}$.
\end{lemma}
\begin{proof}
    Consider a fixed $k$. Let us define the geodesic $y(t) = \Exp_{x_k} \lrp{t\lrp{\delta \beta(x_k) + \delta \phi(x_k) + \sqrt{\delta} \xi(x_k)}}$ for $t\in[0,1]$. From \eqref{e:geodesic_equation_coordinate_christoffel}, we verify that $y(t)$ satisfies the following second order ODE:
    \begin{alignat*}{1}
        & y_k(0) = x_k\\
        & v_k(0) = \delta \beta(x_k) + \sqrt{\delta} \xi(x_k) + \frac{\delta}{2} \phi(x_k)\\
        & d y_k(t) = v_k(t) dt\\
        & d v_k(t) = - \phi(y_k(t),v_k(t))
    \end{alignat*}
    Thus $y_k(1) = x_{k+1}$. Next, we define $\bar{x}_{k+1} := x_k + \delta \beta(x_k) + \sqrt{\delta} \xi_k(x_k)$. We can apply Lemma \ref{l:euclidean_step_to_manifold_step} to bound
    \begin{alignat*}{2}
        & \lrn{x_{k+1} - \bar{x}_{k+1}}_2^2
        &&\leq O\lrp{\delta^6 \lrn{\beta(x_k)}_2^6 + \delta^2}\\
        &\lrn{x_{k+1} - \bar{x}_{k+1} + \frac{1}{2} \phi(x_k,\sqrt{\delta} \xi_k(x_k)) - \frac{\delta}{2}\phi(x_k)}_2
        &&\leq O\lrp{\delta^3 \lrn{\beta(x_k)}_2^3 + \delta^{3/2}}
        \elb{e:t:ondasod:1}
    \end{alignat*}
    where $O\lrp{}$ hides polynomial dependence on $L_A, \frac{1}{\lambda_A}, L_\phi, L_{\phi}',d,L_{\xi,2}$.

    Consider some fixed $k$. From now on, unless otherwise stated, all expectations are wrt $\xi_k$, conditioned on $\xi_0...\xi_{k-1}$. 
    \begin{alignat*}{1}
        & \E{\lrn{x_{k+1} - z_{k+1}}_2^2}\\
        =& \lrn{x_{k} - z_{k}}_2^2 + \linb{x_{k} - z_{k}, \E{x_{k+1} - x_k - (z_{k+1} - z_k)}} + \E{\lrn{x_{k+1} - x_k - (z_{k+1} - z_k)}_2^2}\\
    \end{alignat*}
    We now bound the $\linb{x_{k} - z_{k}, \E{x_{k+1} - x_k - (z_{k+1} - z_k)}}$ term. Notice that
    \begin{alignat*}{1}
        & \E{x_{k+1} - x_k}\\
        =& \E{\bar{x}_{k+1} - x_k +  x_{k+1} - \bar{x}_{k+1}}\\
        =& \E{\delta \beta(x_k) +  x_{k+1} - \bar{x}_{k+1} + \frac{1}{2} \phi(x_k,\sqrt{\delta} \xi_k(x_k)) - \frac{\delta}{2}\phi(x_k) + \sqrt{\delta} \xi_k(x_k)}\\
        =& \E{\delta \beta(x_k) +  x_{k+1} - \bar{x}_{k+1} + \frac{1}{2} \phi(x_k,\sqrt{\delta} \xi_k(x_k)) - \frac{\delta}{2}\phi(x_k)}
    \end{alignat*}
    where the second equality is because $\E{\phi(x_k,\sqrt{\delta} \xi_k(x_k)) - \delta \phi(x_k)} = 0$, and the third inequality is because $\E{\xi_k(x_k)}=0$. We can also verify that $\E{z_{k+1} - z_k} = \delta \beta(z_k)$. Plugging these into $\linb{x_{k} - z_{k}, \E{x_{k+1} - x_k - (z_{k+1} - z_k)}}$, we get
    \begin{alignat*}{1}
        & \linb{x_{k} - z_{k}, \E{x_{k+1} - x_k - (z_{k+1} - z_k)}}\\
        =& \linb{x_{k} - z_{k}, \delta \beta(x_{k}) - \delta \beta(z_{k})} + \linb{x_{k} - z_{k},  \E{x_{k+1} - \bar{x}_{k+1} + \frac{1}{2} \phi(x_k,\sqrt{\delta} \xi_k(x_k)) - \frac{\delta}{2}\phi(x_k)}}\\
        \leq& \delta {L_{\beta,2}'} \lrn{x_k-z_k}_2^2 + \frac{1}{K} \lrn{x_k-z_k}_2^2 + O\lrp{K\delta^6 \lrn{\beta(x_k)}_2^6 + K\delta^{3}}
    \end{alignat*}
    where we use Assumption \ref{ass:beta_lipschitz_euclidean} and Young's inequality from \eqref{e:t:ondasod:1}.

    Finally, we can bound
    \begin{alignat*}{1}
        & \E{\lrn{x_{k+1} - x_k - (z_{k+1} - z_k)}_2^2}\\
        =& \E{\lrn{x_{k+1} -\bar{x}_{k+1} + \bar{x}_{k+1} - x_k - (z_{k+1} - z_k)}_2^2}\\
        \leq& 2\E{\lrn{x_{k+1} -\bar{x}_{k+1}}_2^2} + 2\E{\lrn{\bar{x}_{k+1} - x_k - (z_{k+1} - z_k)}_2^2}\\
        \leq& O\lrp{\delta^6 \lrn{\beta(x_k)}_2^6 + \delta^2} + 2\delta^2 {L_{\beta,2}'}^2 \lrn{x_k - z_k}_2^2 + 2\delta^2 {L_{\beta,2}'}^2 \lrn{x_k - z_k}_2^2 + 2\delta {L_{\xi,2}'}^2 \lrn{x_k - z_k}_2^2
    \end{alignat*}
    where we use \eqref{e:t:ondasod:1} and Assumption \ref{ass:beta_lipschitz_euclidean}.

    Putting everything together, we get
    \begin{alignat*}{1}
        & \E{\lrn{x_{k+1} - z_{k+1}}_2^2}\\
        \leq& \lrp{1 + 2\delta L_{\beta,2}' + 2\delta {L_{\xi,2}'}^2 + \frac{1}{K}}\lrn{x_{k} - z_{k}}_2^2 + O\lrp{K\delta^6 \lrn{\beta(x_k)}_2^6 + K\delta^{3} + \delta^2}
    \end{alignat*}

    Applying the above recursively for $k=0...K-1$, we get
    \begin{alignat*}{1}
        \E{\lrn{x_{K} - z_{K}}_2^2} 
        \leq& \exp\lrp{1 + 2K\delta L_{\beta,2}' + 2K\delta {L_{\xi,2}'}^2} + O\lrp{K\delta^6 \lrp{\sum_{k=0}^{K-1}\E{\lrn{\beta(x_k)}_2^6}} + K^2\delta^{3} + K\delta^2}\\
        =& O\lrp{K\delta^6 \lrp{\sum_{k=0}^{K-1}\E{\lrn{\beta(x_k)}_2^6}} + K^2\delta^{3} + K\delta^2}
    \end{alignat*}
\end{proof}

\subsection{Manifold SDE to Euclidean SDE}
\begin{lemma}\label{l:manifold_sde_to_euclidean_sde}
    Consider the Riemannian manifold $(\Re^d,g)$. Let $\beta$ be a vector field satisfying Assumption \ref{ass:beta_lipschitz}. Given any initial point $x_0$, orthonormal basis $E$ of $T_{x_0}M$, standard Brownian motion $\BB(t)$ and any $T\in \Re^+$, let $x(t):= \Phi(t;x_0,E,\beta,\BB)$ where $\Phi$ is as defined in \eqref{d:x(t)}, $E$ is an arbitrary orthonormal basis of $x(0)$, $\BB$ is a standard Brownian motion over $\Re^d$. Let us also define the Euclidean SDE
    \begin{alignat*}{1}
        d z(t) = \beta(z(t)) dt - \frac{1}{2} \phi(z(t))dt + g(z(t))^{-1/2} d\WW(t)
    \end{alignat*}
    with $z(0) = x(0)$, and where $\WW(t)$ is another standard Brownian motion and $\phi$ is as defined in \eqref{d:phi_for_christoffel_symbol}.

    Then $x$ and $z$ have the same distribution.
\end{lemma}
\begin{proof}
    We will verify that $x(t)$ and $z(t)$ have the same generator. The conclusion follows from Theorem 1.3.6 of \cite{hsu2002stochastic} which states that diffusion measusure with the given generator and initial distribution is unique.

    We have already observed in Lemma \ref{l:Phi_is_diffusion} that $x(t)$ is the diffusion process generated by $L f = \lin{\nabla f, \beta} + \frac{1}{2} \Delta(f)$, where $\Delta$ denotes the Laplace Beltrami operator. We verify that $z(t)$ has the same generator.

    It is important to recall our definition of $\nablab^2, \nablab, \linb{}, \nabla^2, \nabla, \lin{}$ in Section \ref{ss:euclidean_setup}.

    By Ito's Lemma, for any twice continuously differentiable $f$,
    \begin{alignat*}{1}
        d f(z(t)) 
        =& \linb{\nablab f(z(t)), \beta(z(t))} dt - \frac{1}{2}\linb{\nablab f(z(t)), \phi(z(t))} + \frac{1}{2} \tr\lrp{g(z(t))^{-1} \nablab^2 f(z(t))} dt\\
        &\quad + \linb{\nablab f(z(t)), g(z(t))^{-1} d\WW(t)} 
    \end{alignat*}

    By definition of $\linb{}, \nablab, \lin,\nabla$, $\linb{\nablab f(z(t)), \beta(z(t))} = \lin{\nabla f(z(t)), \beta(z(t))}$. 

    Next, we verify that $- \frac{1}{2}\linb{\nablab f(z(t)), \phi(z(t))} + \frac{1}{2} \tr\lrp{g(z(t))^{-1} \nablab^2 f(z(t))} dt = \frac{1}{2} \Delta f (z(t))$. This is because
    \begin{alignat*}{1}
        \Delta f 
        =& \sum_{j,k=1}^d \lrp{g^{-1}}_{j,k} \lrp{\nablab^2 f}_{k,j} - \sum_{j,k,\ell} \lrp{g^{-1}}_{jk} \Gamma^{\ell}_{jk} \lrp{\nabla f}_{\ell}\\
        =& \tr\lrp{g^{-1} \nablab^2 f} + \linb{\nabla f, \phi}
    \end{alignat*}
    where we use the definition of $\phi$ in \eqref{d:phi_for_christoffel_symbol}.

    We have thus verified that
    \begin{alignat*}{1}
        df(z(t)) - \lin{\nabla f(z(t)), \beta(z(t))} - \frac{1}{2} \Delta f(z(t)) = \linb{\nablab f(z(t)), g(z(t))^{-1} d\WW(t)} 
    \end{alignat*}
    is a martingale.
\end{proof}

\subsection{A CLT for Euclidean SDE}
\begin{lemma}\label{l:euclidean_clt}
    Assume $M= \lrp{\Re^d, g}$ satisfies Assumptions \ref{ass:sectional_curvature_regularity}, \ref{ass:ricci_curvature_regularity} and \ref{ass:higher_curvature_regularity}. Let $A = g^{-1}$ satisfy Assumption \ref{ass:g_A_regularity} and \ref{ass:g_A_Lipschitz}. Let $\beta$ and $\xi$ satisfy Assumptions \ref{ass:beta_lipschitz_euclidean} and \ref{ass:euclidean_xi_assumption} (equivalent to Assumption \ref{ass:moments_of_xi}). Let $\phi$ satisfy Assumption \ref{ass:phi_Gamma_regularity}. Let $\beta^g := \beta + \phi$ satisfying Assumption \ref{ass:beta_lipschitz}. Assume that $\beta^g (x^*) = 0$ at some $x^*\in \Re^d$. Let $\xi$ be a vector field and assume there exists constants $L_\xi,L_\xi',L_\xi''$ so that Assumption \ref{ass:regularity_of_xi} holds with probability 1.

    Let $\delta \in \Re^+$ and $K\in \Z^+$, Let $z_0$ be a point, let $\s:=\max\lrbb{1,\dist\lrp{z_0,x^*}}$. Define
    \begin{alignat*}{1}
        & z_{k+1} = z_k + \delta \beta(z_k) + \sqrt{\delta} \xi_k(z_k)
    \end{alignat*}
    and SDE
    \begin{alignat*}{1}
        d y(t) = \beta(y(t))dt + A(y(t))^{1/2} d\WW(t)
    \end{alignat*}
    where $\WW(t)$ is a Brownian motion and where $y(0) = z_0$.
    There exists constant $\C_1$, which depend polynomially on $L_\beta,L_\beta', L_\xi,L_\xi',L_\xi'',L_R,L_R',d,\R,L_{\xi,2}',L_{\beta,2}'$, such that for any positive $T \leq \frac{1}{\C_1}$ and $\delta := {T}^3$ and $K:= T/\delta$, there exists a coupling between $y$ and $\hat{y}$, such that 
    \begin{alignat*}{1}
        \E{\lrn{z_K - y(K\delta)}_2} \leq {O}\lrp{\lrp{K\delta}^{3/2}\lrp{1+\lrn{z_0 - x^*}_2^3}}
    \end{alignat*}
    where $\t{O}$ hides polynomial dependency on $L_\beta,L_\beta', L_\xi,L_\xi',L_\xi'',L_R,L_R',d,\R,L_A, \frac{1}{\lambda_A}, L_\phi, L_{\phi}',d,L_{\xi,2}',L_{\beta,2}'$.
\end{lemma}
\begin{remark}
    Note that the requirement that $\delta = (K\delta)^3$ means that for $K\delta \epsilon$ error, we need $K\delta = \epsilon^{2}$ and $\delta = \epsilon^6$.
\end{remark}

\begin{proof}[Proof of Lemma \ref{l:euclidean_clt}]
    Let us define
    \begin{alignat*}{1}
        x_{k+1} = \Exp_{x_k}\lrp{\delta \beta(x_k) + \frac{\delta}{2} \phi(x_k)+ \sqrt{\delta} \xi_k(x_k)}
    \end{alignat*}
    By Lemma \ref{l:euclidean_walk_to_manifold_walk},
    \begin{alignat*}{1}
        \E{\lrn{x_{K} - z_{K}}_2^2} 
        \leq& O\lrp{K\delta^6 \lrp{\sum_{k=0}^{K-1}\E{\lrn{\beta(z_k)}_2^6}}+ K\delta^2}
    \end{alignat*}
    By Lemma \ref{l:euclidean_walk_l2_distance_bound_without_dissipativity} and Assumption \ref{ass:beta_lipschitz_euclidean}, 
    \begin{alignat*}{1}
        \E{\lrn{\beta(z_k)}_2^6} 
        \leq& 8{L_{\beta,2}'}^6\E{\lrn{z_k - z_0}_2^6} +  8\E{\lrn{\beta(z_0)}_2^6}\\
        \leq& 2^{11}{L_{\beta,2}'}^6 \lrp{{K^3\delta^6 \lrn{\beta(z_0)}_2^6} +  K^3 \delta^3 L_{\xi,2}^6} + 8\E{\lrn{\beta(z_0)}_2^6}\\
        \leq& 9\E{\lrn{\beta(z_0)}_2^6} + 2^{11}{L_{\beta,2}'}^6 K^3 \delta^3 L_{\xi,2}^6
    \end{alignat*}
    where we assume that $K\delta \leq 1/(32L_{\beta,2}')$. Plugging into the earlier bound from Lemma \ref{l:euclidean_walk_to_manifold_walk},
    \begin{alignat*}{1}
        \E{\lrn{x_{K} - z_{K}}_2^2} 
        \leq& O\lrp{K\delta^6 \lrn{\beta(z_0)}_2^6 + K\delta^2}\\
        =& O\lrp{ \lrp{K\delta}^{15} \lrn{\beta(z_0)}_2^6 + \lrp{K\delta}^4}\\
        =& O\lrp{ \lrp{K\delta}^{15} \lrn{z_0 - x^*}_2^6 + \lrp{K\delta}^4}
    \end{alignat*}
    using the fact that $\delta = (K\delta)^3$ by definition. The last line is because $\lrn{\beta(z_0)}_2 \leq L_{\beta,2}' \lrn{z_0 - x^*}_2 + \lrn{\beta(x^*)}_2 \leq L_{\beta,2}' \lrn{z_0 - x^*}_2 + L_{\phi}$.

    By Lemma \ref{l:euclidean_step_to_manifold_step}, we verify that $y(t)$ has the same distribution as $\hat{y}(t) = \Phi(K\delta;y_0,E,\beta,\BB)$ where $E$ is some orthonormal basis of $T_{y_0} M$ and $\BB(t)$ is a standard Brownian motion. By Theorem \ref{t:main_nongaussian_theorem},
    \begin{alignat*}{1}
        \E{\dist\lrp{\hat{y}_{K},{x}_{K}}} \leq {O}\lrp{\lrp{K\delta}^{3/2}\lrp{1+\dist\lrp{z_0,x^*}^3}} = {O}\lrp{\lrp{K\delta}^{3/2}\lrp{1+\lrn{z_0 - x^*}_2^3}}
    \end{alignat*}

    Combining the bounds on $\E{\dist\lrp{\hat{y}_{K},{x}_{K}}}$ and $\E{\lrn{x_{K} - z_{K}}_2^2}$ (which is entirely dominated in the big-O sense), and using Assumption \ref{ass:g_A_regularity}, and recalling that $\hat{y}_K \overset{d}{=} y(K\delta)$, we conclude that there is a coupling between $z$ and $y$ such that
    \begin{alignat*}{1}
        \E{\lrn{z_K - y(K\delta)}_2} \leq {O}\lrp{\lrp{K\delta}^{3/2}\lrp{1+\lrn{z_0 - x^*}_2^3}}
    \end{alignat*}
    
\end{proof}

\subsection{Notation and Assumptions for SGD}
Let $\S$ denote a set of samples, let $p_\S$ denote the population distribution over $\S$. $\ell(x,s): \Re^d \to \S$ denote a loss function, where $x$ can be viewed as parameterizing a model. We assume that that $\ell$ is bounded in the following sense
\begin{assumption}
    \label{ass:L_ell}
    Exists $L_\ell$ such that for all $s,s' \in \S$,
    \begin{alignat*}{1}
        \lrn{\nablab \ell(x,s) - \nablab \ell(x,s')}_2 \leq L_{\ell}
    \end{alignat*}
\end{assumption}

Let $\S_n = \lrbb{s_1,s_2,...s_n}$, where $s_i$ are sampled i.i.d from $p_\S$. Let $\t{s}_1 \sim p_\S$ be sampled independently of $s_1...s_n$. Let $\t{\S}_n=\lrp{\t{s}_1,s_2...s_n}$, i.e. $\t{\S}_n$ is $\S_n$ with the first sample replaced.

Let
\begin{alignat*}{1}
    & \ell_n (x) := \frac{1}{n} \sum_{i=1}^n \ell(x,s_i)\\
    & \t{\ell}_n = \frac{1}{n} \lrp{\ell(x,s_1) + \sum_{i=2}^n \ell(x,s_i)}
    \elb{d:ell_n}
\end{alignat*}

Let us define the matrix $A_n(x)$ to be the noise covariance of gradient under the empirical distribution over $\S_n$, i.e.
\begin{alignat*}{1}
    A_n(x)
    :=& \frac{1}{n} \sum_{i=1}^n \lrp{\nablab \ell(x,s_i) - \nablab \ell(x)}\lrp{\nablab \ell_n(x,s_i) - \nablab \ell_n(x)}^T\\
    =& \frac{1}{n} \sum_{i=1}^n \nablab \ell(x,s_i) \nablab \ell(x,s_i)^T - \nablab \ell_n(x) \nablab \ell_n(x)^T
    \elb{d:An}
\end{alignat*}
Let $\t{A}_n$ denote the noise covariance of gradient under  under the empirical distribution over $\t{\S}_n$
\begin{alignat*}{1}
    \t{A}_n(x)
    :=& \frac{1}{n} \sum_{i=2}^n \lrp{\nablab \ell(x,s_i) - \nablab \t{\ell}_n(x)}\lrp{\nablab \ell(x,s_i) - \nablab \t{\ell}_n(x)}^T\\
    &\quad + \frac{1}{n} \lrp{\nablab \ell(x,\t{s}_1) - \nablab \t{\ell}_n(x)}\lrp{\nablab \ell(x,\t{s}_1) - \nablab \t{\ell}_n(x)}^T\\
    =& \frac{1}{n} \lrp{\nablab \ell(x,s_1) \nablab \ell(x,s_1)^T + \sum_{i=2}^n \nablab \ell(x,s_i) \nablab \ell(x,s_i)^T} - \nablab \t{\ell}_n(x) \nablab \t{\ell}_n(x)^T
    \elb{d:tAn}
\end{alignat*}

\begin{lemma}\label{l:sgd:An-tAn}
    Let $\ell_n,\t{\ell}_n,A_n,\t{A}_n$ be as defined in \eqref{d:ell_n}, \eqref{d:An} and \eqref{d:tAn} respectively. Under Assumption \ref{ass:L_ell}, for all $x$,
    \begin{alignat*}{1}
        & \tr\lrp{A_n(x) - \t{A}_n(x)}^2 \leq \frac{128 d L_\ell^4}{n^2}\\
        & \tr\lrp{A_n^{1/2}(x) - \t{A}_n^{1/2}(x)}^2 \leq \frac{128 d L_\ell^4}{\lambda_A n^2}\\
        & \lrn{\t{\ell}_n(x) - {\ell}_n(x)}_2 \leq \frac{2L_\ell}{n}
    \end{alignat*}
\end{lemma}
\begin{proof}
We verify that
\begin{alignat*}{1}
    &A_n(x) - \t{A}_n(x)\\
    =& \frac{1}{n}\underbrace{\lrp{\nablab \ell(x,\t{s}_1) - \nablab \t{\ell}_n(x)}\lrp{\nablab \ell(x,\t{s}_1) - \nablab \t{\ell}_n(x)}^T}_{\circled{1}} - \frac{1}{n}\underbrace{\lrp{\nablab \ell(x,{s}_1) - \nablab {\ell}_n(x)}\lrp{\nablab \ell(x,{s}_1) - \nablab {\ell}_n (x)}^T}_{\circled{2}}\\
    &\quad + \frac{1}{n}\sum_{i=2}^n \underbrace{\lrp{\nablab \ell_n(x) - \nablab \t{\ell}_n(x)}\lrp{\nablab \ell(x,s_i) - \nablab \t{\ell}_n(x)}^T}_{\circled{3}} + \underbrace{\lrp{\nablab \ell(x,s_i) - \nablab {\ell}_n(x)}\lrp{\nablab \t{\ell}_n(x) - \nablab {\ell}_n(x)}^T}_{\circled{4}}
\end{alignat*}
Under Assumption \ref{ass:L_ell}, we can bound $\lrn{\circled{1}}_2+\lrn{\circled{2}}_2\leq 8 L_\ell^2$. Furthermore, we verify that $\lrn{\t{\ell}_n(x) - {\ell}_n(x)}_2 \leq \frac{2L_\ell}{n}$, so that $\lrn{\circled{3}}_2 + \lrn{\circled{4}}_2 \leq \frac{8 L_{\ell}^2}{n}$. Put together, we can bound $\lrn{A_n(x) - \t{A}_n(x)}_2 \leq \frac{16 L_\ell^2}{n}$ and thus
\begin{alignat*}{1}
    & \tr\lrp{A_n(x) - \t{A}_n(x)}^2 \leq \frac{128 d L_\ell^4}{n^2}\\
    & \lrn{\t{\ell}_n(x) - {\ell}_n(x)}_2 \leq \frac{2L_\ell}{n}
\end{alignat*}
The bound on $\tr\lrp{A_n^{1/2}(x) - \t{A}_n^{1/2}(x)}^2$ then follows from Lemma 1 of \cite{eldan2020clt}.
\end{proof}

\begin{lemma}\label{l:euclidean_norm_divergence_under_constant_perturbation}
    Let $u(x)$ be a vector field and $F(x)$ be a matrix valued function. Assume that for all $x$, $\lrn{\t{u}(x) - u(x)}_2 \leq \epsilon_u$ and $\lrn{\t{F}(x) - F(x)}_F \leq \epsilon_F$. Assume further that $\lrn{u(x) - u(y)}_2 \leq L_u'\lrn{x-y}_2$ and $\tr\lrp{F(x) - F(y)}^2 \leq {L_{F}'}^2\lrn{x-y}_2^2$. Let
    \begin{alignat*}{1}
        & d x(t) = u(x(t)) dt + F(x(t)) dB_t\\
        & d y(t) = \t{u}(y(t)) dt + \t{F}(y(t)) dB_t
    \end{alignat*}
    initialized at some arbitrary points $x_0$ and $y_0$ respectively.

    Then
    \begin{alignat*}{1}
        \E{\lrn{x(T)-y(T)}_2^2} \leq e^{T \lrp{2L_u' + 2L_F' + 1}} \lrp{\lrn{x(0)-y(0)}_2^2 + 8  T\lrp{\epsilon_u^2 + \epsilon_F^2}}
    \end{alignat*}
\end{lemma}

\begin{proof}
    Let us denote by $v(x) := u(x) - \t{u}(x)$ and $H(x) := F(x) - \t{F}(x)$. By Ito's Lemma,
    \begin{alignat*}{1}
        &\frac{d}{dt} \E{\lrn{x(t)-y(t)}_2^2}\\
        \leq& \E{2\linb{x(t) - y(t), \lrp{u(x(t)) - u(y(t))} dt + \lrp{F(x(t)) - F(y(t))}, dB_t}}\\
        &\quad + \E{2\linb{x(t)-y(t), v(y(t)) dt + H(y(t)) dB_t}}\\
        &\quad + \E{2 \tr\lrp{F(x(t)) - F(y(t))}^2 + 2\tr\lrp{H(y(t))}^2}\\
        \leq& \lrp{2L_u'+ 2L_F'} \E{\lrn{x(t)-y(t)}_2^2} + 2\epsilon_u \E{\lrn{x(t) - y(t)}_2} + 2 \epsilon_H^2
    \end{alignat*}

    For any $T$, we can apply Young's inequality to get
    \begin{alignat*}{1}
        &\frac{d}{dt} \E{\lrn{x(t)-y(t)}_2^2}
        \leq \lrp{2L_u'+ 2L_F' + 1} \E{\lrn{x(t)-y(t)}_2^2} + \epsilon_u^2 + 2 \epsilon_H^2
    \end{alignat*}

    Integrating over $[0,T]$, we get
    \begin{alignat*}{1}
        \E{\lrn{x(T)-y(T)}_2^2} \leq e^{T \lrp{2L_u' + 2L_F' + 1}} \lrp{\lrn{x(0)-y(0)}_2^2 + 8  T\lrp{\epsilon_u^2 + \epsilon_F^2}}
    \end{alignat*}    
\end{proof}

\subsection{Main SGD Result}
\label{ss:main_sgd_result}

\begin{lemma}\label{l:sgd_stability_lemma}
    Assume the same assumptions as Theorem \ref{t:sgd_stability}.
    
    Let $\delta \in \Re^+$, define the sequences
    \begin{alignat*}{1}
        & x_{k+1} = x_k + \delta \ell_n(x_k) + \sqrt{\delta} \xi_k(x)\\
        & \t{x}_{k+1} = \t{x}_k + \delta \t{\ell}_n(\t{x}_k) + \sqrt{\delta} \t{\xi}_k(\t{x}_k)
    \end{alignat*}
    Let $f,\alpha$ be as defined in Theorem \ref{t:langevin_mcmc}. There exists constant $\C_1$, which depend polynomially on $L_\beta,L_\beta', L_\xi,L_\xi',L_\xi'',L_R,L_R', \frac{1}{m},d,\R,L_{\xi,2}',L_{\beta,2}'$, such that for any positive \\
    $T \leq \frac{1}{\C_1}\min\lrbb{1, \frac{1}{\lrn{x_0 - x^*}_2^2 + \lrn{\t{x}_0 - \t{x}^*}_2^2}}$ and $\delta := {T}^3$ and $K:= T/\delta$, there exists a coupling between $x$ and $\t{x}$ such that for $j = \lceil \frac{1}{32 T\lrp{1 + 2L_A'/\sqrt{\lambda_A} + L_\beta'}} \rceil$,
    \begin{alignat*}{1}
        \E{f\lrp{\dist\lrp{x_{jK}, \t{x}_{jK}}}}
        \leq& \exp\lrp{-\alpha jK\delta}f\lrp{\dist\lrp{x_0,\t{x}_0}}\\
        &\quad + {O}\lrp{j\lrp{K\delta}^{3/2}\lrp{1+\E{\lrn{x_{0} - x^*}_2^3 + \lrn{\t{x}_{0} - \t{x}^*}_2^3}} + \frac{\lrp{jK\delta}}{n}}
    \end{alignat*}
    where $\t{O}$ hides polynomial dependency on $L_\beta,L_\beta', L_\xi,L_\xi',L_\xi'',L_R,L_R',d,\R,L_A, \frac{1}{\lambda_A}, L_\phi, L_{\phi}',d,L_{\xi,2}',L_{\beta,2}', L_{\ell}$.
\end{lemma}

\begin{proof}[Proof of Lemma \ref{l:sgd_stability_lemma}]
    Under existing assumption on $T$, we assume without loss of further generality that $K\delta = T \leq \frac{1}{32\lrp{1 + 2L_A'/\sqrt{\lambda_A} + L_\beta'}}$. Thus $\frac{1}{32 T \lrp{1 + 2L_A'/\sqrt{\lambda_A} + L_\beta'}} \geq 1$, thus $jK\delta := K\delta \lceil\frac{1}{32K\delta\lrp{1 + 2L_A'/\sqrt{\lambda_A} + L_\beta'}} \rceil \leq 2$. We verify from definition of $j$ that
    \begin{alignat*}{1}
        \frac{1}{32\lrp{1 + 2L_A'/\sqrt{\lambda_A} + L_\beta'}} \leq j K\delta \leq \frac{1}{16\lrp{1 + 2L_A'/\sqrt{\lambda_A} + L_\beta'}}
        \elb{e:t:jkdelta}
    \end{alignat*}
    This will be useful at several places later in this proof.

    Let us define, for $t\in[0,jK\delta]$,
    \begin{alignat*}{1}
        & d {y}(t) = \beta({y}(t))dt + A({y}(t))^{1/2} d\WW(t)\\
        & d \t{y}(t) = \t{\beta}(\t{y}(t))dt + \t{A}(\t{y}(t))^{1/2} d\t{\WW}(t)\\
        & d \bar{y}(t) = {\beta}(\bar{y}(t))dt + {A}(\bar{y}(t))^{1/2} d\bar{\WW}(t)
    \end{alignat*}
    where $\WW,\t{\WW},\bar{\WW}$ are Brownian motions and where $y(0) = x_0$, $\t{y}(0) = \bar{y}(0) = \t{x}_0$. We can bound
    \begin{alignat*}{1}
        & \E{f\lrp{\dist\lrp{x_{jK}, \t{x}_{jK}}}}\\
        \leq& \E{f\lrp{\dist\lrp{y({jK}\delta), \bar{y}({jK}\delta)}}} + \E{\dist\lrp{y({jK}\delta), x_{jK}}} + \E{\dist\lrp{\bar{y}({jK}\delta), \t{x}_{jK}}}\\
        \leq& \E{f\lrp{\dist\lrp{y({jK}\delta), \bar{y}({jK}\delta)}}} + \E{\dist\lrp{y({jK}\delta), x_{jK}}} + \E{\dist\lrp{\t{y}({jK}\delta),\t{x}_{jK}}} + \E{\dist\lrp{\bar{y}({jK}\delta), \t{y}({jK}\delta)}}
        \elb{e:t:alkmlakmd}
    \end{alignat*}
    By Lemma \ref{l:manifold_sde_to_euclidean_sde}, $y(t) \overset{d}{=} \Phi(K\delta;y(0),E,\beta^g,{\BB})$ and $\bar{y}(t)\overset{d}{=} \Phi(K\delta;\bar{y}(0),\t{E},{\beta}^g,\bar{\BB})$. Therefore, by Lemma \ref{l:g_contraction_without_gradient_lipschitz}, there exists a coupling between $y$ and $\bar{y}$ such that
    \begin{alignat*}{1}
        \E{f\lrp{\dist\lrp{y(jK\delta),\bar{y}(jK\delta)}}}
        \leq& \exp\lrp{-\alpha jK\delta}f\lrp{\dist\lrp{x_0,\t{x}_0}}
    \end{alignat*}
    
    Now consider a fixed $i \leq j$. Define $\hat{x}_{(i+1)K} := \Phi(K\delta; x_{iK}, E^{iK}, \beta^g, \hat{\BB}^{iK})$ where $E^{iK}$ is an orthonormal basis at $T_{x_{iK}} M$ and $\hat{\BB}^{iK}$ is a Brownian motion. From Lemma \ref{l:manifold_sde_to_euclidean_sde} and Lemma \ref{l:euclidean_clt}, $\E{\lrn{x_{(i+1)K} - \hat{x}_{(i+1)K}}_2} \leq {O}\lrp{\lrp{K\delta}^{3/2}\lrp{1+\lrn{x_{iK} - x^*}_2^3}}$. By Lemma \ref{l:manifold_sde_to_euclidean_sde} and Lemma \ref{l:euclidean_norm_divergence_under_constant_perturbation} with $\epsilon_u = \epsilon_F = 0$ and $L_u' = L_\beta'$ and $L_F'= L_A'/\sqrt{\lambda_A}$, we can bound \\
    $\E{\lrn{\hat{x}_{(i+1)K} - y((i+1)K\delta)}_2} \leq \exp\lrp{K\delta \lrp{1 + 2L_A'/\sqrt{\lambda_A} + L_\beta'}} \E{\lrn{x_{iK} - y(iK\delta)}_2}$. Therefore, by triangle inequality,
    \begin{alignat*}{1}
        & \E{\lrn{x_{(i+1)K} - y((i+1)K\delta)}_2}\\
        \leq& \exp\lrp{K\delta \lrp{1 + 2L_A'/\sqrt{\lambda_A} + L_\beta'}} \E{\lrn{x_{iK} - y(iK\delta)}_2} +  {O}\lrp{\lrp{K\delta}^{3/2}\lrp{1+\lrn{x_{iK} - x^*}_2^3}}
    \end{alignat*}
    Applying the above recursively for $i=0... j-1$, and noting that from \eqref{e:t:jkdelta},$jK\delta \lrp{1 + 2L_A'/\sqrt{\lambda_A} + L_\beta'}\leq 1/16$, 
    \begin{alignat*}{1}
        \E{\lrn{x_{jK} - y(jK\delta)}_2}
        \leq& 8 \sum_{i=0}^{j-1}  {O}\lrp{\lrp{K\delta}^{3/2}\lrp{1+\E{\lrn{x_{iK} - x^*}_2^3}}}\\
        \leq& {O}\lrp{j\lrp{K\delta}^{3/2}\lrp{1+\E{\lrn{x_{0} - x^*}_2^3}}}
    \end{alignat*}
    where the last inequality is because $\E{\lrn{x_{iK} - x^*}_2^3} \leq O\lrp{\lrn{x_0 - x^*}_2^3}$ by Lemma \ref{l:euclidean_walk_l2_distance_bound_without_dissipativity}, and our assumption that $jK\delta \leq \frac{1}{32L_\beta'}$ from \eqref{e:t:jkdelta}.

    By an identical sequence of steps, we can also show that
    \begin{alignat*}{1}
        & \E{\lrn{\t{x}_{jK} - \t{y}(jK\delta)}_2}
        \leq {O}\lrp{j\lrp{K\delta}^{3/2}\lrp{1+\E{\lrn{\t{x}_{0} - \t{x}^*}_2^3}}}
    \end{alignat*}
    Finally, by Lemma \ref{l:sgd:An-tAn}, we can apply Lemma \ref{l:euclidean_norm_divergence_under_constant_perturbation} with $\epsilon_u = \frac{2L_\ell}{n}$, $\epsilon_F = \frac{16\sqrt{d} L_\ell^2}{\sqrt{\lambda_A} n}$ and $L_u' = L_\beta'$ and $L_F'= L_A'/\sqrt{\lambda_A}$, so that
    \begin{alignat*}{1}
        \E{\lrn{\bar{y}(jK\delta)-\t{y}(jK\delta)}_2^2} 
        \leq& e^{jK\delta \lrp{2L_u' + 2L_F' + 1}} \lrp{\lrn{x(0)-y(0)}_2^2 + 8  jK\delta \lrp{\epsilon_u^2 + \epsilon_F^2}}\\
        =& O\lrp{\frac{jK\delta}{n^2}} = O\lrp{\frac{\lrp{jK\delta}^2}{n^2}}
    \end{alignat*}
    The last line is because by \eqref{e:t:jkdelta}, $\frac{1}{jK\delta} \leq 32\lrp{1 + 2L_A'/\sqrt{\lambda_A} + L_\beta'}$. Note: we picked $j$ to ensure that, while $K\delta$ is small, $jK\delta$ is large -- $O(1)$ in terms of various parameters.

    Plugging everything into \eqref{e:t:alkmlakmd}, (applying Jensen's inequality to $\E{\lrn{\bar{y}(jK\delta)-\t{y}(jK\delta)}_2^2}$)
    \begin{alignat*}{1}
        \E{f\lrp{\dist\lrp{x_{jK}, \t{x}_{jK}}}}
        \leq& \exp\lrp{-\alpha TjK\delta}f\lrp{\dist\lrp{x_0,\t{x}_0}}\\
        &\quad + {O}\lrp{j\lrp{K\delta}^{3/2}\lrp{1+\E{\lrn{x_{0} - x^*}_2^3 + \lrn{\t{x}_{0} - \t{x}^*}_2^3}}} + O\lrp{\frac{\lrp{jK\delta}}{n}}
    \end{alignat*}

    This concludes the proof.
\end{proof}

\begin{theorem}
    \label{t:sgd_stability}
    Let $\ell(x,s): \Re^d \times \S \to \Re$ denote a loss function. Let $\ell_n,\t{\ell}_n,A_n,\t{A}_n$ be as defined in \eqref{d:ell_n}, \eqref{d:An}, \eqref{d:tAn}. Assume $\ell$ satisfies Assumption \ref{ass:L_ell}. Assume exists $L_{\ell}'$ such that for all $x,y,s$, $\lrn{\nablab \ell(x,s) - \nablab \ell(y,s)}_2 \leq L_{\ell}' \lrn{x-y}_2$. Let $\xi$ be a random vector field satisfying Assumptions \ref{ass:beta_lipschitz_euclidean} and \ref{ass:euclidean_xi_assumption} wrt $A$ and $\t{\xi}$ be a random vector field satisfying Assumptions \ref{ass:beta_lipschitz_euclidean} and \ref{ass:euclidean_xi_assumption} wrt $\t{A}$. Let $M= \lrp{\Re^d, g}$ and $\t{M} = \lrp{\Re^d, \t{g}}$, and assume both $M$ and $\t{M}$ satisfy Assumptions \ref{ass:sectional_curvature_regularity}, \ref{ass:ricci_curvature_regularity} and \ref{ass:higher_curvature_regularity}. Let $\phi$ be as defined in \eqref{d:phi_for_christoffel_symbol} wrt $g$ and analogously $\t{\phi}$ be defined wrt $\t{g}$. Assume $\phi$ and $\t{\phi}$ satisfy Assumption \ref{ass:phi_Gamma_regularity}. Let $\beta := \ell_n$ and $\t{\beta} := \t{\ell}_n$. Assume $\beta, \t{\beta}$ satisfy Assumption \ref{ass:euclidean_dissipativity}. Let $\beta^g := \beta + \frac{1}{2}\phi$ and $\t{\beta}^g := \t{\beta} + \frac{1}{2}\t{\phi}$, and assume ${\beta}^g,\t{\beta}^g$ satisfy Assumptions \ref{ass:distant-dissipativity} and \ref{ass:beta_lipschitz}. Finally, assume there exist points $x^*$ and $\t{x}^*$ such that $\ell_n(x^*) = 0$ and $\t{\ell}_n(\t{x}^*)=0$. Let $x_0$ be some point satisfying $\lrn{x_0 - x^*}_2 \leq \R_2$ and let $\t{x}_0$ be some point satisfying $\lrn{\t{x}_0 - \t{x}^*}_2 \leq \R_2$.

    Let $\delta >0$ be a stepsize, let $x_0 = \t{x}_0$, and define the processes
    \begin{alignat*}{1}
        & x_{k+1} = x_k + \delta \ell_n(x_k) + \sqrt{\delta} \xi_k(x)\\
        & \t{x}_{k+1} = \t{x}_k + \delta \t{\ell}_n(\t{x}_k) + \sqrt{\delta} \t{\xi}_k(\t{x}_k)
    \end{alignat*}
    Let $f,\alpha$ be as defined in Theorem \ref{t:langevin_mcmc}. There exists constant $\C_3$, which depends polynomially on \\
	$L_\beta,L_\beta', L_\xi,L_\xi',L_\xi'',L_R,L_R', \frac{1}{m},d,\R,L_{\xi,2}',L_{\beta,2}'$, such that for any positive \\
    $T \leq \frac{1}{\C_3}$ and $\delta := {T}^3$ and $K:= T/\delta$, there exists a coupling between $x$ and $\t{x}$ such that for $j = \lceil \frac{1}{32 T\lrp{1 + 2L_A'/\sqrt{\lambda_A} + L_\beta'}} \rceil$ and for any $i$,
    \begin{alignat*}{1}
        \E{\lrn{x_{ijK} - \t{x}_{ijK}}_2^2}\leq \min\lrbb{ijK\delta, \frac{1}{\alpha}} \cdot O\lrp{\lrp{K\delta}^{1/2}+ \frac{1}{ n}}
    \end{alignat*}
    where ${O}$ hides polynomial dependency on $L_\beta,L_\beta', L_\xi,L_\xi',L_\xi'',L_R,L_R',d,\R,L_A, \frac{1}{\lambda_A}, L_\phi, L_{\phi}',d,L_{\xi,2}',L_{\beta,2}', L_{\ell}$.
\end{theorem}
\begin{proof}
    Under Assumption \ref{ass:L_ell} and our assumption on $\lambda_A$, we verify that both $A(x)$ and $\t{A}(x)$ satisfy Assumption \ref{ass:g_A_regularity} with $L_A = 4 L_{\ell}^2$. Our assumption on $L_{\ell'}$ implies Assumption \ref{ass:g_A_Lipschitz} holds with $L_A' \leq 4 L_{\ell}^2 {L_{\ell}'}^2$. Note that $\beta$ and $\t{\beta}$ satisfy Assumption \ref{ass:beta_lipschitz_euclidean} with $L_{\beta,2}' \leq L_{\ell}'$, and Assumption \ref{ass:beta_lipschitz} with $L_{\beta}' \leq L_A L_{\ell}'$.

    Let $\s$ denote a radius and let $V_k$ denote the event that $\sup_{i\leq k} \max\lrbb{\lrn{x_{k} - x_2^*}_2,\lrn{\t{x}_{k} - \t{x}_2^*}_2}\geq \s$. Recall that $x_2^*$ is by definition a point with $\ell_n(x_2^*) = 0$ and similarly $\t{x}_2^*$ is a point with $\t{\ell}_n(\t{x}_2^*) = 0$. From Lemma \ref{l:euclidean_walk_l2_distance_bound},
    \begin{alignat*}{1}
        \Pr{\max_{i\leq k}  \lrn{x_{K} - x^*_2}_2^2 \geq \s^2} 
        \leq& \exp\lrp{\frac{m_2}{64L_{\xi,2}^2} \lrp{2\R_2^2 + \frac{L_{\beta,2}' \R_2^2}{m_2} - \s^2}}\\
        \leq& \exp\lrp{\frac{m_2}{64L_{\xi,2}^2} \lrp{\frac{4 L_{\beta,2}' \R_2^2}{m_2} - \s^2}}
    \end{alignat*}
    We get the same bound for $\Pr{\max_{i\leq k}  \lrn{\t{x}_{K} - \t{x}^*_2}_2^2 \geq \s^2}$. Together with the second-moment bound from lemma \ref{l:euclidean_walk_l2_distance_bound} and triangle inequality, we verify that for all $k$, $\E{\ind{V_k^c} \lrp{\lrn{x_k - \t{x}_k}_2}} \leq 2\exp\lrp{\frac{m_2}{128L_{\xi,2}^2} \lrp{\frac{4 L_{\beta,2}' \R_2^2}{m_2} - \s^2}} \cdot \sqrt{\frac{L_{\beta,2}'\R_2^2}{m} +\frac{2{L_{\xi,2}}^2}{m}}$. We verify that there is a constant $\C_2=poly(L_{\beta,2}',\R_2,\frac{1}{{m_2}},L_{\xi,2})$, such that for any $T$, $\s= \C_2 \log\lrp{1/T}$ gives $\E{\ind{V_k^c} \lrp{\lrn{x_k - \t{x}_k}_2}} \leq \lrp{T}^{3/2}$.

    Consider an arbitrary step $k$, and condition on the event $V_k$. Let $\C_1$ be the constant from Lemma \ref{l:sgd_stability_lemma}. We verify using Lemma \ref{l:useful_xlogx} that any $T \leq \frac{1}{32 \C_1 \C_2^2 \log\lrp{1 + \C_1 \C_2^2}}$ implies that $T \leq \frac{1}{\C_1 \C_2^2 \log\lrp{T}^2} = \frac{1}{\C_1 \s^2}$. Therefore, there exists a $\C_3$ such that, under the event $V_k$, $T = \frac{1}{\C_3}$ satisfies the requirement of Lemma \ref{l:sgd_stability_lemma}, which then implies that
    \begin{alignat*}{1}
        &\E{\ind{V_{k+jK}}f\lrp{\dist\lrp{x_{k+jK}, \t{x}_{k+jK}}}}\\
        \leq& \E{\ind{V_{k}}f\lrp{\dist\lrp{x_{k+jK}, \t{x}_{k+jK}}}}\\
        \leq& \ind{V_k}\exp\lrp{-\alpha jK\delta}f\lrp{\dist\lrp{x_k,\t{x}_k}}\\
        &\quad + {O}\lrp{j\lrp{K\delta}^{3/2}\lrp{1+\E{\ind{V_k}\lrp{\lrn{x_{k} - x^*}_2^3 + \lrn{\t{x}_{k} - \t{x}^*}_2^3}}} + \frac{\lrp{jK\delta}}{n}}\\
        =& \ind{V_k}\exp\lrp{-\alpha jK\delta}f\lrp{\dist\lrp{x_k,\t{x}_k}} + {O}\lrp{j\lrp{K\delta}^{3/2} + \frac{\lrp{jK\delta}}{n}}
        \elb{e:t:alskmdal}
    \end{alignat*}
    where the expectation is conditional on randomness up to step $k$, and where $j:= \lceil \frac{1}{32 T\lrp{1 + 2L_A'/\sqrt{\lambda_A} + L_\beta'}} \rceil$ is as defined in Lemma \ref{l:sgd_stability_lemma}.

    The last line follows from the definition of $V_k$, as well as the fact that $\lrn{x_2^* - x^*}_2 + \lrn{\t{x}_2^* - \t{x}^*}_2 = O(1)$, which we verify presently: Assumption \ref{ass:phi_Gamma_regularity} implies $\lrn{\beta(x^*)}_2 \leq L_{\phi}$. If $\lrn{x^* - x_2^*}_2 \geq \R_2$, then we are done. So suppose otherwise. Assumption \ref{ass:euclidean_dissipativity} implies that $L_\phi \lrn{x^* - x_2^*}_2 \geq \linb{\frac{1}{2}\phi(x^*), x^* - x_2^*} \geq m_2 \lrn{x^* - x_2^*}_2^2$, so that $\lrn{x^* - x_2^*}_2 \leq \frac{L_{\phi}}{m_2}$. By identical argument, $V_k$ also implies $\lrn{\t{x}_2^* - \t{x}^*}_2 \leq \R_2 + \frac{L_{\phi}}{m_2}$.

    Applying \eqref{e:t:alskmdal} recursively over $k$ being a integral multiple of $jK$, we get that for any $i$,
    \begin{alignat*}{1}
        &\E{\ind{V_{ijK}}f\lrp{\dist\lrp{x_{ijK}, \t{x}_{ijK}}}}\\
        \leq& \exp\lrp{-\alpha ijK\delta}\E{\ind{V_{0}}f\lrp{\dist\lrp{x_0,\t{x}_0}}} + {O}\lrp{\sum_{r=0}^{i-1} \exp\lrp{-\alpha rjK\delta} \lrp{j\lrp{K\delta}^{3/2} + \frac{\lrp{jK\delta}}{n}}}\\
        =& O\lrp{\frac{1 - \exp\lrp{-\alpha r j K \delta}}{\alpha} \cdot \lrp{\lrp{K\delta}^{1/2}+ \frac{1}{ n}}}\\
        =& \min\lrbb{ijK\delta, \frac{1}{\alpha}} \cdot O\lrp{\lrp{K\delta}^{1/2}+ \frac{1}{ n}}
    \end{alignat*}
    Finally, recall our earlier bound that $\E{\ind{V_k^c} \lrp{\lrn{x_k - \t{x}_k}_2}} \leq \lrp{K\delta}^{3/2} = \min\lrbb{ijK\delta, \frac{1}{\alpha}} \cdot O\lrp{\lrp{K\delta}^{1/2}+ \frac{1}{ n}}$. 
    
    Summing these two bounds, and using the fact that $r \leq O\lrp{\frac{1}{\alpha}}f(r)$ gives our conclusion.

\end{proof}

\subsection{Stability and generalization}\label{ss:gen-gap}
\begin{lemma}\label{lemma:stability}
    For a given geodesically $L-$lipschitz function $\ell$, any n samples $s_1, s_2, ..., s_n$ (distribution free), and two distributions of parameters $x_0 \sim \pi_0, x_1 \sim \pi_1$, we have 
    \begin{align*}
       \mathbb{E}_{x_0 \sim \pi_0, x_1 \sim \pi_1} \left[\tfrac{1}{n}\sum_{i=1}^n \ell(s_i, x_0) - \ell(s_i, x_1)\right] \le L W_{1} {(x_0, x_1)}.
    \end{align*}
\end{lemma}
\begin{proof}
Let $\pi$ be a coupling between $\pi_0$ and $\pi_1$. In other word $\int \pi(x, y) dy = \pi_0(x), \int \pi(x, y) dx = \pi_1(y) $. Then for any such coupling we have
\begin{align*}
       \mathbb{E}_{x_0 \sim \pi_0, x_1 \sim \pi_1} \left[\tfrac{1}{n}\sum_{i=1}^n \lrp{\ell(s_i, x_0) - \ell(s_i, x_1)}\right] &= \inf_{\pi} \mathbb{E}_{ (x_0 , x_1) \sim \pi} \left[\tfrac{1}{n}\sum_{i=1}^n \lrp{\ell(s_i, x_0) - \ell(s_i, x_1)}\right] \\
       &\le \inf_{\pi}  \mathbb{E}_{ (x_0 , x_1) \sim \pi} \left[\tfrac{1}{n}\sum_{i=1}^n L \dist(x_0, x_1)\right] \\
       &= L W_1(x_0, x_1) 
\end{align*}
\end{proof}
\paragraph{Proof for Lemma~\ref{thm:gen-gap}} 
\begin{proof}
Denote $\cA(S)$ as the algorithm output based on dataset $S$. Denote $S' = \{s_1', ..., s_n'\}$, $S^{i} = \{s_1, ..., s_{i-1}, s_i', s_{i+1}, s_n\}$, where $s_i, s_i'$ are independently sampled from the population distribution. Then
\begin{align*}
    \mathbb{E}_{S, S' \cA} \left[\frac{1}{n}\sum_{i=1}^n \ell(s_i', \cA(S)) - \frac{1}{n}\sum_{i=1}^n \ell(s_i, \cA(S))\right] &= \mathbb{E}_{S, \cA} \left[ \mathbb{E}_{s'} \left[\ell(s', \cA(S))\right] - \frac{1}{n}\sum_{i=1}^n \ell(s_i, \cA(S))\right]\\
    &= \mathbb{E}_{S, s_i', \cA} \left[  \frac{1}{n}\sum_{i=1}^n \ell(s_i', \cA(S)) - \frac{1}{n}\sum_{i=1}^n \ell(s_i, \cA(S)) \right]\\
    &= \mathbb{E}_{S, s_i', \cA} \left[  \frac{1}{n}\sum_{i=1}^n \lrp{\ell(s_i', \cA(S)) -  \ell(s_i, \cA(S))} \right]\\
    &= \mathbb{E}_{S, s_i', \cA} \left[  \frac{1}{n}\sum_{i=1}^n \lrp{\ell(s_i, \cA(S^i)) -  \ell(s_i, \cA(S)) }\right] \\
    &\le \mathbb{E}_{S, s_i', \cA} \left[ L W_1(\cA(S^i), \cA(S))\right]
\end{align*}
The first line uses linearity of expectation. The third line is by the fact that $\cA(S)$ is exchangeable in $x_i$ and hence
\begin{align*}
   \mathbb{E}_{S, s_i', \cA} \left[\ell(s_i, \cA(S^i)) \right] =    \mathbb{E}_{S, s_i', \cA} \left[\ell(s_i', \cA(S)) \right].
\end{align*}
The fourth line follows by Lemma~\ref{lemma:stability}. Then the statement follows by upper-bounding $W_1(\cA(S^i), \cA(S))$ by Theorem~\ref{t:sgd_stability}.
\end{proof}

\subsection{Euclidean Tail Bounds}
\begin{assumption}[$L_2$ one-point dissipativity]
    \label{ass:euclidean_dissipativity}
    There exists a $x^*_2 \in \Re^d$ such that $\beta(x^*_2) = 0$. There exist constants $m_2, \R_2$ and for all $x: \lrn{x-x^*_2}_2 \geq \R_2$, 
    \begin{alignat*}{1}
        \linb{\beta(x), x - x_2^*} \leq - m_2 \lrn{x - x_2^*}_2^2
    \end{alignat*}
\end{assumption}
\begin{lemma}\label{l:euclidean_walk_l2_distance_bound_without_dissipativity}
    Let $\beta$ satisfy Assumption \ref{ass:beta_lipschitz_euclidean}. Let
    \begin{alignat*}{1}
        x_{k+1} =  x_k + \beta(x_k) + \sqrt{\delta} \xi_k(x_k)
    \end{alignat*}
    with arbitrary $x_0$. Assume and $K\delta \leq 1/(32L_{\beta,2}')$. Then for all $k$,
    \begin{alignat*}{1}
        \E{\lrn{x_{k} - x_0}_2^{8}} \leq {2^{10}K^4\delta^8 \lrn{\beta(x_0)}_2^8} + 2^{10} K^4 \delta^4 L_{\xi,2}^8
    \end{alignat*} 
\end{lemma}
\begin{proof}
    \begin{alignat*}{1}
        \lrn{x_{k+1} - x_0}_2^2
        =& \lrn{x_{k} - x_0}_2^2 + 2\linb{x_k - x_0, \beta(x_k) + \sqrt{\delta} \xi_k(x_k)} + \lrn{x_{k+1} - x_0}_2^2
    \end{alignat*}
    Notice that 
    \begin{alignat*}{1}
        \linb{x_k - x_0, \beta(x_k)} 
        \leq& 2\delta L_{\beta,2}'\lrn{x_k - x_0}_2^2 + \frac{\delta}{L_{\beta,2}'} \lrn{\beta(x_0)}_2^2
    \end{alignat*}
    and $\lrn{x_{k+1} - x_0}_2^2 \leq 4\delta^2 {L_{\beta,2}'}^2 \lrn{x_{k+1} - x_0}_2^2 + 4\delta^2 \lrn{\beta(x_0)}_2^2 + 4\delta L_{\xi,2}^2 \leq \delta L_{\beta,2}' \lrn{x_{k+1} - x_0}_2^2 + 4\delta^2 \lrn{\beta(x_0)}_2^2 + 4\delta L_{\xi,2}^2 \leq$ by assumption that $\delta \leq 1/(16 L_{\beta,2}')$. Put together,
    \begin{alignat*}{1}
        \lrn{x_{k+1} - x_0}_2^2 \leq \lrp{1 + 2\delta L_{\beta,2}'}\lrn{x_{k} - x_0}_2^2 + \delta^2 \lrn{\beta(x_0)}_2^2 +  2\delta L_{\xi,2}^2 + 2\linb{x_k - x_0, \sqrt{\delta} \xi_k(x_k)}
    \end{alignat*}
    squaring both sides, applying Young's inequality, then squaring both sides again and applying Young's inequality again,
    \begin{alignat*}{1}
        \lrn{x_{k+1} - x_0}_2^8 
        \leq& \lrp{1 + 16 \delta L_{\beta,2}' + \frac{1}{K}}\lrn{x_{k} - x_0}_2^8 + {2^{8}K^3\delta^8 \lrn{\beta(x_0)}_2^8} + 2^8 K^3 \delta^4 L_{\xi,2}^8 + \circled{*}
    \end{alignat*}
    where $\circled{*}$ has 0-mean conditioned on $x_k$. Therefore,
    \begin{alignat*}{1}
        \E{\lrn{x_{k+1} - x_0}_2^8} \leq \lrp{1+ 16 \delta L_{\beta,2}' + \frac{1}{K}} \E{\lrn{x_{k} - x_0}_2^8} + {2^{8}K^3\delta^8 \lrn{\beta(x_0)}_2^8} + 2^8 K^3 \delta^4 L_{\xi,2}^8
    \end{alignat*}
    Applying the above recursively, we get
    \begin{alignat*}{1}
        \E{\lrn{x_{k} - x_0}_2^8} 
        \leq& {2^{10}K^4\delta^8 \lrn{\beta(x_0)}_2^8} + 2^{10} K^4 \delta^4 L_{\xi,2}^8
    \end{alignat*}
\end{proof}

\begin{lemma}\label{l:euclidean_walk_l2_distance_bound}
    Let 
    \begin{alignat*}{1}
        x_{k+1} =  x_k + \beta(x_k) + \sqrt{\delta} \xi_k(x_k)
    \end{alignat*}
    Assume Assumption \ref{ass:euclidean_dissipativity}, Assumption \ref{ass:beta_lipschitz_euclidean}, and $\delta \leq m_2/(16{L_{\beta,2}'}^2)$. Then for all $K$,
    \begin{alignat*}{1}
        & \E{\lrn{x_{K} - x^*_2}_2^2} \leq \frac{L_{\beta,2}'\R_2^2}{m} +\frac{2{L_{\xi,2}}^2}{m}\\
        & \Pr{\max_{k\leq K}  \lrn{x_{K} - x^*_2}_2^2 \geq t^2} \leq \exp\lrp{\frac{m_2}{64L_{\xi,2}^2} \lrp{2\lrn{x_{0} - x^*_2}_2^2 + \frac{L_{\beta,2}' \R_2^2}{m_2} - t^2}}
    \end{alignat*} 
\end{lemma}
\begin{proof}
    \begin{alignat*}{1}
        \lrn{x_{k+1} - x^*_2}_2^2
        =& \lrn{x_{k} - x^*_2}_2^2 + 2\linb{x_k - x^*_2, \beta(x_k) + \sqrt{\delta} \xi_k(x_k)} + \lrn{x_{k+1} - x^*_2}_2^2
    \end{alignat*}
    Notice that 
    \begin{alignat*}{1}
        \linb{x_k - x^*_2, \beta(x_k)} 
        \leq& \ind{\lrn{x_k - x^*_2}_2 \geq \R_2} \lrp{-\delta m_2 }\lrn{x_k - x^*_2}_2 + \ind{\lrn{x_k - x^*_2}_2 \leq \R_2} \lrp{\delta L_{\beta,2}' }\lrn{x_k - x^*_2}_2\\
        \leq& -\delta m_2\lrn{x_k - x^*_2}_2 + 2\delta L_{\beta,2}'\R_2^2
    \end{alignat*}
    and $\lrn{x_{k+1} - x^*_2}_2^2 \leq 2\delta^2 {L_{\beta,2}'}^2 \lrn{x_{k+1} - x^*_2}_2^2 + 2\delta L_{\xi,2}^2 \leq \delta m_2/2 \lrn{x_{k+1} - x^*_2}_2^2 + 2\delta L_{\xi,2}^2 \leq$ by assumption that $\delta \leq m_2/(16 {L_{\beta,2}'}^2)$. Put together,
    \begin{alignat*}{1}
        \lrn{x_{k+1} - x^*_2}_2^2 \leq \lrp{1- \delta m_2/2}\lrn{x_{k} - x^*_2}_2^2 + \delta L_{\beta,2}' \R_2^2 +  2\delta L_{\xi,2}^2 + 2\linb{x_k - x^*_2, \sqrt{\delta} \xi_k(x_k)}
    \end{alignat*}
    Applying the above recursively gives us our first result.

    We now apply Lemma \ref{l:doob_maximal} with
    \begin{alignat*}{1}
        & q_k = \frac{m_2}{32L_{\xi,2}^2} \lrn{x_{k+1} - x^*_2}_2^2 \qquad \eta_k = \frac{m_2}{32L_{\xi,2}^2} \cdot 2\linb{x_k - x^*_2, \sqrt{\delta} \xi_k(x_k)}\\
        & a = - \delta m_2/2 \qquad b =  \frac{m_2}{32L_{\xi,2}^2}\lrp{\delta L_{\beta,2}' \R_2^2 +  2\delta L_{\xi,2}^2} \qquad c = \frac{\delta m_2}{4} \qquad d=0
    \end{alignat*}
    We verify that $a+c<0$, so that
    \begin{alignat*}{1}
        \Pr{\max_{k\leq K} q_{k} \geq t^2} \leq \exp\lrp{\exp\lrp{K (a+c)}q_0 + \frac{b+d}{\lrabs{a+c}} - \frac{t^2}{2}}
    \end{alignat*}
    Plugging in the terms,
    \begin{alignat*}{1}
        & \Pr{\max_{k\leq K}  \frac{m_2}{32L_{\xi,2}^2} \lrn{x_{K} - x^*_2}_2^2 \geq t^2} \leq \exp\lrp{\frac{m_2}{32L_{\xi,2}^2} \lrn{x_{0} - x^*_2}_2^2 + \frac{ L_{\beta,2}' \R_2^2}{32{L_{\xi,2}'}^2} - t^2/2}\\
        \Leftrightarrow \qquad
        & \Pr{\max_{k\leq K}  \lrn{x_{K} - x^*_2}_2^2 \geq t^2} \leq \exp\lrp{\frac{m_2}{64L_{\xi,2}^2} \lrp{2\lrn{x_{0} - x^*_2}_2^2 + \frac{L_{\beta,2}' \R_2^2}{m_2} - t^2}}
    \end{alignat*}
\end{proof}

\end{document}